\documentclass[12pt]{amsart}
\usepackage{amsfonts,amssymb, amscd, latexsym, graphicx, psfrag, color,float}
\usepackage{bbm}
\usepackage{ifthen}
\usepackage[all]{xy}

\usepackage{tikz-cd}

\usepackage{comment}
\usepackage[dvipsnames]{xcolor}

\usepackage{booktabs,enumerate}
\usepackage{multirow}
\usepackage{mathrsfs}

\usepackage{subcaption}
\usepackage{mathtools}

\usepackage{dsfont}

\usepackage[margin=1in,marginparwidth=0.8in, marginparsep=0.1in]{geometry}

\usepackage[bookmarks=true, bookmarksopen=true,%
bookmarksdepth=3,bookmarksopenlevel=2,%
colorlinks=true,%
linkcolor=blue,%
citecolor=blue,%
filecolor=blue,%
menucolor=blue,%
urlcolor=blue]{hyperref}

\DeclareMathOperator{\Loc}{Loc}

\DeclareMathOperator{\fib}{fib}

\newcommand*{\hrlen}{10}
\newcommand*{\hramp}{3}

\usepackage{tikz, tikz-3dplot}
\usetikzlibrary{matrix,shapes,arrows,arrows.meta,calc,topaths,intersections,hobby,positioning,decorations.pathreplacing,decorations.pathmorphing,fit,patterns}
\usetikzlibrary{patterns,patterns.meta}

\tikzset{
asdstyle/.style={blue,thick},
righthairs/.style={postaction={decorate,draw,decoration={border,amplitude=\hramp,segment length=\hrlen,angle=-90,pre=moveto,pre length=\hrlen/2}}},
lefthairs/.style={postaction={decorate,draw,decoration={border,amplitude=\hramp,segment length=\hrlen,angle=90,pre=moveto,pre length=\hrlen/2}}},
righthairsnogap/.style={postaction={decorate,draw,decoration={border,amplitude=\hramp,segment length=\hrlen,angle=-90}}},
lefthairsnogap/.style={postaction={decorate,draw,decoration={border,amplitude=\hramp,segment length=\hrlen,angle=90}}},
graphstyle/.style={thick},
arrowstyle/.style={thick,decorate,decoration={snake,amplitude=1.7,segment length=10pt,post length=.5mm,pre length=0}},
genmapstyle/.style={thick,-stealth'},
arrhdstyle/.style={thick},
exceptarcstyle/.style={red, ultra thick},
dualquiverstyle/.style={thick,->},
|/.tip={Bar[width=.8ex]}
}

\newtheorem{dummy}{dummy}[section]
\newtheorem{lemma}[dummy]{Lemma}
\newtheorem{theorem}[dummy]{Theorem}

\newtheorem{conjecture}[dummy]{Conjecture}
\newtheorem{corollary}[dummy]{Corollary}
\newtheorem{proposition}[dummy]{Proposition}
\theoremstyle{definition}
\newtheorem{definition}[dummy]{Definition}

\newtheorem{example}[dummy]{Example}
\newtheorem{remark}[dummy]{Remark}

\newcommand{\bA}{\mathbb{A}}
\newcommand{\bC}{\mathbb{C}}

\newcommand{\bN}{\mathbb{N}}
\newcommand{\bP}{\mathbb{P}}

\newcommand{\bR}{\mathbb{R}}
\newcommand{\bZ}{\mathbb{Z}}

\newcommand{\bbL}{\mathbb{L}}

\newcommand{\cA}{\mathcal{A}}
\newcommand{\cB}{\mathcal{B}}

\newcommand{\cD}{\mathcal{D}}
\newcommand{\cE}{\mathcal{E}}
\newcommand{\cF}{\mathcal{F}}
\newcommand{\cG}{\mathcal{G}}
\newcommand{\cH}{\mathcal{H}}

\newcommand{\cK}{\mathcal{K}}
\newcommand{\cL}{\mathcal{L}}
\newcommand{\cM}{\mathscr{M}}
\newcommand{\cN}{\mathcal{N}}
\newcommand{\cP}{\mathcal{P}}

\newcommand{\cO}{\mathcal{O}}
\newcommand{\cS}{\mathcal{S}}
\newcommand{\cT}{\mathcal{T}}

\newcommand{\cX}{\mathcal {X}}

\newcommand{\calM}{\mathcal{M}}

\newcommand{\fR}{\mathfrak{R}}

\newcommand{\dsk}{\mathds{k}}

\newcommand{\moduli}[1]{\calM \left(#1 \right)}

\newcommand{\op}{\operatorname}

\newcommand{\N}{N}

\newcommand{\T}{\mathbf{T}}

\newcommand{\conv}{{\op{conv}}}

\newcommand{\Ext}{\mathrm{Ext}}

\newcommand{\frakt}{\mathfrak{t}}

\newcommand{\Hom}{\mathrm{Hom}}

\renewcommand{\log}{{\op{log}}}

\newcommand{\Perf}{\mathcal{P}\mathrm{erf}}

\DeclareMathOperator{\Sh}{Sh}

\newcommand{\Spec}{\mathrm{Spec}\,}
\renewcommand{\SS}{\mathit{SS}}

\newcommand{\Coh}{\mathrm{Coh}}

\newcommand{\Gm}{\mathbb{G}_{\mathrm{m}}}
\newcommand{\Ga}{\mathbb{G}_{\mathrm{a}}}

\definecolor{sayMH}{RGB}{39, 117, 227}
\newcommand{\sayMH}[1]{{\color{sayMH}#1}}

\DeclareMathOperator{\End}{End}
\DeclareMathOperator{\Aut}{Aut}

 \numberwithin{equation}{section}
\numberwithin{figure}{section}

\newcommand{\red}[1]{{\color{red}#1}}
\newcommand{\blue}[1]{{\color{blue}#1}}

\DeclareMathOperator{\colim}{colim}

\usepackage{tikz}

\newcommand{\wtd}{\widetilde}

\tikzset{
  top aligned/.style={
    baseline={([yshift=-\ht\strutbox]current bounding box.north)}
  }
}

\title[Counting Curves]{Mirror Counts of Spectral Curves}

\author[Tom Graber, Mingyuan Hu and Eric Zaslow]{Tom Graber$^1$, Mingyuan Hu$^2$
and Eric Zaslow$^3$\\
\\
{\tiny ${}^1$  Department of Mathematics, California Institute of Technology}\\
{\tiny ${}^{2}$  Centre for Quantum Mathematcs, University of Southern Denmark}\\
{\tiny ${}^{3}$  Department of Mathematics, Northwestern University}}

\begin{document}

\begin{abstract}

Given a convex lattice polygon $\Delta\subset \bR^2,$
let $N_\Delta$ be the count of rational, nodal curves
in the linear system of the ample line
bundle $L_\Delta$ on the toric variety $\bP_\Delta$
defined by $\Delta$, having fixed intersection with the
toric boundary.
We show that the relative Jacobian of the linear system defines an
integrable system that has a mirror dual in the language
of constructible sheaves on a two-torus microsupported
on a Legendrian link.
In this setting, we define the dual
counting problem using an analogue of rulings
of Legendrian links in three-space, then prove equivalence
with $N_\Delta$.
We perform calculations of $N_\Delta$
in several examples using constructible methods, 
tropical curve counting, and localization in logarithmic
Gromov-Witten theory to 
demonstrate the equality of the different approaches.

Moreover, the rulings give rise to a stratification of the moduli of constructible sheaves.  We conjecture that this ruling decomposition recovers the refined tropical invariants of Block-G\"ottsche, and hence encodes higher-genus logarithmic Gromov--Witten invariants, by a theorem of Bousseau.




\end{abstract}

\maketitle


\setcounter{tocdepth}{1}
\tableofcontents

\section{Introduction and Summary}
\label{sec:intro}




Let $\bP_\Delta$ be a toric surface with ample line bundle $L_\Delta$ defined by a convex lattice polygon $\Delta \subset \bR^2$.  In this setting, one can define a host of 
enumerative problems.  Kontsevich \cite{Kontsevich1995}
counted nodal curves incident to generic points.  Mikhalkin \cite{Mikhalkin2005Enumerative} gave an interpretation and equivalent count in
tropical geometry.  Nishinou-Siebert \cite{NS} and Mandel-Ruddat \cite{MR}
consider incidence conditions with the toric
boundary from the tropical and logarithmic
Gromov-Witten \cite{GS} perspective.
Welschinger \cite{Welschinger:2005:RealSymplecticInvariants}
defined
invariants of real curves with incidence
conditions given by real points as well as
possibly conjugate
pairs, and these were computed tropically
by Mikhalkin \cite{Mikhalkin2005Enumerative}
and Shustin \cite{Shustin:2006:TropicalWelschinger},
respectively.
G\"ottsche-Shende defined refinements
of the Severi degrees, i.e.~of the counts of curves
of fixed geometric genus in the linear system 
\cite{GottscheShende:2014:RefinedCurveCounting}.
Block-G\"ottsche \cite{BlockGottsche:2016:RefinedCurveCounting}
interpreted these refined
invariants tropically, and they
were proved to be invariant
by Itenberg-Mikhalkin \cite{ItenbergMikhalkin:2013:BlockGottsche}.
Bousseau \cite[Theorems 1 and 6]{Bousseau:2019:TropicalRefined}
proved that counting tropical curves with refined
multiplicities recovers the Block-G\"ottsche invariants.
Finally, Blomme \cite{Blomme} found a
computationally effective recursion relation satisfied
by these refined counts.

We recall that when the ambient surface is not
a toric surface but rather K3, the count
of nodal curves was performed in \cite{YZ} by interpreting the
relative Jacobian as a Beauville integrable system
describing a moduli of branes, then exploiting string duality
to count branes on the heterotic side.  We know of no
nonperturbative string duality in the present setting, but we do have mirror symmetry, and can use
it to find a dual integrable
system, namely the Goncharov-Kenyon \cite{GK} system, as explained in \cite{TWZ}.  From this perspective,
we are counting spectral curves.

We use this mirror perspective to define combinatorial analogues of some of the above enumerative
questions, using the language of constructible sheaves.
We then prove equivalence of combinatorial and
algebraic definitions,
in the unrefined case of genus-zero curves.  We also add localization
in logarithmic Gromov-Witten theory to the list of
calculational methods, and perform detailed
computations in several examples.
We conjecture an equivalence in higher genus.

Here is a sketch.
A toric surface $\bP_\Delta$
with ample line bundle $L_\Delta$
is determined by a lattice polygon $\Delta\subset \bR^2.$
The genus of a generic curve $C$ in the linear system $|L_\Delta|$ is
equal to the number of interior lattice points, while
the number of exterior lattice points equals the
number of points of $C$ along the toric divisor at
infinity.  Fixing this intersection to be a finite set
of points, the
reduced (non-compact) linear system is a $g$-dimensional family of
genus-$g$ curves, and we want to count the number that are
rational and nodal.  (Later we will consider
some variants, but the bulk of this paper focuses on this,
the simplest case.)  To make this concrete,
consider the following simple example.

\begin{example}
\label{ex:p2ex}
    When $\Delta = \mathrm{conv}((0,0),(3,0),(0,3)) =$
        \begin{tikzpicture}[scale=.2]
  \fill[blue!40] (0,0) -- (3,0) -- (3,3) -- cycle;

  \draw[very thick] (0,0) -- (3,0) -- (3,3) -- cycle;

  \draw[gray, thick] (-0.3,0) -- (3.3,0);
  \draw[gray, thick] (-0.3,1) -- (3.3,1);
  \draw[gray, thick] (-0.3,2) -- (3.3,2);
  \draw[gray, thick] (-0.3,3) -- (3.3,3);

  \draw[gray, thick] (0,-0.3) -- (0,3.3);
  \draw[gray, thick] (1,-0.3) -- (1,3.3);
  \draw[gray, thick] (2,-0.3) -- (2,3.3);
  \draw[gray, thick] (3,-0.3) -- (3,3.3);

  \foreach \x in {0,1,2,3} {
    \foreach \y in {0,1,2,3} {
      \fill (\x,\y) circle (3pt);
    }
  }

\end{tikzpicture},
    $\bP_\Delta = \bP^2,$ $L_\Delta = \cO_{\bP^2}(3),$ $n(\Delta)=9$,
    so we are counting geometric genus-zero
    plane curves of degree $3$ meeting the coordinate
    lines precisely at precisely $9$ points:  $8$ of which are chosen, the ninth being determined by having to solve a cubic equation.  This is the same as the
    number of singular cubic curves defined by homogeneous
    polynomials of degree three in $X, Y, Z,$ with all terms but $aXYZ$ fixed.  The answer is the degree in $a$ of the discriminant:  $9.$
\end{example}

Contrary to the above example,
it is typically too difficult to perform
exact computations in commutative algebra
to find the number and type of singularities of
a family of curves, so one employs alternate but equivalent
definitions of $N_\Delta.$  
The following are reviewed in Section \ref{sec:defs}.

\begin{itemize}
    \item $N^{\mathrm{Sev}}_\Delta$ is the number of points in (i.e., the degree of) the Severi variety of geometric genus-zero curves in the restricted linear system --- see, e.g., \cite{CaporasoHarris:1998:CountingPlaneCurves,Tyomkin24}.
    \item $N^{\mathrm{GW}}_\Delta$ counts genus-zero curves with fixed intersection with the toric divisor via logarithmic Gromov-Witten theory --- see, e.g., \cite[Section 2.6] {Bousseau:2019:TropicalRefined}.
    \item $N^{\mathrm{trop}}_\Delta$ is a corresponding count via the enumeration of tropical curves as in \cite{Mikhalkin2005Enumerative} and \cite{MR}.
    \item $N^{\mathrm{Coh}}_\Delta$ is the Euler characteristic of the relative Jacobian of the restricted linear system, equivalent because the Euler characteristic of the Jacobian of a curve is only nonzero if it has geometric genus zero --- see, e.g. \cite{Beauville1999} and Appendix \ref{sec:beauville}. Line bundles on a curve pushforward to coherent sheaves on $\bP_\Delta$, so this is a family of coherent sheaves, hence the name.
    \item $N^{\mathrm{Sh}}_\Delta$ is the Euler characteristic of a mirror dual family of constructible sheaves.
    \item Given $\Delta,$ we define the combinatorial notion of a \emph{rational ruling}.  $N^{\mathrm{rul}}_\Delta$ is the number of rational rulings.
\end{itemize}

Our main result is that the ruling and constructible
sheaf definitions are
equivalent to all the previous ones:

\vskip0.1in
\noindent {\bf Main Theorem.} (See Theorem
\ref{thm:mainthm}.)
The first two equalities in the following are true:
$$N^{\mathrm{rul}}_\Delta = N_\Delta^{\mathrm{Sh}} 
=N_\Delta^{\mathrm{trop}}=N_\Delta^{\mathrm{GW}} =N_\Delta^{\mathrm{Coh}}=N^{\mathrm{Sev}}_\Delta.$$
(The other equalities are previously known -- see
the discussion around Theorem \ref{thm:mainthm}.)

As discussed earlier, these enumerative problems have a history in the tropical and algebraic settings.
We believe our constructible approach is novel.  
Another novelty here is that we perform several calculations in logarithmic Gromov-Witten theory using the localization theorem in this setting, which is work in progress of the first-named author.
We give considerable background in Section \ref{sec:notation},
but here is the story in brief.

 The relative Jacobian of the restricted linear system is related by the spectral transform of Goncharov-Kenyon \cite{GK} to the integrable system of a dimer model of a bipartite graph $\Gamma$ a two-torus --- or equivalently \cite{TWZ} to a mirror family of constructible sheaves with singular support defined by zigzag curves of $\Gamma$.   Using techniques adapted from the study of Legendrian knots as in \cite{STZ}, we define the notion of a \emph{ruling} for such sheaves and associated \emph{ruling surface}.  We show that a constructible sheaf in the family gives rise to a ruling.  The ruling surface plays the role of the underlying curve in the mirror linear system.  The moduli space $\cM_\triangle^\mathrm{Sh}$
 then decomposes into rulings --- each one $\rho$ with
 its genus $g(\rho)$ and number of returns $r(\rho)$ --- in a way analogous to the stratification of a linear system of curves by genus:

\vskip0.1in
 \noindent{\bf Theorem.}  (See Theorem \ref{thm:ruling_decomposition}.)
\[
\cM_\triangle^\mathrm{Sh} \cong \coprod_{\rho \in \mathfrak R} (\bC^*)^{2g(\rho)} \times \bC^{r(\rho)}.
\]
 This decomposition underlies many of the results
 that follow.  Then $N^{\mathrm{Sh}}_\Delta$ is the count of rational (genus-zero)
rulings $N^\mathrm{rul}_\Delta$ via a mirror version of the equality
$N^{\mathrm{Coh}}_\Delta = N^{\mathrm{Sev}}_\Delta.$ 
This is the content of Theorem \ref{thm:rulings}, subsumed in the main theorem.

As mentioned above, each ruling
has a corresponding topological surface, analogous to the
conjugate Lagrangians of \cite{STWZ}.  The rational rulings are the genus-zero case,
but we can also count rulings of genus $g$.  Experimental evidence suggests that these
counts recover the Block-G\"ottsche refinement \cite{BlockGottsche:2016:RefinedCurveCounting} obtained by counting tropical curves with $q$-multiplicity.
We conjecture the following.

\vskip0.1in
\noindent {\bf Conjecture.} (See Conjecture \ref{conj:Ruling_refined}.)
Let $N_{\triangle,g}$ be the number of genus-$g$ rulings and define the ruling polynomial $R_\triangle(z) =
\sum_g N_{\triangle,g}z^{2g}$.  Then setting $z = q^{\frac{1}{2}}-q^{-\frac{1}{2}},$ we have
\[R_{\triangle}(z) = N_{\triangle}(q),\]
where $N_{\triangle}(q)$ are the refined tropical/log GW counts of \cite{BlockGottsche:2016:RefinedCurveCounting}/\cite{Bousseau:2019:TropicalRefined}.
We also make a conjecture about the
leading behavior of the $N_{\triangle,g}$
in the co-genus, somewhat analogous
to the conifold gap conjecture in topological
strings --- see
Conjecture \ref{conj:gap}.

Further, the decomposition of the mirror moduli space $\cM^{\mathrm{Sh}}_\triangle$ of constructible sheaves by genus of the ruling surface allows us to
express the Hirzebruch genus $\chi_{-y}$ in terms of $R_\triangle(z)$:
we have $$\chi_{-q}(\cM^{\mathrm{Sh}}_\triangle) = q^{g_\triangle} R_\triangle(z),$$
where $g_\triangle$ is the number of interior
lattice points of $\triangle$
--- see Proposition \ref{prop:chi=R}. Therefore, the conjecture above can be viewed as a variant of~\cite[Conjecture 1.1]{Nicaise-Payne-Schroeter}. 


\begin{remark}
    We think of the stratification
    of rulings by genus as being
    dual to the Severi stratification of
    curves in the restricted linear system by
    genus, though we have not proven this.  In fact, one is tempted to think of
    the holomorphic curves and the mirror conjugate Lagrangians as being ``the same,'' related by hyper-K\"ahler rotation on the complement of the toric anticanonical divisor --- see
    Examples \ref{ex:coameba} and \ref{ex:2c}.
\end{remark}

\begin{remark}
    The simplicity of the description of the coefficients $N_{\triangle,g}$ suggests
    to us that the integrable system may
    be the ``right'' perspective on these enumerative problems.
\end{remark}

The story behind the various set-ups is told in Section \ref{sec:notation}, where
we focus on the generally less familiar relation to constructible sheaves.  Section \ref{sec:defs} briefly reviews the different definitions and sketches their equivalence.  Section \ref{sec:logGW} discusses
the log GW problem, focusing on localization.  Section \ref{sec:tropical} reviews the
tropical approach and re-proves
invariance with a $q$-refined Pl\"ucker relation. 
Section \ref{sec:constructible} develops the
combinatorial approach by constructible sheaves,
defines the notions of ruling and ruling surface,
proves our main theorem, and states the conjecture.
In Section \ref{sec:examples}, we calculate
numbers by three methods for several examples
in some detail. 
The appendices contains more calcluations.

\subsection*{Acknowledgements}

We are very grateful to Helge Ruddat, who was involved in this project in its early stages,
contributed key ideas, 
and explained to us many aspects
of the tropical and algebraic approaches.
We would also like to thank many others for sharing their thoughts and perspectives on this problem, in particular
Jim Bryan,
Roger Casals, Bohan Fang, Paul Hacking, Wenyuan Li, Davesh Maulik, Justin Sawon, Vivek Shende, David Treumann, and Harold Williams.
{\bf Generative AI} has been used at several points, all noted in the text, to assist with computer programming, with computer searches and in making numerical conjectures.  The proofs are our own.
The work of EZ has been supported in part by NSF grants DMS-2104087 and DMS-2505652. The work of MH has been supported by the Villum Fonden grant Villum Investigator 37814.

\section{Set-up, Background, Narrative}
\label{sec:notation}

Let us describe the geometric set-up and recall some
background material to establish notation and construct the narrative of the paper.

\subsection{Toric Geometry}
We recall the usual toric package:  $N$ is a lattice, $M = M^\vee$ the dual lattice, $N_\bR := N\otimes_\bZ \bR$ and
$M_\bR := M\otimes_\bZ \bR$ are the corresponding real
vector spaces, $\Delta\subset M_\bR$ is a lattice
polytope and $\Sigma = \Sigma_\Delta \subset N_\bR$
the corresponding normal fan, and
$\bP := \bP_\Sigma$ the associated toric.

We sometimes discuss toric
Deligne-Mumford stacks and will use
the following notion of stacky fan \cite{BorisovChenSmith2005}.
Let $N$ a finitely-generated abelian group,
not necessarily free, and $\Sigma \subset N_\bR$ be a fan, where
$N_\bR := N\otimes_\bZ \bR$.  
We assume the one-dimensional cones
$\{\rho_1,...,\rho_r\}$ span $N_\bR$.
Let $\{b_1,...,b_r\}$
be a set of points in $N$ and suppose given
a group homomorphism $\beta: \bZ^r \to N_\bR$ 
such that the image $\widetilde{b}_i$ of the $i$th
direction vector spans $\rho_i$.
The tuple ${\bf \Sigma} := (N,\Sigma,\beta)$ is a stacky fan.

Note that given a free abelian group $N$ and a 
(standard) fan $\Sigma$, we can define a stacky fan ${\bf \Sigma}$ after choosing a set of not necesssarily primitive vectors $b_i$ spanning $\rho_i$, respectively.
For instance, when $N = \bZ^2$
and we are given a lattice polytope $\triangle = 
\mathrm{conv}\{p_1,...,p_r\},$
in $\bR^2$, with $p_1,...,p_r$
oriented counterclockwise around $\triangle,$
the normal fan $\Sigma(\triangle)$ has rays
defined by $\rho_i = \mathrm{Span}J(p_{i+1}-p_i),$
where $J$ is the complex structure on
$\bR^2\cong \bC.$
We can then define a stacky fan $\bf\Sigma (\triangle)$
by taking $b_i = J(p_{i+1}-p_i)$.
Thus any non-primitve edge gives rise
to stackiness.

\subsection{Coherent-Constructible Correspondence}
\label{subsec:CCC}
The algebraic geometry
of toric varieties can be related
to the combinatorial
geometry on a real torus $T:= M_\bR/M$
through the category
of constructible
sheaves on a real torus.
The geometry of the fan defines
a singular support condition for
constructible sheaves 
Let ${\bf \Sigma} = (N,\Sigma,\beta)$
be a stacky fan in the special case
where $N$ is free, as when ${\bf \Sigma} = 
{\bf \Sigma}(\triangle)$ described above.

The theorem of \cite{Ku}, following
\cite{Bondal2006,FLTZ}, gives
an equivalence $Coh(\bP_\triangle)\cong
Sh_{\underline{\Lambda}_{{\bf \Sigma}(\triangle)}}(T^2)$,
where $\underline{\Lambda}_{{\bf \Sigma}(\triangle)} \subset T^*T^2$ is defined as follows.
For each cone $\sigma = \mathrm{Span}_\bR(\{b_1,...,b_k) \in \Sigma$,
define $N_\sigma := \mathrm{Span}_\bZ(\{b_1,...,b_k) $
and $M_\sigma = \Hom(N_\sigma,\bZ)\subset M_\bR.$
Then
\begin{equation}
\label{eq;toricstack}
    \underline{\Lambda}_{{\bf \Sigma}(\triangle)} :=
\sum_{\sigma\subset \Sigma}(\sigma^\perp + M_\sigma)
\times (-\sigma) \subset (M_\bR/M) \times N_\bR 
\cong T^*T.
\end{equation}
When we want to refer to the associated Legendrian in the cocircle bundle $T^\infty T = \partial^\infty T^*T,$ we remove the underline $\underline{\Lambda}\rightsquigarrow \Lambda.$
We denote by ${\bf \Sigma}^\circ$
the fan obtained by keeping only $0$- and $1$-
cones (which deletes the toric fixed points)
and $\Lambda_{\bf \Sigma}(\triangle)^\circ$
for the associated Legendrian.

\subsection{Bipartite Graphs and The Dimer Model}

We now assume $N$ is of rank two,
so $T \cong T^2$ is
the standard two-torus, which we equip
with the standard orientation.  Let $\Gamma\subset T$ be a bipartite graph with
vertex set $V = B\sqcup W$ and edge
set $E.$
Goncharov-Kenyon \cite{GK}
define from this data a dimer statistical
system, as well as an algebraic
torus chart $T_\Gamma$
in a cluster
Poisson integrable system.
The dimer model counts dimer covers $c$,
i.e.~subsets $c\subset E(\Gamma)$ of edges
meeting each vertex exactly once, each with a given
weight, defined as follows.
Choose a set of edge weights ${\bf w}:E\to \bC^*$,
then define the total
weight ${\bf w}(c) = \prod_{e\in c}{\bf w}(e)$.
Since edges are canonically oriented from
black to white vertices,
we can think of $c$ as a one-chain, then
the dimer condition says $\partial(c)=
\sum_{w\in W}w-\sum_{b\in B}b.$
Let $c_0$ be a reference dimer.  Then
$c-c_0$ is a one-cycle.
After
a universal shift of weight by $1/w(c_0)$, we can weight
$c$ by ${\bf w}(c)/{\bf w}(c_0)$.  
The partition function for this dimer system
is $\cP = \sum_c {\bf w}(c).$
By interpreting
edge weights as parallel transports for a 
local system on $\Gamma$, the weighting
is by monodromy of the one-cycle $c-c_0.$
Note the gauge group $(\bC^*)^V(\Gamma)$
acts trivially on the weights of cycles,
and the space of dimer models is then
the space of edge weights modulo gauge
is $(\bC^*)^E/(\bC^*)^V\cong
\mathrm{Loc}_1(\Gamma)$, the algebraic torus $T_\Gamma$
of rank-one
local systems on $\Gamma.$
The space of edge weights comes equipped
with a Poisson structure defined by an
 intersection form, not on $T^2$
but on a conjugate surface that we now describe.

\subsection{Zigzags and the conjugate surface}

A zigzag path $\gamma$ is a smoothing in $T^2$ of an
oriented loop $\eta$ on $\Gamma\subset T^2$ which minimally turns left at black vertices and right at white ones,
such that $\gamma$ intersects $\eta$ once, transversely,
in each of its edges.
Each undirected edge of $\Gamma$ spawns two oriented zigzags, and thus the sum of all zigzag paths is zero both as a chain and homology class.  Thus the collection of $[\gamma]\in H_1(T)\cong \bZ^2$ defines a polygon $\Delta_\Gamma$ defined by starting at the origin
and traversing the vectors $[\gamma]$ in counterclockwise order.  

Then the complement $T^2\setminus \bigcup \gamma$ of the union of zigzags is a disjoint collection of regions:  ``null'' regions, which contain no vertex, as well as one region for each vertex of $\Gamma$.  The regions surrounding vertices have consistent orientation
of their boundaries --- opposite for the regions around white (resp.)~black vertices ---
while the orientations of segments around null regions alternate.
The conjugate surface
$S_\Gamma$ is the oriented topological
surface with boundary obtained by
gluing these regions along their corners with a twist ---
see Figure \ref{fig:zz}.

    
\begin{figure}[ht]
\begin{tikzpicture}[rotate=180]

\pgfmathsetmacro{\eps}{0.15}
\pgfmathsetmacro{\a}{0.5}


\newcommand{\Ax}[1]{#1}
\newcommand{\Ay}[1]{#1-sin(pi*#1/2 r)/2}

\newcommand{\Bx}[1]{#1-sin(pi*#1/2 r)/2}
\newcommand{\By}[1]{#1}


\fill[red!40]
plot[domain=0:1, samples=120, variable=\t]
({\Ax{\t}}, {\Ay{\t}})
-- (1,1)
-- plot[domain=1:0, samples=120, variable=\t]
({\Bx{\t}}, {\By{\t}})
-- cycle;


\fill[blue!40]

plot[domain=-1:0, samples=120, variable=\t]
({\Ax{\t}}, {\Ay{\t}})

-- (-1,-1) 

-- plot[domain=-1:0, samples=120, variable=\t]
({\Bx{\t}}, {\By{\t}})

-- (-1,-1) 

-- cycle;


\draw[very thick, black, ->]
plot[domain=1:-1, samples=120, variable=\t]
({\Ax{\t}}, {\Ay{\t}});

\draw[very thick, black, ->]
plot[domain=-1:1, samples=120, variable=\t]
({\Bx{\t}}, {\By{\t}});


\draw[thick]
(-\a,-\a) -- (\a,\a);


\filldraw[fill=black, draw=black, line width=0.6pt]
(-\a,-\a) circle (3pt);

\filldraw[fill=white, draw=black, line width=0.6pt]
(\a,\a) circle (3pt);

\end{tikzpicture}
\qquad\qquad
  \begin{tikzpicture}[scale=.9]
  \pgfmathsetmacro{\A}{1}
  \pgfmathsetmacro{\B}{1}
  \pgfmathsetmacro{\a}{2*\A-1}
  \pgfmathsetmacro{\b}{2*\B-1}
\foreach \x in {-1,0,...,\a}{  
\draw (\x+1/2,-1/2) -- (\x + 1/2,\b + 1/2);}
\foreach \y in {-1,0,...,\b}{  
\draw (-1/2,\y+1/2) -- (\a + 1/2,\y + 1/2);}
\foreach \x in {0,1,...,\a}{
\foreach \y in {0,1,...,\b}{
\pgfmathparse{mod(\x,2)} \let\xm\pgfmathresult 
\pgfmathparse{mod(\y,2)} \let\ym\pgfmathresult 
\ifthenelse{\equal{\xm}{1.0} \AND \equal{\ym}{1.0}}
{
\draw[draw=blue,fill=blue,opacity=.4]  (\x - 1/2,\y -1/2) rectangle (\x + 1/2,\y + 1/2);
\draw (\x-1/2,\y-1/2) -- (\x+1/2,\y+1/2);\draw (\x-1/2,\y+1/2) -- (\x+1/2,\y-1/2);
\filldraw (\x,\y) circle (3pt);  
\draw[->] (\x-1/2,\y-1/2) -- (\x,\y-1/2);
\draw[->] (\x+1/2,\y-1/2) -- (\x+1/2,\y);
\draw[->] (\x+1/2,\y+1/2) -- (\x,\y+1/2);
\draw[->] (\x-1/2,\y+1/2) -- (\x-1/2,\y);
}{};
\ifthenelse{\equal{\xm}{0.0} \AND \equal{\ym}{0.0}}
{
\draw[draw=red,fill=red,opacity=.4]  (\x - 1/2,\y -1/2) rectangle (\x + 1/2,\y + 1/2);
\draw (\x-1/2,\y-1/2) -- (\x+1/2,\y+1/2);\draw (\x-1/2,\y+1/2) -- (\x+1/2,\y-1/2);
\draw[draw=black,fill=white] (\x,\y) circle (3pt); 
\draw[->] (\x-1/2,\y-1/2) -- (\x-1/2,\y);
\draw[->] (\x-1/2,\y+1/2) -- (\x,\y+1/2);
\draw[->] (\x+1/2,\y+1/2) -- (\x+1/2,\y);
\draw[->] (\x+1/2,\y-1/2) -- (\x,\y-1/2);
}{};

}
\draw[->] (-1/2,1/2)--(-1/2,1);
\draw[->] (3/2,-1/2)--(1,-1/2);
\draw[->] (3/2,-1/2)--(3/2,0);
\draw[->] (1/2,3/2)--(0,3/2);
}
   \end{tikzpicture}
\caption{Zigzag paths and conjugate surface near an edge (left) and for a torus graph $\Gamma$ (right) with 
$\Lambda_\Gamma = \Lambda_{{\bf \Sigma}(\Delta)^\circ}$ where
$\Delta = \square$ and 
$\bP_{\Delta} = \bP^1\times\bP^1$.  }
\label{fig:zz}
\end{figure}
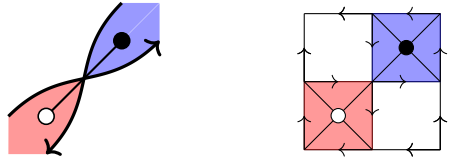
Then $S_\Gamma$ is homeomorphic to a thickening of $\Gamma$
with a twist in each ribbon edge.
The intersection form $\langle\;,\;\rangle$ on $S_\Gamma$ defines the Poisson
structure on $\mathrm{Loc}_1(S_\Gamma)\cong \mathrm{Loc}_1(\Gamma)\cong \Hom(H_1(\Gamma;\bZ);\bC^*) \cong (\bC^*)^{b_1(\Gamma)}$:  let $a,b\in H_1(\Gamma;\bZ)$ and let $z^a,z^b \in \cO(\mathrm{Loc}_1(\Gamma)).$  Then $\{z^a,z^b\}
=\langle a, b\rangle z^{a+b}.$
The boundary curves of $S_\Gamma$ are precisely the zigzags $\gamma$, and these span the kernel of the intersection form, 
so $z^\gamma$ are the Casimirs of the Poisson structure.
The symplectic leaves are thus specified by fixing
the monodromies along zigzags.

\subsection{The Goncharov-Kenyon Integrable System}

Goncharov-Kenyon \cite{GK} 
show that given this Poisson structure,
there is an integrable system structure
on the space of edge weights whose
Hamiltonians are the partition functions
for the dimer models over dimers $c$ such
that $[c-c_0] \in H_1(\Gamma)$ is fixed.
Further, they show that if you
alter $\Gamma\rightsquigarrow \Gamma'$
by a square move and adjust edge weights by
a birational cluster coordinate transformation,
then the dimer models are equivalent on the
overlap of torus charts.  In this way the
charts glue to a cluster Poisson 
integrable system.  See Appendix \ref{sec:beauville} for details
on the complete integrability of this system.

\subsection{Embedding the Dimer Model
into Constructible Sheaves}

Using the standard orientation and complex structure
on $T^2,$ we can rotate the unit tangents along the orientation 
$\begin{tikzpicture}
    \draw[->,thick,blue](0,0)--(.55,0);
    \draw[thick,blue] (.55,0)--(1,0);
\end{tikzpicture}$
of a zigzag $\gamma$ to the right by 90 degrees
$\begin{tikzpicture}
    \draw[thick,blue](0,0)--(1,0);
    \draw[thick,blue] (1/4,0)--(1/4,-.15);
    \draw[thick,blue] (2/4,0)--(2/4,-.15);
    \draw[thick,blue] (3/4,0)--(3/4,-.15);
\end{tikzpicture}$
to define a zigzag Legendrian $\Lambda_\Gamma \subset T^\infty T^2,$ and define
$\Lambda_{\Gamma}:= \bigcup_\gamma \Lambda_\gamma$ as
the union over all zigzags Legendrians.
A main construction in \cite{STWZ} used in \cite{TWZ}
is the definition of a constructible sheaf on $T^2$
with singular support in $\Lambda_{\Gamma}$.
Let $B$ (resp.~$W$) be the union of black (resp.~white) regions, and let $i$ be the inclusion map to $T^2$.  
Let $i_*\bC_B$ (resp.~$i_!\bC_W$) be the
corresponding standard (resp.~costandard) constructible
sheaf.  The singular support of this sheaf includes
$\Lambda_{\Gamma}$, but also contains the
inward (resp.~outward) normal
vectors at every corner.  However, by taking an extension of
these sheaves we can cancel the excess singular support
and construct an object in $\mathrm{Sh}_{\Lambda_\Gamma}(T^2).$
The key observation is that
there is a locally defined extension class at each corner
(a 2d version of the fact from 1d that $\mathrm{Hom}^1(\bC_{x\geq 0},\bC_{x< 0})\cong \bC$).
The corners are in one-to-one correspondence with edges of
$\Gamma,$ so we can reinterpret the collection of edge
weights $w$ as defining a non-trivial extension,
thus defining a sheaf
$F({\bf w})\in \mathrm{Sh}^1_{\Lambda_\Gamma}(T^2)$ fitting into the exact
sequence $i_*\bC_B\to F({\bf w}) \to i_!\bC_W[1].$
Recall there is a microlocal monodromy map $\mathrm{Sh}_\Lambda(T)\to \mathrm{Loc}(\Lambda)$.
Because the ranks of $F({\bf w})$ by one, the image under
this map lands in $\mathrm{Loc}_1(\Lambda_\Gamma),$
with the
microlocal monodromy is defined by the zigzag weights
under the restriction $\mathrm{Loc}_1(\Gamma)\cong
\mathrm{Loc}_1(S_\Gamma)\to \mathrm{Loc}_1(\partial S_\Gamma)
\cong \mathrm{Loc}_1(\Lambda_\Gamma).$
We write $\mathrm{Sh}^1_{\Lambda_\Gamma}(T^2)$ for
sheaves with rank one microlocal monodromy.

This construction via weights ${\bf w}$ thus defines
an open embedding $\mathrm{Loc}_1(\Gamma)\to \mathrm{Sh}^1_{\Lambda_\Gamma}(T^2).$
Further, square moves $\Gamma \rightsquigarrow \Gamma'$
induce \emph{isotopies} $\Lambda_\Gamma \cong \Lambda_{\Gamma'},$
which by \cite{GKS} define equivalences $\mathrm{Sh}_{\Lambda_\Gamma}\cong \mathrm{Sh}_{\Lambda_{\Gamma'}}$.
In this way the charts of the Goncharov-Kenyon
cluster integrable system glue together to the single
space of objects in $\mathrm{Sh}^1_{\Lambda_\Gamma}(T^2)$
(at least up to the usual issue of
codimension-two loci that sit outside
all cluster charts).

\subsection{Spectral Transform, Beauville Integrable System
and Mirror Symmetry}

From the data of a bipartite graph $\Gamma \subset T^2,$
Goncharov and Kenyon also construct \cite{GK} a dual integrable
system, which we now describe.
The edge weights ${\bf w}$, together with $T^2$ monodromies
$x$ and $y$ and some discrete choices, can be assembled into a
Kasteleyn matrix $K(x,y;{\bf w}): \bC^{|B|}\to \bC^{|W|}$
This matrix generically has full rank
so $\det(K(x,y)=0$ defines a curve $C_0\subset (\bC^*)^2$
depending on ${\bf w},$
together with a line bundle $L_0\to C_0$
whose with fiber $\ker K(x,y)$ at the point $(x,y)\in C_0.$

Now recall the zigzag paths define a polygon $\triangle$ and toric variety $\bP_\triangle \supset (\bC^*)^2$, so $C_0$ has a compactification $i:C\hookrightarrow \bP_\triangle$,
and the line bundle $L_0$ compactifies to $L\to C$,  giving a family of coherent
sheaves $i_*L \in Coh(\bP_\triangle)$ 
Further, $C$ lies in the linear system $|L_\triangle|$, and the map $i_*L\mapsto C$
defines a torus fibration from the family of
coherent sheaves to a sub-linear system, with
fibers a Picard variety of line bundles of fixed degree.

\begin{remark}[Comments on stackiness]
\label{rmk:coh-stack}
Suppose $\partial \Delta$ is a union of edges of length $d_i$ and choose a labeling $p_{i,j},$ $j\in \{0,...d_i-1\}$ for each
edge.  Identifying the set $\{0,...,d_i-1\}$ with
the characters of $\bZ/d_i\bZ$ --- write $\chi_k$ for the
character defined by $k$ --- we arrive at a lift of
our sheaves to the toric stack.  
To see this, first write $\pi:\bP_{\bf{\Sigma}(\Delta)}\to\bP_\Delta$ for the morphism from the
toric stack to its coarse moduli space.
All the stackiness of $\bP_{\bf{\Sigma}(\Delta)}$
is along the boundary divisors $D_i$, which have
generic stabilizers
$\bZ/d_i\bZ.$
Then
for each $i_*L\in \mathrm{Coh}(\bP_\Delta)$ there is
a lift $i_*\widetilde{L}\in\mathrm{\Coh}(\bP_{\bf{\Sigma}(\Delta)})$ to the toric stack, defined as follows.
First assign the character $\chi_{p_{i,j}}$ to the
point $p_{i,j}\in C\cap \partial \bP_\Delta.$
The data of $L\to C$ together with a character for $\bZ/d_i\bZ$ at each point of $C$ along the boundary divisor $D_i$ defines
a lift of $L\to C$ to a line bundle $\widetilde{L}\to \widetilde{C}$
over the stacky preimage $\widetilde{C}$
of $C$. 
Then $i_*\widetilde{L}$ is our
desired lift.  Conversely,
given $\widetilde{L}\to \widetilde{C}$,
the pushforward $\pi_*(i_*\widetilde{L})$
defines a line bundle on $C$, with the relevant
character data on the boundary.  
The lift is therfore
a section of a morphism $\pi_*$ of categories.
The morphism $\pi_*$ induces
a morphism of moduli spaces of sheaves
of the form $i_*L$ over curves $C$ --- since it is
a bijection of objects, it is an isomorphism.
If we need to distinguish these spaces, we refer to the
moduli space of sheaves on the stack as $\cM^{\mathrm{Stack}}$.
This discussion shows $\cM^{\mathrm{Coh}}\cong \cM^{\mathrm{Stack}}.$
\end{remark}

In \cite{TWZ} it was shown that, provided $\Gamma$
satisfies a technical condition of \emph{consistency}, this transform is 
an instance of mirror symmetry in the form
of the coherent-constructible correspondence.
Specifically, one uses the orientation
on the torus to convert oriented zigzag paths
$\begin{tikzpicture}
    \draw[->,thick,blue](0,0)--(.55,0);
    \draw[thick,blue] (.55,0)--(1,0);
\end{tikzpicture}$
into co-oriented paths
$\begin{tikzpicture}
    \draw[thick,blue](0,0)--(1,0);
    \draw[thick,blue] (1/4,0)--(1/4,-.15);
    \draw[thick,blue] (2/4,0)--(2/4,-.15);
    \draw[thick,blue] (3/4,0)--(3/4,-.15);
    
\end{tikzpicture}$
or equivalently
Legendrian curves in the
cocircle bundle $T^\infty T^2.$
One result of \cite{TWZ} is that when
$\Gamma$ is consistent, there is an
isotopy of the zigzag Legendrian $\Lambda_\Gamma$
to the 1-cone $\Lambda_{{\bf \Sigma}(\triangle)^\circ}$ of the Legendrian
of the toric stack associated to $\triangle.$
By \cite{GKS} this results in an equivalence
of constructible sheaf categories for the
two Legendrians. 

Then under the mirror duality $\Coh(\bP_{\bf{\Sigma}(\triangle)})\cong
\Sh_{\Lambda_{{\bf \Sigma}(\triangle)}}(T^2)$,  
the condition that a coherent sheaf avoids fixed points restricts its support to be a curve or points, and a key result of \cite{TWZ} is to
characterize the constructible-sheaf image
of the pure one-dimensional
sheaves arising from the spectral
transform.

Recall again that on general grounds, we have
a map $Sh_\Lambda \to Loc(\Lambda)$
and that 
we focus on the subcategory of constructible
sheaves which are ``microlocal rank one,''
i.e.~land in $Loc_1(\Lambda)$ (in cohomological
degree zero).
It is proven in \cite{TWZ}
that the image
of pure one-dimensional
sheaves under the spectral transform
is 
$Sh^1_{\Lambda_{{\bf \Sigma}(\triangle)^\circ}}(T^2)$,
where the superscript indicates microlocal rank one.
Further, the monodromies of the zigzag paths
are the Casimirs of the
Poisson structure, becuase
the zigzags represent the boundary cycles of the
conjugate surface, which span the kernel
of the intersection form.
Thus fixing an image in $Loc_1(\Lambda)$
determines a symplectic leaf, and
as shown in \cite{GK} and \cite{TWZ}, this is dual
to fixing the intersection of $C$ with the
toric boundary.

\section{Counting Nodal
Spectral Curves}
\label{sec:defs}

We have arrived at the central question of the
paper.  Having described in Section \ref{sec:notation}
a complete integrable
system of curves $C$ and their Jacobians ---
equivalently sheaves on $\bP_\triangle$
pushed forward from line bundles
on curves --- as well as the mirror
dual system of constructible sheaves,
we ask:  how many curves are rational and nodal?
Call the answer $N_\Delta$.
We offer several equivalent definitions of $N_\Delta.$ 
On the algebraic side,
we begin with a classical definition
in terms of Severi varieties,
then use logarithmic Gromov-Witten theory,
then pass to an equivalent
formulation via tropical geometry.
On the constructible side, we employ techniques
from the relation between constructible
sheaves and Legendrian knots (similarly
to \cite[Sections 5.2 and 6.5]{STZ}) to give a more
combinatorial definition.  

These methods are
explained in more detail in the subsequent
sections, but briefly:
\begin{itemize}
\item Let $\Delta$ be a convex polygon with $g$ interior
points and $L_\Delta \to \bP_\Delta$ the corresponding
line bundle.  Let $|L'_\Delta|$ denote the (non-closed) $g$-dimensional
sub-linear system
of curves $i:C\subset \bP_\Delta$ of genus $g$
in $|L_\Delta|$ with fixed
generic intersection on the toric boundary.  Let $\cS_0\subset |L_\Delta'|$
denote the compact, zero-dimensional Severi subvariety consisting of curves of geometric genus zero.  Define $$N^{\mathrm{Sev}}_\Delta := \#\cS_0.$$

\item Write $\cM^{\mathrm{Coh}}$ for the moduli of sheaves of the form $i_*L,$ where $C$ is a curve in the reduced linear system $|L_\Delta'|$ defined above, and $L\to C$ is a line bundle of fixed degree.  $\cM^{\mathrm{Coh}}$ is isomorphic
to the compactified relative Jacobian of $|L_\Delta'|$ ---
see e.g.~\cite{Beauville1999} and Appendix \ref{sec:beauville}.  Then  \begin{equation*}
    \label{eq:NDelta} N_\Delta^{\mathrm{Coh}} := \chi(\cM^{\mathrm{Coh}}).
\end{equation*}
It is known that generically, all singularities are nodal ---see \cite{Tyomkin24,ChristHeTyomkin2020,KleimanShende2011,Harris1986}. Since Jacobians of all curves $C$ of non-zero geometric genus have vanishing Euler characteristic, only the Jacobians of curves in the Severi variety $\cS_0$ contribute, and as a result $N^{\mathrm{Coh}}_\Delta$ contribute, and each count for $1$.  Thus $N^{\mathrm{Coh}}_\Delta = N^{\mathrm{Sev}}_\Delta.$
\item Write $\overline{\cM}_{0,n}^{\mathrm log}(\bP_\Delta,\beta)_d$ for the moduli space of log stable
maps of degree $\beta = \beta(\Delta)$ Poincar\'e dual to $c_1(L_\Delta),$
and with ``type'' $d$ defined by $n$ points on the toric boundary. 
Here $n$ is the lattice length
of the perimeter $\partial \triangle$,
and $d$ encodes the log structure
defined as follows:
if the $i$th boundary component of $\partial \triangle$ has lattice length $d_i$, we send $d_i$
marked points to the
corresponding toric divisor $D_i$.
Write $ev_p^*\phi_\rho(p))$ for the pullback of the
point class on the divisor $\rho(p)$ to which $p$ is sent.
See Section \ref{sec:logGW} for details.
Then
$$N_\Delta^{\mathrm{GW}} := \int_{\overline{\cM}_{0,n}^{\mathrm{log}}(\bP,\beta)_d}\prod_{i=1}^{n-1}ev^*_{p_i}\phi(\rho(p_i)).$$
It follows from \cite[Theorem 1.2]{MR} that
this question is enumerative, the answer
being the number of immersed nodal rational curves
of degree $\beta,$ i.e.~$N^{\mathrm{GW}}_\Delta = N^{\mathrm{Sev}}_\Delta.$
\item Write $\cT_\triangle$ for the set of tropical maps
from a genus-zero tropical tree into $\bR^2$, with the following conditions.
As above, put $d_i$ for the lattice length
of the $i$th edge of $\partial\triangle$,
with $n=\sum_i d_i.$  We require
$d_i$ external edges to be parallel
to the $i$th ray of the fan $\Sigma(\triangle),$ with $n-1$ of
them fixed at infinity generically.
Let $m(\tau)$ be
the multiplicity of a tropical tree.  Then
$$N_\Delta^{\mathrm{trop}} := \sum_{\tau\in\cT_\Delta}m(\tau)$$
It follows from Mikhalkin's theorem
\cite[Theorem 1]{Mikhalkin2005Enumerative}, together
with the work of \cite{MR} incorporating incidence conditions,
that $N^{\mathrm{trop}}_\Delta=
N^{\mathrm{GW}}_\Delta.$
\item We can identify $\cM^{\mathrm{Coh}}$ with $\cM^{\mathrm{Stack}}$ as in
Remark \ref{rmk:coh-stack}.
Then use the coherent constructible correspondence
for toric stacks to find an isomorphic
mirror moduli space.  Specifically,
write $\cM^{\mathrm{Sh}}\cong \cM^{\mathrm{Coh}}$ for the
moduli space of constructible sheaves
mirror to coherent sheaves in $\cM^{\mathrm{Stack}}$ under the equivalence of \cite{TWZ}. 
Then
$$N_\Delta^{\mathrm{Sh}} := \chi(\cM^{\mathrm{Sh}}).$$
Then the equivalence of categories \cite{TWZ} establishes
isomorphism $\cM^{\mathrm{Sh}}\cong \cM^{\mathrm{Stack}},$ which in
turn is isomorphic to $\cM^{\mathrm{Coh}},$ whence the equality $N^{\mathrm{Coh}}_\Delta = N^{\mathrm{Sh}}_\Delta$.
\item Given $\Delta$, we define in
Section \ref{sec:rulings} (see Definition
\ref{def:rulings on T}) the notion of a \emph{ruling}, a combinatorial object that captures the underlying geometry of constructible sheaves in $\cM^{\mathrm{Sh}}$.  We then focus on \emph{rational} rulings, and define
$$N^{\mathrm{rul}}_\Delta := \text{the number of rational rulings.}$$   
Analogously to the stratification of a linear system of
curves by genus, we have
$\cM^{\mathrm{Sh}}_\Delta = \bigsqcup_\rho \cM^{\mathrm{Sh},\rho}_\Delta$, with 
$\cM^{\mathrm{Sh},\rho}_\Delta \cong (\bC^*)^{2g(\rho)}\times \bC^{r(\rho)}$, where $g(\rho)$ is the genus of a
\emph{ruling surface}, so rational rulings having
$g(\rho)=0.$
The equality $N^{\mathrm{Sh}}_\Delta = N^{\mathrm{rul}}_\Delta$ 
then follows from additivity of the Euler characteristic,
analogous to $N^{\mathrm{Coh}}_\Delta=N^{\mathrm{Sev}}_\Delta.$
See Theorem \ref{thm:rulings}.
\end{itemize}

These equivalences give the following equalities, the first
two of which are new (see Theorem \ref{thm:rulings}).
\begin{theorem}
\label{thm:mainthm}
$$N^{\mathrm{rul}}_\Delta = N_\Delta^{\mathrm{Sh}} 
=N_\Delta^{\mathrm{trop}}=N_\Delta^{\mathrm{GW}} =N_\Delta^{\mathrm{Coh}}=N^{\mathrm{Sev}}_\Delta.$$
\end{theorem}
We sometimes write $N_\Delta$ for any of these quantities.

The following two sections, respectively on the log GW and tropical approaches, briefly highlight aspects that are relevant for our calculations.  Then, in Section \ref{sec:constructible} we develop the constructible-sheaf side of the story
in detail.

\section{Definition
via Logarithmic Gromov-Witten Theory}
\label{sec:logGW}

Let $N$ be a rank-two lattice, $N_\bR := N\otimes_\bZ \bR;$ let $M=\Hom(N,\bZ)$ be the dual lattice and $M_\bR:=M\otimes_\bZ \bR.$
Let $\Delta\subset M_\bR$ be a convex lattice polygon and let $L_\Delta \to \bP$ be the corresponding equivariant ample
line bundle over the toric Fano surface $\bP = \bP_\Sigma$ defined by the normal fan $\Sigma\subset N_\bR$ of $\Delta.$

The data defines a logarithmic Gromov-Witten problem.  Recall that Gross-Siebert \cite{GS} construct
a moduli space $\overline{\cM}^{\rm log}(C,X)$
of basic log stable maps from a domain log scheme $C$ over base log scheme $W$ to a projective variety $X$ over a base log scheme $S$, and in some cases show this family is proper over $S$. To wit, $X=\bP_\Delta$ is such a case.  
The space of log stable maps carries a perfect obstruction theory,
relative to a universal log curve.
This family then has a virtual
fundamental class, with which log Gromov-Witten invariants 
can be defined.

An important novelty in the logarithmic case is that components of $\cM^\log(C,X)$ can be labeled not only by the standard decorations --- the arithmetic genus $g$ of the domain curve, number of marked points $n$, and homology $\beta$ of the image --- but also by \emph{type}, which includes incidence data of the marked points, as we soon describe.

First, let us fix the usual decorations.  
We require $g=0$.
We require $n= \#(\partial\Delta \cap M)$ to be the number of
boundary points, equal also to the lattice perimeter of $\Delta$. 
More specifically, we fix an isomorphism between marked points
$p$ and unit intervals of the perimeter (unimodular line segments of $\partial \Delta)$, and therefore
also boundary divisors $D_{\rho(p)}$ corresponding to rays $\rho(p)\in \Sigma(1).$
Then a given ray $\rho$ will be assigned to as many marked points as the lattice length of the corresponding edge of $\Delta$.
Finally,
$\beta = \beta(\Delta)$ is defined to be the homology class uniquely determined by its intersection with the toric divisors, as follows:
$\beta\cdot D_\rho$ is the lattice length of $D_\rho = \deg(L_\Delta\vert_{D_\rho}),$ or in other words $\beta\in H_2(\bP,\bZ)$ is Poincar\'e dual to $c_1(L_\Delta).$

We now use the \emph{type} to further specify a component of $\cM^{\log}(\bP).$
We recall that at a point $q$ of the domain curve, the log structure determines the stalk of the ghost sheaf of $C$.  
To remind, let $Q$ be the (universal, basic)
monoid of the log point $W$ --- see \cite[Construction 1.16]{GS}.  Then the ghost
sheaf of $C$ at $q$ is either
$Q$ (if $q$ is generic), $Q\oplus \bN$ (if $q$ is
a marked point) or $Q\oplus_\bN (\bN\oplus \bN)$ (if $q$ is a node), where in the last case the map from $\bN\to \bN\oplus \bN$ is diagonal and the image of $1$ in $Q$ is denoted $\rho_q.$ 
At the image $p := f(q)\in X$ we have the stalk $P$ of the ghost sheaf of $X$.  The log morphism determines a map $P$ to the ghost
sheaf of $C$ at $q$.

In particular, for each marked point,
a log morphism gives a map from $P$ to $\bN,$ so an element of the dual cone of $P$. 
In our toric case, this proceeds as follows.
Let $\sigma := \sigma_q$ be the smallest cone of $\Sigma$ such that $p\in U_\sigma = \mathrm{Spec}(\sigma^\vee \cap M).$
Then $P = \sigma^\vee \cap M$ and the log morphism determines an element $v_q$ of $\sigma \cap N.$  The collection of data $\sigma_q$,
where $q$ is a marked point, is part of the type.  We require
$\sigma_q$ to be precisely $\rho(p),$ the ray assigned to $p$ above, and require $v_p$ to be
a generator.  This fixes a simple incidence
condition along the corresponding divisor.  (Non-generators would encode higher tangencies.)
Let us call this topological type $d = d(\Delta).$  We will indicate the corresponding component of the moduli of maps with a subscript $d(\Delta).$\footnote{In the language
of \cite[Definition 1.10]{GS}, above we have
fixed the data of the ``type'' at marked
points, but the type also includes
the data $\{u_q\}$ at the nodes $q$ of $C$.  We
do not fix this in our logarithmic Gromov-Witten
problem, meaning we sum over types which
include any such choice of
data at nodes.}

Thus far we have constructed
$$\overline{\cM}^\log_{0,n(\Delta)}(\bP_{\Sigma(\Delta)},\beta(\Delta)_{d(\Delta)}.$$  Our notation
emphasizes that all decorations depend only on $\Delta,$ though from now on we will omit the $\Delta.$
This space comes equipped with evaluation maps, but note that from the type $d$ we have automatically that
$ev_p(f) = f(p) \in D_{\rho(p)}$.  As a result, we can pull back a point class $\phi_{\rho(p)}$ generating $H^2(D_{\rho(p)},\bZ)$.
This imposes the incidence condition that the marked point $p$ be mapped to a specific point on the boundary divisor $D_{\rho(p)}$.
Finally, then, we can define the logarithmic Gromov-Witten invariant
\begin{equation}
    N^{\mathrm{GW}}_\Delta := \int_{\overline{\cM}_{0,n}^\log(\bP,\beta)_d}\prod_{i=1}^{n-1} ev_i^* \phi_{\rho(p_i)}.
\end{equation}
Note the $n$th point is not constrained ---
similar to the $3d-1$ point conditions (not $3d$)
in Kontsevich's original work on $\bP^2$ --- as the image of that point
is constrained so that the image of all points together represent $L_\Delta \vert_{\partial\bP}.$

As a reality check, note the virtual dimension
is equal to the dimension of the linear system of the pullback of $L_\triangle$ to a genus-zero curve, i.e.~$n-1$, and the number of codimension-one point conditions is indeed also $n-1$.

\subsection{Localization}

\subsubsection{The Localization Theorem}
Our computations via localization follow the work
in progress of the first-named author.  As a result,
we do not include proofs of the localization formula, but
instead give details on how the localization is carried out.\footnote{The interested reader may view \href{https://www.youtube.com/watch?v=7WPrQbu4tBg}{this talk} on the topic.}

Let us recall the localization formula for torus equivariant integrals in the algebraic setting.
We have a torus $T = (\mathbb{G}_m)^n$ acting on a smooth projective variety $X$ with fixed
locus $X^T$ and normal bundle $N_{X^T/X}$.  Now in the equivariant
Chow ring, we have
$A_T^*(\mathrm{pt}) = \mathbb{Q}[\lambda_1,...,\lambda_n]$ and
$A_T(X^T)=A(X^T)\otimes \mathbb{Q}[\lambda_1,...,\lambda_n]$.  Further,
if we write $i:X^T\to X$ for the inclusion, then
\begin{equation}
    \label{eq:loc}
    i_*: A^T_*(X^t)_\lambda \to A^T_*(X)_\lambda
\end{equation} is an isomorphism (after inverting the $\lambda_i$), and then 
the Bott residue formula reducing integration over $X$ to integration over $X^T$ can be deduced from the self-intersection formula
\begin{equation}
\label{eq:ab}
    i^*i_*- = e(N_{X^T/X})\cap -.
\end{equation}
as follows.  Since $e(N_{X^T/X})$ is invertible in $A_T^*(X^T)$ (because all
torus weights in normal directions are nonzero), we can set
$\alpha = \frac{1}{e(N_{X^T/X})}$
and write this as $i^* i_* \alpha = [X^T],$ which
implies $i_* \alpha = [X].$  As a result,
we have the localization formula $\int_X p = p\cap i_*\alpha = i^*p \cap \alpha,$ where $\alpha = e(N_{X^T/X})^{-1}$. When $X$ is not smooth, we don't have the self intersection formula at our disposal.  Moreover, in Gromov-Witten theory one wants to compute integrals not against an ordinary fundamental class, but against a virtual fundamental class.  It is still true that by 
Equation \eqref{eq:loc} there must be a unique class $\alpha$ on the fixed locus whose push forward is the virtual fundamental class.  In the setting of logarithmic stable maps, it is more difficult to give a simple description of this class, but we will compute it by toric methods in some simple examples later.

\subsubsection{The fixed locus}
\label{sec:fixedlocus}

In ordinary Gromov-Witten theory, the connected components of the space of torus fixed maps to a toric variety are indexed by some decorated graphs.  In the case of logarithmic maps, the combinatorial data is more interesting, it is given by certain types of tropical curves in the toric fan of the target variety.  These tropical curves arise from a basic construction in log GW theory.  The tropicalization of a log stable map over a log point is constructed from the morphisms of stalks of ghost sheaves.  
It assigns to each such map a
tropical curve as follows.  The underlying graph of the tropical curve is the dual graph of the source curve (with an unbounded edge for each marked point).
A log morphism $f:C\to X$ over a log point $*_\log$ (left, below) determines at $q\in C$
morphisms of ghost sheaves:
$$\xymatrix{C\ar[r]^f\ar[d]&X\\ \ast_\log}
\qquad\xymatrix{G&\ar[l] P\\ Q\ar[u]}$$
where the monoid $G$ at $q$ can be $Q \oplus 0$ (if $q$ is a generic point), $Q\oplus \bN$ (marked point), or $Q \oplus_\bN (\bN\oplus \bN$ (node).  Here $P$ is the stalk
of the ghost sheaf
at $P=f(Q)$ and we have
$P = \sigma^\vee \cap M$ where the cone $\sigma$
labels the largest torus-invariant stratum
containing $p$ (or the smallest cone such that $p\in U_\sigma$).  If we take $Q = \bN$, then for each irreducible component of $C$, at the generic point we get a morphism $P\to \bN$ which is exactly a point in the dual monoid to $P$, which is simply a lattice point in the dual cone.  At each marked point, we will get an additional map $P \to \bN$ giving another vector in the dual cone which we interpret as a ray emanating from the vertex and heading to infinity in a direction contained within the same cone.  At a node, we will get two equal and opposite slopes.  Putting together the marked points for each component and drawing in the line segments for the nodes and the rays for the marked point, we arrive at a balanced tropical curve whose vertices are all in maximal or codimension one faces of the toric fan.  (This last condition is because a torus fixed map sends every irreducible component either to a torus fixed point of the target or to a one dimensional torus orbit. At a nontoric fixed point we would simply not have this last condition, but the rest of the discussion would be identical.) A given point of the moduli space will correspond not to a single tropical curve, but rather to the monoid of all tropical curves where we fix the data of the dual graph of the curve, which cone in the toric fan each vertex maps to, and what the slopes of all the edges and rays are.  
We can think of the data as a family of
tropical curves parametrized by the set of
monoid homomorphisms $Q\to \bN$ or as a single generalized tropical curve where the $\N$ is replaced by $Q$.
 The balanced condition on the tropical curve is proved in \cite[Prop.~1.14]{GS}. where this tropicalization construction is discussed in more detail and more generality.

The monoid $Q$ determines part of the formal local geometry of the moduli space near the corresponding point.  Up to some finite groups, there is a natural morphism from the space of log stable curves to the quotient of the affine toric variety defined by $Q$ by its associated torus.  In the case of genus zero log stable maps to a toric variety, this morphism is smooth, so the local geometry of the moduli space is mostly determined by this monoid, and the determination of the contribution of the fixed locus in the setting of localization becomes essentially a toric computation.
 The
procedure is perhaps best explained by the
examples of Section \ref{sec:examples}, in
particular that of Section \ref{sec:poly2b} which
is worked out in great detail.

\section{Definition and Computation 
via Tropical Geometry}
\label{sec:tropical}

After the seminal work of Mikhalkin \cite{Mikhalkin2005Enumerative}, the counting of holomorphic curves in toric surfaces can be approached from a tropical perspective. 
On the open set we have a hypersurface in an algebraic torus, and we can
consider for the moment that problem in any dimension. So
fix a Laurent polynomial $f = \sum_{m\in \bZ^n} a_m z^m$ and
set $A = \{m\in \bZ^n \mid a_m \neq 0\}.$  Given a function $v:A\to \bR$ we define the family $f_t = \sum a_m t^{-v(m)}z^m$,
the hypersurfaces $X_f = f_t^{-1}(0)$ and amebas $\mathcal A_t = {\rm Log}(X_t)\subset \bR^n.$
Then as $t\to \infty$ the scaled amebas $\frac{1}{\log t}\mathcal A_t$ converge in the Gromov-Hausdorff topology to the \emph{tropical ameba} $\mathcal A^{trop}$ \cite[Theorem 4.6]{Mikhalkin-amoeba}, defined by the tropical semi-ring as follows.

Basically, the tropical semi-ring only considers valuations (leading degrees) of Puiseaux series $\bC\{t\}$.  The valuation of the variable $z$ is written as the real variable $x$.
In this way, products become sums and sums become ``max,'' because that's how leading powers work.  The tropicalization ${\rm trop}(f_t)$ is thus ${\rm max}_{m\in A}(v(m) + m\cdot x).$   The ``tropical hypersurface'' or ``tropical
ameba'' is the critical locus of ${\rm trop}(f_t),$ namely the locus where
the maximum is achieved by more than one term.

\begin{example}
Take $f_t = z_1 + z_2 + 1$ and $v \equiv 0.$  Then $f_t = x_1 + x_2 + 1$ and ${\rm trop}(f_t) = {\rm max}\{x_1,x_2,0\}.$  The tropical ameba is therefore $\{x_1 = x_2 \geq 0\}\cup \{x_1 = 0, x_2 \leq 0\}\cup \{x_1 \leq 0,x_2=0\}$:
\begin{tikzpicture}[scale=.2]
\draw[very thick] (0,0) -- (1.414,1.414);
\draw[very thick] (0,0) -- (0,-2);
\draw[very thick] (0,0) -- (-2,0);
\end{tikzpicture}
\end{example}

Tropical amebas have this graph structure --- generically trivalent --- with a balancing condition:  the sum of vectors along rays emanating from each vertex must vanish. (The vectors include the data of multiplicities of primitive vectors.)  If $\mathcal A$ is a tropical ameba of a genus-$g$ curve, then $g = {\rm dim}H_1(\mathcal A).$

Though the picture is appealing, holomorphic curves are not in one-to-one correspondence with tropical curves.  To convert a counting problem from curves to tropical curves, one must study multiplicites:  how many algebraic curves impinge upon the same tropical curve?  In particular, tropical Gromov-Witten theory determines the appropriate multiplicity of tropical curves satisfying the tropicalization of the algebraic problem --- see, e.g., \cite{MandelRuddat2023}.
So we are counting genus-zero tropical curves $\cT$ with certain multiplicities --- specifically, $\mathrm{mult}(\cT) = \prod_v \mathrm{m_v},$ where the vertex multiplicity is $m_v = |\det(v_1\vert v_2)|$ where $v_1$ and $v_2$ are two of the three vectors emanating from vertex $v$. 

The condition of fixed toric boundary amounts to fixing the external rays at infinity,
 an enumerative
condition of the kind considered in \cite{MR}.  For example, in the case of curves of bi-degree $(2,2)$ in $\bP^1\times \bP^1,$ the tropical amebas have eight external rays or spokes, two in each of the four compass directions:
\begin{tikzpicture}[scale=.4]
    \draw (.35,.25)--(.55,.25)--(.55,.65)--(.35,.65)--cycle;
    \draw (.35,.66)--(.2,.8);
    \draw (-.3,.8)--(.2,.8)--(.2,1.3);
    \draw (.55,.65)--(.8,.9);
    \draw (.8,1.3)--(.8,.9)--(1.3,.9);
    \draw (.3,.2)--(.35,.25);
    \draw (-.3,.2)--(.3,.2)--(.3,-.3);
    \draw (.55,.25)--(.7,.1);
    \draw (.7,-.3)--(.7,.1)--(1.3,.1);
\end{tikzpicture}.
We can independently fix seven of these spokes at infinity, i.e.~outside a compact set.  In the general case,
we fix all but one external
leg at infinity.  

Let us recall Mikhalkin's theorem.  
Let $\Delta\subset \bR^2$ be a convex lattice polytope, let $\Sigma = \Sigma(\Delta)$ be the corresponding normal fan and let $\bP = \bP_\Sigma$ the associated toric Fano surface
with ample line bundle $\cL\to \bP.$
Let $\beta = c_1(\cL)^\vee \in H_2(\bP)$
be a curve class,
let $n:= |\partial \Delta \cap \bZ^2|-1$ and let $Q = \{q_1,...,q_n\}$ be
a collection of $n$ points of $\bR^2$
in general
position.  Let
$K^\mathrm{GW}_\Delta = \int_{\cM_{0,n}(\bP,\beta)} \prod_{i=1}^n \mathrm{ev}_i^*(P^\vee),$ where $P^\vee \in H_4(\bP;\bZ)$
is the fundamental/point class, i.e.~Poincar\'e dual to a point $P$ in $\bP$.
Let $K^\mathrm{trop}_\Delta$ be the sum
over tropical cures passing through $Q,$
counted with multiplicity.
Then Mikhalkin shows \cite[Theorem 1]{Mikhalkin2005Enumerative}:
$$K^\mathrm{GW}_\Delta = K^\mathrm{trop}_\Delta.$$
In other words, the incidence conditions
$\mathrm{ev}^*_i(P^\vee)$ can be realized
tropically by the condition that the tropical
curve passes through the $q_i$.

We want a variant of this theorem
allowing for incidence conditions for points
mapping to the toric boundary $\partial \bP.$
Logarithmic Gromow-Witten theory
is the proper setting.
Specifically, since the \emph{type} of a log map includes
the data of a map of monoids, the types encode the
information of the monoids at the point to which the
marked points are mapped.  In other words, the type
determines which boundary (toric) divisor the marked point
gets mapped to:  so $ev_i$ maps to a chosen boundary
divisor, and we can impose boundary incidence conditions
as usual,
by pulling back a point class from the corresponding boundary divisor.

Then according to
\cite{vGGR,MR}, the tropical count
equals the log Gromov-Witten count,
as stated in Section \ref{sec:defs}.
We now explain
in some more detail
how to impose these boundary
conditions tropically.

\subsection{Tropical Counts}
\label{sec:tropcounts}

We are interested in tropical curves $\cT$ with fixed external rays at infinity.  To formulate this, suppose that
$\Delta = \conv(v_1,...,v_s)$, with $v_i, i\in \bZ/s\bZ$ in counterclockwise order around $\partial \Delta$ and put $w_i = J(v_i-v_{i+1})$
the outward normals vectors, where $J$ is the
complex structure on $\bR^2$.  Now write $\hat{w_i}$ for the primitive vector in the direction
of $w_i,$ with $w_i = d_i \hat{w}_i,$ where $d_i$ is thus the lattice length of the $i$th edge of $\Delta.$  So $\sum_i d_i = n+1.$ 
Define the ray vectors 
$r_{i,j}, i = 1,...,s$ and $j = 1,...,d_i$
by $r_{i,j} = \hat{w}_i.$
Now 
recall that the external $n + 1 = \sum_i d_i$ rays of a tropical
curve $\cT\subset \bR^2$ corresponding to a holomorphic curve in class $\beta$ are required to be in the directions $r_{i,j}$.  We put $n$
boundary conditions on $\cT$ as follows.

First for a pair $(k,l),1\leq l\leq d_k,$ 
define $\bR_{(k,l)}:= \bR.$
Now fix a pair $(i,j), 1\leq j\leq d_i,$
then put $\widetilde{H} = 
\bigoplus\limits_{(k,l)\neq (i,j)} \bR_{(k,l)}$
and write $H\subset \widetilde{H}$ for
the subset \emph{excluding}, for all $(k,l)\neq (k,l')$, any points
whose projection 
$(x,y)$ to the $\bR_{(k,l)}\times \bR_{(k,l')}$
factor has $x=y.$  (For counts with higher tangencies,
one would relax this exclusion.)  See Figure \ref{fig:bc}.
\begin{figure}[ht]
\begin{tikzpicture}
    \pgfmathsetmacro{\a}{sqrt(5)}
    \draw[thick,blue] (0,0)--(1,0)--(1,2)--(0,0);
    \draw (0,-1)--(1,-1);
    \draw (2,0)--(2,2);
    \draw (-2/\a,1/\a)--(1-2/\a,2+1/\a);
    \fill[blue] (0,0) circle (.06cm);
    \fill[blue] (1,0) circle (.06cm);
    \fill[blue] (1,1) circle (.06cm);
    \fill[blue] (1,2) circle (.06cm);
    \draw[ultra thick, red] (.8,-.8)--(.8,.2)--(1.8,.2);
    \draw[ultra thick, red] (.8,.2)--(.5,.5)--(1.8,.5);
    \draw[ultra thick, red] (.5,.5)--(.5-2/\a,.5+1/\a);
    \fill[black] (2,.2) circle (.04cm);
    \fill[black] (2,.5) circle (.04cm);
    \fill[black] (.8,-1) circle (.04cm);
    \node[right] at (2,.1) {$h_{2,1}$};
    \node[right] at (2,.6) {$h_{2,2}$};
    \node[below] at (.8,-1) {$h_{1,1}$};

\end{tikzpicture}
\caption{Here $\Delta = \conv\{(0,0),(1,0),(1,2)\}$, so $\hat{w}_1 = (0,-1), \hat{w}_2 = (1,0), \hat{w}_3 = (-2,1).$  Note $d_1 = 1, d_2 = 2, d_3 = 1.$  Tropical curve shown in red for boundary conditions given by $(i,j) = (3,1)$ and
$h_{1,1} = .8, h_{2,1} = .2,
h_{2,2} =.5.$  Note the conditions
exclude the case $h_{2,1}=h_{2,2}$ of coincident right rays.
This unique tropical curve
has intercept with the third edge
determined to be $-0.3$. 
}
\label{fig:bc}
\end{figure}
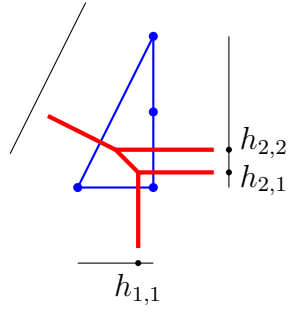

Now let $h = (h_{(k,l)})\in H.$
We require of our tropical curves $\cT$
that the $(k,l)$ external ray coincides with the ray $\gamma_{(k,l)}(t) = h_{(k,l)} J\hat{w}_k + t\hat{w}_k, t>0$, outside a compact set.  Thus the $h_{(k,l)}$ are the ``heights'' or ``intercepts'' of these external rays, parametrizing their intersecions with the polygon ``at infinity.''

Fixing $h\in H,$
let $N_\Delta^{\mathrm{trop},h}$ be the count of tropical curves $\cT$ (having internal vertices of degree greater than two) with boundary condition described by $h\in H$ --- i.e., having ray $(k,l)$ coinciding with $\gamma_{(k,l)}$ outside a compact set --- counted with
multiplicity.
Let $N_\triangle^{\mathrm{trop},h}(q)$
be the counts with $q$-multiplicity,
where the local vertex contribution
at a trivalent vertex $v$ is as follows.
call the (not necessarily primitive) edge vectors from the vertex $v_1,v_2,v_3$, oriented counterclockwise
and summing to zero.
Then $m_v(q) = [\det(v_1,v_2)]_q$, where
$[n]_q := \frac{q^{\frac{n}{2}} - q^{-\frac{n}{2}}}{q^\frac{1}{2}-q^{-\frac{1}{2}}}.$  Then $\mathrm{mult}(\cT)(q)=\prod_{v} m_v(q),$ and so setting $q=1$ recovers the standard multipliity:  $N_{\triangle}^{\mathrm{trop},h}(q=1) = N_\triangle^{\mathrm{trop},h}.$

Note that as the tropical
image of a holomorphic map --- not necessarily
embedded curve --- $\cT$ need only
be an immersed tropical curve. 
In particular, while the domain of $\cT$
is a tree, there may be loops in our
diagrams, which illustrate the image in $\bR^2.$  In particular, overcrossings
appear as four-valencies with opposite
pairs of edges.

\subsection{Invariance}
\label{sec:invariance}

We want to show that the tropical counts
defined with these constraints
are independent of the choice of boundary conditions at infinity.
In fact, the same
statement is true when we count
with $q$-multiplicities.
As we shall see, this ultimately is
established by the Pl\"ucker relation
and its $q$-analogue.
Let us emphasize that
this proof is essentially the same as
that of Itenberg-Mikhalkin \cite[Section 3]{ItenbergMikhalkin:2013:BlockGottsche}, 
\cite[Section 8.2]{Bousseau:2019:TropicalRefined} and Blomme \cite[Section 3.2]{Blomme}.  
We include it for completeness.
If $q=1,$
the following proposition also follows generally from \cite{MR}.
Write $N_\triangle^{\mathrm{trop},h}(q)$ for the tropical
count with $q$-multiplicities.

\begin{proposition}    $N_\Delta^{\mathrm{trop},h}(q)$ is independent of generic $h$, and of the initial choice $(i,j).$
\end{proposition}
\begin{proof}
We first fix $(i,j)$, $1\leq i\leq s$ and $1\leq j \leq d_i.$
\begin{enumerate}[Step 1)]
\item First note that while $\cH$ is not connected,
the product of permutation groups $G = S_{d_i-1}\times\prod_{k\neq i}S_{d_k}$ acts transitively
on the set of connected components, and the invariants
are unchanged under the action of $G$ composed with a
simultaneous relabeling of endpoints.  So we can
prove invariance by proving invariance under continuous
paths $h:I\to H,$ i.e.~$N^{\mathrm{trop},h(0)}_\Delta(q)
= N^{\mathrm{trop},h(1)}(q).$
\item Next we note that the number and degrees of
$\cT$ are bounded by Lemma \ref{lem:comblemma}, so
there are finitely many possible topological types
of $\cT.$  It follows that since all the constraints
on $\cT$ are continuous, there are open chambers
within $H$
consisting of trivalent trees of
fixed topological type,
and that in any one-parameter family
of boundary conditions $h(t),$
$t\in [0,1],$
there are a finite number of
critical moments $t_1,...,t_m$ 
separating intervals in which the corresponding trees are trivalent
and isotopic to one another.
With generic boundary conditions,
trees at critical moments (walls) will have
a single four-valency.
Trees on opposite sides of the wall are
not necessarily in bijection with
each other.
\item 
By concatenating paths, we may assume that our family has a 
single critical moment, and by openness of the chambers
that the only one $h_{(k,l)}$ is not constant.
We have reduced the problem to varying the boundary
condition of a single external ray and a single
moment where the collection of tropical curves
changes.  
\item Let us call the varying ray $a$ and the unconstrained
ray $b$.  Since $\cT$ is a tree, there is a unique path
$P\subset \cT$ from $a$ to $b$.  The neighborhood of $P$
in $\cT$ looks like a fishbone, pictured with a blue crooked spine
and black vertebrae/spokes in Figure \ref{fig:fishbone}, left.
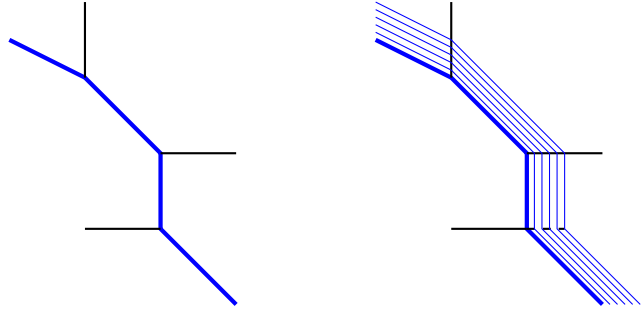
\begin{figure}
    \centering
\begin{tikzpicture}
    \draw[ultra thick,blue](1,-1)--(0,0)--(0,1)--(-1,2)--(-2,2.5);
    \draw[thick,black](0,1)--(1,1);
    \draw[thick,black](-1,2)--(-1,3);
    \draw[thick,black](0,0)--(-1,0);
\end{tikzpicture}
\qquad\qquad
\begin{tikzpicture}
    \draw[ultra thick,blue](1,-1)--(0,0)--(0,1)--(-1,2)--(-2,2.5);
    \draw[thick,black](0,1)--(1,1);
    \draw[thick,black](-1,2)--(-1,3);
    \draw[thick,black](0,0)--(-1,0);
    \draw[thick,black,dashed](0,0)--(.5,0);
    \foreach \i in {.1,.2,.3,.4,.5}
    {\draw[blue](1+\i,-1)--(\i,0)--(\i,1)--(-1,2+\i)--(-2,2.5+\i);}
\end{tikzpicture}
    \caption{Portion of a path $P\subset \cT$ (left, in blue) and its deformation (right).}
    \label{fig:fishbone}
\end{figure}

\item Now as we vary $a,$ the path $P$ moves in its
``normal bundle.''  
The generic
tree will be trivalent.
Specifically, the local
neighborhood of each vertex of $P$ is 
three rays $\bR_+u_1,\bR_+u_2\bR_+u_3$
emanating from the origin
with ray vectors $u_1,u_2,u_3$ summing
to zero.  Say $\bR_+u_1\cup \bR_+u_2$ is the
intersection of $P$ with the neighborhood, with $\bR_+u_3$
the black spoke.  The family $\left(t u_3 + \bR_+u_1\right)\cup (t u_3 + \bR_+ u_2)$,
with $t\in (-\epsilon,\epsilon),$
deforms the blue spine $P$, with black spoke $t u_3 + \bR_+u_3.$
See Figure \ref{fig:fishbone}, right.
Adjacent neighborhoods can be deformed similarly, with
no obstructions since $P$ is a path and $\cT$ is a tree.

\item This deformation extends until the lifted
path of $P$ meets another vertex or an edge of
$P$ shrinks to zero.  Thus the critical moment happens
when an edge (either of $P$ or a spoke) shrinks to zero along a lifted path.  Since this obstruction is local,
we can analyze a neighborhood of a shrinking edge,
so near a four-valent vertex.
\item So let $u_1,...,u_4 \in \bZ^2$ obey
$\sum u_i = 0$ and consider $\cT = \bigcup_i \bR_+u_i$ with $P$ the union of rays $u_a$ and $u_b$, i.e.~$b$ is fixed and we vary the
intercept at infinity of the initial ray $a$.
Note by Lemma \ref{lem:comblemma}
there are only three possible abstract
trivalent trees with four leaves, labeled
by the three possible pairings of the leaves.

Now there are two
cases:  $a$ and $b$ are consecutive in the cyclic
order or not.  
\begin{itemize}
\item
If they are consecutive, then so are the other two rays, and up to $SL_2(\bZ),$
we may assume the other rays to be 
in the directions $(1,0)$ and $(0,1)$ and
we label them $1$ and $2$.
Note we make no assumptions about whether $u_1$ and $u_2$ are primitive.  Up to a reflection,
we may further assume that $a$ is labeled $4$ and $b$ is $3$.
Let us consider the case that $u_4$ is in
the fourth quadrant and $u_3$ is in the second
quadrant,
and let us parametrize the family by
varying the fourth ray as $tu_1 + \bR_+u_4$.
If $t<0$ then
two of the three possible abstract trees
have a planar immersion with these boundary conditions.  If $t>0$ only one  possibility
is realized.  See Figure \ref{fig:trees}.
\begin{figure}
    \centering
\begin{tikzpicture}
\draw (1,0)--(0,0)--(0,1);
\draw (1,-2)--(0,0)--(-2,1);
\draw[ultra thick,blue] (.7,-2)--(-.2,-.2)--(-2,.7);
\draw[ultra thick,black] (-.2,-.2)--(0,0);
\draw[ultra thick,black] (0,0)--(1,0);
\draw[ultra thick,black] (0,0)--(0,1);
\node[right] at (1,0) {$1$};
\node[above] at (0,1) {$2$};
\node[left] at (-2,.7) {$3$};
\node[below] at (.7,-2) {$4$};

\end{tikzpicture}
\quad
\begin{tikzpicture}
\draw (1,0)--(0,0)--(0,1);
\draw (1,-2)--(0,0)--(-2,1);
\node[right] at (1,0) {$1$};
\node[above] at (0,1) {$2$};
\node[left] at (-2,.7) {$3$};
\node[below] at (.7,-2) {$4$};

\draw[ultra thick,blue] (.7,-2)--(0,-.6);
\draw[ultra thick,blue] (-.6,0)--(-2,.7);
\draw[ultra thick,black] (-.6,0)--(1,0);
\draw[ultra thick,black] (0,-.6)--(0,1);
\draw[ultra thick,blue] (0,-.6)--(-.6,0);

\end{tikzpicture}
\qquad\qquad
\begin{tikzpicture}
\draw (1,0)--(0,0)--(0,1);
\draw (1,-2)--(0,0)--(-2,1);
\node[right] at (1,0) {$1$};
\node[above] at (0,1) {$2$};
\node[left] at (-2,.7) {$3$};
\node[below] at (.7,-2) {$4$};

\draw[ultra thick,blue] (1.3,-2)--(.3,0);
\draw[ultra thick, black] (.3,0)--(1,0);
\draw[ultra thick,blue] (.3,0)--(0,.3);
\draw[ultra thick,blue] (0,.3)--(-2,1.3);
\draw[ultra thick,black] (0,.3)--(0,1);
\end{tikzpicture}

    \caption{Trees (thick) appearing in the deformation of a four-valent vertex (black, thin), with path $P$ (blue) involving consecutive rays.  Two trees (one immersed) appear on one side of the deformation (left), while one appears on the other (right).  Invariance is due to an equality of determinants provided by the
    Pl\"ucker relation: $[12][34] + [13][42] = [41][23].$}
    \label{fig:trees}
\end{figure}
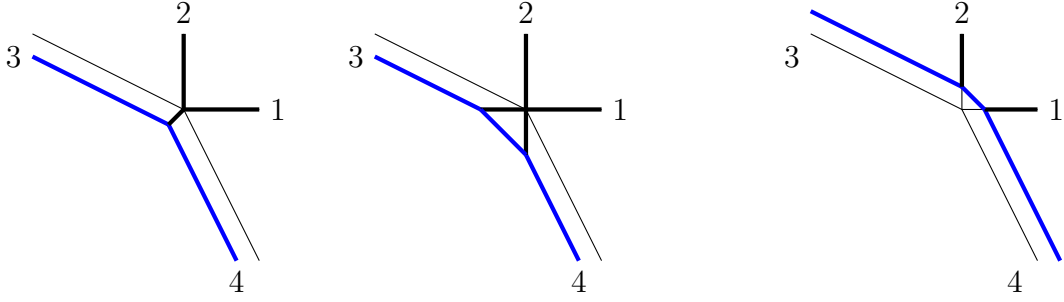
Now for the main point:  summing the contributions
when $t<0$ and $t>0$.  Let us write $[ij]$ for $\det(u_i|u_j),$ and we will be careful only to
take $u_i$ and $u_j$ in order so that the determinants we write are all positive.
When $t<0$ the multiplicities of tropical curves in Figure \ref{fig:trees} are
$[12][34]$ and $[13][42]$,
while for $t>0$ we have $[41][23].$
But by the Pl\"ucker relation for
the Grassmannian $\mathrm{Gr}(2,4)$
we have
precisely
$$[12][34]+[13][42]=[41][23].$$
For the $q$-analogue, the
same relation holds:  $[12]_q[34]_q + [13]_q[42]_q = [41]_q[23]_q$.  This is
not true generally as in the Pl\"ucker
case but holds for four integer vectors
\emph{whose
sum is zero}.
(For vectors in the stated quadrants as specified, we have written the relation so that all terms are positive.)
The general
$q$-relation is $$[i]_q[j+k]_q + [j]_q[k-i]_q = [i+j]_q[k]_q.$$  
To see that the six determinants
satisfy the linear equalities allowing
them to be expressed as above
in terms of 
$i,$ $j$ and $k$, consider the polygon
whose (counterclockwise) edge vectors
are rotations of the ray vectors by 90 degrees (so the areas remain the same):
$$
\begin{tikzpicture}[scale=.5]
    \coordinate (A) at (4,5);
    \coordinate (B) at (4,2);
    \coordinate (AB) at (4,7/2);
    \coordinate (BC) at (2,1);
    \coordinate (CD) at (1,5/2);
    \coordinate (DA) at (3,5);
    \coordinate (C) at (0,0);
    \coordinate (D) at (2,5);
    \coordinate (E) at (2,2);
    \coordinate (X) at (2,0);
    \coordinate (Y) at (0,3);
    \node[right] at (AB) {$1$};
    \node[below right] at (BC) {$4$};
    \node[above left] at (CD) {$3$};
    \node[above] at (DA) {$2$};
    \filldraw[draw=black,->,thick, fill=blue!70,fill opacity = .50] (A)--(B)--(D)--cycle; 
\filldraw[draw=black,thick,fill=blue!70,fill opacity = .50] (B)--(C)--(D)--cycle; 
 \node at (4,-2) {$[12]+[34]=[23]+[41]$};
 \begin{scope}[xshift=5cm]
    \coordinate (AB) at (4,7/2);
    \coordinate (BC) at (2,1);
    \coordinate (CD) at (1,5/2);
    \coordinate (DA) at (3,5);
    \node[right] at (AB) {$1$};
    \node[below right] at (BC) {$4$};
    \node[above left] at (CD) {$3$};
    \node[above] at (DA) {$2$};
     \coordinate (A) at (4,5);
    \coordinate (B) at (4,2);
    \coordinate (C) at (0,0);
    \coordinate (D) at (2,5);
    \coordinate (E) at (2,2);
    \coordinate (X) at (2,0);
    \coordinate (Y) at (0,3);
    \filldraw[draw=black,thick, fill=red!70,fill opacity = .50] (A)--(D)--(C)--cycle; 
 \filldraw[draw=black,thick, fill=red!70,fill opacity = .50] (A)--(C)--(B)--cycle; 
 \end{scope}
 \end{tikzpicture}
 \qquad
\begin{tikzpicture}[scale=.5]
    \coordinate (AB) at (4,7/2);
    \coordinate (BC) at (2,1);
    \coordinate (CD) at (1,5/2);
    \coordinate (DA) at (3,5);
    \node[right] at (AB) {$1$};
    \node[below right] at (BC) {$4$};
    \node[above left] at (CD) {$3$};
    \node[above] at (DA) {$2$};
    \coordinate (A) at (4,5);
    \coordinate (B) at (4,2);
    \coordinate (C) at (0,0);
    \coordinate (D) at (2,5);
    \coordinate (E) at (2,2);
    \coordinate (X) at (2,0);
    \coordinate (Y) at (0,3);
    \filldraw[draw=black,thick, fill=blue!70,fill opacity = .50] (B)--(D)--(C)--cycle; 
 \draw[black,thick] (D)--(A)--(B)--cycle; 
 \node at (4,-2) {$[34]=[12]+[42]+[13]$};
 \begin{scope}[xshift=5cm]
    \coordinate (AB) at (4,7/2);
    \coordinate (BC) at (2,1);
    \coordinate (CD) at (1,5/2);
    \coordinate (DA) at (3,5);
    \node[right] at (AB) {$1$};
    \node[below right] at (BC) {$4$};
    \node[above left] at (CD) {$3$};
    \node[above] at (DA) {$2$};
    \coordinate (A) at (4,5);
    \coordinate (B) at (4,2);
    \coordinate (C) at (0,0);
    \coordinate (D) at (2,5);
    \coordinate (E) at (2,2);
    \coordinate (X) at (2,0);
    \coordinate (Y) at (0,3);
    \filldraw[draw=black,thick, fill=red!70,fill opacity = .50] (B)--(D)--(C)--cycle; 
     \draw[black,thick] (D)--(A)--(B)--cycle; 
     \draw[black,thick] (D)--(E)--(B);
     \draw[black,thick] (C)--(E);
     \end{scope}
 \end{tikzpicture}
 \quad
\begin{tikzpicture}[scale=.5]
    \coordinate (AB) at (4,7/2);
    \coordinate (BC) at (2,1);
    \coordinate (CD) at (1,5/2);
    \coordinate (DA) at (3,5);
        \coordinate (CX) at (1,0);
    \coordinate (CY) at (0,3/2);
    \node[below] at (CX) {$2$};
    \node[left] at (CY) {$1$};
    \node[right] at (AB) {$1$};
    \node[below right = -2mm] at (BC) {$4$};
    \node[above left = -1.5mm] at (CD) {$3$};
    \node[above] at (DA) {$2$};
    \coordinate (A) at (4,5);
    \coordinate (B) at (4,2);
    \coordinate (C) at (0,0);
    \coordinate (D) at (2,5);
    \coordinate (E) at (2,2);
    \coordinate (X) at (2,0);
    \coordinate (Y) at (0,3);
    \filldraw[draw=black,thick, fill=blue!70,fill opacity = .50] (A)--(C)--(X)--(B)--cycle; 
        \filldraw[draw=black,thick, fill=red!70,fill opacity = .50] (A)--(D)--(Y)--(C)--cycle; 
 \draw[black,thick] (A)--(D)--(Y)--(C)--(X)--(B)--cycle;
 \draw[black,thick] (C)--(E)--(B);
 \draw[black,thick] (E)--(D);
 \draw[black,thick] (B)--(C)--(D);
 \node at (2,-2) {$[42]+[41]=[13]+[23]$};
 \end{tikzpicture}
$$

This concludes the proof of invariance of counts with $q$-multiplicities, in this case.

A similar enumeration of possibilities can
be made if $u_4$ is in the fourth quadrant
and $u_3$ is in the third.  These are essentially
all the cases, after some relabeling of diagrams
and/or rotating the figures.
\item
If $a$ and $b$ are not consecutive, then
again the same figures contribute, only
different parts are black and blue.
\end{itemize}

\end{enumerate}

Finally, by interchanging the roles of $a$ and $b$ in the above argument, the same proof works when $a$ is the free and $b$ the constrained (but varying) ray.

\end{proof}

\begin{remark}
    Due to independence of $h\in H$ (and of the pair $(i,j)$), we hereafter drop $h$ from the supercript and simply write $N_\Delta^{\mathrm{trop}}$ instead of $N_\Delta^{\mathrm{trop},h}.$
\end{remark}
Mikhalkin's theorem, together with \cite{vGGR,MR}, then gives
$$N^{\mathrm{GW}}_\Delta = N_\Delta^{\mathrm{trop}}.$$
See Section \ref{sec:examples} for examples.

\section{Definitions via Constructible Sheaves}
\label{sec:constructible}


In this section, all categories are dg-derived unless otherwise specified.
 We write $\underline{\Hom}$ for the dg-enhanced hom complex, $\End(X)$ for $H^0\underline \Hom(X, X)$, and $\Aut(X)$ for the automorphism group consisting of invertible elements in $\End(X)$. We work over the ground field $\dsk = \bC$.

 For background on microlocal sheaf theory we refer to \cite{KS90} and \cite{STZ}. We use $\SS(\cF)$ to denote the singular support (at infinity) of a sheaf $\cF$ on a manifold $X$. Since we only consider constructible sheaves, $\SS(\cF)$ is a (possibly singular) Legendrian in the cosphere bundle $S^*X$. 
 For a locally closed subset $Z$, by $\dsk_Z$ we mean the lower shriek of the constant sheaf on $Z$. We say $\dsk_Z$ is standard (resp. costandard) if $Z$ is open (resp. closed).

\subsection{The toric homological mirror symmetry}

We recall the coherent-constructible correspondence (CCC) of toric stacks~\cite{Bondal2006, FLTZ, FLTZ2, Ku}. 

Let \(\triangle\) be a Newton polygon, and $\T$ be the $2$ dimensional real torus. Fix the following notations:
\begin{itemize}
    \item Let \(\cX_\triangle\) be the corresponding (noncompact) toric stack, given by the stacky fan with only $0$ and $1$ dimensional cones, as defined in Section~\ref{subsec:CCC}. Let \(\cD\) be its boundary divisor. 
    \item Let \(\Lambda \subset S^*\T\) be the FLTZ skeleton corresponding to \(\cX_\triangle\), which is a union of Legendrian $S^1$. See Section~\ref{subsec:CCC}. 
\end{itemize}

By the CCC, we have the following commutative diagram:
\begin{equation}
    \begin{tikzcd}
        \Perf_\mathrm{prop} (\mathcal{X}_{\triangle}) \arrow[r, "CCC"] \arrow[d, "i^*"'] & \Sh^b(\T, \Lambda) \arrow[d, "\mu"] \\
        \Perf_\mathrm{prop} (\cD) \arrow[r, "\sim"] & \Loc (\Lambda)  
    \end{tikzcd}
\end{equation}
Here \( \Perf_\mathrm{prop}\) is the category of perfect complex with proper support, and \(\Sh^b (\T, \Lambda)\) is the category of (stalkwise perfect) constructible sheaves on $\T$ with singular supports contained in \( \Lambda\).

We have
\[
\cD  = \coprod_i {D}_i \times B \bZ/ k_i.
\]
where each \({D}_i\) is canonically isomorphic to \(\bC^*\), and \(k_i + 1\) is the number of integral points on the corresponding edge of \(\triangle\) . For each character \(\chi_j\) of \(\bZ/ k_i\), we choose a point \(p_{ij}\in {D}_i\). This defines an object \(\xi = \bigoplus k_{p_{ij}}^{\chi_j}  \in \Perf (\cD)\). We also ask the product of the coordinates of all \(p_{ij}\) to be \(1\).

Let \( \mathcal J_{\triangle,\xi}\) be the moduli stack of \(1\)-dimensional pure sheaves \(E \in \Perf(\cX_\triangle)\) with a fixed degree, such that \(i^*E \simeq \xi\). Since for a generic choice of $\xi$, such a sheaf $E$ is supported on a reduced and irreducible curve, $\End(E) \simeq \bC$. Hence $\mathcal J_{\triangle,\xi}$ is a $\Gm$-gerbe over its coarse moduli space $|\mathcal J_{\triangle,\xi}|$ (denoted by  $\cM^\mathrm{Coh}$ in Section~\ref{sec:defs}). As discussed in Section~\ref{sec:defs}, we have
\begin{equation}
     N_\triangle^\mathrm{Coh} = \chi \left( |\mathcal J_{\triangle,\xi}| \right).
\end{equation}


We want to characterize \(\mathcal J_{\triangle,\xi}\) on the A-side. The choice of \(\xi\) is equivalent to choosing a rank-\(1\) local system concentrated in degree \(-1\), which we denote by \(\eta \in \Loc (\Lambda)\). Let \(\Sh(\T, \Lambda)^\eta_0\) be the full subcategory of \(\Sh^b(\T, \Lambda)\) consisting of objects corresponding to  points in \(\mathcal J_{\triangle,\xi}\) under the CCC.

    Denote $q: S^*\T \rightarrow \T$, and choose a point $\mathfrak{t} \in \T \backslash q(\Lambda)$. Pick an object $\cF_0 \in \Sh(\T, \Lambda)_0^\eta$, and assume that \( \chi( (\cF_0)_\mathfrak{t}))  = d\). Here $\chi$ is means the Euler number. Note that in many cases we would assume that $\Lambda$ is the zig-zag path of a dimer model, as in \cite{TWZ}. The we can choose $\mathfrak{t}$ to be a point in an alternating region, and $d = 0 $.

\begin{lemma}\label{lem:characterization of Sh_0 by Hom with Loc}
A sheaf \(\cF \in \Sh^b(\T,\Lambda)\) lies in \(\Sh(\T,\Lambda)^\eta_0\) if and only if the following conditions hold:
\begin{enumerate}
    \item \label{lemfirst} \(\mu(\cF)=\eta\);
    \item \label{lemsecond} \(\chi(\cF_\frakt)\) = d;
    \item  \label{lemthird} for every rank-\(1\) local system \(\cL \in \Loc (\T)\) concentrated in degree \(0\), we have 
    \begin{equation}\label{eq:conditions on HiHom}
        \begin{cases}
            H^i \underline{\Hom} (\cL, \cF) = 0 & \text{for all \(i \le -2\);} \\
            H^i \underline{\Hom} (\cF, \cL    ) = 0 & \text{for all \(i \le 1\)}. 
        \end{cases}
    \end{equation}. 
    \end{enumerate}
\end{lemma}
\begin{proof}
    Let the mirror of \(\cF\) be \(E \in \Perf_\text{prop}(\cX)\). Condition (1) is equivalent to \(E|_\cD = \xi\). By \cite{Zhou19}, the stalk at $\mathfrak{t}$ is corepesented by a the mirror of a line bundle. Hence condition (2) is equivalent to \(\chi(L, E) = d\) for some line bundle $E$, which fixes the degree of $E$. Note that for a point \(x\in (\bC^*)^2 \subset \cX\), the mirror of the skyscraper sheaf \(\bC_x\) is a rank-\(1\) local system \( \cL[2] \in \Loc(\T)\) placed in degree \(-2\). Thus (3) is equivalent to 
    \begin{equation}\label{eq:homs between E and k_x}
    \begin{cases}
        \Ext^i(\bC_x, E) = 0 & \text{for all \(i \le 0\) and all \(x\in (\bC^*)^2\);}  \\
        \Ext^i(E, \bC_x) =0 & \text{for all \(i < 0\) and all \(x\in (\bC^*)^2\).}
    \end{cases}
    \end{equation}
    which is also equivalent to \(E\) being an ordinary coherent sheaf and pure.
\end{proof}

We use \(\Sh^1(\T, \Lambda)_0\) to denote the subcategory of \(\Sh^b(\T, \Lambda)\) consisting of sheaves of microlocal stalk being $k[1]$, and satisfying the conditions (2) and (3) in Lemma~\ref{lem:characterization of Sh_0 by Hom with Loc}. Under CCC, constructible sheaves in \(\Sh^1(\T, \Lambda)_0\) are mirror to the pure 1d coherent sheaves in \(\cX_\triangle\) with fixed degree.

Following \cite[\S 2.4]{STWZ} and \cite{TV07}, we consider the derived 
stack of microlocal rank \(1\) sheaves \( \bR \calM_1 (\T, \Lambda)\). Furthermore, we can consider the substack corresponding to objects in \(\Sh(\T, \Lambda )_0^\eta\), which we denote by \(\calM_{\triangle, \eta}\). By CCC, we have 
\[
\mathcal J_{\triangle,\xi} \cong \calM_{\triangle, \eta},
\]
which implies that 
\[
N_\triangle^\mathrm{Coh} = \chi \left( | \calM_{\triangle, \eta} | \right).
\]


Analogous to the Legendrian case, we define \emph{rulings} on the torus and show that there exists a ruling decomposition:
\[
|\calM_{\triangle, \eta}| = \coprod_{\rho \in \fR } \bC^{r(\rho)} \times (\bC^*)^{2g(\rho)}
\]
Thus, we can compute \(N_\triangle^\mathrm{Coh}\) using combinatorial data on the A-side.

\subsection{Rulings}
\label{sec:rulings}
Let us define rulings on \(T^2\). We can do this by lifting to \(\bR^2\). The difference from the Legendrian link case is that there are  no cusps. However, the singular support could point in any direction.

Let us denote 
\[
\begin{tikzcd}
S^*\bR^2 \arrow[r, "\widetilde{p}"] \arrow[d, "\widetilde{q}"] & S^*\T \arrow[d, "q"] \\
\bR^2 \arrow[r, "p"] & \T
\end{tikzcd}
\]
and \(\underline{\Lambda} = q (\Lambda)\),  \(\widetilde{\Lambda} = \widetilde p^{-1}(\Lambda)\), \( \widetilde{\underline{\Lambda}} = p^{-1} (\underline{\Lambda})\). 


Pick a direction \( v \in \bZ^2\), such that $v$ is not tangent to any point in $ \underline{\Lambda}$. We rotate the diagram if necessary, such that $v$ is pointing upward. In particular, a point on \(\Lambda\) (and \(\widetilde{\Lambda}\)) is either pointing ``upward'' or ``downward''. After a small GKS perturbation, we  can assume that each vertical \(S^1\) is not tangent to $ \underline \Lambda$ and  contains at most one crossing. 


\begin{remark}
    In many examples, it is convenient to isotope $\Lambda$ to the zig-zag of some dimer model, as in \cite{TWZ, GK}. However, we do not know if we can always do this isotopy and find the direction $v$ simultaneously. In any case, our results do not require \(\Lambda\) to arise from the zig-zag paths of a dimer model.
\end{remark}

\begin{definition}\label{def:rulings on T}
	A \emph{ruling} (with respect to \(v\) and $\mathfrak{t}$) is a \(\bZ^2\)-equivariant decomposition of \(\widetilde{{\Lambda}}\) into pairs of paths.  Each pair of paths is colored either red or blue, and bounds a region of the same color, satisfying: 
    \begin{enumerate}
  
     \item    Locally at a smooth point of the edge, only one of the following could happen:
    
    \begin{tikzpicture}
    
        \filldraw[blue, opacity=.4] (0, 0) rectangle (2, -1); 

        \draw[blue, line width = 2pt] (0, 0) -- (2, 0); 
        \draw (0.5, 0) -- (.5, -.2);
        \draw (1, 0) -- (1, -.2);
        \draw (1.5, 0) -- (1.5, -.2);

        \begin{scope}[shift = { ( 3, 0)} ]
        
            \filldraw[blue, opacity=.4] (0, 0) rectangle (2, -1); 
            \draw[blue, line width = 2pt] (0, -1) -- (2, -1); 
            \draw (0.5, -1) -- (.5, -.8);
            \draw (1, -1) -- (1, -.8);
            \draw (1.5, -1) -- (1.5, -.8);

        \end{scope}
        
        \begin{scope}[shift = { ( 6, 0)} ]
        \filldraw[red, opacity=.4] (0, 0) rectangle (2, -1);
        \draw[red, line width = 2pt] (0, 0) -- (2, 0); 
        \draw (0.5, 0) -- (.5, .2);
        \draw (1, 0) -- (1, .2);
        \draw (1.5, 0) -- (1.5, .2);

        \end{scope}

        \begin{scope}[shift = { ( 9, 0)} ]
        \filldraw[red, opacity=.4] (0, 0) rectangle (2, -1); 
        \draw[red, line width = 2pt] (0, -1) -- (2, -1); 
        \draw (0.5, -1) -- (.5, -1.2);
        \draw (1, -1) -- (1, -1.2);
        \draw (1.5, -1) -- (1.5, -1.2);
        \end{scope}
    \end{tikzpicture}\\
    In other words, the codirection induced by \(\widetilde{\Lambda}\) points inward along the boundary of blue regions and outward along the boundary of red regions. This also indicates that a downward pointing path can only pair with an upward pointing path.

    \item Locally at a crossing, the ruling configuration looks like one of the following:

    \begin{enumerate}[(i)] \label{enumerate:ruling at a crossing}

    \item \emph{Overlaps} 
    
    \begin{tikzpicture}[yscale = .4, xscale = 1.3]
        \draw[line width = 2pt, blue] (-1, 1) -- (1, -1);
        \draw[line width = 2pt, red] (-1, -1) -- (1, 1); 
        \fill[red, opacity = .4] (-1, -1) -- (-1, -3) -- (1, -3) -- (1, 1); 
        \fill[pattern={Lines[angle=90, distance=1.2mm, line width=0.5pt]}, pattern color=blue] (-1, -3) -- (-1, 1) -- (1, -1) -- (1, -3);
    \end{tikzpicture}
    \qquad
    \begin{tikzpicture}[yscale = .4, xscale = -1.3]
        \draw[line width = 2pt, blue] (-1, 1) -- (1, -1);
        \draw[line width = 2pt, red] (-1, -1) -- (1, 1); 
        \fill[red, opacity = .4] (-1, -1) -- (-1, -3) -- (1, -3) -- (1, 1); 
        \fill[pattern={Lines[angle=90, distance=1.2mm, line width=0.5pt]}, pattern color=blue] (-1, -3) -- (-1, 1) -- (1, -1) -- (1, -3);
    \end{tikzpicture}
    \qquad
    \begin{tikzpicture}[yscale = -.4, xscale = 1.3]
        \draw[line width = 2pt, blue] (-1, 1) -- (1, -1);
        \draw[line width = 2pt, red] (-1, -1) -- (1, 1); 
        \fill[red, opacity = .4] (-1, -1) -- (-1, -3) -- (1, -3) -- (1, 1); 
        \fill[pattern={Lines[angle=90, distance=1.2mm, line width=0.5pt]}, pattern color=blue] (-1, -3) -- (-1, 1) -- (1, -1) -- (1, -3);
    \end{tikzpicture}
    \qquad
    \begin{tikzpicture}[yscale = -.4, xscale = -1.3]
        \draw[line width = 2pt, blue] (-1, 1) -- (1, -1);
        \draw[line width = 2pt, red] (-1, -1) -- (1, 1); 
        \fill[red, opacity = .4] (-1, -1) -- (-1, -3) -- (1, -3) -- (1, 1); 
        \fill[pattern={Lines[angle=90, distance=1.2mm, line width=0.5pt]}, pattern color=blue] (-1, -3) -- (-1, 1) -- (1, -1) -- (1, -3);
    \end{tikzpicture}

    \vspace{1em}

    \begin{tikzpicture}[yscale = -.4, xscale = -1.3]
        \draw[line width = 2pt, Blue] (-1, 1) -- (1, -1);
        \draw[line width = 2pt, blue] (-1, -1) -- (1, 1); 
        \fill[pattern={Lines[angle=90,distance=1mm,line width=0.3pt]}, pattern color=blue] (-1, -1) -- (-1, 2) -- (1, 2) -- (1, 1); 
        \fill[Blue, opacity = .4] (-1, -2) -- (-1, 1) -- (1, -1) -- (1, -2); 
    \end{tikzpicture}
    \qquad
    \begin{tikzpicture}[yscale = -.4, xscale = -1.3]
        \draw[line width = 2pt, red!70!black] (-1, 1) -- (1, -1);
        \draw[line width = 2pt, red] (-1, -1) -- (1, 1); 
        \fill[pattern={Lines[angle=90,distance=1mm,line width=0.3pt]}, pattern color=red] (-1, -1) -- (-1, 2) -- (1, 2) -- (1, 1); 
        \fill[red!70!black, opacity = .4] (-1, -2) -- (-1, 1) -- (1, -1) -- (1, -2); 
    \end{tikzpicture}

    \item \emph{Departures}
    
    \begin{tikzpicture}[yscale = .4, xscale = 1.3]
        \draw[line width = 2pt, Blue] (-1, 1) -- (1, -1);
        \draw[line width = 2pt, blue] (-1, -1) -- (1, 1); 
        \draw[line width = 2pt, Blue] (-1, -2) -- (1, -2); 
        \draw[line width = 2pt, blue] (-1, -3) -- (1, -3); 
        \fill[pattern={Lines[angle=90,distance=1mm,line width=0.3pt]}, pattern color=blue] (-1, -1) -- (-1, -3) -- (1, -3) -- (1, 1); 
        \fill[Blue, opacity = .4] (-1, -2) -- (-1, 1) -- (1, -1) -- (1, -2); 
    \end{tikzpicture}
    \qquad
    \begin{tikzpicture}[yscale = -.4, xscale = 1.3]
        \draw[line width = 2pt, Blue] (-1, 1) -- (1, -1);
        \draw[line width = 2pt, blue] (-1, -1) -- (1, 1); 
        \draw[line width = 2pt, Blue] (-1, -2) -- (1, -2); 
        \draw[line width = 2pt, blue] (-1, -3) -- (1, -3); 
        \fill[pattern={Lines[angle=90,distance=1mm,line width=0.3pt]}, pattern color=blue] (-1, -1) -- (-1, -3) -- (1, -3) -- (1, 1); 
        \fill[Blue, opacity = .4] (-1, -2) -- (-1, 1) -- (1, -1) -- (1, -2); 
    \end{tikzpicture}
    \qquad
    \begin{tikzpicture}[yscale = .4, xscale = 1.3]
        \draw[line width = 2pt, red!70!black] (-1, 1) -- (1, -1);
        \draw[line width = 2pt, red] (-1, -1) -- (1, 1); 
        \draw[line width = 2pt, red!70!black] (-1, -2) -- (1, -2); 
        \draw[line width = 2pt, red] (-1, -3) -- (1, -3); 
        \fill[pattern={Lines[angle=90,distance=1mm,line width=0.3pt]}, pattern color=red] (-1, -1) -- (-1, -3) -- (1, -3) -- (1, 1); 
        \fill[red!70!black, opacity = .4] (-1, -2) -- (-1, 1) -- (1, -1) -- (1, -2); 
    \end{tikzpicture}
    \qquad
    \begin{tikzpicture}[yscale = -.4, xscale = 1.3]
        \draw[line width = 2pt, red!70!black] (-1, 1) -- (1, -1);
        \draw[line width = 2pt, red] (-1, -1) -- (1, 1); 
        \draw[line width = 2pt, red!70!black] (-1, -2) -- (1, -2); 
        \draw[line width = 2pt, red] (-1, -3) -- (1, -3); 
        \fill[pattern={Lines[angle=90,distance=1mm,line width=0.3pt]}, pattern color=red] (-1, -1) -- (-1, -3) -- (1, -3) -- (1, 1); 
        \fill[red!70!black, opacity = .4] (-1, -2) -- (-1, 1) -- (1, -1) -- (1, -2); 
    \end{tikzpicture}
    
    \vspace{1em}

    \begin{tikzpicture}[yscale = .4, xscale = -1.3]
        \draw[line width = 2pt, red] (-1, 1) -- (1, -1);
        \draw[line width = 2pt, blue] (-1, -1) -- (1, 1); 
        \fill[blue, opacity = .4] (-1, -1) -- (-1, 2) -- (1, 2) -- (1, 1); 
        \fill[red, opacity = .4] (-1, -2) -- (-1, 1) -- (1, -1) -- (1, -2); 
    \end{tikzpicture}
    \qquad
    \begin{tikzpicture}[yscale = -.4, xscale = -1.3]
        \draw[line width = 2pt, red] (-1, 1) -- (1, -1);
        \draw[line width = 2pt, blue] (-1, -1) -- (1, 1); 
        \fill[blue, opacity = .4] (-1, -1) -- (-1, 2) -- (1, 2) -- (1, 1); 
        \fill[red, opacity = .4] (-1, -2) -- (-1, 1) -- (1, -1) -- (1, -2); 
    \end{tikzpicture}
    \qquad

    \item \emph{Returns}:
   
    \begin{tikzpicture}[yscale = .4, xscale = -1.3]
        \draw[line width = 2pt, blue] (-1, 1) -- (1, -1);
        \draw[line width = 2pt, Blue] (-1, -1) -- (1, 1); 
        \draw[line width = 2pt, blue] (-1, -2) -- (1, -2); 
        \draw[line width = 2pt, Blue] (-1, -3) -- (1, -3); 
        \fill[Blue, opacity = .4] (-1, -1) -- (-1, -3) -- (1, -3) -- (1, 1); 
        \fill[pattern={Lines[angle=90,distance=1mm,line width=0.3pt]}, pattern color=Blue] (-1, -2) -- (-1, 1) -- (1, -1) -- (1, -2); 
    \end{tikzpicture}
    \qquad
    \begin{tikzpicture}[yscale = -.4, xscale = -1.3]
        \draw[line width = 2pt, blue] (-1, 1) -- (1, -1);
        \draw[line width = 2pt, Blue] (-1, -1) -- (1, 1); 
        \draw[line width = 2pt, blue] (-1, -2) -- (1, -2); 
        \draw[line width = 2pt, Blue] (-1, -3) -- (1, -3); 
        \fill[Blue, opacity = .4] (-1, -1) -- (-1, -3) -- (1, -3) -- (1, 1); 
        \fill[pattern={Lines[angle=90,distance=1mm,line width=0.3pt]}, pattern color=Blue] (-1, -2) -- (-1, 1) -- (1, -1) -- (1, -2); 
    \end{tikzpicture}
    \qquad
    \begin{tikzpicture}[yscale = .4, xscale = -1.3]
        \draw[line width = 2pt, red] (-1, 1) -- (1, -1);
        \draw[line width = 2pt, red!70!black] (-1, -1) -- (1, 1); 
        \draw[line width = 2pt, red] (-1, -2) -- (1, -2); 
        \draw[line width = 2pt, red!70!black] (-1, -3) -- (1, -3); 
        \fill[red!70!black, opacity = .4](-1, -1) -- (-1, -3) -- (1, -3) -- (1, 1); 
        \fill[pattern={Lines[angle=90,distance=1mm,line width=0.3pt]}, pattern color=red](-1, -2) -- (-1, 1) -- (1, -1) -- (1, -2); 
    \end{tikzpicture}
    \qquad
    \begin{tikzpicture}[yscale = -.4, xscale = -1.3]
        \draw[line width = 2pt, red] (-1, 1) -- (1, -1);
        \draw[line width = 2pt, red!70!black] (-1, -1) -- (1, 1); 
        \draw[line width = 2pt, red] (-1, -2) -- (1, -2); 
        \draw[line width = 2pt, red!70!black] (-1, -3) -- (1, -3); 
        \fill[red!70!black, opacity = .4](-1, -1) -- (-1, -3) -- (1, -3) -- (1, 1); 
        \fill[pattern={Lines[angle=90,distance=1mm,line width=0.3pt]}, pattern color=red](-1, -2) -- (-1, 1) -- (1, -1) -- (1, -2); 
    \end{tikzpicture}

    \vspace{1em}

    \begin{tikzpicture}[yscale = .4, xscale = 1.3]
        \draw[line width = 2pt, red] (-1, 1) -- (1, -1);
        \draw[line width = 2pt, blue] (-1, -1) -- (1, 1); 
        \fill[blue, opacity = .4] (-1, -1) -- (-1, 2) -- (1, 2) -- (1, 1); 
        \fill[red, opacity = .4] (-1, -2) -- (-1, 1) -- (1, -1) -- (1, -2); 
    \end{tikzpicture}
    \qquad
    \begin{tikzpicture}[yscale = -.4, xscale = 1.3]
        \draw[line width = 2pt, red] (-1, 1) -- (1, -1);
        \draw[line width = 2pt, blue] (-1, -1) -- (1, 1); 
        \fill[blue, opacity = .4] (-1, -1) -- (-1, 2) -- (1, 2) -- (1, 1); 
        \fill[red, opacity = .4] (-1, -2) -- (-1, 1) -- (1, -1) -- (1, -2); 
    \end{tikzpicture}

        \item Switches of type A:

        \label{item:crossing_b_meets_r} 
        
        \begin{tikzpicture}
        \begin{scope}[rotate = 45]
            \draw[line width = 2pt, red] (-1, 0) -- (0, 0) -- (0, -1);
            \draw[line width = 2pt, blue] (0, 1) -- (0, 0) -- (1, 0); 
            \draw (.5, 0) -- (.5, .2);
            \draw (-.5, 0) -- (-.5, .2);
            \draw (0, .5) -- (.2, .5);
            \draw (0, -.5) -- (.2, -.5);
            \filldraw[blue, opacity=.4] (0, 0) -- (1, 0) -- (0, 1); 
            \filldraw[red, opacity=.4] (0, 0) --(-1, 0) -- (0, -1);
        \end{scope}
        \begin{scope}[shift = {(3, 0)}, rotate = 225]
            \draw[line width = 2pt, red] (-1, 0) -- (0, 0) -- (0, -1);
            \draw[line width = 2pt, blue] (0, 1) -- (0, 0) -- (1, 0);
            \draw (.5, 0) -- (.5, .2);
            \draw (-.5, 0) -- (-.5, .2);
            \draw (0, .5) -- (.2, .5);
            \draw (0, -.5) -- (.2, -.5);
            \filldraw[blue, opacity=.4] (0, 0) -- (1, 0) -- (0,1); 
            \filldraw[red, opacity=.4] (0, 0) -- (-1, 0) -- (0, -1);
        \end{scope}
        \begin{scope}[shift = {(6, 0)}, rotate = 135]
            \draw[line width = 2pt, red] (-1, 0) -- (0, 0) -- (0, -1);
            \draw[line width = 2pt, blue] (0, 1) -- (0, 0) -- (1, 0);
            \draw (.5, 0) -- (.5, .2);
            \draw (-.5, 0) -- (-.5, .2);
            \draw (0, .5) -- (.2, .5);
            \draw (0, -.5) -- (.2, -.5);
            \filldraw[blue, opacity=.4] (0, 0) -- (1, 0) -- (0,1); 
            \filldraw[red, opacity=.4] (0, 0) -- (-1, 0) -- (0, -1);
        \end{scope}
        \begin{scope}[shift = {(9, 0)}, rotate = -45]
            \draw[line width = 2pt, red] (-1, 0) -- (0, 0) -- (0, -1);
            \draw[line width = 2pt, blue] (0, 1) -- (0, 0) -- (1, 0);
            \draw (.5, 0) -- (.5, .2);
            \draw (-.5, 0) -- (-.5, .2);
            \draw (0, .5) -- (.2, .5);
            \draw (0, -.5) -- (.2, -.5);
            \filldraw[blue, opacity=.4] (0, 0) -- (1, 0) -- (0,1); 
            \filldraw[red, opacity=.4] (0, 0) -- (-1, 0) -- (0, -1);
        \end{scope}
        \end{tikzpicture}.
        
        We call the first two switches \emph{vertical}, and the last two \emph{horizontal}. 
        \item\label{item:switch}
        
        \vspace{1em}
        Switches of type B:

\begin{tikzpicture}[xscale = .7, yscale = .4]
    \draw[blue, line width = 2pt] (-2, -1) -- (0, 0) -- (2, -1); 
    \draw[blue, line width = 2pt] (-2, -3) -- (2, -3); 
    \fill[pattern={Lines[angle=90,distance=1mm,line width=0.3pt]}, pattern color = blue] (-2, -1) -- (-2, -3) -- (2, -3) -- (2, -1) -- (0, 0); 
    \draw[Blue, line width = 2pt] (-2, 1) -- (0, 0) -- (2,1); 
    \draw[Blue, line width = 2pt] (-2, -2) -- (2, -2); 
    \fill[Blue, opacity = .4] (-2, 1) -- (-2, -2) -- (2, -2) -- (2, 1) -- (0, 0);
\end{tikzpicture}
\qquad
\begin{tikzpicture}[xscale = .7, yscale = .4]

    \begin{scope}[shift = {(6, -2)}, rotate = 180]
    \draw[blue, line width = 2pt] (-2, -1) -- (0, 0) -- (2, -1); 
    \draw[blue, line width = 2pt] (-2, -3) -- (2, -3); 
    \fill[pattern={Lines[angle=90,distance=1mm,line width=0.3pt]}, pattern color = blue] (-2, -1) -- (-2, -3) -- (2, -3) -- (2, -1) -- (0, 0); 
    \draw[Blue, line width = 2pt] (-2, 1) -- (0, 0) -- (2,1); 
    \draw[Blue, line width = 2pt] (-2, -2) -- (2, -2); 
    \fill[Blue, opacity = .4] (-2, 1) -- (-2, -2) -- (2, -2) -- (2, 1) -- (0, 0);
    \end{scope}
\end{tikzpicture}
\qquad
\begin{tikzpicture}[xscale = .7, yscale = .4]
    \draw[red, line width = 2pt] (-2, -1) -- (0, 0) -- (2, -1); 
    \draw[red, line width = 2pt] (-2, -3) -- (2, -3); 
    \fill[pattern={Lines[angle=90,distance=1mm,line width=0.3pt]}, pattern color = red] (-2, -1) -- (-2, -3) -- (2, -3) -- (2, -1) -- (0, 0); 
    \draw[red!70!black, line width = 2pt] (-2, 1) -- (0, 0) -- (2,1); 
    \draw[red!70!black, line width = 2pt] (-2, -2) -- (2, -2); 
    \fill[red!70!black, opacity = .4] (-2, 1) -- (-2, -2) -- (2, -2) -- (2, 1) -- (0, 0);
\end{tikzpicture}
\qquad
\begin{tikzpicture}[xscale = .7, yscale = -.4]
    \draw[red, line width = 2pt] (-2, -1) -- (0, 0) -- (2, -1); 
    \draw[red, line width = 2pt] (-2, -3) -- (2, -3); 
    \fill[pattern={Lines[angle=90,distance=1mm,line width=0.3pt]}, pattern color = red] (-2, -1) -- (-2, -3) -- (2, -3) -- (2, -1) -- (0, 0); 
    \draw[red!70!black, line width = 2pt] (-2, 1) -- (0, 0) -- (2,1); 
    \draw[red!70!black, line width = 2pt] (-2, -2) -- (2, -2); 
    \fill[red!70!black, opacity = .4] (-2, 1) -- (-2, -2) -- (2, -2) -- (2, 1) -- (0, 0);
\end{tikzpicture}

    \end{enumerate}

    \item Every path points either upward or downward throughout.  Assume all paths go from left to right. 
    We call the region bounded by a pair of paths an \emph{eye}\footnote{We borrow this name from the rulings of Legendrian links, though in this case such a region does not necessarily look like an eye.}.
    Then each eye is either
    \begin{enumerate}
        \item 
    unbounded and retracts to a horizontal line, or 
    \item 
\[   
    \begin{cases}
    \text{starts with }
\begin{tikzpicture}[baseline={(current bounding box.center)}]
        \begin{scope}[ rotate = -45]
            \draw[line width = 2pt, red] (-.5, 0) -- (0, 0) -- (0, -.5);
            \draw[line width = 2pt, blue] (0, 1) -- (0, 0) -- (1, 0);
            
            \filldraw[blue, opacity=.4] (0, 0) -- (1, 0) -- (0,1); 
            \filldraw[red, opacity=.4] (0, 0) -- (-.5, 0) -- (0, -.5);
        \end{scope}
        \end{tikzpicture}
        \text{ and ends with }
        \begin{tikzpicture}[baseline={(current bounding box.center)}]
            \begin{scope}[rotate = 135]
            \draw[line width = 2pt, red] (-.5, 0) -- (0, 0) -- (0, -.5);
            \draw[line width = 2pt, blue] (0, 1) -- (0, 0) -- (1, 0);
            
            \filldraw[blue, opacity=.4] (0, 0) -- (1, 0) -- (0,1); 
            \filldraw[red, opacity=.4] (0, 0) -- (-.5, 0) -- (0, -.5);
        \end{scope}
        \end{tikzpicture}, 
        & \text{if it is blue;}\\[30pt]
\text{starts with }  
        \begin{tikzpicture}[baseline={(current bounding box.center)}]
            \begin{scope}[rotate = 135]
            \draw[line width = 2pt, red] (-1, 0) -- (0, 0) -- (0, -1);
            \draw[line width = 2pt, blue] (0, .5) -- (0, 0) -- (.5, 0);
            
            \filldraw[blue, opacity=.4] (0, 0) -- (.5, 0) -- (0, .5); 
            \filldraw[red, opacity=.4] (0, 0) -- (-1, 0) -- (0, -1);
        \end{scope}
        \end{tikzpicture} 
        \text { and ends with } 
        \begin{tikzpicture}[baseline={(current bounding box.center)}]
        \begin{scope}[ rotate = -45]
            \draw[line width = 2pt, red] (-1, 0) -- (0, 0) -- (0, -1);
            \draw[line width = 2pt, blue] (0, .5) -- (0, 0) -- (.5, 0);
            
            \filldraw[blue, opacity=.4] (0, 0) -- (.5, 0) -- (0, .5); 
            \filldraw[red, opacity=.4] (0, 0) -- (-1, 0) -- (0, -1);
        \end{scope}
        \end{tikzpicture}, 
     & \text{if it is red.}
\end{cases}
\]
 \end{enumerate}
 In the case (a) we call the eye an \emph{annular eye}, and in case (b) we call the eye a \emph{disk eye}.

\item  (The balancing condition) Over the point $\mathfrak{t}$, the number of blue eyes over it is $d$ plus the number of red eyes over it. 

\end{enumerate}

\begin{remark}\label{rmk:crossing_types}
    The crossings of $\Lambda$ can be classified into two types by the codirection:
    \begin{enumerate}
        \item crossings of same codirection:
        \begin{tikzpicture}[xscale = .6, yscale = .3]
            \draw[thick] (-1, -1) -- (1, 1);
            \draw[thick] (-1, 1) -- (1, -1);
            \draw (-.5, -.5) -- (-.4, -.9);
            \draw (.5, -.5) -- (.4, -.9);
            \draw (.5, .5) -- (.6, .1);
            \draw (-.5, .5) -- (-.6, .1);
        \end{tikzpicture}
        \quad or \quad
        \begin{tikzpicture}[xscale = .6, yscale = -.3]
            \draw[thick] (-1, -1) -- (1, 1);
            \draw[thick] (-1, 1) -- (1, -1);
            \draw (-.5, -.5) -- (-.4, -.9);
            \draw (.5, -.5) -- (.4, -.9);
            \draw (.5, .5) -- (.6, .1);
            \draw (-.5, .5) -- (-.6, .1);
        \end{tikzpicture}.
        \vspace{1em}
        \item crossings of opposite codirections:
        \begin{tikzpicture}[xscale = .6, yscale = .3]
            \draw[thick] (-1, -1) -- (1, 1);
            \draw[thick] (-1, 1) -- (1, -1);
            \draw (-.6, -.6) -- (-.5, -1);
            \draw (-.6, .6) -- (-.5, 1);
            
            \draw (.6, -.6) -- (.7, -.2);
            \draw (.6, .6) -- (.7, .2);
        \end{tikzpicture}
        \quad or \quad
        \begin{tikzpicture}[xscale = -.6, yscale = .3]
            \draw[thick] (-1, -1) -- (1, 1);
            \draw[thick] (-1, 1) -- (1, -1);
            \draw (-.6, -.6) -- (-.5, -1);
            \draw (-.6, .6) -- (-.5, 1);
            
            \draw (.6, -.6) -- (.7, -.2);
            \draw (.6, .6) -- (.7, .2);
        \end{tikzpicture}.
    \end{enumerate}
    One observes that, at a crossing of the same codirection, a ruling can be a departure, a reture, a vertical switch of type A, or a switch of type B;  at an crossing of opposite codirections, a ruling can be an overlap, or a horizontal switch of type A. 
    
\end{remark}

\begin{remark}
Readers familiar with rulings of Legendrian links may notice that our definitions of returns, departures, and the normality condition for type B switches are slightly different from their counterparts in the Legendrian-link setting. This is because, in our setting, the codirections of the singular support may point both upward and downward.
\end{remark}

\end{definition}

\begin{definition}[The ruling surface and its skeleton] Let \(\rho\) be a ruling of \(\Lambda\). Following \cite{Kal08, STZ}, we associate an (orientable) topological surface \(\widetilde{\Sigma}(\rho)\) whose boundary can be identified with \(\widetilde{\Lambda}\), which descends to a topological surface \(\Sigma(\rho)\) whose boundary can be identified with \(\Lambda\). The construction is similar to the Legendrian link case. First take the disjoint union of all eyes \(\coprod_\alpha E_\alpha \). At each switch (\eqref{item:switch} or~\eqref{item:crossing_b_meets_r} in Definition~\ref{def:rulings on T}), glue the incident eyes by a half-twisted strip:

\begin{tikzpicture}[ line width = 1pt, xscale = .7, yscale = .5]

    \draw[blue] (-2, 2) -- (0, 0)-- (-2, -2);
    
    \fill[blue, opacity = .4] (-2, 2) -- (0, 0)-- (-2, -2);

    \draw[red] (2.4, -2) -- (0.4,0) -- (2.4, 2);
    \filldraw [red, opacity = .4] (2.4, -2) -- (0.4,0) -- (2.4, 2);

    \draw (-.3, .3) -- (.7, -.3);
    \draw (-.3, -.3) -- (.7, .3);

\end{tikzpicture}
\qquad
\begin{tikzpicture}[line width = 1pt, xscale = .6, yscale = .8]
    \draw[blue] (-.5, .5) -- (1, -1) -- (2, 0) -- (3, -1) -- (4.5, .5);
    \draw[Blue] (1,1.4) -- (2,0.4) -- (3, 1.4);
    \draw (1.7, .7) -- (2.3, -.3);
    \draw (2.3, .7) -- (1.7, -.3);
    \fill[blue, opacity=.3] (-1, 1) -- (1, -1) -- (2, 0) -- (3, -1) -- (5, 1);
    \fill[Blue, opacity = .5] (1, 1.4) -- (2, .4) -- (3, 1.4);
\end{tikzpicture}
\qquad
\begin{tikzpicture}[line width = 1pt, xscale = .6, yscale = .8]
    \draw[red] (-.5, .5) -- (1, -1) -- (2, 0) -- (3, -1) -- (4.5, .5);
    \draw[red!70!black] (1,1.4) -- (2,0.4) -- (3, 1.4);
    \draw (1.7, .7) -- (2.3, -.3);
    \draw (2.3, .7) -- (1.7, -.3);
    \fill[red, opacity=.3] (-1, 1) -- (1, -1) -- (2, 0) -- (3, -1) -- (5, 1);
    \fill[red!70!black, opacity = .5] (1, 1.4) -- (2, .4) -- (3, 1.4);
\end{tikzpicture}
.

To understand the topology of the ruling surface more easily, we also define the \emph{skeleton} of the topological surface \(\widetilde{\Sigma}(\rho)\). The skeleton is a \(\bZ^2\)-equivariant graph  \( \widetilde{\Gamma}(\rho)\) to which  \(\widetilde{\Sigma}(\rho)\) retracts. It descends to a graph \(\Gamma(\rho)\), which we call the \emph{skeleton} of \(\Sigma(\rho)\). 
The construction is as follows:
\begin{enumerate}
    \item Each disk eye retracts to a point in the interior. Take such a point and call it the \emph{core} of the eye. 
    \item Each annular eye retracts to a line, homeomorphic to \(\bR\). Take such a line and call it the \emph{core} of the eye. 
    \item If two eyes meet at a switch, then we put an edge connecting the corresponding two cores. Note that how we draw this edge does not affect the topological type of the resulting graph. 
\end{enumerate}
\end{definition}

We will draw a ruling on the torus \(\T\) when there is no ambiguity in lifting it to \(\bR^2\). 
For a ruling $\rho$, denote by $g(\rho)$ the genus of the ruling surface $\Sigma(\rho)$, and by $r(\rho)$ the number of returns.

\begin{example}
\label{ex:o22stdruling}
     Assume $\Lambda$ is the zig-zag of a dimer model. 
     Then we can associate a ruling called the \emph{standard ruling}, as follows. Take the blue eyes to be the regions that contain black vertices, and the red eyes to be the regions that contain white vertices. The ruling surface is homeomorphic to the spectral curve, and can be realized as a Lagrangian filling of \(\Lambda\) (\cite{TWZ}). The skeleton of the ruling surface is exactly the bipartite graph.  An example is illustrated in Figure~\ref{fig:newtonpoly_bipartite_ruling}. (Technically, we should perturb the Legendrian a bit, such that the crossing points do not sit on the same vertical line.)  
\end{example}

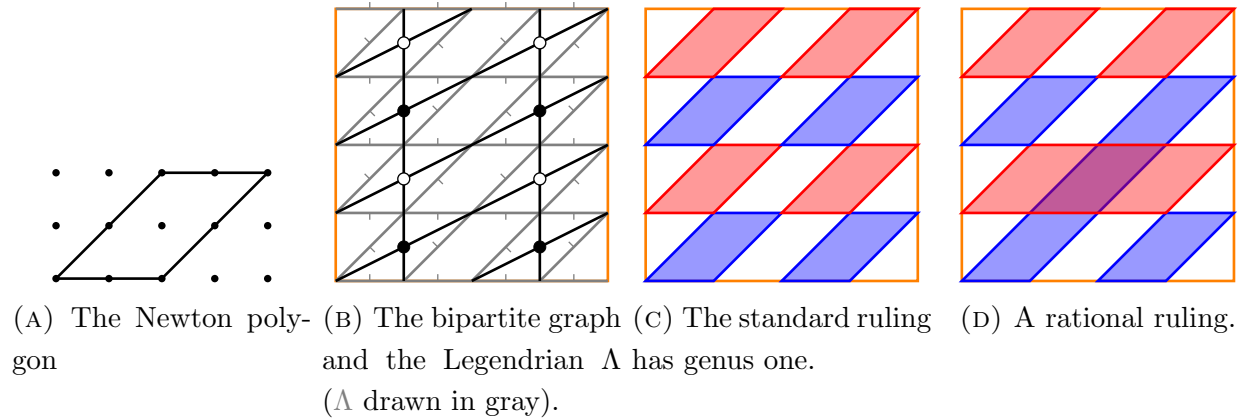
\begin{figure}[htbp] 
    \centering
    \begin{subfigure}[t]{.24\linewidth}
    \centering
    \begin{tikzpicture}[scale = .7, line width = 1pt]
  \foreach \x in {0,1,2,3, 4}
    \foreach \y in {0,1,2}
      \fill (\x,\y) circle (2pt);
      
      \draw (0, 0) -- (2, 0) -- (4, 2) -- (2, 2) -- cycle;

\end{tikzpicture}
    \caption{The Newton polygon}
    \end{subfigure}
\begin{subfigure}[t]{.24\linewidth}
    \centering
    \begin{tikzpicture}[scale = .45, line width = 1pt]
    \draw[ color = orange] (0,0) rectangle(8, 8);
    \draw[gray] (0, 0) -- (8, 0);
    \draw[gray] (0, 2) -- (8, 2);
    \draw[gray] (0, 4) -- (8, 4); 
    \draw[gray] (0, 6) -- (8, 6);
    \draw[gray] (0, 8) -- (8, 8);

    \draw[gray] (0, 0) -- (8, 8);
    \draw[gray] (2, 0) -- (8, 6);
    \draw[gray] (4, 0) -- (8, 4);
    \draw[gray] (6, 0) -- (8, 2);

    \draw[gray] (0, 2) -- (6, 8);
    \draw[gray] (0, 4) -- (4, 8);
    \draw[gray] (0, 6) -- (2, 8);

    \draw[gray, line width = .5pt] (1, 0) -- (1, .2);
    \draw[gray, line width = .5pt] (3, 0) -- (3, .2);
    \draw[gray, line width = .5pt] (5, 0) -- (5, .2);
    \draw[gray, line width = .5pt] (7, 0) -- (7, .2);

    \draw[gray, line width = .5pt] (1, 2) -- (1, 1.8);
    \draw[gray, line width = .5pt] (3, 2) -- (3, 1.8);
    \draw[gray, line width = .5pt] (5, 2) -- (5, 1.8);
    \draw[gray, line width = .5pt] (7, 2) -- (7, 1.8);
    
    \draw[gray, line width = .5pt] (1, 4) -- (1, 4.2);
    \draw[gray, line width = .5pt] (3, 4) -- (3, 4.2);
    \draw[gray, line width = .5pt] (5, 4) -- (5, 4.2);
    \draw[gray, line width = .5pt] (7, 4) -- (7, 4.2);
    
    \draw[gray, line width = .5pt] (1, 6) -- (1, 5.8);
    \draw[gray, line width = .5pt] (3, 6) -- (3, 5.8);
    \draw[gray, line width = .5pt] (5, 6) -- (5, 5.8);
    \draw[gray, line width = .5pt] (7, 6) -- (7, 5.8);

    \draw[gray, line width = .5pt] (1, 8) -- (1, 8.2);
    \draw[gray, line width = .5pt] (3, 8) -- (3, 8.2);
    \draw[gray, line width = .5pt] (5, 8) -- (5, 8.2);
    \draw[gray, line width = .5pt] (7, 8) -- (7, 8.2);

    \draw[gray, line width = .5pt] (1, 1) -- (1.15, .85);
    
    \draw[gray, line width = .5pt] (1, 3) -- (.85, 3.15);
    
    \draw[gray, line width = .5pt] (1, 5) -- (1.15, 4.85);
    \draw[gray, line width = .5pt] (1, 7) -- (.85, 7.15);

    \draw[gray, line width = .5pt] (5, 1) -- (5.15, .85);
    \draw[gray, line width = .5pt] (5, 3) -- (4.85, 3.15);
    \draw[gray, line width = .5pt] (5, 5) -- (5.15, 4.85);
    \draw[gray, line width = .5pt] (5, 7) -- (4.85, 7.15);

    \draw[gray, line width = .5pt] (3, 1) -- (2.85, 1.15); 
    \draw[gray, line width = .5pt] (3, 3) -- (3.15, 2.85); 
    \draw[gray, line width = .5pt] (3, 5) -- (2.85, 5.15); 
    \draw[gray, line width = .5pt] (3, 7) -- (3.15, 6.85); 

    \draw[gray, line width = .5pt] (7, 1) -- (6.85, 1.15); 
    \draw[gray, line width = .5pt] (7, 3) -- (7.15, 2.85); 
    \draw[gray, line width = .5pt] (7, 5) -- (6.85, 5.15); 
    \draw[gray, line width = .5pt] (7, 7) -- (7.15, 6.85); 

    \draw (2, 0) -- (2, 8);
    \draw (6, 0) -- (6, 8);
    \draw (0, 0) -- (8, 4);
    \draw (0, 2) -- (8, 6);
    \draw (0, 4) -- (8, 8);
    \draw (4, 0) -- (8, 2);
    \draw (0, 6) -- (4, 8);

    \filldraw[fill = black, draw = black, line width = .3 pt] (2, 1) circle (5pt); 
    \filldraw[fill = black, draw = black, line width = .3 pt] (6, 1) circle (5pt); 
    \filldraw[fill = black, draw = black, line width = .3 pt] (2, 5) circle (5pt); 
    \filldraw[fill = black, draw = black, line width = .3 pt] (6, 5) circle (5pt); 

    \filldraw[fill = white, draw = black, line width = .5pt] (2, 3) circle (5pt); 
    \filldraw[fill = white, draw = black, line width = .5pt] (6, 3) circle (5pt); 
    \filldraw[fill = white, draw = black, line width = .5pt] (2, 7) circle (5pt); 
    \filldraw[fill = white, draw = black, line width = .5pt] (6, 7) circle (5pt);

\end{tikzpicture}
    \caption{The bipartite graph and the Legendrian \(\Lambda\) (\textcolor{gray}{\(\Lambda\)} drawn in gray).}
\end{subfigure}
\begin{subfigure}[t]{.24\linewidth}
    \centering
    \begin{tikzpicture}[scale = .45]
    \draw[line width = 1pt, color = orange] (0,0) rectangle(8, 8);
    \filldraw[draw = blue, line width = 1pt, fill = blue, fill opacity = .4 ] (0, 0) -- (2,0) -- (4, 2) -- ( 2, 2) -- cycle; 
    \filldraw[draw = blue, line width = 1pt, fill = blue, fill opacity = .4 ] (4, 0) -- (6,0) -- (8, 2) -- (6, 2) -- cycle; 

    \filldraw[draw = red, line width = 1pt, fill = red, fill opacity = .4 ] (0, 2) -- (2, 2) -- (4, 4) -- ( 2, 4) -- cycle; 
    \filldraw[draw = red, line width = 1pt, fill = red, fill opacity = .4 ] (4, 2) -- (6,2) -- (8, 4) -- (6, 4) -- cycle; 

    \filldraw[draw = blue, line width = 1pt, fill = blue, fill opacity = .4 ] (0, 4) -- (2, 4) -- (4, 6) -- (2, 6) -- cycle; 
    \filldraw[draw = blue, line width = 1pt, fill = blue, fill opacity = .4 ] (4, 4) -- (6, 4) -- (8, 6) -- (6, 6) -- cycle; 

    \filldraw[draw = red, line width = 1pt, fill = red, fill opacity = .4 ] (0, 6) -- (2, 6) -- (4, 8) -- (2, 8) -- cycle; 
    \filldraw[draw = red, line width = 1pt, fill = red, fill opacity = .4 ] (4, 6) -- (6, 6) -- (8, 8) -- (6, 8) -- cycle; 
    
\end{tikzpicture}
    \caption{The standard ruling has genus one.}
\end{subfigure}
\begin{subfigure}[t]{.25\linewidth}
    \centering
    \begin{tikzpicture}[scale = .45]
    \draw[line width = 1pt, color = orange] (0,0) rectangle(8, 8);
    \filldraw[draw = blue, line width = 1pt, fill = blue, fill opacity = .4 ] (0, 0) -- (2, 0) -- (8, 6) -- (6, 6) -- cycle; 
    \filldraw[draw = blue, line width = 1pt, fill = blue, fill opacity = .4 ] (4, 0) -- (6,0) -- (8, 2) -- (6, 2) -- cycle; 

    \filldraw[draw = red, line width = 1pt, fill = red, fill opacity = .4 ] (0, 2) -- (6, 2) -- (8, 4) -- ( 2, 4) -- cycle;

    \filldraw[draw = blue, line width = 1pt, fill = blue, fill opacity = .4 ] (0, 4) -- (2, 4) -- (4, 6) -- (2, 6) -- cycle;

    \filldraw[draw = red, line width = 1pt, fill = red, fill opacity = .4 ] (0, 6) -- (2, 6) -- (4, 8) -- (2, 8) -- cycle; 
    \filldraw[draw = red, line width = 1pt, fill = red, fill opacity = .4 ] (4, 6) -- (6, 6) -- (8, 8) -- (6, 8) -- cycle; 
    
\end{tikzpicture}
    \caption{A rational ruling.}
\end{subfigure}
    \caption{Examples of rulings. The orange square is understood as the fundamental domain of a torus. The Newton polygon corresponds to the line bundle \(\cO(2, 2)\) on \(\bP^1 \times \bP^1\).} 
    \label{fig:newtonpoly_bipartite_ruling}
\end{figure}

\begin{example}
\label{ex:oablocal}
    Continuing with $\cO_{\bP^1\times \bP^1}(a,b)$ as in
    Example \ref{ex:o22stdruling} (but drawn perpendicular for clarity), the following local
    models of rulings are illustrative (a more detailed example can be found in Example~\ref{ex:o23}):
\[    \begin{tikzpicture}[scale = .6]
        \filldraw[draw = blue,fill=blue,fill opacity = .4] (1,0)--(2,0)--(2,3)--(1,3)--cycle;
        \filldraw[draw = red,fill=red,fill opacity = .4] (0,1)--(3,1)--(3,2)--(0,2)--cycle; 
        \node at (1.5,-1.5) {filling a hole;};
        \node at (1.5,-2.4) {genus drops by one}; 
        \begin{scope}[shift = {(6,0)}]
        \filldraw[draw = blue, fill = blue, fill opacity = .4]
        (1,1)--(4,1)--(4,4)--(1,4)--cycle;
        \filldraw[draw = red, fill = red, fill opacity = .4]
        (0,2)--(5,2)--(5,3)--(0,3)--cycle;
        \filldraw[draw = red, fill = red, fill opacity = .4]
        (2,0)--(2,5)--(3,5)--(3,0)--cycle;
        \node at (2.5,-1.5) {an ``Aztec diamond'';};
        \node at (2.5,-2.4) {genus drops by two};         
        \end{scope}
        \end{tikzpicture} \]
            
\end{example}

\begin{example}
A more complicated example of a ruling is shown in Figure~\ref{fig:example of ruling}. See more details in Example~\ref{ex:2c}.

\begin{figure}[htbp]
    \centering
    \begin{tikzpicture}
    \begin{scope}[shift = {(-1, 1)}, scale = .8] 
        \draw[line width = 1pt] (0, 0) -- (4, 0) -- (1, 2) -- cycle; 
        \fill (0,0) circle(2pt);
        \fill (1,0) circle(2pt);
        \fill (2,0) circle(2pt);
        \fill (3,0) circle(2pt);
        \fill (4,0) circle(2pt);
        \fill (0,1) circle(2pt);
        \fill (3,1) circle(2pt);
        \fill (1,1) circle(2pt);
        \fill (2,1) circle(2pt);
        \fill (4,1) circle(2pt);
        \fill (0,2) circle(2pt);
        \fill (1,2) circle(2pt);
        \fill (2,2) circle(2pt);
        \fill (4,2) circle(2pt);
        \fill (3,2) circle(2pt);
        
    \end{scope}
\end{tikzpicture}
\qquad \qquad
\begin{tikzpicture}
    \begin{scope}[shift = {(0, 0)}, scale = .7]
        \draw[line width = 1pt, color = orange] (0,0) rectangle(6, 6);

        \draw (0, 0) -- (6, 0);
        \draw (0, 6) -- (6, 6);
        \draw (0,1.5) -- (6, 1.5);
        \draw (0, 3) -- (6, 3);
        \draw (0, 4.5) -- (6, 4.5); 
        \draw (0, 0) -- (3, 6);
        \draw (3, 0) -- (6, 6);
        \draw (0, 3) -- (4.5, 0); 
        \draw (4.5, 6) -- (6, 5); 
        \draw (0, 5) -- (6, 1); 
        \draw (0, 1) -- (1.5, 0); 
        \draw (1.5, 6) -- (6, 3); 
        \fill[pattern={Lines[angle=90,distance=1mm,line width=0.3pt]}, pattern color = red] (0, 0) -- (4.5/4, 4.5/2) -- (4.5, 0) -- cycle;
        \filldraw[fill = red!70!black, opacity=.4] (0, 1.5) -- ( 6, 1.5) -- (6, 3) -- (2.625, 5.25) --(4.5/4, 4.5/2) -- (0, 3) -- cycle;

        \filldraw[ fill = red, opacity=.4] (0, 4.5) -- (.75, 4.5) --(0, 5) -- cycle;
        \filldraw[fill = red, opacity=.4] (5.25, 4.5) -- (6, 4.5) -- (6, 5) --(5.625, 5.25); 

        \draw[line width = 2pt, red!70!black]  (0, 1.5) -- ( 6, 1.5) ;
        \draw[line width = 2pt, red!70!black] (6, 3) -- (2.625, 5.25) --(4.5/4, 4.5/2) -- (0, 3) -- cycle;
        
        \draw[line width = 2pt, red] (0, 0) -- (4.5/4, 4.5/2) -- (4.5, 0) -- cycle;

        \draw[line width = 2pt, red] (0, 4.5) -- (.75, 4.5) --(0, 5);
        \draw[line width = 2pt, red] (5.25, 4.5) -- (6, 4.5) -- (5.25, 4.5) --(5.625, 5.25) -- (6, 5);

        \fill[pattern={Lines[angle=90,distance=1mm,line width=0.3pt]}, pattern color = blue] (.75, 4.5) -- (4.125, 2.25) -- (5.25, 4.5) -- cycle;
    
        \filldraw[fill = Blue, opacity=.5] (1.5, 0) -- (3, 0) -- (4.125, 2.25) -- (3,  3) -- (0, 3) -- (0, 1) --  cycle;
        \filldraw[fill = Blue, opacity=.5] (6, 1) -- (3, 3) -- (6, 3) -- cycle;
        
        \filldraw[fill = Blue, opacity=.5] (1.5, 6) -- (2.625, 5.25) -- (3, 6) -- cycle; 

        \filldraw[ fill = blue, opacity=.4] (4.5, 6) -- (6, 6) -- (5.625, 5.25); 

        \draw[line width = 2pt, Blue] (1.5, 6) -- (2.625, 5.25) --(3, 6);
        \draw[line width = 2pt, Blue] (0, 1) -- (1.5, 0); 
        \draw[line width = 2pt, Blue] (3, 0) -- (4.125, 2.25) -- (6, 1);
        \draw[line width = 2pt, Blue] (0, 3) -- (6,3); 
        \draw[blue, line width = 2 pt] (4.5, 6) -- (6, 6) -- (5.625, 5.25) -- cycle; 
        \draw[ blue, line width = 2pt] (.75, 4.5) -- (4.125, 2.25) -- (5.25, 4.5) -- cycle;

\end{scope}
\end{tikzpicture}
\qquad \qquad
\begin{tikzpicture}
    \begin{scope}[shift = {(10, 0)}, scale = .7]
    \draw[line width = 1pt, color = orange] (0,0) rectangle(6, 6);

    \draw[blue, line width = 1pt] (0, 2.7) -- (6, 2.7); 
    \coordinate (B1) at (3.7, 3.7); 
    \coordinate (B2) at (5.4, 5.7);  

    \draw[red, line width = 1pt] (0, 1.8) -- (6, 1.8);
    \coordinate (R1) at (1.5, 1);
   
    \coordinate (R2) at (5.8, 4.8);

    \draw[line width = 1pt] (1.5, 1) -- (1.5, 1.8);
    \draw[line width = 1pt] (0, 0) -- (R1); 
    \draw[line width = 1pt] (R2) -- (B2); 
    \draw[line width = 1pt] (R1) -- (4.5, 0); 
    \draw[line width = 1pt] (R2) -- (B1); 
    \draw[line width = 1pt] (R2) -- (6, 4.7);
    \draw[line width = 1pt] (0, 4.7) -- (B1); 
    
    \draw[line width = 1pt] (B2) -- (6, 6); 
    \draw[line width = 1pt] (B1) -- (3.7, 2.7);
    \draw[line width = 1pt] (B2) -- (4.5, 6); 

    \draw[line width = 1pt] (2.3, 2.7) -- (2.3, 1.9); 
    \draw[line width = 1pt] (2.3, 1.7) -- (2.3, .85);
    \draw[line width = 1pt] (2.3, .65) -- (2.3, 0); 

    \draw[line width = 1pt] (2.3, 6) .. controls (2.3, 5) and (3, 5) .. (3, 4); 
    \draw[line width = 1pt] (3, 3.7) -- (3, 2.8); 
    \draw[line width = 1pt] (3, 2.6) -- (3, 1.8); 

     \fill[red]  (R1) circle(2pt); 
     \fill[red]  (R2) circle(2pt); 
     \fill[blue] (B2) circle(2pt); 
     \fill[blue] (B1) circle(2pt);
     
\end{scope}

    \end{tikzpicture}

    \caption{An example of a ruling (middle) and the skeleton of the corresponding ruling surface (right). The Newton polygon is drawn on the left.}
    \label{fig:example of ruling}
\end{figure}
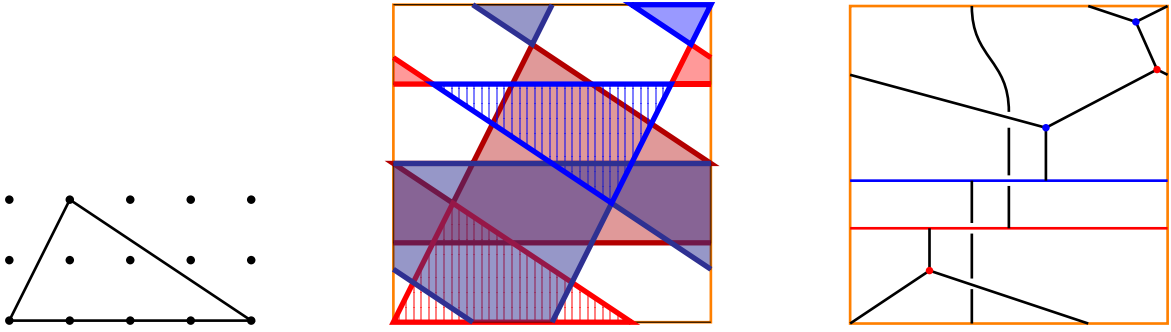
\end{example}

\subsection{From Sheaves to Rulings}

We explain how to get a ruling from a sheaf \(\cF \in \Sh(\T, \Lambda)^\eta_0\). 

First recall a standard fact in microlocal sheaf theory:
\begin{proposition}[{\cite[Corollary 4.4.3]{Guillermou23}}]\label{prop:ruling in 1d}
A constructible sheaf \(\cE \in \Sh^b(\bR) \) admits a unique decomposition:
\begin{equation}
    \cE = \bigoplus_{\alpha} \bC_{I_\alpha}[d_\alpha] 
\end{equation}
where \( \{I_\alpha \subset \bR\}_{\alpha\in A}\) is a locally finite family of intervals (open, closed, or half-open), and \(d_\alpha\) are finite integers.

Moreover, let \(e:\bR \to S^1\) be the covering map. A constructible sheaf \(\cE \in \Sh^b     (S^1)\) admits a unique decomposition 
\begin{equation}
\cE \cong \bigoplus_{\alpha\in A} 
e_* \bC_{I_\alpha}[d_\alpha] \oplus \cL
\end{equation}
where \(\cL \in \Loc(S^1) \). Since \(\Loc (S^1) \simeq k[x^{\pm 1}]\text{-}mod \) is of global dimension \(1\), \(\cL\) also splits:
\[
\cL \cong  \bigoplus_{\beta} \cL_\beta [ d_\beta ] 
\]
where each \(\cL_\beta\) is an abelian local system on \(S^1\).
\end{proposition}

\begin{lemma}[{Non-characteristic deformation \cite[Proposition 2.11]{Zhou19} \cite{KS90}}]
\label{lem:non_char_def}
Let \(I\) be an open interval in \(\bR\), and \(\cF, \cG\) be sheaves on \( M\times I\) with compact support. Take \(\cF_t = \cF|_{ M \times \{t \} } \) and \(\cG_t = \cG|_{ M \times \{t\}}\). Assume that
\begin{enumerate}
\item \(\SS^\infty(\cF)\), \(\SS^\infty(\cG)\), and \( T_M^*M \times T^*I\) are pairwise disjoint;
\item
\( \SS^\infty ( \cF_t) \cap \SS^\infty(\cG_t) = \emptyset\), for each  \(t \in I\); 

\end{enumerate}
Then there is a natural isomorphism
\[
 \underline \Hom( \cF_t, \cG_t) \cong \underline \Hom (\cF_s, \cG_s), \quad \text{for \(s,t \in I\). }
\]

\end{lemma}

After fixing the direction \(v \in \bZ^2\), we have an alternative characterization of the subcategory \(\Sh(\T, \Lambda)_0^\eta \subset \Sh^c(\T)\).

\begin{lemma}\label{lem:characterization of Sh_0 via restriction on S1}
Let \(\cF \in \Sh^c(\T) \) be a constructible sheaf satisfying conditions (\ref{lemfirst}) and (\ref{lemsecond}) in Lemma~\ref{lem:characterization of Sh_0 by Hom with Loc}. Then \(\cF \in \Sh(\T, \Lambda)_0^\eta\) if and only if 
\begin{enumerate}
\setcounter{enumi}{2}
\item[(\,$3'$)] \label{lemthreeprime}
\renewcommand{\theenumi}{\arabic{enumi}'}
    for every \(S^1\)-leaf, the restriction of \(\cF\) has no  local system direct summand.
\end{enumerate}
In fact, condition (\ref{lemthird}) is equivalent to 
\begin{enumerate}
    \item[(\,$3''$)] There exists an \(S^1\)-leaf, on which the restriction of \(\cF\) has no local system direct summand. 
\end{enumerate}
\end{lemma}

\begin{proof}
    Let us first show the equivalence of ($3''$) and ($3'$). If \( \cF_{\{a\} \times S^1}\) has a local system summand \(\cL\), then let \(\cL_t\) be the pushforward of \(\cL\) under the embedding \( \iota_t: \{t\} \times S^1 \hookrightarrow \T\).  Since by our assumption,  \(\pi(\Lambda)\) is nowhere tangent to the vertical direction \(v\), \(\cL_t\) and \(\cF\) satisfy the assumption of Lemma~\ref{lem:non_char_def}. Therefore  \(\underline \Hom (\iota_t^{-1} \cF, \cL ) \simeq \underline \Hom( \cF, \cL_t) \simeq \underline{\Hom} (\cF|_{\{a\}\times S^1}, \cL) \)  for any \(t\), which implies \(\cL\) is also a summand of \(\cF|_{\{t\} \times S^1}\).


     {Assume \(\cF \in \Sh(\T, \Lambda)^\eta_0 \).} 
    We fix an identification  \(\T = S^1 \times S^1\), where the vertical direction is \(v\), and then subdivide the torus into small ``strips'', as illustrated in Figure~\ref{fig:cut the torus into strips}. More precisely, let us choose a sequence of points \(a_0, a_1,\ldots,a_{n-1} \in S^1=\mathbb{R}/\mathbb{Z}\) together with a sufficiently small number \(\epsilon\in \mathbb{R}_{>0}\), such that each strip \(\T_j \coloneqq (a_j - \epsilon, a_{j+1} + \epsilon) \times S^1\) contains exactly one crossing of \(\Lambda\), and their intersections \(\T^\epsilon_j \coloneqq (a_j - \epsilon, a_j + \epsilon)\times S^1\) contain no crossings. 
    Here the subscripts are understood modulo $n$. 
    
    









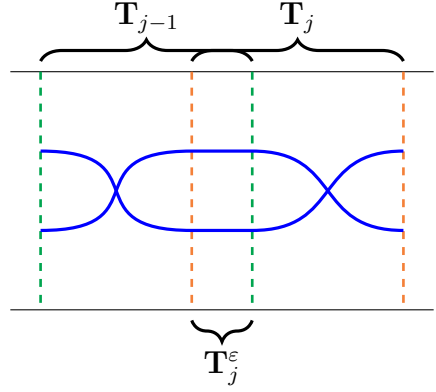
\begin{figure}[htbp]
    \centering
    \begin{tikzpicture}[xscale = .8, yscale = .7,  decoration={brace,amplitude=8pt}]
    \draw[line width=1.2pt, decorate]
    (-.5, 1.7) -- (3,1.7)
    node[midway,above=6pt] {${\bf T}_{j-1}$};
    \draw[line width=1.2pt, decorate]
    (2, 1.7) -- (5.5,1.7)
    node[midway,above=6pt] {${\bf T}_{j}$};
  
    \draw (-1, 1.5) -- (6, 1.5);
    \draw (6, -3) -- (-1, -3); 
    \draw[dashed, line width = 1pt, Green] (-.5, 1.5) -- (-.5, -3);
    \draw[dashed, line width = 1pt, Green] (3, 1.5) -- (3, -3);
    \draw[dashed, line width = 1pt, Orange] (2, 1.5) -- (2, -3);
    \draw[dashed, line width = 1pt, Orange] (5.5, 1.5) -- (5.5, -3);

    \draw[line width=1.2pt, decorate, decoration={brace,mirror,amplitude=8pt}] 
    (2,-3.2) -- (3,-3.2)
    node[midway,below=7pt] {${\bf T}_{j}^\varepsilon$};
    
    \draw[blue, very thick] 
    (-.5, 0) .. controls (1.5, 0) and (0, -1.5) .. (2, -1.5) -- (3, -1.5) .. controls (4.5, -1.5) and (4, 0) .. (5.5, 0);

    \draw[blue, very thick]  (-.5, -1.5) .. controls (1.5, -1.5) and (0, 0) .. (2, 0) -- (3, 0) .. controls (4.5, 0) and (4, -1.5) .. (5.5, -1.5);
    
    \end{tikzpicture}
    \caption{Cutting the torus into small strips.}
    \label{fig:cut the torus into strips}
\end{figure}

    Suppose (\,$3''$) is not satisfied. Then we have \(\cF|_{\T_j} \simeq \cL_j \oplus \cdots \) and \(\cF|_{\T_j^\epsilon} \simeq \cL_j^\epsilon \oplus \cdots\) for each \(j\). 
    Again by the non-characteristic deformation lemma, we have \( \cL_j|_{\T_j^\epsilon} \simeq \cL_j^\epsilon\) and  \(\cL_j|_{\T^\epsilon_{j+1}} \simeq \cL_{j+1}^\epsilon\). 
    In particular, all the \(\cL_j\) can be regarded as the same local system \( \cA \in \Loc( S^1) \).
    In other words, we have \(\cL_j \simeq \dsk_{(a_j - \epsilon, a_{j+1} + \epsilon)} \boxtimes \cA \) and \( \cL_j^\epsilon \simeq \dsk_{ (a_j-\epsilon, a_j +\epsilon)} \boxtimes \cA\). 

     Assume  \(  \cA = \cA_0 [-k] \oplus \cdots \), where \(\cA_0\) is an underived ordinary local system.
    Take a rank-\(1\) local system \(\cB[-k]\) with a non-zero map \(\cA \rightarrow \cB[-k]\). 
    Globally we can glue a local system on \(\T\):
    \begin{align*}
        \cL
        &\coloneqq \lim \left( 
        \bigoplus \dsk_{[a_j-\epsilon, a_{j+1} + \epsilon ]} \boxtimes \cB \rightrightarrows
        \bigoplus \dsk_{[a_j-\epsilon, a_{j} + \epsilon ]} \boxtimes \cB \right) 
    \end{align*}
    where \(\iota_j\) and \(\iota_j^\epsilon\) are the embeddings of \(\T_j\) and \(\T_j^\epsilon\) in \(\T\).  Note that $\cL$ is a rank-$1$ local system because it has constant stalk everywhere, hence it has the form $\cL \simeq \cN \boxtimes \cB$. 
    

    Since
    \begin{align*}
    \underline{\Hom}\left(\cF, \, \cN \boxtimes \cB[-k]\right) \simeq 
    \lim
    &\left( 
    \bigoplus \underline{\Hom}
    \left({\iota_j}_* \iota_j^{-1} \cF, \,  \dsk_{[a_j-\epsilon, a_{j+1} + \epsilon ]} \boxtimes \cB[-k] \right)
    \right.\\ 
    &\rightrightarrows \bigoplus 
    \left.
    \underline{\Hom}\left({\iota_j^\epsilon}_* (\iota_j^\epsilon)^{-1} \cF, \,  \dsk_{[a_j - \epsilon, a_j+\epsilon]} \boxtimes \cB[-k] \right)  \right)
    \end{align*}
    the non-trivial map \(\cA \rightarrow \cB[-k]\) yields a non-trivial map in {\(H^0 \underline{\Hom} (\cF, \,\cN\boxtimes \cB[-k])\)}. Thus by assumption \eqref{eq:conditions on HiHom}, we must have \(k\le-2\). 

    On the other hand, we can take a rank-\(1\) local system \(\cB'[-k] \in \Loc(S^1)\) concentrated in degree \(k\), with a nontrivial element in \( H^0 \underline{\Hom} (\cB'[-k], \cA)\).
    Then consider the local system on \(\T\):
    \begin{align*}
        \cN' \boxtimes \cB'
        &\coloneqq \colim \left( \bigoplus \dsk_{(a_j-\epsilon, a_{j+1} + \epsilon)} \boxtimes \cB'
        \leftleftarrows \bigoplus \dsk_{(a_j-\epsilon, a_{j} + \epsilon )} \boxtimes \cB'  \right).
    \end{align*}
     Similarly, globally there is a non-trivial map in \(H^0 \underline{\Hom} (\cN' \boxtimes \cB'[-k], \cF) \), which implies that \(k \ge -1\). Hence we get a contradiction. 
    
    \vspace{1em}
    Now assume that \(\cF\) satisfies (\,$3'$).
    Again denote its mirror by \(E\in \Perf_\text{prop}(\cX)\). Let \(\cL\in \Loc(\T)\) be a rank-\(1\) local system in degree \(0\). We have 
    \begin{equation*}
        \cL \simeq \colim \left( \bigoplus {\iota_j}_! \iota_j^{-1}\cL \leftleftarrows \bigoplus {\iota_j^{\epsilon}}_!(\iota_j^{\epsilon})^{-1} \cL \right)
    \end{equation*}
    One way to see this equivalence is that this is a map from the colimit diagram to \(\cL\), inducing an isomorphism on each fiber.     
    Thus 
    \[
        \underline{\Hom}(\cL, \cF) \simeq \lim \left( \bigoplus \underline{\Hom} 
        \left({\iota_j}_! \iota_j^{-1}\cL, \, \cF
        \right) \rightrightarrows 
        \bigoplus \underline{\Hom} 
        \left( 
        {\iota_j^{\epsilon}}_!(\iota_j^{\epsilon})^{-1} \cL, \, \cF
        \right)
        \right).
    \]
    By Proposition~\ref{prop:ruling in 1d} and the microlocal stalk condition, we know that 
    \[\underline{\Hom} \left({\iota_j^\epsilon}_! (\iota_j^\epsilon)^{-1}\cL, \, \cF
        \right) \simeq \underline{ \Hom}\left( (\iota_j^\epsilon)^{-1} \cL, \, (\iota_j^\epsilon)^{-1}\cF \right)
    \]
    is concentrated in degree \(-1\). By the non-characteristic deformation lemma, we know that 
    \[\underline{\Hom} 
    \left({\iota_j}_! \iota_j^{-1}\cL, \, \cF
    \right) \]
    is also concentrated in degree $-1$. Therefore, all
    \( H^i \underline{\Hom} (\cL, \cF ) \) vanish for \(i \le -2\). 

    We also have
    \begin{equation*}
        \cL \simeq \lim \left( \bigoplus {\iota_j}_* \iota_j^{-1}\cL \rightrightarrows \bigoplus {\iota_j^{\epsilon}}_*(\iota_j^{\epsilon})^{-1} \cL \right)
    \end{equation*}
    By a parallel argument, we have \( H^i \underline{\Hom} (\cF, \cL) = 0 \) for \(i \le 1\).   Thus condition (3) in Lemma~\ref{lem:characterization of Sh_0 by Hom with Loc} holds. This completes the proof.

\end{proof}

Take a sheaf \(\cF \in \Sh^1_{\Lambda}(T^2)_0\).
Combining Proposition~\ref{prop:ruling in 1d} and Lemma~\ref{lem:characterization of Sh_0 via restriction on S1} with the microlocal stalk condition, we see that the restriction of \({\cF}\) to a vertical \(S^1\) decomposes as
\[
\bigoplus_{i} e_* \dsk_{U_i}[2] \oplus  \bigoplus_{j} e_* \dsk_{Z_j}[1],
\]
where \(e\) is the universal cover of \(S^1\), and  \(U_i\) and \(Z_j\) are open and closed intervals on $\bR$, respectively. In particular, the stalks of \(\cF\) are concentrated in degree \(-1\) and $-2$. All sheaves on \(\T_i^\epsilon\) decompose in a similar way. We also classify sheaves on~\( \T_i \):

\begin{proposition}\label{prop:indecomposable}
  Let \(e:\widetilde \T_i\to \T_i\) be the universal cover, and let
\(\widetilde{\Lambda}_i=e^{-1}(\Lambda_i)\). 
Denote by \(\Sh^{\le 1}(\T_i, \Lambda_i)\) the subcategory of \(\Sh(\T_i, \Lambda_i)\) consisting of sheaves that have no local system summands, and whose microstalks along \(\Lambda_i\) are either \(\dsk[1]\) or \(0\).  An indecomposable object in \(\Sh^{\le 1} (\T_i, \Lambda_i)\) has one of the following forms:

\begin{enumerate}
    \item The pushforward under \(e\) of a standard or costandard sheaf:
    \[
        e_* \dsk_{\widetilde Z} \, [n]
    \]
    where \(\widetilde Z\subset \widetilde \T_i\) is a closed or open subset, and \(n\) is an integer, such that \( e_* \dsk_{\widetilde Z} [n] \in \Sh^{\le 1}(\T_i, {\Lambda_i})\).

    \item The pushforward under \(e\)
    of one of the following sheaves in \(\Sh^1\left(\widetilde{\T_i}, \widetilde{\Lambda_i} \right ) \) 
    \begin{equation}\label{eq:indecomposable_two}
     \begin{tikzpicture}[line width = 1pt, scale = .6]
            \draw[line width = 1pt] (-2, 1) -- (2, -1);
            \draw[line width = 1pt] (-2, -1) -- (2, 1);
            \draw[line width = 1pt] (-2, 2) -- (2, 2);
            \draw[line width = 1pt] (-2, -2) -- (2, -2);
            \draw (-1, .5) -- (-.9, .7);
            \draw (1, -.5) -- (1.1, -.3);
            \draw (-1, -.5) -- (-1.1, -.3);
            \draw (1, .5) -- (.9, .7); 
            \draw (0, -2) -- (0, -2.25);
            \draw (0, 2) -- (0, 1.75); 

            \fill[red, opacity = .4] (-2, -1) -- (-2, -2) -- (2, -2) -- (2, -1) -- (0, 0); 
            \fill[blue, opacity = .4] (-2, 2) -- (2, 2) -- (2, 1) -- (0, 0) -- (-2, 1);
            \node at (0, 1) { \(k[1]\)};
            \node at (0, -1) {\(k[2]\)};
            \node at (1.5, 0) {\(0\)};
            \node at (-1.5, 0) {\(0\)}; 

            \draw[dashed, Green] (-2, 2.5) -- (-2, -2.5);
            \draw[dashed, Green] (2, 2.5) -- (2, -2.5);
            
        \end{tikzpicture}
        \qquad  
        \begin{tikzpicture}[line width = 1pt, xscale = .6, yscale = -.6]
            \draw[line width = 1pt] (-2, 1) -- (2, -1);
            \draw[line width = 1pt] (-2, -1) -- (2, 1);
            \draw[line width = 1pt] (-2, 2) -- (2, 2);
            \draw[line width = 1pt] (-2, -2) -- (2, -2);
            \draw (-1, .5) -- (-.9, .7);
            \draw (1, -.5) -- (1.1, -.3);
            \draw (-1, -.5) -- (-1.1, -.3);
            \draw (1, .5) -- (.9, .7); 
            \draw (0, -2) -- (0, -2.25);
            \draw (0, 2) -- (0, 1.75); 

            \fill[red, opacity = .4] (-2, -1) -- (-2, -2) -- (2, -2) -- (2, -1) -- (0, 0); 
            \fill[blue, opacity = .4] (-2, 2) -- (2, 2) -- (2, 1) -- (0, 0) -- (-2, 1);
            \node at (0, 1) { \(k[1]\)};
            \node at (0, -1) {\(k[2]\)};
            \node at (1.5, 0) {\(0\)};
            \node at (-1.5, 0) {\(0\)}; 

            \draw[dashed, Green] (-2, 2.5) -- (-2, -2.5);
            \draw[dashed, Green] (2, 2.5) -- (2, -2.5);
            
        \end{tikzpicture}
        \qquad
        \begin{tikzpicture}[line width = 1pt, scale = .8]
            \draw[line width = 1pt] (-2, 1) -- (2, -1);
            \draw[line width = 1pt] (-2, -1) -- (2, 1);
            
            \draw (-1, .5) -- (-.9, .7);
            \draw (1, -.5) -- (1.1, -.3);
            \draw (-1, -.5) -- (-.9, -.7);
            \draw (1, .5) -- (1.1, .3); 

            \fill[red, opacity = .4] (-2, -1) -- (-2, 1) -- (0, 0);
            \fill[blue, opacity = .4] (0, 0) -- (2, 1) -- (2, -1);
            \node at (1.5, 0) {\small \(k[1]\)};
            \node at (-1.5, 0) {\small\(k[2]\)}; 
            
            \draw[dashed, Green] (-2, 1.5) -- (-2, -1.5);
            \draw[dashed, Green] (2, 1.5) -- (2, -1.5);
            
        \end{tikzpicture}
        \qquad
        \begin{tikzpicture}[line width = 1pt, xscale = - .8, yscale = .8]
            \draw[line width = 1pt] (-2, 1) -- (2, -1);
            \draw[line width = 1pt] (-2, -1) -- (2, 1);
            
            \draw (-1, .5) -- (-.9, .7);
            \draw (1, -.5) -- (1.1, -.3);
            \draw (-1, -.5) -- (-.9, -.7);
            \draw (1, .5) -- (1.1, .3); 

            \fill[red, opacity = .4] (-2, -1) -- (-2, 1) -- (0, 0);
            \fill[blue, opacity = .4] (0, 0) -- (2, 1) -- (2, -1);
            \node at (1.5, 0) {\small \(k[1]\)};
            \node at (-1.5, 0) {\small \(k[2]\)}; 

            \draw[dashed, Green] (-2, 1.5) -- (-2, -1.5);
            \draw[dashed, Green] (2, 1.5) -- (2, -1.5);
            
        \end{tikzpicture}.
        \end{equation}
        
        \item The pushforward under \(e\) of one of the following
        \begin{equation}\label{eq:indecomposable_three}
        \begin{tikzpicture}[line width = 1pt, xscale = .7, yscale = -.7  ]
            \draw[line width = 1pt] (-2, 1) -- (2, -1);
            \draw[line width = 1pt] (-2, -1) -- (2, 1);
            \draw[line width = 1pt] (-2, 2) -- (2, 2);
            \draw[line width = 1pt] (-2, 3.5) -- (2, 3.5);
            
            \draw (-1, .5) -- (-.9, .7);
            \draw (1, -.5) -- (1.1, -.3);
            \draw (-1, -.5) -- (-1.1, -.3);
            \draw (1, .5) -- (.9, .7); 
            \draw (0, 3.5) -- (0, 3.25);
            \draw (0, 2) -- (0, 1.75); 

            
            \node at (0, 1) { \(k^2[1]\)};
            \node at (0, 2.5) {\(k[1]\)};
            \node at (1.5, 0) {\small \(k[1]\)};
            \node at (-1.5, 0) {\small \(k[1]\)}; 

            \draw[dashed, Green] (-2, 4) -- (-2, -1.5);
            \draw[dashed, Green] (2, 4) -- (2, -1.5);

            \fill[blue, opacity = .2] (-2, 3.5) -- (2, 3.5) -- (2, -1) -- (0, 0) -- (-2, -1);
            \fill[blue, opacity = .2] (-2, 1) -- (-2, 2) -- (2, 2) -- (2, 1) -- (0, 0); 
            
        \end{tikzpicture}
        \quad
        \begin{tikzpicture}[line width = 1pt, scale = .7]
            \draw[line width = 1pt] (-2, 1) -- (2, -1);
            \draw[line width = 1pt] (-2, -1) -- (2, 1);
            \draw[line width = 1pt] (-2, 2) -- (2, 2);
            \draw[line width = 1pt] (-2, 3.5) -- (2, 3.5);
            
            \draw (-1, .5) -- (-.9, .7);
            \draw (1, -.5) -- (1.1, -.3);
            \draw (-1, -.5) -- (-1.1, -.3);
            \draw (1, .5) -- (.9, .7); 
            \draw (0, 3.5) -- (0, 3.25);
            \draw (0, 2) -- (0, 1.75); 

            
            \node at (0, 1) { \(k^2[1]\)};
            \node at (0, 2.5) {\(k[1]\)};
            \node at (1.5, 0) {\small \(k[1]\)};
            \node at (-1.5, 0) {\small \(k[1]\)}; 

            \draw[dashed, Green] (-2, 4) -- (-2, -1.5);
            \draw[dashed, Green] (2, 4) -- (2, -1.5);

            \fill[blue, opacity = .2] (-2, 3.5) -- (2, 3.5) -- (2, -1) -- (0, 0) -- (-2, -1);
            \fill[blue, opacity = .2] (-2, 1) -- (-2, 2) -- (2, 2) -- (2, 1) -- (0, 0); 
            
        \end{tikzpicture}
        \qquad
        \begin{tikzpicture}[line width = 1pt, xscale = .7, yscale = -.7  ]
            \draw[line width = 1pt] (-2, 1) -- (2, -1);
            \draw[line width = 1pt] (-2, -1) -- (2, 1);
            \draw[line width = 1pt] (-2, 2) -- (2, 2);
            \draw[line width = 1pt] (-2, 3.5) -- (2, 3.5);
            
            \draw (-1, .5) -- (-1.1, .3);
            \draw (1, -.5) -- (.9, -.7);
            \draw (-1, -.5) -- (-.9, -.7);
            \draw (1, .5) -- (1.1, .3); 
            \draw (0, 3.5) -- (0, 3.75);
            \draw (0, 2) -- (0, 2.25); 

            
            \node at (0, 1) { \(k^2[2]\)};
            \node at (0, 2.8) {\(k[2]\)};
            \node at (1.5, 0) {\small \(k[2]\)};
            \node at (-1.5, 0) {\small \(k[2]\)}; 

            \draw[dashed, Green] (-2, 4) -- (-2, -1.5);
            \draw[dashed, Green] (2, 4) -- (2, -1.5);

            \fill[red, opacity = .2] (-2, 3.5) -- (2, 3.5) -- (2, -1) -- (0, 0) -- (-2, -1);
            \fill[red, opacity = .2] (-2, 1) -- (-2, 2) -- (2, 2) -- (2, 1) -- (0, 0); 
            
        \end{tikzpicture}
        \quad
        \begin{tikzpicture}[line width = 1pt, xscale = .7, yscale = .7  ]
            \draw[line width = 1pt] (-2, 1) -- (2, -1);
            \draw[line width = 1pt] (-2, -1) -- (2, 1);
            \draw[line width = 1pt] (-2, 2) -- (2, 2);
            \draw[line width = 1pt] (-2, 3.5) -- (2, 3.5);
            
            \draw (-1, .5) -- (-1.1, .3);
            \draw (1, -.5) -- (.9, -.7);
            \draw (-1, -.5) -- (-.9, -.7);
            \draw (1, .5) -- (1.1, .3); 
            \draw (0, 3.5) -- (0, 3.75);
            \draw (0, 2) -- (0, 2.25); 

            
            \node at (0, 1) { \(k^2[2]\)};
            \node at (0, 2.8) {\(k[2]\)};
            \node at (1.5, 0) {\small \(k[2]\)};
            \node at (-1.5, 0) {\small \(k[2]\)}; 

            \draw[dashed, Green] (-2, 4) -- (-2, -1.5);
            \draw[dashed, Green] (2, 4) -- (2, -1.5);

            \fill[red, opacity = .2] (-2, 3.5) -- (2, 3.5) -- (2, -1) -- (0, 0) -- (-2, -1);
            \fill[red, opacity = .2] (-2, 1) -- (-2, 2) -- (2, 2) -- (2, 1) -- (0, 0); 
            
        \end{tikzpicture}.
        \end{equation}
        In each case, the kernels of the three maps \( k^2 [1] \rightarrow k[1]\) or the images of three maps \( k[2] \rightarrow k^2 [2] \) are distinct. Hence the sheaves do not split. 

        \end{enumerate}

\end{proposition}

\begin{proof}
    It is clear from the descriptions that the listed objects are indecomposable.
    In the following, we assume \(\cF\in \Sh^{\le 1} (\T_i , \Lambda_i)\) is indecomposable, and show that \(\cF\) must be one of the cases listed above.

    Denote \(\widetilde \cF \coloneqq e^{-1} \cF\). It suffices to show that $\widetilde \cF$ has a summand of the prescribed form. If $\widetilde \cF$ has such a summand $\cE$, then we can consider 
    \[
        e_! \cE \longrightarrow  \cF \longrightarrow e_* \cE. 
    \]
    Note that \(e_!\cE \simeq e_*\cE\) since $\cE$ has compact support. Denote \(\widetilde \T_i^\epsilon = e^{-1} (\T_i^\epsilon) \) and $\widetilde \T_{i+1}^\epsilon = e^{-1} (\T_{i+1}^\epsilon)$.  
    One can see that the above is a retraction by restricting to $\widetilde \T_i^\epsilon$ and $\widetilde \T_{i+1}^\epsilon$. Hence $e_*\cE$
    is a summand of $\cF$, which forces $\cF = e_*\cE$ since \(\cF\) is indecomposable. 
    
    Write \(\Lambda_i = \Lambda^{\mathrm{cr}} \sqcup \Lambda^\mathrm{hor}\), where \(\Lambda^\mathrm{cr}\) is the crossing part (which has two connected components), and \(\Lambda^\mathrm{hor}\) is the horizontal part.  Choose a lift of $\Lambda^\mathrm{cr}$, which we denote by $\widetilde \Lambda^\mathrm{cr}_0$. 
    Consider the  pairing between the components in \(e^{-1}(\Lambda_i)\) induced by $\wtd \cF|_{\widetilde \T_i^\epsilon}$.

    \textbf{Case I:} the two components of $\widetilde \Lambda^\mathrm{cr}_0$ are paired with each other. 

    \textbf{Case II:} the two components of $\widetilde \Lambda^\mathrm{cr}_0$ are paired with some other horizontal components. 

    \textbf{Case III:} the two components of $\widetilde \Lambda^\mathrm{cr}_0$ are paired with components in other lifts of \(\Lambda^\mathrm{cr}\). 

    In \textbf{Case I} and \textbf{II}, take
     $\wtd \Lambda^\mathrm{cr} \coloneqq \wtd \Lambda^\mathrm{cr}_0$. In \textbf{Case III}, 
    take $\wtd \Lambda^\mathrm{cr}$ to be the union of \(\wtd \Lambda^\mathrm{cr}_0\) and the other lifts of $ \Lambda^\mathrm{cr}$ that are paired with \(\wtd \Lambda^\mathrm{cr}_0\).

    Consider two smooth proper maps from \(\wtd \T_i\) to itself, denoted by $f$ and $g$.  The map \(g\) collapses neighborhoods of all lifts of $\Lambda^{\mathrm{cr}}$ 
    except $\widetilde \Lambda^\mathrm{cr}$. The map $f$ collapses a neighborhood of $\widetilde \Lambda^{\mathrm{cr}}$. An illustration is shown in the following figure.
    \[
    \begin{tikzpicture}[xscale = 1, yscale = .5]
    \draw[dashed, Green] (0, -4) -- (0, 9);
    \draw[dashed, Green] (2, -4) -- (2, 9);

    \draw[line width = 1pt] (0, 0) -- (2, 0);
    \draw[line width = 1pt] (0, .5) -- (2, .5);
    \draw[line width = 1pt] (0, -.5) -- (2, -.5); 
    
    \draw[line width = 1pt, Purple] (0, 2) -- (2, 3);
    \draw[line width = 1pt, Purple] (0, 3) -- (2, 2);
    
    \draw[line width = 1pt] (0, -2) -- (2, -3);
    \draw[line width = 1pt] (0, -3) -- (2, -2);

    \draw[line width = 1pt] (0, 4.5) -- (2, 4.5); 
    \draw[line width = 1pt] (0, 5) -- (2, 5);
    \draw[line width = 1pt] (0, 5.5) -- (2, 5.5); 
    
    \draw[line width = 1pt] (0, 7) -- (2, 8);
    \draw[line width = 1pt] (0, 8) -- (2, 7);

    \draw[-Stealth] (3, 2.5) -- (4, 2.5); 
    \node[above] at (3.5, 2.5) { $g$ };

    \node[left] at (0, 2.5) {\textcolor{Purple}{$\widetilde \Lambda^{\mathrm{cr}}$}};

    \begin{scope}[shift = {(5, 0)}]
    \draw[dashed, Green] (0, -4) -- (0, 9);
    \draw[dashed, Green] (2, -4) -- (2, 9);

    \draw[line width = 1pt] (0, 0) -- (2, 0);
    \draw[line width = 1pt] (0, .5) -- (2, .5);
    \draw[line width = 1pt] (0, -.5) -- (2, -.5); 
    
    \draw[line width = 1pt, Purple] (0, 2) -- (2, 3);
    \draw[line width = 1pt, Purple] (0, 3) -- (2, 2);
    
    \draw[line width = 3pt] (0, -2.5) -- (2, -2.5);

    \draw[line width = 1pt] (0, 4.5) -- (2, 4.5); 
    \draw[line width = 1pt] (0, 5) -- (2, 5);
    \draw[line width = 1pt] (0, 5.5) -- (2, 5.5); 
    
    \draw[line width = 3pt] (0, 7.5) -- (2, 7.5);

    \draw[-Stealth] (3, 2.5) -- (4, 2.5); 
    \node[above] at (3.5, 2.5) { $f$ };
    \end{scope}

    \begin{scope}[shift = {(10, 0)}]
    \draw[dashed, Green] (0, -4) -- (0, 9);
    \draw[dashed, Green] (2, -4) -- (2, 9);

    \draw[line width = 1pt] (0, 0) -- (2, 0);
    \draw[line width = 1pt] (0, .5) -- (2, .5);
    \draw[line width = 1pt] (0, -.5) -- (2, -.5); 
    
    \draw[line width = 3pt, Purple] (0, 2.5) -- (2, 2.5);
    
    \draw[line width = 3pt] (0, -2.5) -- (2, -2.5);

    \draw[line width = 1pt] (0, 4.5) -- (2, 4.5); 
    \draw[line width = 1pt] (0, 5) -- (2, 5);
    \draw[line width = 1pt] (0, 5.5) -- (2, 5.5); 
    
    \draw[line width = 3pt] (0, 7.5) -- (2, 7.5);
    \end{scope}
    \end{tikzpicture}
\]  

    Now we consider \(f_* g_* \widetilde \cF\). By Proposition~\ref{prop:ruling in 1d},  it has a canonical decomposition. We take \(\wtd \cF_0\) to be the direct sum of all summands whose singular support does not intersect with the image of \(\widetilde \Lambda^\mathrm{cr}\). Then we apply the same trick: consider the maps
    \[
    f^{-1} \wtd \cF_0 \longrightarrow  g_*\wtd \cF \longrightarrow f^! \wtd \cF_0. 
    \]
    Since $f$ is locally homeomorphic near $\SS (\wtd \cF_0)$, we get \(f^! \wtd \cF_0 \simeq f^{-1} \wtd \cF_0\). By restricting to the left and right boundaries and a stalkwise check, one can see that the above is a retraction. Hence $ f^{-1} \wtd \cF_0$ is a component of $g_*\wtd \cF$. 
    Denote its complement by $ \wtd \cE_0$, so we can write $g_*\wtd \cF \simeq f^{-1} \widetilde \cF_0 \oplus \wtd \cE_0$. 
    Then by a similar argument, \( \cE \coloneqq g^{-1}\wtd \cE_0\) is a direct summand of \(\widetilde \cF\).  
    Notice that the sheaf \( \cE\) is compactly supported with all microstalks being  $\dsk[1]$.
    We only need to classify all such sheaves on $\widetilde \T_i$ in each case.

    Let us first fix some notation.
     Consider the naive \(t\)-structure on \( \Sh(\T_i, \Lambda_i) \). The stalks of \(\cE\) are concentrated in degrees \(-1\) and \(-2\). There is an exact triangle:
    \[
        \tau_{\le -2} \cE \longrightarrow \cE \longrightarrow \tau_{\ge -1} \cE \longrightarrow .
    \]
    Denote \( \cE_\mathrm{blue} \coloneqq \big(\tau_{\ge -1} \cE \big) \) and \( \cE_\mathrm{red} \coloneqq \big(\tau_{\le -2} \cE\big) \). Then \(\cE_\mathrm{blue}\) is concentrated in degree \(-1\) and \(\cE_\mathrm{red} \) is concentrated in degree \(-2\), and \(\cE\) is an extension of \(\cE_{\mathrm{blue}}\) by $\cE_{\mathrm{red}}$.

    Let \( \SS ( - )^\mathrm{sm}\) be the singular support removing the codirections at the crossing. Then 
    \begin{equation}\label{eq:decompose_SS}
        \SS(\cE)^\mathrm{sm} \subset \SS(\cE_\mathrm{red})^\mathrm{sm} \sqcup \SS(\cE_\mathrm{blue})^\mathrm{sm}
    \end{equation}
    By construction, $\cE_\mathrm{red}$ and $\cE_\mathrm{blue}$ are constructible with respect to the stratification given by $\Lambda$. In particular, if we look at a neighborhood of a point in $ \SS (\cE)^\mathrm{sm}$, the microlocal rank $1$ condition tells us that \eqref{eq:decompose_SS} is indeed an equation, i.e., $\SS(\cE)^\mathrm{sm} = \SS(\cE_\mathrm{red})^\mathrm{sm} \sqcup \SS(\cE_\mathrm{blue})^\mathrm{sm}$.

    \smallskip

    \noindent \textbf{Case I.}  $\SS(\cE)$ has the form:
    \[
      \begin{tikzpicture}[ yscale = .6]
        \draw[line width = 1pt] (0, 0) -- (2, 1);
        \draw[line width = 1pt] (0, 1) -- (2, 0);
        \draw[dashed, Green] (0, -1) -- (0, 2);
        \draw[dashed, Green] (2, -1) -- (2, 2);
    \end{tikzpicture}. 
    \]
    The crossing is of opposite codirections (see Remark~\ref{rmk:crossing_types}).
    The sheaf must be of the left two types in \eqref{eq:indecomposable_two}.


    \smallskip
    \noindent \textbf{
      Case II(A):} $\SS(\cE)$ has the form
      \[
      \begin{tikzpicture}[top aligned, yscale = .5, baseline={(current bounding box.center)}]
        \draw[line width = 1pt] (0, 0) -- (2, 0);
        \draw[line width = 1pt] (0, 1) -- (2, 1);
        \draw[line width = 1pt] (0, 3) -- (
        2, 4);
        \draw[line width = 1pt] (0, 4) -- (2, 3);
        \draw[dashed, Green] (0, -1) -- (0, 5);
        \draw[dashed, Green] (2, -1) -- (2, 5);
    \end{tikzpicture}. 
    \qquad or \qquad
    \begin{tikzpicture}[ yscale = -.5, baseline={(current bounding box.center)} ]
        \draw[line width = 1pt] (0, 0) -- (2, 0);
        \draw[line width = 1pt] (0, 1) -- (2, 1);
        \draw[line width = 1pt] (0, 3) -- (
        2, 4);
        \draw[line width = 1pt] (0, 4) -- (2, 3);
        \draw[dashed, Green] (0, -1) -- (0, 5);
        \draw[dashed, Green] (2, -1) -- (2, 5);
    \end{tikzpicture}. 
    \]
    The crossing is of the same codirection. By symmetry, we only need to consider the first case. 
    
    If \(\cE_\mathrm{blue}\) and $\cE_\mathrm{red}$ are both non-zero, they must have the form $\cE_{\mathrm{blue}} = 
    \begin{tikzpicture}[scale = .6]
        \draw[thick] (0, 1) -- (2, 2);
        \draw[thick] (0, 0) -- (2, 0);
        \fill[blue, opacity = .4] (0, 0) -- (2, 0) -- (2, 2) -- (0, 1) -- cycle;
        \node at (1, .6) {\small$k[1]$};
    \end{tikzpicture}$ and
    $\cE_\mathrm{red} =
    \begin{tikzpicture}[scale = .6]
        \draw[thick] (0, 0) -- (2, 0);
        \draw[thick] (0, 2) -- (2, 1);
        \fill[red, opacity = .4] (0, 0) -- (2, 0) -- (2, 1) -- (0,2);
        \node at (1, .6) {\small$k[2]$};
    \end{tikzpicture}
    $,
    or the other way around.
    One can easily compute that there are no extensions between them, hence \(\cE\) must split as their direct sum. 

    Suppose \(\cE \simeq \cE_{\mathrm{blue}}\) (the case \(\cE \simeq \cE_{\mathrm{red}}\) is similar). Since $\cE$ has microstalk $k[1]$ and compact support, the stalks of $\cE$ must be the same as the first picture in \eqref{eq:indecomposable_three}. Let us denote 
    the stalks in each region and the maps between them  as in \eqref{eq:stalks}. 
\begin{equation}\label{eq:stalks}
        \begin{tikzpicture}[ scale =.8, baseline={(current bounding box.center)}]
        \draw[line width = 1pt] (-1, 0) -- (0, 0) .. controls (1.5, 0) and (.5, 1) .. (2, 1) -- (3, 1); 
        \draw[line width =1pt] (-1, 1) -- (0, 1) .. controls (1.5, 1) and (.5, 0) .. (2, 0) -- (3, 0);
        \draw [line width = 1pt] (-1, -3) -- (3, -3); 
        \draw [line width = 1pt] (-1, -2) -- (3, -2);
        
        \draw (-.5, 1) -- (-.5, .8);
        \draw (-.5, 0) -- (- .5, -.2);

        \draw (2.5, 1) -- (2.5, .8);
        \draw (2.5, 0) -- (2.5, -.2);
        
        \draw (2.5, -2) -- (2.5, -1.8);
        \draw (2.5, -3) -- (2.5, -2.8);
        \draw (-.5, -2) -- (-.5, -1.8); 
        \draw (-.5, -3) -- (-.5, -2.8);

        \node[anchor = east] at (.9, .5) {\small \(W\)};
        \node[anchor = west] at (1.1, .5) {\small \(E\)};

        \node at (1, -1) {\small \(M\)};
        \node at (1, -2.5) {\small $S$};

        \draw[dashed, Green] (-1, -3.5) -- (-1, 1.5);
        \draw[dashed, Green] (3, -3.5) -- (3, 1.5);
    \end{tikzpicture}.
    \qquad
    \begin{tikzcd}
    W &  & E \\
    & M \arrow[ul, "f_{W}"] \arrow[ur, "f_{E}"'] \ar[d, "f_{S}"]& \\
    & S &
    \end{tikzcd}.
    \end{equation}
Note that by the singular support condition, $\ker(f_W) \neq \ker(f_E)$. If $\ker(f_S)$ coincides with either $\ker(f_W)$ or $\ker(f_E)$, the sheaf $\cE$ would split as a direct sum of two costandard sheaves. Suppose $\ker(f_S)$, $\ker(f_W)$, and $\ker(f_E)$ are transverse to each other. Then we get the sheaf described in $(3)$ of Proposition~\ref{prop:indecomposable}. 

\smallskip
\noindent \textbf{Case II(B) :} $\SS(\cE)$ has the form:
\[
      \begin{tikzpicture}[top aligned, yscale = .5, baseline={(current bounding box.center)}]
        \draw[line width = 1pt] (0, 6) -- (2, 6);
        \draw[line width = 1pt] (0, 1) -- (2, 1);
        \draw[line width = 1pt] (0, 3) -- (
        2, 4);
        \draw[line width = 1pt] (0, 4) -- (2, 3);
        \draw[dashed, Green] (0, 0) -- (0, 7);
        \draw[dashed, Green] (2, 0) -- (2, 7);
    \end{tikzpicture}. 
    \]
    The crossing is of opposite codirections.
    If $\cE_\mathrm{red}$ and $\cE_\mathrm{blue}$ are both non-zero, the only cases in which $\cE_\mathrm{blue}$ and $\cE_\mathrm{red}$ can have a non-trivial extension are the two left-hand cases in \eqref{eq:indecomposable_two}. 

    Suppose $\cE \simeq \cE_\mathrm{blue}$ or $\cE \simeq \cE_\mathrm{red}$. Say $\cE \simeq \cE_\mathrm{blue}$. Then by the microlocal stalk condition, it must have the form:
    \[
     \begin{tikzpicture}[ yscale = .8, xscale = 1.5, baseline=(current bounding box.center)]
        \draw[line width = 1pt] (0, 5) -- (2, 5);
        \draw[line width = 1pt] (0, 2) -- (2, 2);
        \draw[line width = 1pt] (0, 3) -- (2, 4);
        \draw[line width = 1pt] (0, 4) -- (2, 3);
        \draw[dashed, Green] (0, 1.5) -- (0, 5.5);
        \draw[dashed, Green] (2, 1.5) -- (2, 5.5);
        \node[right] at (0, 3.5) {\tiny $k^2[1]$};
        \node[left] at (2, 3.5) { $0$};
        \node at (1, 2.5) { $k[1]$ }; 
        \node at (1, 4.4) { $k[1]$ };
        \fill[blue, opacity = .3] (0, 4) -- (2, 3) -- (2, 2) -- (0, 2); 
        \fill[blue, opacity = .3] (0, 3) -- (2, 4) -- (2, 5) -- (0, 5); 
    \end{tikzpicture}
    \qquad \text{or} \qquad
    \begin{tikzpicture}[ yscale = .8, xscale = -1.5, baseline=(current bounding box.center)]
        \draw[line width = 1pt] (0, 5) -- (2, 5);
        \draw[line width = 1pt] (0, 2) -- (2, 2);
        \draw[line width = 1pt] (0, 3) -- (2, 4);
        \draw[line width = 1pt] (0, 4) -- (2, 3);
        \draw[dashed, Green] (0, 1.5) -- (0, 5.5);
        \draw[dashed, Green] (2, 1.5) -- (2, 5.5);
        \node[left] at (0, 3.5) {\tiny $k^2[1]$};
        \node[right] at (2, 3.5) { $0$};
        \node at (1, 2.5) { $k[1]$ }; 
        \node at (1, 4.4) { $k[1]$ };
        \fill[blue, opacity = .3] (0, 4) -- (2, 3) -- (2, 2) -- (0, 2); 
        \fill[blue, opacity = .3] (0, 3) -- (2, 4) -- (2, 5) -- (0, 5); 
    \end{tikzpicture}.
    \]
    which splits into a direct sum of two costandard sheaves. 

    \smallskip
    \noindent \textbf{Case III.} Recall that the decomposition of $\wtd \cF|_{\wtd \T^\epsilon_i}$ induces a pairing between the components of \(e^{-1}\Lambda|_{\wtd\T_i^\epsilon}\).
    Denote the two components of $\wtd \Lambda_0^\mathrm{cr}$ by $\wtd \Lambda_0^\mathrm{cr} = \Lambda_0^\mathrm{u} \coprod \Lambda_0^\mathrm{d}$, as in the figure below.     
    \[
      \begin{tikzpicture}[ yscale = .7, xscale = 1.5]
        \draw[line width = 1pt] (0, 0) -- (2, 1);
        \draw[line width = 1pt] (0, 1) -- (2, 0);
        \draw[line width = 1pt, dashed] (0, 3) -- (
        2, 4);
        \draw[line width = 1pt] (0, 4) -- (2, 3);

       \draw[line width = 1pt, dashed] (0, -2) -- (2, -3);
       \draw[line width = 1pt] (0, -3) -- (2, -2); 

        \draw[line width = 1.5pt, blue]  (0, -3) -- (.5, -2.75) ;

        \draw[line width = 1.5pt, blue] (.5, .75) -- (0, 1);
        
        \fill [pattern={Lines[angle=90,distance=1mm,line width=0.3pt]}, pattern color = blue] (0, -3) -- (.5, -2.75) -- (.5, .75) -- (0, 1); 

        \draw[line width = 1.5pt, Blue](0, 0) -- (.5, .25) ; 
        \draw[line width = 1.5pt, Blue]  (.5, 3.75) -- (0, 4); 
        
        \fill[Blue, opacity = .4] (0, 0) -- (.5, .25) -- (.5, 3.75) -- (0, 4); 
        
       \node[left] at (0, 1) {\small $\Lambda_0^\mathrm{u}$};
       \node[left] at (0, 0) {\small $ \Lambda_0^\mathrm{d}$};

       \node[left] at (0, 3) {\small $\Lambda_2^\mathrm{d}$};
       \node[left] at (0, 4) {\small $ \Lambda_2^\mathrm{u}$};

       \node[left] at (0, -2) {\small $\Lambda_1^\mathrm{u}$};
       \node[left] at (0, -3) {\small $ \Lambda_1^\mathrm{d}$};
        
        \draw[dashed, Green] (0, -4) -- (0, 5);
        \draw[dashed, Green] (2, -4) -- (2, 5);

        \end{tikzpicture}. 
        \]
        Since \(e_* \cE \in \Sh^{\le 1} (\T_i, \Lambda_i)\), the component $\Lambda_0^\mathrm{u}$ must be paired with some component of \(e^{-1} (  \Lambda_0^\mathrm{cr})\) with positive slope, which we denote by $\Lambda_1^\mathrm{d}$. Assume $\Lambda_1^\mathrm{d}$ is below $\Lambda_0^\mathrm{u}$ (the other case is similar). Since the sheaf $\wtd \cF$ is $\bZ$-equivariant, $\Lambda_0^\mathrm{d}$ must be paired with some component of \(e^{-1} ( \wtd \Lambda_0^\mathrm{cr})\) with negative slope, living above $\Lambda_0^\mathrm{d}$. Denote it by $\Lambda_2^\mathrm{u}$. Therefore, the singular support of $\cE$ is contained in the four components, and the situation is the same as the latter case in \textbf{Case II(B)}. In particular, $\cE$ is a direct sum of two costandard sheaves. 
        \end{proof}

\begin{corollary}\label{cor:sheaf_to_ruling}
    A sheaf $\cF \in \Sh^1(\T, \Lambda )_0$ uniquely determines a ruling on $\T$.   
\end{corollary}

\subsection{From Ruling to Sheaf}

For a ruling \(\rho\), let \( \calM^\rho_{\triangle, \eta}\) denote the submoduli parametrizing sheaves that give rise to the ruling \(\rho\). Then there is a \textit{ruling decomposition}:
\begin{equation}\label{eq:ruling decomposition}
\left|\calM_{\triangle, \eta} \right| \cong \coprod_{\rho \in \mathfrak{R}(\Lambda) } \left|\calM^{\rho}_{\triangle, \eta} \right| 
\end{equation}
where \(\mathfrak{R}(\Lambda)\) is the set of all rulings.
As in the Legendrian link case, we expect each stratum to be
\begin{equation}
 \left|\calM^\rho_{\triangle, \eta} \right| \cong (\bC^*)^{2g(\rho)} \times \bC^{r (\rho)}.
\end{equation}
The \emph{standard ruling} obtained from the bipartite graph should give rise to an open dense stratum, isomorphic to the moduli of local systems on the bipartite graph. In particular, we can compute the Euler characteristic by counting the rulings such that \(g (\rho) = 0\).

\begin{remark}
    A ruling surface $\Sigma(\rho)$ can be viewed as an immersed Lagrangian filling of the Legendrian $\Lambda$ in $T^* \T$. Heuristically, under the sheaf-Fukaya correspondence~\cite{NZ, GPS3}, one can view $\calM^\rho_{\triangle, \eta}$ as the sheaf quantization of this immersed Lagrangian. The torus part $(\bC^*)^{2g(\rho)}$ comes from the topology of the surface, and the affine part $\bC^{r(\rho)}$ comes from the immersed points. 
\end{remark}

Denote
\[ 
\widetilde{\bbL} \coloneqq \left\{ (z_i)_{i \in \pi_0 (\Lambda ) } \in   \left|\calM (\Loc^1 (\Lambda))\right| \simeq (\bC^*)^{\pi_0(\Lambda)} \mid \prod_i z_i = 1\right\}
\] 
and its dense open subset:
\[
\bbL \coloneqq \left\{ (z_i)_{i} \in \widetilde{\mathbb{L}} \mid  \text{ for every nonempty proper subset } I \subset \pi_0 (\Lambda), \ \prod_{i\in I} z_i \neq 1 \right\}
\]

We assume $\eta$ is \emph{generic}, meaning \(\eta \in \bbL\). Let \( \Loc^1\left(\Sigma (\rho) \right)_\eta \) be the subcategory of rank-\(1\) local systems on \(\Sigma(\rho)\) whose restriction to \(\Lambda\) is \(\eta\). Note that \({ \Loc^1 ( \Sigma(\rho))_\eta }   \) is nonempty if and only if \(\Sigma(\rho)\) is connected. 
Let $\moduli{\Loc^1(\Sigma(\rho))_\eta}$ be its moduli stack. 
If we identify \(\Sigma (\rho)\) with a genus \(g\) surface with \(|\pi_0(\Lambda)|\) disks removed, then we have
\[
\moduli{\Loc^1(\Sigma(\rho))_\eta} \simeq (\bC^*)^{2g} \times B\Gm
\]


\begin{theorem}\label{thm:ruling_decomposition}
For a ruling \(\rho\), there is a map
\begin{equation}\label{eq:from M_rho to loc}
    f_\rho: \calM^\rho_{\triangle, \eta} \longrightarrow
    \calM \left( \Loc^1 ( \Sigma (\rho) )_\eta \right).
\end{equation}
The fiber is isomorphic to an affine space \(\bC^{r(\rho)}\). Here $r(\rho)$ is the number of returns of $\rho$. Hence we have 
\[
    |\calM_{\triangle,\eta}^\rho| \simeq (\bC^*)^{2 g(\rho)} \times \bC^{r(\rho)}. 
\]

\end{theorem} 

By a \emph{rational ruling}, we mean a ruling \(\rho\) with $g(\rho) = 0$. Define 
\[
N_\triangle^\mathrm{rul} = \# \{ \text{rational rulings} \}
\]
As a corollary, we have 
\begin{theorem} 
\label{thm:rulings}
   \( N^\mathrm{rul}_\triangle = \chi( |\calM_{\triangle, \eta}|) = N_\triangle^\Coh\). 
\end{theorem}

The remainder of this section is devoted to proving Theorem~\ref{thm:ruling_decomposition}. The idea of the proof already appeared in the proof of Lemma~\ref{lem:characterization of Sh_0 via restriction on S1}. Namely, we cut the torus \(\T\) into small ``strips'' and glue the moduli spaces of sheaves.




Fix a ruling \(\rho\). Let
\(
\Sh^1(\T_j^\epsilon,\Lambda_j^\epsilon)_\rho
\)
(resp. \(\Sh^1(\T_j,\Lambda_j)_\rho\))
denote the full subcategory of
\(\Sh(\T_j^\epsilon,\Lambda_j^\epsilon)\)
\(\bigl(\text{resp. } \Sh(\T_j,\Lambda_j)\bigr)\)
spanned by objects whose associated local ruling is \(\rho\).
By Proposition~\ref{prop:indecomposable}, each of these categories
contains a unique object. Denote it by $ \cF_j^\epsilon$ (resp. $\cF_j$). 
By Proposition~\ref{prop:ruling in 1d}, each  \(\cF_j^\epsilon\) splits into a direct sum of \( e_* \dsk_U[2] \) and \(e_* \dsk_Z[1]\).

Take:
\begin{equation*}
    \calM^\rho_j \coloneqq B \Aut(\cF_j), \qquad \calM_j^{\rho, \epsilon} \coloneqq B \Aut(\cF_j^\epsilon).
\end{equation*}
Then \(\calM_j^\rho\) is the moduli stack of  \(\Sh^1(\T_j, \Lambda_j)_\rho\), and likewise \(\calM_j^{\rho, \epsilon}\) is the moduli stack of \(\Sh^1(\T_j^\epsilon, \Lambda_j^\epsilon)_\rho\).

Let \(\Sh^1(\T, \Lambda)_\rho\) be the subcategory of \(\Sh^1(\T, \Lambda)_0\) which induces the ruling \(\rho\) (without conditions on the microlocal monodromies), and let \(\calM^\rho_\triangle\) be the corresponding moduli.  
By sheaf properties,  we have:
\begin{equation}\label{eq:gluing_of_sheaf_moduli}
    \calM_{\triangle}^\rho
    =
    \lim \left(  \prod_j \calM_j^\rho \rightrightarrows \prod_j \calM_j^{\rho,\epsilon} \right). 
\end{equation}

Now let \(\iota\) be the natural map \(\iota: \Sigma(\rho) \longrightarrow \T\), and denote \( \Sigma(\rho)_i \coloneqq \iota^{-1} (\T_i) \) and \(\Sigma(\rho)_i^\epsilon \coloneqq \iota^{-1}(\T^\epsilon_i) \). Then
\[
\calM ( \Loc^1 ( \Sigma (\rho) ) ) = \lim \left ( \prod_i \calM (\Loc^1 (\Sigma(\rho)_i) ) \rightrightarrows \prod_i \calM ( \Loc^1 (\Sigma(\rho)_i^\epsilon ) )  \right).
\]

The natural functors \(\Loc^1(\Sigma(\rho)_i) \rightarrow \Sh^1(\T_i, \Lambda_i)_\rho \) and  \(\Loc^1(\Sigma(\rho)_i^\epsilon) \rightarrow \Sh^1(\T_i^\epsilon, \Lambda_i)_\rho \) induce 
\begin{align*}
    \calM_j^\rho &\longrightarrow \calM (\Loc^1(\Sigma(\rho)_j)) \\
    \calM_j^{\rho, \epsilon} &\longrightarrow \calM (\Loc^1(\Sigma(\rho)_j^\epsilon)) 
\end{align*}
and hence the global map:
\begin{equation}
    \calM_\triangle^\rho \longrightarrow \calM(\Loc^1( \Sigma(\rho)) ). 
\end{equation}
Taking the Cartesian diagram
\[
\begin{tikzcd}
\calM_{\triangle, \eta}^\rho \arrow[r, "f_\rho"] \arrow[d] 
  &  \moduli{\Loc^1 (\Sigma(\rho))_\eta} \arrow[d] \\
\calM_\triangle^\rho \arrow[r] 
  & \moduli{\Loc^1 (\Sigma(\rho))}
\arrow[from=1-1, to=2-2, phantom, "\lrcorner", very near start]
\end{tikzcd}
\]
yields the desired map~\eqref{eq:from M_rho to loc}. 

Now we take a point \(\cL\) in \(\moduli{\Loc^1(\Sigma(\rho))_\eta} \), and consider the fiber of \(f_\rho\) in~\eqref{eq:from M_rho to loc} over \(\cL\).

Denote
\begin{align*}
    \cK_j
    &\coloneqq \fib \big( \calM_j^\rho
    \longrightarrow \moduli{\Loc^1( \Sigma (\rho)_j)} \big) \\
    \cK_j^\epsilon
    & \coloneqq \fib \big(\calM_j^{\rho, \epsilon} \longrightarrow \moduli{ \Loc^1(\Sigma(\rho)_j^\epsilon)} \big) \\
    \cK 
    &\coloneqq \fib \big(\calM_{\triangle}^\rho 
    \longrightarrow \moduli{ \Loc^1( \Sigma(\rho) ) }\big). 
\end{align*}
Note that since both $ \calM_\triangle^\rho $ and \(\moduli{\Loc^1(\Sigma(\rho))}\) have a \(\Gm\)-gerbe, \(\cK\) is indeed an algebraic variety. 

    Since limits commute with limits, we have:
\begin{equation}\label{eq:K_limit}
\cK = \lim \left( \prod_i \cK_i \rightrightarrows \prod_i \cK_i^\epsilon \right) 
\end{equation}

Let the unique object in \(\Sh^1(\T_i^\epsilon, \Lambda_i^\epsilon)_\rho \) be \( \cF_i^\epsilon = \bigoplus_j \cE_{i_j}\), where each \(\cE_{i_j}\) is isomorphic to an eye sheaf \(e_* \dsk_{U}[2]\) or \(e_* \dsk_{Z}[1]\).  We define a partial order \(\prec\) on \( \{\cE_{i_j}\} \), given by:
\begin{itemize}
    \item \(e_*\dsk_{Z}[1] \prec  e_*\dsk_{Z'}[1]\) if \( Z \supset \mathbf{d}_m ( {Z'})\) for some \(m\in \bZ\), where \(\mathbf{d}_m\) is the deck transformation \( (x, y) \mapsto (x, y+m)\) of the covering map \( (a_i, a_{i+1}) \times \bR  \rightarrow (a_i, a_{i+1}) \times S^1 \). 
    \item \(e_*\dsk_{U}[2] \prec e_*\dsk_{U'}[2]\) if \(U \subset \mathbf{d}_m(U')\) for some \(m\in \bZ\). 
    \item \(e_*\dsk_Z[1] \prec e_*\dsk_U[2] \) if \( Z\cap \mathbf{d}_m (U) \neq \emptyset \) for some \(m \in \bZ\). 
\end{itemize}
Apparently, this is a well-defined partial order, and \( \Hom (\cE_{i_j}, \cE_{i_j'} ) \neq 0\) if and only if \( \cE_{i_j} \prec \cE_{i_j'}\). Thus, as a \(\bC\)-module, we have 
\begin{equation}\label{eq:End_F_epsilon}
\End (\cF_i^\epsilon) = \bigoplus_{i_j} \End( \cE_{i_j}) \oplus \bigoplus_{\cE_{i_j}\prec  \, \cE_{i_j'} } \Hom (\cE_{i_j}, \cE_{i_j'}). 
\end{equation}
Note that each \(\End(\cE_{i_j}) \simeq \bC\). If all \( \Hom (\cE_{i_j}, \cE_{i_j'}) \simeq \bC\), then \(\End (\cF_i^\epsilon)\) is a subring of upper triangular matrices.  

      As an algebraic variety, \(\Aut (\cF_i^\epsilon) \simeq \prod_{i_j} \bC^* \times \prod_{ \cE_{i_j} \prec \cE_{i_j'} } \Hom (\cE_{i_j}, \cE_{i_j'})  \). Let 
\[ G_i^\epsilon \coloneqq \ker \left( \Aut (\cF_i^\epsilon) \rightarrow \prod_{i_j} \bC^*  \right).
\]
Then as a variety, \( G_i^\epsilon \) is isomorphic to an affine space \(\bC^{k_i} \simeq \prod_{\cE_{i_j}\prec  \, \cE_{i_j'} } \Hom (\cE_{i_j}, \cE_{i_j'})\), and we have: \( \cK_i^\epsilon \simeq B G_i^\epsilon\).




\begin{lemma}\label{lem:fiber_seq_of_Moduli}
    The fiber of \(\calM^\rho_j \rightarrow \calM^{\rho, \epsilon}_{j}\) is 
    \begin{enumerate}
        \item \(\Gm\), if the crossing in \(T_j\) is a ``vertical'' switch of type A, i.e. the ruling looks like one of the following:
       \[ 
        \begin{tikzpicture}
        \begin{scope}[rotate = 45]
            \draw[line width = 2pt, red] (-1, 0) -- (0, 0) -- (0, -1);
            \draw[line width = 2pt, blue] (0, 1) -- (0, 0) -- (1, 0); 
            \draw (.5, 0) -- (.5, .2);
            \draw (-.5, 0) -- (-.5, .2);
            \draw (0, .5) -- (.2, .5);
            \draw (0, -.5) -- (.2, -.5);
            \filldraw[blue, opacity=.4] (0, 0) -- (1, 0) -- (0, 1); 
            \filldraw[red, opacity=.4] (0, 0) --(-1, 0) -- (0, -1);
        \end{scope}
        \begin{scope}[shift = {(3, 0)}, rotate = 225]
            \draw[line width = 2pt, red] (-1, 0) -- (0, 0) -- (0, -1);
            \draw[line width = 2pt, blue] (0, 1) -- (0, 0) -- (1, 0);
            \draw (.5, 0) -- (.5, .2);
            \draw (-.5, 0) -- (-.5, .2);
            \draw (0, .5) -- (.2, .5);
            \draw (0, -.5) -- (.2, -.5);
            \filldraw[blue, opacity=.4] (0, 0) -- (1, 0) -- (0,1); 
            \filldraw[red, opacity=.4] (0, 0) -- (-1, 0) -- (0, -1);
        \end{scope}
        \end{tikzpicture}
        \]
        or if the crossing in $T_j$ is a switch of type B. 
        \item \(\Ga\), if the crossing in \(T_j\) is a return.

        \item trivial, for other cases. 
    \end{enumerate}
\end{lemma}
\begin{proof}
    We only need to show the map
    \begin{equation}\label{eq:map of auto}
        \Aut (\cF_{j}) \longrightarrow \Aut(\cF_j^{\epsilon})
    \end{equation}
    is always injective and the quotient has the prescribed form. 

    We can cancel a switch of type B by a sequence of moves in Figure~\ref{fig:killing a B-switch}. 
    Note that this process only creates a new vertical switch of type A, and does not introduce new departures or returns. Thus, we can assume that there are no switches of type B at the beginning.

    \begin{figure}[htbp]
    
    \centering
    \begin{tikzpicture}[scale =.6, line width = 2pt]
       \draw[Blue] (-1, 1) -- (0, 1) -- (3, 0) -- (6, 1) -- (7, 1);
        \draw[blue, line width =1pt] (-1, -1) -- (0, -1) -- (3, 0)-- (6, -1) -- (7, -1);

        \draw[Blue] (-1, -2.5) -- (7, -2.5);
        \draw[blue] (-1, -3) -- (7, -3);

        \fill[Blue, opacity = .4] ( -1, 1) -- (0, 1) -- (3, 0) -- (6, 1) -- (7, 1) -- (7, -2.5) -- (-1, -2.5); 

        \fill[pattern={Lines[angle=-60,distance=1mm,line width=0.3pt]}, pattern color = blue] (-1, -1) -- (0, -1) -- (3, 0) -- (6, -1) -- (7, -1) -- (7, -3) -- (-1, -3); 
        
        \begin{scope}[yscale = 1, shift = {(0, -1.5)}] 
        \draw (-.5, 0) -- (-.5, -.15); 
        \draw (0, 0) -- (0, -.15);
        \draw (.5, 0) -- (.5, -.15);
        \end{scope}

        \begin{scope}[yscale = -1, shift = {(0, 2)}] 
        \draw (-.5, 0) -- (-.5, -.15); 
        \draw (0, 0) -- (0, -.15);
        \draw (.5, 0) -- (.5, -.15);
        \end{scope}

        \draw (-1, -1.5) -- (7, -1.5); 
        \draw (-1, -2) -- (7, -2); 

    \begin{scope}[shift = {(10, 0)} ]
        \draw[Blue] (-1, 1) -- (0, 1) -- (3, 0) -- (6, 1) -- (7, 1);
        \draw[blue, line width =1pt] (-1, -1) -- (0, -1) -- (3, 0)-- (6, -1) -- (7, -1);

        \draw[Blue] (-1, -2.5) -- (7, -2.5);
        \draw[blue] (-1, -3) -- (7, -3);

        \fill[Blue, opacity = .4] ( -1, 1) -- (0, 1) -- (3, 0) -- (6, 1) -- (7, 1) -- (7, -2.5) -- (-1, -2.5); 

        \fill[pattern={Lines[angle=-60,distance=1mm,line width=0.3pt]}, pattern color = blue] (-1, -1) -- (0, -1) -- (3, 0) -- (6, -1) -- (7, -1) -- (7, -3) -- (-1, -3); 

        \draw (-1, -2) -- (0, -2) .. controls (1, -2) and (0, 2) .. (2.5, 2) -- (3.5, 2) .. controls (6, 2) and (5, -2) .. (6, -2) -- (7, -2); 

        \draw (-1, -1.5) -- (0, -1.5) .. controls (1, -1.5) and (0, -3.5) .. (2.5, -3.5) -- (3.5, -3.5) .. controls (6, -3.5) and (5, -1.5) .. (6, -1.5) -- (7, -1.5);
        
        \draw[line width = 1pt] (2.5, 2) -- (2.5, 2.3); 
        \draw[line width = 1pt] (3, 2) -- (3, 2.3); 
        \draw[line width = 1pt] (3.5, 2) -- (3.5, 2.3); 

        \draw[line width = 1pt] (2.5, -3.5) -- (2.5, -3.8); 
        \draw[line width = 1pt] (3, -3.5) -- (3, -3.8); 
        \draw[line width = 1pt] (3.5, -3.5) -- (3.5, -3.8); 
        
    \end{scope}

    \begin{scope}[shift = {(0, -7)} ]
        \draw[Blue] (-1, 1) -- (0, 1) -- (3, 0) -- (6, 1) -- (7, 1);
        \draw[blue, line width =1pt] (-1, -1) -- (0, -1) -- (3, 0)-- (6, -1) -- (7, -1);

        \draw[Blue] (7, -2.5) -- (5, -2.5) -- (4, -1) -- (2, -1) -- (1, -2.5)  -- (-1, -2.5);
        \draw[blue] (-1, -3) -- (7, -3);

        \fill[Blue, opacity = .4] ( -1, 1) -- (0, 1) -- (3, 0) -- (6, 1) -- (7, 1) -- (7, -2.5) -- (5, -2.5) -- (4, -1) -- (2, -1) -- (1, -2.5)  -- (-1, -2.5); 

        \fill[pattern={Lines[angle= -60,distance=1mm,line width=0.3pt]}, pattern color = blue] (-1, -1) -- (0, -1) -- (3, 0) -- (6, -1) -- (7, -1)  -- (7, -3) -- (-1, -3); 

        \draw[line width = .2 pt] (-1, -2) -- (0, -2) .. controls (1, -2) and (0, 2) .. (2.5, 2) -- (3.5, 2) .. controls (6, 2) and (5, -2) .. (6, -2) -- (7, -2); 

        \draw[line width = .2 pt]  (-1, -1.5) -- (0, -1.5) .. controls (1, -1.5) and (0, -3.5) .. (2.5, -3.5) -- (3.5, -3.5) .. controls (6, -3.5) and (5, -1.5) .. (6, -1.5) -- (7, -1.5);

        \end{scope}

        \begin{scope}[shift = {(10, -7)}]

        \draw[line width = .2 pt] (-1, -2) -- (0, -2) .. controls (1, -2) and (0, 2) .. (2.5, 2) -- (3.5, 2) .. controls (6, 2) and (5, -2) .. (6, -2) -- (7, -2); 

        \draw[line width = .2 pt]  (-1, -1.5) -- (0, -1.5) .. controls (1, -1.5) and (0, -3.5) .. (2.5, -3.5) -- (3.5, -3.5) .. controls (6, -3.5) and (5, -1.5) .. (6, -1.5) -- (7, -1.5);

        \draw[blue] (-1, -3) -- (7, -3);

        \draw[blue] (-1, -1) -- (1, -1) -- (3, 0)-- (5, -1) -- (7, -1); 

        \fill[pattern={Lines[angle=-60,distance=1mm,line width=0.3pt]}, pattern color = blue] (-1, -1) -- (1, -1) -- (3, 0) -- (5, -1) -- (7, -1) -- (7, -3) -- (-1, -3);

        \draw[Blue] (-1, -2.5) -- (1, -2.5)  -- (15/8, 9/16) -- (1, 1) -- (-1, 1); 
        \fill[Blue, opacity = .4] (1, 1) -- (-1, 1) -- (-1, -2.5) -- (1, -2.5)  -- (15/8, 9/16);

        \begin{scope}[shift = {(6, 0)}, xscale = -1]
            \draw[Blue] (-1, -2.5) -- (1, -2.5)  -- (15/8, 9/16) -- (1, 1) -- (-1, 1); 
            \fill[Blue, opacity = .4] (1, 1) -- (-1, 1) -- (-1, -2.5) -- (1, -2.5)  -- (15/8, 9/16);
        \end{scope}

        \draw[red] (15/8, 9/16) -- (2, 1) -- (4, 1) -- (33/8, 9/16) -- (3, 0) -- cycle; 
        \fill[red, opacity = .4] (15/8, 9/16) -- (2, 1) -- (4, 1) -- (33/8, 9/16) -- (3, 0); 
        
        \end{scope}
    \end{tikzpicture}

        \caption{Using GKS moves to kill a switch of type B}
        \label{fig:killing a B-switch}
    \end{figure}
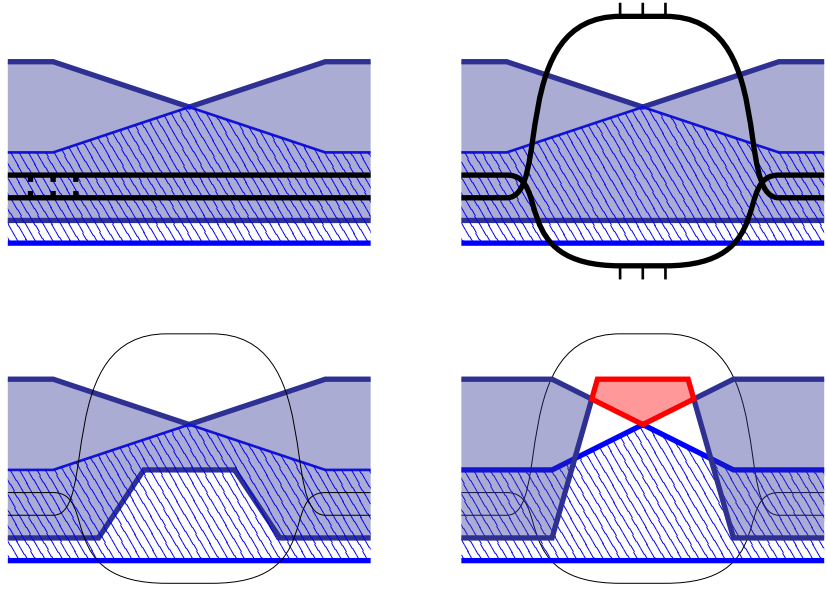

    Now let us consider the map~\eqref{eq:map of auto}. Assume \(\cF_j \simeq \cE \oplus \cF\), where $\cE$ is the minimal component of $\cF_j$ whose singular support contains $\Lambda_j^\mathrm{cr}$. For each type of crossing, the corresponding $\cE$ is given in Proposition~\ref{prop:indecomposable}. 
    We have
    \[
    \End(\cF_j) \simeq \begin{pmatrix}
    \End( \cF)  & \Hom(\cF, \cE) \\
    \Hom(\cE, \cF) & \End(\cE)
    \end{pmatrix}
    \]
    Denote \( \cE^\epsilon =  \cE|_{\T^\epsilon_j} \) and \(  \cF^{\epsilon} =  \cF|_{\T^\epsilon_j} \). Note that $\cF$ has the form $\cF \simeq e_* \left(\bigoplus_a \dsk_{U_a}[2] \oplus \bigoplus_b \dsk_{Z_b}[1] \right)$, where $e$ is the universal cover and $U_a$ and $Z_b$ are open and closed sets, respectively. In particular, the restriction $\End(\cF) \rightarrow \End(\cF^\epsilon)$ is an isomorphism. Moreover, since $q(\SS(\cF)) \cap q(\SS(\cE)) = \emptyset$ (recall $q$ is the projection to the base), we have 
    \[
    \SS ( \mathcal{H}om (\cE, \cF))  \subset \SS (\cF) \cup \SS(\cE) 
    \]
    where $ \mathcal{H}om (\cE, \cF) $ is the internal hom (which is a sheaf). In particular, $\SS ( \mathcal{H}om (\cE, \cF)) $
    has no horizontal codirection. By the non-characteristic deformation lemma~\ref{lem:non_char_def}, the restriction map $\Hom (\cE, \cF) \rightarrow \Hom(\cE^\epsilon, \cF^\epsilon)$ is an isomorphism. Similarly, we have $\Hom(\cF, \cE) \xrightarrow{\sim} \Hom(\cF^\epsilon, \cE^\epsilon)$. Therefore, we only need to show the map 
    \begin{equation}\label{eq:map_aut_E}
    \Aut( \cE) \longrightarrow \Aut (\cE^\epsilon)
    \end{equation}
    is an injection, and it has the prescribed quotient. 
    Note that we can write $\cE \simeq e_* \cE_0$ and $ \cE^\epsilon \simeq e_*\cE_0^\epsilon$. Then 
    $\End(\cE) \simeq \bigoplus_{m \in \bZ} \Hom( (\mathbf{d}_m)_*\cE_0, \cE_0) $, where $\mathbf{d}_m$ is the deck transformation of the covering map. 
    If $\cE_0$ is in \textbf{Case III} in the proof of Proposition~\ref{prop:indecomposable}, a direct computation shows that \( \Hom( (\mathbf{d}_m)_* \cE_0,  \cE_0) \simeq 0 \simeq \Hom( (\mathbf{d}_m)_* \cE_0^\epsilon, \cE_0^\epsilon) \) unless \(m = 0\). For \textbf{Case I} and \textbf{Case~II}, a similar non-characteristic deformation argument shows that 
    $\Hom ( (\mathbf{d}_m)_* \cE_0, \cE_0) \rightarrow \Hom ( (\mathbf{d}_m)_* \cE_0^\epsilon, \cE_0^\epsilon) $
    is an isomorphism for \(m \neq 0\). Hence we only need to consider \( \Aut(\cE_0) \rightarrow \Aut(\cE_0^\epsilon)\). 
    Note that $\cE_0$ and $\cE_0^\epsilon$ are sheaves on a small interval times $\bR$. The desired result follows from a direct case-by-case computation. 
\end{proof}

\begin{corollary}\label{cor:fiber_of_Ki}
    The fiber of \( \cK_i \to \cK_i^\epsilon \) is \(\Ga\) for a return, and trivial otherwise. 
\end{corollary}

Now we are ready to give the proof of Theorem~\ref{thm:ruling_decomposition}. 

\begin{proof}[Proof of Theorem~\ref{thm:ruling_decomposition}] 
We only need to show that $\cK \simeq \bC^{ r(\rho)}$.

Denote \({}_0\cK_{j} \coloneqq \lim \left(\prod_{i=0}^{j-1} \cK_i \rightrightarrows \prod_{i = 0}^{j} \cK_i^\epsilon \right) \).
Following \cite[Definition 2.19]{STWZ}, we consider the framed version \({}_0\cK_j^{fr}\) and  \(\cK^{fr}\)  defined by
\begin{equation*}
\begin{tikzcd}
    {}_0\cK_j^{fr} \ar[r] \ar[d] & {}_0\cK_j \ar[d]\\
    \Spec \bC \ar[r] & \cK_0^\epsilon
    \arrow[from=1-1, to=2-2, phantom, "\lrcorner", very near start]
\end{tikzcd}
\qquad
\begin{tikzcd}
    \cK^{fr} \ar[r] \ar[d] & \cK \ar[d]\\
    \Spec \bC \ar[r] & \cK_0^\epsilon
    \arrow[from=1-1, to=2-2, phantom, "\lrcorner", very near start]
\end{tikzcd}
\end{equation*}
Note that since $\cK_0^\epsilon \simeq B G_0^\epsilon$, as a variety we have \(\cK^{fr} \simeq G_0^\epsilon\times \cK\). In other words, $\cK^{\mathrm{fr}}$ is a trivial $G_0^\epsilon$ bundle over $\cK$.

For each \(j \ge 0\), we have a fiber product:
\[
\begin{tikzcd}
    {}_0\cK^{fr}_{j+1} \arrow[r] \arrow[d] \arrow[dr, phantom, "\lrcorner", very near start]  & \cK_{j} \arrow[d]\\
      {}_0\cK_{j}^{fr} \arrow[r] & \cK_{j}^\epsilon 
\end{tikzcd}
\]
Thus, we can start with \({}_0\cK_0^{fr} = \{*\}\) and construct \({}_0\cK_{n}^{fr}\) inductively. 
By Corollary~\ref{cor:fiber_of_Ki}, we have \({}_0\cK^{fr}_{n} \cong \bC^{r(\rho)}\). Then we can obtain \(\cK^{fr}\) by gluing:
\[
\begin{tikzcd}
    \cK^{fr} \arrow[r] \arrow[d] \arrow[dr, phantom, "\lrcorner", very near start]  & \cK_n^\epsilon  \arrow[d]\\
      {}_0\cK_{n}^{fr} \arrow[r] & \cK_{n}^\epsilon \times \cK_{0}^\epsilon .
\end{tikzcd}
\]
Recall that the subscript is understood modulo \(n\), hence \(\cK_n^\epsilon = \cK^\epsilon_0\), and \(\cK^\epsilon_n \longrightarrow \cK_n^\epsilon \times \cK_0^\epsilon \) is the diagonal map.  This implies 
\begin{align*} 
\fib (\cK^{fr} \rightarrow {}_0\cK_{n}^{fr}) &\cong \fib \left( \cK_n^\epsilon {\longrightarrow} \cK^\epsilon_n \times \cK^\epsilon_n \right) \\
&\cong \fib \left( B G_n^\epsilon \stackrel{\Delta}{\longrightarrow} B G_n^\epsilon \times B G_n^\epsilon \right) \\
&\cong G_n^\epsilon.
\end{align*}
Thus we get \(\cK^{fr} \cong \bC^{r(\rho)} \times G_0^\epsilon\), which implies that \(\cK \cong \cK^\mathrm{fr} / G_0^\epsilon \cong \bC^{r(\rho)}\).  
\end{proof}

We make the following observation.

\begin{proposition}\label{prop:g_plus_r}
    For any ruling $\rho$, let $g(\rho)$ be the genus of \(\Sigma(\rho)\), and let $r(\rho)$ be the number of returns. Then
    \[
    g_\triangle = g(\rho) + r(\rho), 
    \]
    where $g_\triangle = \#  (\mathrm{Int }\ \triangle \cap \bZ^2)$.  
\end{proposition}
\begin{proof}
    Again we can assume there are no switches of type B. 

    Let us fix some notation:
    \begin{itemize}
        \item $d(\rho)$: the number of departures. 
        \item $C$: the number of  same-codirection crossings. 
        \item $V(\rho)$: the number of vertical switches. 
        \item $H(\rho)$: the number of horizontal switches. 
        \item $E_d(\rho)$: the number of disk eyes. 
        \item $B = |\pi_0(\Lambda)| = | \partial \triangle \cap \bZ^2|$. 
    \end{itemize}
    
    First note that \(d(\rho) = r(\rho)\). This is because, by~\eqref{eq:End_F_epsilon}, the dimension of $\End(\cF^\epsilon_j)$ increases by $1$ when passing through a departure,  decreases by $1$ when passing through a return, and remains the same in other cases. 

    One also observes that $E_d(\rho) = H(\rho)$, because each disk eye is connected to two horizontal switches, and each horizontal switch is adjacent to two disk eyes. 

    We have 
    \[
     2 - 2 g(\rho) - B = \chi (\Sigma(\rho) ) = E_d(\rho) - (V(\rho) + H(\rho) ) = - V(\rho)
    \]
    and (by Remark~\ref{rmk:crossing_types})
    \[
    C = V(\rho) + d(\rho) + r(\rho) = V(\rho) + 2 r(\rho). 
    \]
    Hence
    \[ 
        r(\rho) + g(\rho) = \frac{C - B +2}{2}. 
    \]

    On the other hand, we claim that
    \[
     C = 2 \,\mathrm{Area} (\triangle) .
    \]
    Since
    $
    B = \# (\partial \triangle \cap \bZ^2)
    $,
    by Pick's theorem, the desired result follows from
    \[
    g_\triangle = \mathrm{Area} (\triangle) - \frac{B}{2} + 1 = \frac{C - B + 2}{2}. 
    \]

    Now we prove the claim. Let the edges of \(\triangle\) be $\{v_1, \dots, v_n\}$. By assumption, none of the \(v_i\) is vertical. We assume that \(v_1, \dots, v_k\) point to the left, and \(v_{k+1}, \dots, v_n\) point to the right. Then
    \begin{align*}
         C 
         &= \sum_{1\le i < j\le k} \det (v_i, v_j) + \sum_{ k+1\le s < t \le n} \det (v_s , v_t)\\
         &= \sum_{1\le i < j\le k} \det (v_i, v_j) + \sum_{ k+1\le s < t \le n} \det (v_s , v_t) + \det \left( \sum_{i\le k} v_i, \sum_{s\ge k+1} v_s\right)\\
         &= \sum_{1\le i <j \le n} \det(v_i, v_j) = 2 \mathrm{Area}(\triangle). 
    \end{align*}
    The last equality is the Gauss area formula.  
\end{proof}

The following proposition applies to an infinite class of examples including
linear systems with generic curves of
arbitrary genus.  See Figure \ref{fig:example of ruling} in Section
\ref{sec:rulings} for an illustration,
and Section \ref{sec:unitedgetriangles} for more discussion.

\begin{proposition}
\label{prop:triangulerulingcount}
Let \(\triangle\) be a triangle with primitive edges.  Assume that the corresponding Legendrian \(\Lambda\) is a union of geodesics coming from a dimer model.  Then $$N_\triangle^\mathrm{rul} = 2\cdot \mathrm{Area}(\triangle),$$
where area is taken in lattice units.
\end{proposition}

\begin{proof}
    In this case, \(\Lambda = \coprod_{i=1}^3 S^1 \). Hence, for a rational ruling \(\rho\), the surface \(\Sigma(\rho)\) is a \(3\)-punctured sphere, and its skeleton satisfies \(\chi (\Gamma(\rho)) = -1\).  If \(\Gamma(\rho)\) has \(k\) vertices, denoted by \(v_1, \dots, v_k\), then
\[
-1 \;=\; k - \frac{1}{2}\sum_i \deg v_i  
   \;\le\; k - \frac{3k}{2}  
   \;=\; -\frac{k}{2}, 
\]
which forces \(k = 2\), and each vertex must have degree \(3\).

Note that in this case, the boundary components of each eye are zig-zag paths of geodesics, and every eye must be adjacent to at least \(3\) switches. Therefore, if there exists an annular eye, the number of vertices of \(\Gamma\) would be at least \(3\), which leads to a contradiction.

Hence, the only possibility is that \(\rho\) has two contractible eyes, both triangular, and connected along three switches. By the balancing condition, the areas of these two triangles must be the same, and the lengths of their edges are half the lengths of the corresponding edges of \(\triangle\). Fixing a crossing, there is a unique pair of such triangles meeting along this crossing, and hence it determines a rational ruling. Conversely, every rational ruling consists of a pair of blue and red triangles meeting along \(3\) distinguished crossings. Therefore, the total number of rational rulings is the number of crossings divided by \(3\), which is equal to twice the area of the triangle.

\end{proof}

\subsection{The ruling polynomial and refined invariants.} We say a ruling $\rho$ is of genus $g$ if $\Sigma(\rho)$ has genus $g$.  
Define 
\[
N_{\triangle, g} \coloneqq \# \{ \rho \in \mathfrak{R}(\Lambda) \mid g(\rho) = g \}
\]
and the ruling polynomial
\[
R_\triangle(z) \coloneqq \sum_g N_{\triangle, g} z^{2g}
\]

Following \cite{BlockGottsche:2016:RefinedCurveCounting}, we define refined tropical counting:
\[
N_\triangle(q) = \sum_{T \in \cT_\triangle} \prod_{V\in V(T)} [m_V]_q. 
\]
Here $\cT_\triangle$ is the set of tropical trees with fixed boundary conditions as in Section~\ref{sec:tropical}, $V(T)$ is the set of vertices of a tropical tree $T$, $m_V$ is the multiplicity at the vertex $V$, and $ [n]_q = \frac{q^{n/2} - q^{-n/2}}{q^{1/2} - q^{-1/2}}$ is the quantum integer.  

\begin{conjecture}\label{conj:Ruling_refined}
    Setting $z = q^{1/2} - q^{-1/2}$, we have
    \[
    R_\triangle(z) = \N_\triangle(q). 
    \]
\end{conjecture}

    Setting $q=1$ in this conjecture recovers Theorem~\ref{thm:rulings}. 

    With programming help from ChatGPT, we were able to count rulings and tropical curves using Python. Conjecture~\ref{conj:Ruling_refined} is confirmed in many cases, including all the Newton polygons in Table~\ref{table:data}, $\cO_{\bP^1\times \bP^1}(2,4)$, $\cO_{\bP^2}(4)$, and $\cO_{\bP^2}(5)$. In particular, for $\cO_{\bP^2}(5)$, we have\footnote{In this particular case, the formula for $N_\triangle(q)$ stated in \cite[Section 5]{Blomme} appears to be incorrect. Applying the recursion relation in \cite[Theorem 3.4]{Blomme} instead yields the formula given here.}
    \begin{align*}
    R_\triangle(z) & = 25875+18165z^2+7050z^4+1750z^6 + 275z^8+25z^{10}+z^{12} \\
    & = q^6 + 13q^5 + 91 q^4 + 455 q^3 + 1745 q^2 + 5273 q + 10719\\
    & + 5273q^{-1}+1745q^{-2}+455q^{-3}
    +91q^{-4}+13q^{-5}+q^{-6} 
    = N_\triangle(q). 
    \end{align*}
These calculations provide good evidence for the conjecture.

    The refined curve counting has a motivic interpretation \cite{GottscheShende:2014:RefinedCurveCounting, Nicaise-Payne-Schroeter}.
    Following \cite{Nicaise-Payne-Schroeter}, we consider the Hirzebruch genus, denoted by \(\chi_{-y}\). For the definition we refer to \cite{Nicaise-Payne-Schroeter}. We use the fact that \(\chi_{-y}\) is additive, and 
    \[
        \chi_{-y} \left(\bC^r \times (\bC^*)^{2g}\right) = y^r(y-1)^{2g}. 
    \]

    \begin{proposition}\label{prop:chi=R}
    Setting $  z = y^{1/2} - y^{-1/2} $, we have
        \[
        \chi_{-y} ( |\calM_{\triangle, \eta}| ) = y^{g_\triangle} R_\triangle (z). 
        \]
    \end{proposition}
    
    \begin{proof}
    This follows from the ruling decomposition and Proposition~\ref{prop:g_plus_r}:
    \begin{align*}
        \chi_{-y}\left( |\calM_{\triangle, \eta}| \right) &= \sum_\rho y^{r(\rho)} (y-1)^{2 g(\rho)}\\
        &=  \sum_\rho y^{g_\triangle}(y^{1/2} - y^{-1/2})^{2 g(\rho)} \\
        &= y^{g_\triangle} R_\triangle(z).
    \end{align*}
    \end{proof}

    As we have discussed before, $|\calM_{\triangle, \eta}|$ (denoted as \(\cM^{\mathrm{Sh}}\) in Section~\ref{sec:defs}) is isomorphic to the relative Jacobian $|\mathcal J_{\triangle, \xi}|$ (denoted as $\cM^\mathrm{Coh}$ in Section~\ref{sec:defs}).  
    Therefore, Conjecture~\ref{conj:Ruling_refined} is equivalent to
    \begin{equation}
    \label{eq:N=chi}
    N_{\triangle} (q) = q^{-g_\triangle} \chi_{-q} ( |\mathcal{J}_{\triangle, \xi}| ) . 
    \end{equation}
    This can be viewed as a variant of~\cite[Conjecture 1.1]{Nicaise-Payne-Schroeter}.

\begin{remark}
    In \cite{Bousseau:2019:TropicalRefined}, it is shown that the refined tropical counting corresponds to higher genus log GW invariants. More precisely, applying \cite[Theorem 6]{Bousseau:2019:TropicalRefined}, we get
    \[
    N_{\triangle}(q) =  \left(2 \sin \frac{\hbar}{2} \right )^{2 - b} \sum_{g\ge 0 } N_{\triangle, g}^{\log, \lambda}   \hbar^{2 g + b-2},
    \]
    where
    $b = | \partial \triangle \cap \bZ^2|$, $q = e^{i\hbar}$, and $N_{\triangle, g}^{\log, \lambda}$ is the genus $g$ log GW invariants with $\lambda$-class insersion. See, for example, ~\cite{Bousseau:2019:TropicalRefined} for the definitions. 
    Hence, Conjecture~\ref{conj:Ruling_refined} implies that the ruling decomposition would recover higher genus curve counting information. 
\end{remark}

\subsection{The gonality gap condition}
Take \(x = z^{-2}\) and $A_\triangle(x) = x^{g_\triangle} R_\triangle(z)$, i.e.
\[
A_\triangle(x) = \sum_{\delta \ge 0} N_{\triangle, g_\triangle - \delta} x^\delta.
\]
Realizing $\Lambda$ as the zig-zag paths of a bipartite graph, one finds that 
$N_{\triangle, g_{\triangle} - 1}$ is the number of ways to fill in a ``hole''. Hence $N_{\triangle, g_{\triangle} - 1} = \#\{\text{alternating regions}\} = 2 \operatorname{Area}(\triangle)$. In other words, we have
$A_\triangle(x) = 1 + (2 \operatorname{Area}(\triangle)) x + \cdots$.

More generally, put 
\[
    B(x)=\frac{1+\sqrt{1+4x}}2
    =1+x-x^2+2x^3-5x^4+\cdots
    =\frac{q}{q-1}
\]
and let $\gamma(\triangle)$ be the gonality of  a general irreducible curve
with Newton polygon \(\triangle\), namely the least degree of a nonconstant map to
\(\bP^1\). For example, if $\cL_\triangle = \cO_{\bP^2}(d)$, then $\gamma(\triangle) = d-1$, and if $\cL_\triangle = \cO_{\bP^1\times \bP^1}(a, b)$ then $\gamma(\triangle) = \operatorname{min}\{a, b\}$. The gonality can also be read off from the Newton polygon~\cite{CastryckCools2017}.
Computational evidence suggests the following, which has a similar form as the conifold gap condition --- see e.g. \cite{GhoshalVafa} and \cite[Section 2.2 and Equation 79]{HuangKlemmQuackenbush}. 
\begin{conjecture}
\label{conj:gap}
\[
A_\triangle(x) \equiv B(x)^{2 \operatorname{Area}(\triangle)}  \pmod{x^{\gamma(\triangle)}}. 
\]

\end{conjecture}

\begin{remark}
    For each
    edge \(E\subset \triangle\), choose a partition
    \[
    \lambda^E
    = \bigl(\lambda^E_1,\ldots,\lambda^E_{\ell(\lambda^E)}\bigr)
    \vdash \ell_{\bZ}(E),
\]
where \(\ell_{\bZ}(E)\) is the lattice length of \(E\) and
\(\ell(\lambda^E)\) is the length of the partition.  Denote
$
    \boldsymbol\lambda=(\lambda^E)_E
$.
Recall the counting defined in~\cite{Bousseau:2019:TropicalRefined}: 
\[
    N_{\triangle,\boldsymbol\lambda}(q)
    =
    \sum_{\Gamma\in\mathcal T_{\triangle,\boldsymbol\lambda}}
    \prod_{v\in V(\Gamma)}[m_v]_q,
\]
where $\cT_{\triangle,\boldsymbol\lambda}$ is the set of tropical trees with fixed weighted ends.

Set
$
    m=\sum_{E\subset \triangle}\ell(\lambda^E)
$,
$
    H=2 \operatorname{Area}(\triangle)-m+2
$,
$
    \epsilon\equiv H\pmod 2    
$ where
$\epsilon\in\{0,1\}
$,
and define \(A^{\mathrm{trop}}_{\triangle,\boldsymbol\lambda}(x)\) by
\[
    N_{\triangle,\boldsymbol\lambda}(q)
    =
    (q^{1/2}+q^{-1/2})^\epsilon z^{H-\epsilon}
    A^{\mathrm{trop}}_{\triangle,\boldsymbol\lambda}(z^{-2}).
\]
Then $A^{\mathrm{trop}}_{\triangle,\boldsymbol\lambda}$ is a polynomial, and the computational evidence also supports the following:
\begin{equation}
\label{eq:gap_weighted}
A^\mathrm{trop}_{\triangle,\boldsymbol\lambda}(x)
    \equiv
    {B(x)^{2\operatorname{Area}(\triangle)}}{(1+4x)^{-\epsilon/2}}  
    \pmod{x^{\gamma(\triangle)}}.
\end{equation} 
Assuming Conjecture~\ref{conj:Ruling_refined}, this would recover Conjecture~\ref{conj:gap}
by taking all parts of \(\lambda^E\) to be \(1\).

\end{remark}

\section{Examples}
\label{sec:examples}

We label convex lattice polygons $\Delta$ according to the Figures of \cite{DW}.

Analyzing our invariant $N_\Delta$ from localization or degeneration requires an analysis of tropical curves with fixed number of external rays.  The following combinatorial lemma will be useful.

\begin{lemma}
\label{lem:comblemma}
    Let $T$ be a tree with $b$ leaves and internal vertex set $I$.  For $i\in I$, $d_i$ for the degree of vertex $i.$  Then
    $$b = 2 + \sum_{i\in I} (d_i - 2).$$

    In particular, if all $v := |I|$ internal vertices are trivalent, then $$v = b-2.$$
    
    \begin{proof} This
    follows from the Euler characteristic $|V|-|E|=1$ as well as $2|E| = b + \sum_{i\in I} d_i.$
    \end{proof}
\end{lemma}

\subsection{Polygon \texorpdfstring{$2b$}{2b}}
\label{sec:poly2b}

We begin with an example that we will work in detail.
Subsequent examples will be treated more briefly.

Let $\Delta = \conv(\{(0,0),(2,1),(3,-1)\}$.
The associated toric surface is defined by the complete
normal fan $\Sigma$, with ray vectors $v_1 = (1,-2), v+2 = (-2,-1), v_3 = (1,3).$


The polygon and toric fan are shown in
Figure \ref{fig:2b} below.
\begin{figure}[ht]
\begin{tikzpicture}
    \foreach\i in {-1,0,...,4}{
    \foreach\j in {-2,-1,...,2}{
    \fill (\i,\j) circle (.1cm);   }}
    \draw[very thick] (0,0)--(2,1)--(3,-1)--(0,0);
        \fill (0,-3) circle (0cm);
    \node[left] at (0,0) {$P_{31}$};
    \node[above] at (2,1) {$P_{12}$};
    \node[right] at (3,-1) {$P_{23}$};
    \node at (1-.23,1/2+.15) {$C_1$};
        \node at (5/2+.2,0+.2) {$C_2$};
        \node at (3/2-.05,-1/2-.2) {$C_3$};
        
\end{tikzpicture}\qquad\qquad\qquad
\begin{tikzpicture}
\pgfmathsetmacro{\shift}{2}
    \fill (0,0) circle (.1cm);
    \draw[very thick,->] (0,0)--(1,-2);
    \node[right] at (1,-2) {$1$};
    \node[above right] at (1,3) {$3$};
    \draw[very thick,->] (0,0)--(1,3);
        \node[left] at (-2,-1) {$2$};
    \draw[very thick,->] (0,0)--(-2,-1);
        \node[right] at (0.05,0) {$5$};
        \node[above left] at (0,0) {$5$};
        \node[below] at (-0.05,-0.05) {$5$};
    \fill (0,3) circle (0cm);
\end{tikzpicture}
\caption{Polgon and fan for Figure 2b of \cite{DW}.}
\label{fig:2b}
\end{figure}

The ray vectors $v_i$ are in bijection with torus-fixed curves $C_i$. 
We write $P_{ij}$ for the fixed points $C_i\cap C_j$.
The intersection pairing of distinct curves are obtained from
the index of the lattice of adjacent ray vectors $v_i$ and $v_{i+1}$,
which is displayed in the figure above.
The intersection pairing matrix is
\[
\begin{pmatrix}
 \tfrac{1}{5} & \tfrac{1}{5} & \tfrac{1}{5} \\
 \tfrac{1}{5} & \tfrac{1}{5} & \tfrac{1}{5}  \\
 \tfrac{1}{5} & \tfrac{1}{5} & \tfrac{1}{5} \\
\end{pmatrix}
\]

The off-diagonal elements are clear from the local geometry and the index indicated in the figure.  As for the self-intersections, these are determined by
equations setting principal divisors to zero.  Each lattice vector $m\in M$ determines a character $\chi^m$ on $T$ and thus a rational function on $\bP$ and therefore principal divisor $\sum_i \langle m,v_i\rangle D_i$--- hence an equation
by setting intersections with it to zero.
Taking $m = (1,0)$ gives
$D_1 - 2D_2 + D_3= 0$.
Taking $m = (0,1)$ gives
the equation $-2D_1 - D_2 + 3D_3 = 0.$
Together, we derive $D_1 = D_2 = D_3$,
when $D_i\cdot D_j = \frac{1}{5}$ for all $i,j.$

Then our linear system has five effective representatives ---
equivalently, five equivariant global sections --- in terms of the \(C_i\). They are in correspondence
with lattice points of $\Delta$:
\[
\begin{aligned}
&(0,0):\phantom{-1} 5C_2 \\
&(1,0):\phantom{-1} C_1 + 3C_2 + C_3 \\
&(2,0):\phantom{-1} 2C_1 + C_2 + 2C_3\\
&(2,1):\phantom{-1} 5C_3 \\
&(3,-1):\phantom{1} 5C_1.
\end{aligned}
\]
The expressions are found as follows:  given a lattice point, the coefficient of \(C_i\)
in the corresponding divisor
is the order of vanishing of the corresponding sections,
given by the value of the primitive linear form defining the \(i\)-th edge at that lattice point.
(Think about the order of vanishing of $y^2$ along
the $x$ axis: it's $2$.)  In this case,
all of these curve classes are equivalent to 
$5C_1.$

To compute the torus weights of the sections,
given weights $\lambda_1,\lambda_2$
parametrizing a one-parameter (or even two-parameter!) subgroup of $T = 
N\otimes_\bZ \bC^*,$ then the section
defined by the monomial $x^ay^b$ has weight $a\lambda_1 + b\lambda_2$.
This will be useful for the localization calculation.

\vskip0.1in
\noindent{\bf Localization Calculation.}

We now describe the different types of tropical
curves in the fixed locus.  
Recall that we pull back two point classes from two
chosen toric divisors (the last point condition
is never imposed, as its image is determined from the others by the condition that the class of the
fixed points along the toric divisor $\bP^\infty$ represent
a class in the local system pulled back to $\bP^\infty$).
Let us then leave point $1$ unspecified
and impose the incidence conditions
on the marked points $2$ and $3$ by pulling back
point classes from $C_2$ and $C_3,$ respectively.
(This labeling choice makes the pictures prettier.)
We choose equivariant lifts of the these point classes
by requiring the corresponding points on $C_2$ and $C_3$ to be torus
fixed points, and we choose them both to be $P_{23},$ the point $C_2\cap C_3.$
Write $P_{ij}^\vee \in H_T^2(C_i)$ for the equivariant lift
of the Poincar\'e dual of the fixed point $P_{ij}\in C_i$, and note this is \emph{not} $i\leftrightarrow j$ symmetric.
Thus we are computing the log GW invariant
$$\int_{\cM_{0,3}^\mathrm{log}(\bP_\triangle,\beta)}\mathrm{ev_2}^*(P_{23}^\vee)\mathrm{ev_3}^*(P_{32}^\vee).$$
The toric fixed points correspond to maximal
cones of the fan, $\Sigma,$ so here the relevant cone is $\sigma_{23}.$

We are finally ready to describe the fixed locus,
using the results of Section \ref{sec:fixedlocus}.
A component of a fixed locus is labeled by
a decorated graph (as in the usual localization
calculation of Gromov-Witten theory) as well as the
``type'' of a tropical curve, in the language of
\cite{GS}.
Let's call the tropical curve $\cT.$
Note by Lemma \ref{lem:comblemma} we know
that $\cT$ must have a single trivalent
vertex and any number of bivalent vertices.

The vertices of $\cT$ are in correspondence with generic
points (components)
of the domain curve.  These can be mapped to
maximal cones,
if the component is collapsed to a fixed point (whose
cone defines the toric neighborhood), or to rays if the
component
is mapped nontrivially to a toric divisor.
Because in our example
there are only three marked points,
a fixed log stable map with image homology class $\beta$
(for us $\beta$ is the Poincar\'e dual of $c_1(L_\triangle)$)
must have exactly
one bivalent vertex, and by
stability exactly one contracted component.
Let us call $C$ the nontrivially mapped component
and $C'$ the contracted component.
If $C$ contains the marked point $2$ (which gets mapped
to $C_2$) then the marked point $3$ is on $C'$, so
the fixed map contracts $C'$ to $P_{23}$.
Therefore,
the trivalent vertex of $\cT$ carries
the cone label $\sigma_{23}$, while the bivalent
vertex (from $C$) carries the label of the ray $\bR v_2$.
If on the other hand, $C$ contains the marked
point $3$, then the bivalent vertex has
label of the ray $\bR v_3.$
We color $\cT$ blue for the first case and red for
the second and plot them in Figure \ref{fig:tropicalfixed}.


\begin{figure}[ht]
\begin{tikzpicture}
\pgfmathsetmacro{\shift}{2}
    \fill (0,0) circle (.1cm);
    \draw[very thick,->] (0,0)--(1,-2);
    \node[right] at (1,-2) {$1$};
    \node[above right] at (1,3) {$3$};
    \draw[very thick,->] (0,0)--(1,3);
        \node[left] at (-2,-1) {$2$};
    \draw[very thick,->] (0,0)--(-2,-1);
        \node[right] at (0.05,0) {$5$};
        \node[left] at (-0.05,0.1) {$5$};
        \node[below] at (-0.05,-0.05) {$5$};
    \fill (0,3) circle (0cm);
    \draw[very thick,blue,->] (-1,0)--(-3,-1);
        \draw[very thick,blue,->] (-1,0)--(0,-2);
            \draw[very thick,blue,->] (-1,0)--(0,3);
    \draw[very thick,red,->] (0,1)--(-2,0);
        \draw[very thick,red,->] (0,1)--(1,-1);
            \draw[very thick,red,->] (0,1)--(1,4);
    \draw[dashed,thick](0,0)--(-1.5,3);
    \fill[blue] (-4/5,-2/5) circle (.1cm);
    \fill[red] (1/5,3/5) circle (.1cm);
\end{tikzpicture}
\qquad
\begin{tikzpicture}
    \draw[very thick,blue] (0,0)--(2,0);
    \draw[very thick,blue] (1.5,-.5)--(1.5,2);
    \node[blue,left] at (0,0) {$C_2$};
    \fill[blue] (0.5,0) circle (.1cm);
    \fill[blue] (1.5,0.7) circle (.1cm);
    \fill[blue] (1.5,1.4) circle (.1cm);
    \node[blue,below] at (0.5,0) {$1$};
    \node[blue,left] at (1.5,0.7) {$3$};
    \node[blue,left] at (1.5,1.4) {$2$};
        \node[blue,above] at (1.5,2) {$P_{23}$};
    
\end{tikzpicture}
\qquad
\begin{tikzpicture}
    \draw[very thick,red] (0,0)--(2,0);
    \draw[very thick,red] (.5,-.5)--(.5,2);
    \node[red,right] at (2,0) {$C_3$};
    \fill[red] (1.5,0) circle (.1cm);
    \fill[red] (0.5,0.7) circle (.1cm);
    \fill[red] (0.5,1.4) circle (.1cm);
    \node[red,below] at (1.5,0) {$1$};
    \node[red,left] at (0.5,0.7) {$3$};
    \node[red,left] at (0.5,1.4) {$2$};
    \node[red,above] at (.5,2) {$P_{23}$};
    
\end{tikzpicture}
\caption{The tropical curves corresponding to the two fixed
maps, and their corresponding intersection graphs.  The
vertical components $C'$ are contracted to a fixed point and the horizontal components $C$ are mapped to the invariant curve as
indicated.}
\label{fig:tropicalfixed}
\end{figure}
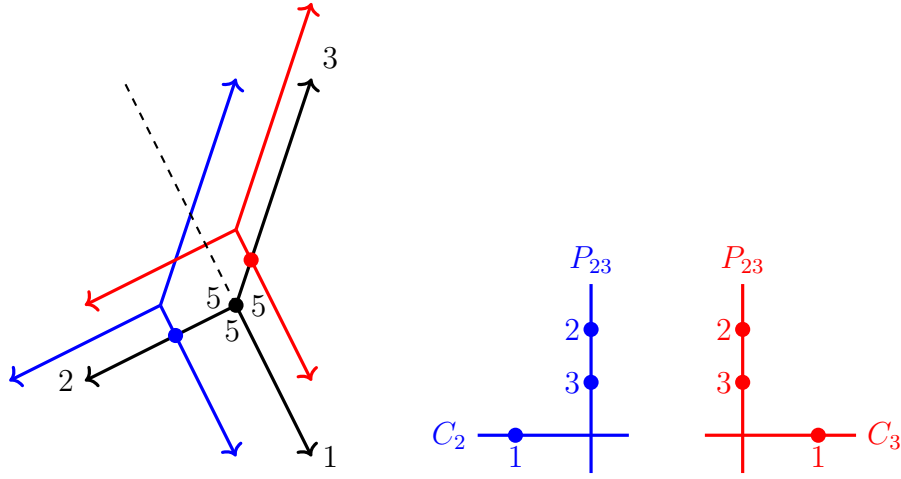

The localization calculation consists of two parts:  the numerator and denominator.
\begin{itemize}
   \item The contribution of the numerator.
   The numerator is the restriction to the fixed
   locus of the point class along the divisor.
   The weight of a point class pulled back from a point $P$ equals the character of the action of the torus
   on the tangent space at $P$.  For example, a point $P_{23}$ at the intersection of $C_2$ and $C_3$ lies on $C_2$ and
   the character for the action of the torus on $T_{P_{23}}C_2$ is given by the minimal integral linear form on the fan which vanishes along
   the ray $\bR v_2$ and is non-negative on the cone $\sigma_{23}$, here $2b-a.$
   For the blue graph, the numerator is the product of the two point classes, the one just considered
   and the one pulled back from $C_3,$ which gives $b-3a$. In total we have $(2b-a)(b-3a)$ for the numerator.
   The red graph has the same numerator.
   \item The contribution of the denominator.
   The local toric geometry around the fixed point $f$
   in the moduli space of log stable maps is determined by the monoid at $f$.  This monoid is the cone described by the possible locations for the trivalent vertex of $\cT.$   For the blue graph, the bounding rays are in the directions $(-2,-1)$ and $(-1,2)$  However, the integral structure of the monoid is different, because the intersection of the blue line with ray 2 must be an integral point.  Thus we see that the allowed positions for the vertex of the tripod consist of those lattice points in the interior of the cone which can be realized as a positive integer combination of the minimal vectors along the rays.  This corresponds to a smooth toric surface with an action of a torus $\tilde T$ which is a 5 fold cover of $T$.  This torus is exactly the torus for which the action on the original surface lifts to an action on the stable map.  For this torus, there will be extra characters which give a dual basis to the basis of the lattice spanned by our primitive vectors along the rays. The primitive characters defining the two edges of our cones will be $\red{-}\frac{a-2b}{5}$ and $\red{-}\frac{2a+b}{5}$ and the denominator term is just the product of the two corresponding weights.  The total contribution is then the reciprocal $\frac{5^2}{(a-2b)(2a+b)}.$
   
   For the red graph, the denominator calculation gives
   $\frac{5^2}{(3a-b)(-2a-b).}$
   
   \item Automorphisms:
   For each graph, the corresponding map has a group of automorphisms of order 5.
   This is easy to see for the underlying stable map --- the automorphism acts trivially on the contracted component of $C$ but on the noncontracted component acts by multiplication by a 5th root of unity in coordinates where the node and the marked point are 0 and 
$\infty$ ---  but it is more delicate to compute the automorphisms of the log stable map (which in general do not agree with its ordinary stable map counterpart).  We will give some idea of what is involved in lifting the automorphism to the log setting.  We will do this for the point corresponding to the blue tropical curve in Figure \ref{fig:tropicalfixed}.  The point of the moduli space associated to this point corresponds to a morphism from a log point whose monoid $Q$ is the dual to the monoid of tropical curves of that blue type.  Such a tropical curve is determined by the location of its trivalent vertex.  This is constrained to lie between the dotted line and ray 2, it must be integral, and the intersection of the tropical curve with ray 2 must be integral.  This gives an index 5 submonoid in the set of integral points between $(-2,-1)$ and $(-1,2)$, in fact the free monoid spanned by those two vectors.  An action of $\bZ/5$ on $Q\times \bC^*$ which induces the trivial action on $Q$ amounts to a moprhism $Q \to \mu_5$.  In order for this action to be compatible with the log map, it is necessary that the kernel of this morphism contain the image of the map from $\sigma_{23}^\vee$ to $Q$ induced by the map from the sheaf of monoids on the toric surface to the sheaf of monoids at the generic point of the contracted component of $C$.  This is an index 5 submonoid, so there is essentially a unique action of $\bZ/5$ on $Q \times \bC^*$ which could be compatible with the given action of $\bZ/5$ on $C$.  In this case, that action on $Q$ does lead to a unique lift of the action to action on the log curve compatible with morphism $f$. 
   (Due to the length of the automorphism calculation --- even with this abbreviated explanation --- we simply state the result in subsequent examples.)
   
   The upshot of this calculation is that we divide each term by 5.
\end{itemize}
Adding up the contributions from both graphs , we
get:
\begin{align*}
N_\Delta &=(2b-a)(b-3a)\left[\frac{5^2}{(a-2b)(2a+b)} + \frac{5^2}{(3a-b)(-2a-b)}\right]\cdot\frac{1}{5} \\
&= \frac{5(2b-a)(b-3a)}{(a-2b)(2a+b)(3a-b)}
     \left(3a-b - (a-2b)\right) \\
&= 5
\end{align*}

\vskip0.1in
\noindent{\bf Tropical Curve Counting.}  Fixing two rays of a tropical curve at infinity requires, by Lemma \ref{lem:comblemma}, that the tropical
curve be a translate of the one in Figure \ref{fig:2b}.
The multiplicity is $5.$

\vskip0.1in
\noindent{\bf Constructible Sheaf Counting.}
The count (5) of rational rulings in this example follows
from Proposition \ref{prop:triangulerulingcount}.
We illustrate one in Figure \ref{fig:nonproper ruling} below.  The others
are all translates by the indicated
vectors.  

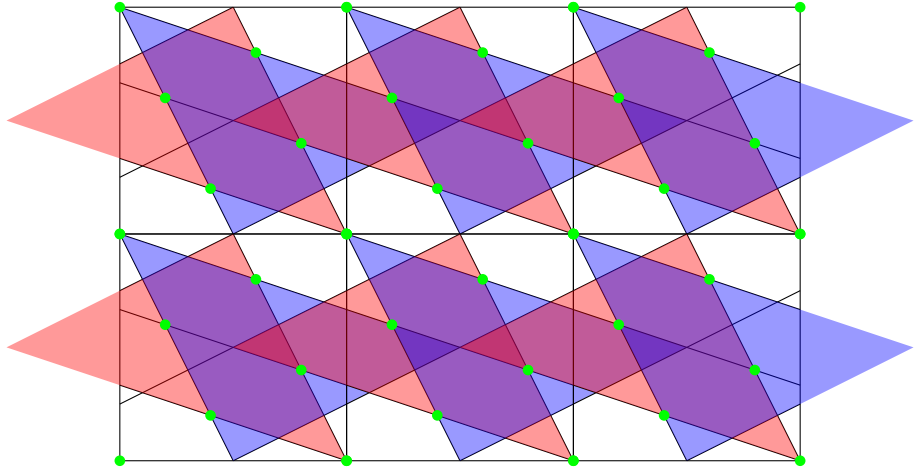
\begin{figure}[h]
         \begin{tikzpicture}[scale = .5]
         \foreach \i in {0,6,12}{
         \foreach \j in {0,-6}{

        \begin{scope}[shift = {(\i, \j)}]

    \draw[black] (0, 0) rectangle (6, 6); 
    \draw (0, 6) -- (6, 4);
    \draw (0, 4) -- (6, 2); 
    \draw (0, 2) -- (6, 0); 

    \draw (0, 6) -- (3, 0);
    \draw (3, 6) -- (6, 0); 
    \draw (0, 4.5) -- (3, 6);
    \draw (3, 0) -- (6, 1.5);
    \draw (0, 1.5) -- (6, 4.5);

    \fill[red, opacity = .4] (6, 0) -- (-3, 3) -- (3, 6);
    \fill[blue, opacity= .4] (0, 6) -- (3, 0) -- (9, 3);
    \end{scope}
    }}
         \foreach \i in {0,6,12}{
         \foreach \j in {0,-6}{

        \begin{scope}[shift = {(\i, \j)}]

    \fill[green,fill opacity=1] (0,0) circle (4pt);
    \fill[green,fill opacity=1] (6,0) circle (4pt);
    \fill[green,fill opacity=1] (0,6) circle (4pt);
    \fill[green,fill opacity=1] (6,6) circle (4pt);
    \fill[green,fill opacity=1] (6/5,18/5) circle (4pt);
    \fill[green,fill opacity=1] (12/5,6/5) circle (4pt);
    \fill[green,fill opacity=1] (18/5,24/5) circle (4pt);
    \fill[green,fill opacity=1] (24/5,12/5) circle (4pt);

    \end{scope}
    }}

        \end{tikzpicture}
        
        \caption{One of five rational rulings for Polygon 2b.  The
        others are translates by vectors indicated with green dots. The picture is drawn on the universal cover of $\T$, and each square is a fundamental domain.}
        \label{fig:nonproper ruling}
    \end{figure}


\subsection{Unit-edge triangles}
\label{sec:unitedgetriangles}

We have the following generalizations.
Suppose given a triangle $\Delta$ with unit-length edges.  Up to translation and $SL_2(\bZ)$, we can assume it is the convex hull of $(0,0), (1,0)$ and $(k,n)$, where $n>0.$  Then still applying
$\left(\begin{smallmatrix} 1&1\\0&1
\end{smallmatrix}\right)$ repeatedly, we can assume $0<k<n$, and the unit edge condition
gives $\gcd(k,n) = \gcd(k-1,n) = 1.$
For instance, the example of Section \ref{sec:poly2b} corresponds to $k = 4, n = 5.$  

We note the unit-edge
condition together with the $\gcd$ condition implies that $n$ must be odd.  Then Pick's theorem gives the number of interior points, which is the genus of the generic curve of the linear system, to be $\frac{1}{2}(n-1).$
This family of examples thus contains generic curves of every genus.

These examples all have dimer models that can be constructed
explicitly follows.  After isotoping the zigzags to straight-line geodesics, we can specify the models by describing the triangles surrounding black and white vertices.  For the choice above, these are formed by the projection from $\bR^2$ to $\bR^2/\bZ^2$ of the lines
$y = \frac{1}{2},$ $y = \frac{n}{k}x,$ $y = \frac{n}{k-1}x$.
Figure \ref{fig:primitiveedgetriangle}
gives a picture when $k=2,n=3$.

\begin{figure}[ht]
    \centering
    \begin{tikzpicture}

    \begin{scope}[shift = {(7, .3)}, scale = .7]
        \foreach\i in {-2, ...,2}{
    \foreach\j in {0,...,5}{
    \fill (\i,\j) circle (2pt);   }}

    \draw[very thick] (-1, 1) -- (0, 1) -- (1, 4) -- cycle; 

    \node[above] at (1, 4) {\small \( (k=2, n=3)\)}; 
    \end{scope}

\pgfmathsetmacro{\u}{4}
    \draw[thick,dotted] (0,0)--(\u,0)--(\u,\u)--(0,\u)--(0,0);
    \draw[thick] (0,\u/2)--(\u,\u/2);
    \draw[thick] (0,0)--(\u/3,\u);
    \draw[thick] (\u/3,0)--(2*\u/3,\u);
    \draw[thick] (2*\u/3,0)--(\u,\u);
    \draw[thick] (0,0)--(2*\u/3,\u);
    \draw[thick] (2*\u/3,0)--(\u,\u/2);
    \draw[thick] (0,\u/2)--(\u/3,\u);
    \draw[thick] (\u/3,0)--(\u,\u);
    \foreach\i in {0,\u/3,2*\u/3}{
        \draw[blue] (\u/6+ \i,\u/2+\u/6)--(\u/6+\i,\u/2-\u/6);
        \draw[blue] (\u/6+ \i,\u/2+\u/6)--
        (\u/3+\i,\u);
        \draw[blue] (\u/6+ \i,\u/2-\u/6)--
        (\i,0);
        \draw[blue] (\u/6+ \i,\u/2-\u/6)--
        (\u/3+\i,\u/2);
        \draw[blue] (\u/6+ \i,\u/2+\u/6)--
        (\i,\u/2);
        \fill[blue] (\u/6+ \i,\u/2+\u/6) circle (2pt);
        \filldraw[blue,fill=white] (\u/6+\i,\u/2-\u/6) circle (2pt);};
        \end{tikzpicture}
    \caption{Left:  dimer model (blue) and zigzag curves (black) for primitive edge triangle (right), when $k=2$ and $n=3$.}
    \label{fig:primitiveedgetriangle}
\end{figure}
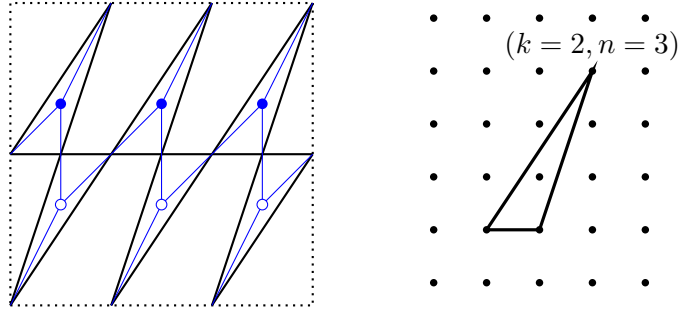

As for tropical curves, recall that the point
conditions on the boundary determine the location of the rays of the tropical curves.  Specifically, for each marked point $p_i$ in the domain curve $C$, the incidence condition $f(p_i) \in D_\rho$ determines via the $\mathrm{Log}$ map an equivalence class in $N_\bR/(\bR\rho)$
labeling the $i$th external ray of the tropical curve, with direction $\rho.$

\vskip0.1in
\noindent{\bf Tropical Curve Counting.}
Unit-edges triangles have $b=3$, so Lemma \ref{lem:comblemma} requires that all tropical curves have just one vertex.  Further, the two point conditions determine two of the external rays of the curve, and therefore the location of the vertex and the
tropical curve itself.  We conclude there is a unique
tropical curve contributing to the Mikhalkin point
count.  The multiplicity is $\left|\det(v_1,v_2)\right|,$
where $v_1,v_2,v_3$ are the fan vectors.
In the standard presentation, this quantity
is $\bigl|\det\binom{1\,k}{0\,n}\bigr| = n.$

\vskip0.1in
\noindent{\bf Constructible Sheaf Counting.}

The count of rational rulings
for these examples is given
by Proposition \ref{prop:triangulerulingcount}
to be twice the lattice area of the triangle.
So the sheaf count for the $(k,n)$ triangle
$\Delta$, with its unit base and height $n$, also gives $n$.

\vskip0.1in
\noindent{\bf Localization Calculation.}
The calculation proceeds along the lines of the example of Section \ref{sec:poly2b}
except with three ray vectors $(n,-k),$ $(-n,k-1),$ and $(0,1).$
The numerator contribution for both graphs is $a(-(k-1)a-nb).$
Here $|\mathrm{Aut}|=n.$
In total, we get
\begin{align*}
N_\Delta &= a(-(k-1)a-nb)\left[\frac{n^2}{(-(k-1)a-nb)(ka+nb)} + \frac{n^2}{(-ka-nb)a}\right]\cdot \frac{1}{n} \\
&= \frac{na(-(k-1)a-nb)}{(-(k-1)a-nb)(ka+nb)a}\left(a+(k-1)a+nb\right) \\
&= n
\end{align*}

All methods give $n$.

\subsection{Polygon \texorpdfstring{$3b$}{3b}}
\label{sec:poly3b}

Let $\Delta = \conv(\{(0,0),(0,1),(3,0),(2,-1)\}$.
The associated toric surface is defined by the complete
normal fan $\Sigma$, with ray vectors 
$v_1 = (1,0), v_2 = (-1,-3), v_3 = (-1,1), v_4 = (1,2).$


The polygon and toric fan are shown below.
$$
\begin{tikzpicture}
    \foreach\i in {-1,0,...,4}{
    \foreach\j in {-2,-1,...,2}{
    \fill (\i,\j) circle (.1cm);   }}
    \draw[very thick] (0,0)--(2,-1)--(3,0)--(0,1)--(0,0);
        \fill (0,-3) circle (0cm);
\end{tikzpicture}\qquad\qquad\qquad
\begin{tikzpicture}
\pgfmathsetmacro{\shift}{2}
    \fill (0,0) circle (.1cm);
    \draw[very thick,->] (0,0)--(1,0);
    \node[right] at (1,0) {$1$};
    \draw[very thick,->] (0,0)--(1,2);
    \node[above right] at (1,2) {$4$};
    \draw[very thick,->] (0,0)--(-1,1);
        \node[above left] at (-1,1) {$3$};
    \draw[very thick,->] (0,0)--(-1,-3);
    \node[below left] at (-1,-3) {$2$};
        \node[above right] at (0.1,0) {$2$};
        \node[above] at (-0.05,0.1) {$3$};
        \node[left] at (0,-.1) {$4$};
        \node[below right] at (0,0) {$3$};
    \fill (0,3) circle (0cm);
\end{tikzpicture}
$$

The ray vectors $v_i$ are in correspondence with the torus-fixed curves, $C_i$.
The torus fixed points are $P_{ij} = C_i\cap C_j.$
The intersection pairings of distinct curves are obtained from
the index of the lattice of adjacent ray vectors $v_i$ and $v_{i+1}$,
which is displayed in the figure above.
The intersection pairing matrix is
\[
\begin{pmatrix}
 -\tfrac{1}{6} & \tfrac{1}{3} & 0 & \tfrac{1}{2} \\
 \tfrac{1}{3} & \tfrac{1}{12} & \tfrac{1}{4} & 0 \\
 0 & \tfrac{1}{4} & \tfrac{1}{12} & \tfrac{1}{3} \\
 \tfrac{1}{2} & 0 & \tfrac{1}{3} & -\tfrac{1}{6}
\end{pmatrix}
\]
The off-diagonal elements are clear from the local geometry and the index indicated in the figure.  As for the self-intersections, these are determined by
equations setting principal divisors to zero.  Each lattice vector $m\in M$ determines a character $\chi^m$ on $T$ and thus a rational function on $\bP$ and therefore principal divisor $\sum_i \langle m,v_i\rangle D_i$--- hence an equation
by setting intersections with it to zero.
Taking $m = (1,0)$ gives
$D_1 - D_2 - D_3 + D_4 = 0$.
Intersecting this with $D_1$, for example,
gives $D_1\cdot D_1 = \frac{1}{3}-\frac{1}{3} = -\frac{1}{6}.$
With $D_2$ gives
$D_2\cdot D_2 = \frac{1}{3}-\frac{1}{4} = \frac{1}{12}.$  With $D_3$ gives $D_3\cdot D_3 = \frac{1}{3}-\frac{1}{4} = \frac{1}{12}$ 
and with $D_4$ gives $D_4\cdot D_4 = \frac{1}{3}-\frac{1}{2}=-\frac{1}{6}.$
The element $m = (0,1)$ gives
$-3D_2 + D_3 + 2D_4 = 0.$
The above answers
are consistent with dotting this
second equation with the $D_i$.

Then our linear system has six effective representatives --- or six equivariant global sections --- in terms of the \(C_i\). They are in correspondence
with lattice points of $\Delta$:
\[
\begin{aligned}
&(0,0):\phantom{-1}3C_2 + 3C_3, \\
&(1,0):\phantom{-1} C_1 + 2C_2 + 2C_3 + C_4, \\
&(2,0):\phantom{-1} 2C_1 + C_2 + C_3 + 2C_4, \\
&(3,0):\phantom{-1} 3C_1 + 3C_4, \\
&(0,1):\phantom{-1} 4C_3 + 2C_4, \\
&(2,-1):\phantom{1} 2C_1 + 4C_2.
\end{aligned}
\]
The expressions are found as follows:  given a lattice point, the coefficient of \(C_i\) in the corresponding divisor is the value of the primitive linear form defining the \(i\)-th edge at that lattice point.

To compute the torus weights of the sections,
given weights $\lambda_1,\lambda_2$
parametrizing a one-parameter subgroup of $T = 
N\otimes_\bZ \bC^*,$ then the section
defined by the monomial $x^ay^b$ has weight $a\lambda_1 + b\lambda_2$.
This will be useful for the localization calculation.

\vskip0.1in
\noindent{\bf Localization Calculation.}

 We again write $P_{ij}^\vee \in H_T^2(C_i)$ for the equivariant lift
of the Poincar\'e dual of the fixed point $P_{ij}\in C_i$, and
consider $\int_{\cM^\mathrm{log}_{0,4}(\bP_,\beta)}\mathrm{ev^*}_2(P^\vee_{23})\cdot \mathrm{ev^*_3}(P_{34}^\vee)\cdot \mathrm{ev^*_4}(P_{41}^\vee)$.
To prep the calculation, we should compute for the numerator (no matter the graph) the torus weights at
$T_{P_{23}}C_2,$ $T_{P_{23}}C_3$, and $T_{P_{41}}C_4$.
These are computed as in previous calculations to be
$$T_{P_{23}}C_2: -3a+b\qquad  T_{P_{34}}C_3: a+b \qquad T_{P_{41}}C_4:  2a-b$$
So the numerator is the product of these
three terms.

Note that the tropical curve cannot have a single four-valent vertex, due to the point conditions on $2$ and $4$ appearing at opposite corners.  (This is not the case
for other equivariant lifts.)
In addition, since the curve class requires that there are exactly two non-contracted components, there are not enough marked points for a stable curve with three or more contracted components.  The tropical curve
therefore has two vertices.
With two domain points mapping to $P_{23},$ we know there is one contracted component mapping to $P_{23}.$  
Given the other point conditions, 
only the following combinatorial
types of maps remain.

$$
\begin{tikzpicture}
    \draw[thick] (-1.2,0)--(.2,0);
    \draw[thick] (0,-.2)--(0,1.2);
    \draw[thick] (-.2,1)--(1.2,1);
    \draw[thick] (1,-.2)--(1,1.2);
    \node[below] at (-.5,0) {$3$};
        \node[above] at (.5,1) {$4$};
        \fill[black] (-1,0) circle (.1cm);
        \fill[black] (0,.5) circle (.1cm);
        \fill[black] (1,0) circle (.1cm);
        \fill[black] (1,.5) circle (.1cm);
\end{tikzpicture}
\hskip1in
\begin{tikzpicture}
    \draw[thick] (-1.3,-.2)--(.2,.03);
    \draw[thick] (-2.2,.03)--(-.7,-.2);
    \draw[thick] (0,-.2)--(0,1.2);
    \draw[thick] (-.2,1)--(1.2,1);
    \draw[thick] (1,-.2)--(1,1.2);
    \draw[thick] (.8,0)--(2.2,0);
    \node[below] at (-1.5,0) {$1$};
    \node[below] at (-.5,0) {$2$};
        \node[above] at (.5,1) {$3$};
    \node[below] at (1.5,0) {$4$};
        \fill[black] (-2,0) circle (.1cm);
        \fill[black] (0,.5) circle (.1cm);
        \fill[black] (2,0) circle (.1cm);
        \fill[black] (1,.5) circle (.1cm);
\end{tikzpicture}
$$

Here are the corresponding combinatorial types of tropical curves.

\[
\begin{tikzpicture}[scale=.6] 
\pgfmathsetmacro{\x}{0}
\pgfmathsetmacro{\y}{.7}
\pgfmathsetmacro{\l}{1.2}
\pgfmathsetmacro{\xx}{\x + \l}
\pgfmathsetmacro{\yy}{\y + \l}
\pgfmathsetmacro{\tval}{\x/4+\y/4}
\pgfmathsetmacro{\sval}{\y-2*\x}

\pgfmathsetmacro{\px}{\x + \tval*(-1)}
\pgfmathsetmacro{\py}{\y + \tval*(-3)}
\pgfmathsetmacro{\qx}{\x + \sval}
\pgfmathsetmacro{\qy}{\y + \sval}

    \draw[thick, dotted] (0,0)--(1,3);
    \draw[very thick,->] (0,0)--(1,0);
    \node[right] at (1,0) {$1$};
    \draw[very thick,->] (0,0)--(-1,-3);
    \node[below] at (-1,-3) {$2$};
    \draw[very thick,->] (0,0)--(-1,1);
        \node[above left] at (-1,1) {$3$};
        \node[above] at (1,2) {$4$};
    \draw[very thick,->] (0,0)--(1,2);
    \fill[red] (0,3) circle (0cm);
    \fill[red] (\x-\tval,\y-3*\tval) circle (.1cm);
    \fill[red] (\x+\sval,\y+\sval) circle (.1cm);
    \draw[very thick,->,red] (\x,\y)--(\x-1,\y-3);
    \draw[very thick,->,red] (\x,\y)--(\x-1,\y+1);
    \draw[very thick,red] (\x,\y)--(\xx,\yy);
    \draw[very thick,->,red] (\xx,\yy)--(\xx+1,\yy);
    \draw[very thick,->,red] (\xx,\yy)--(\xx+1,\yy+2);
    \end{tikzpicture}
\hskip1in
\begin{tikzpicture}[scale=.6]  
\pgfmathsetmacro{\x}{-1}
\pgfmathsetmacro{\y}{-2.5}
\pgfmathsetmacro{\l}{4}
\pgfmathsetmacro{\xx}{\x}
\pgfmathsetmacro{\yy}{\y + \l}
\pgfmathsetmacro{\tval}{-\x-\y}
\pgfmathsetmacro{\sval}{-\xx + \yy/2}
\pgfmathsetmacro{\uval}{\y-3*\x} 
\pgfmathsetmacro{\vval}{-\y/2} 

\pgfmathsetmacro{\px}{\x + \tval*(-1)}
\pgfmathsetmacro{\py}{\y + \tval*(-3)}
\pgfmathsetmacro{\qx}{\x + \sval}
\pgfmathsetmacro{\qy}{\y + \sval}

    \draw[thick, dotted] (0,0)--(-1,-2);
    \draw[very thick,->] (0,0)--(1,0);
    \node[right] at (1,0) {$1$};
    \draw[very thick,->] (0,0)--(-1,-3);
    \node[below] at (-1,-3) {$2$};
    \draw[very thick,->] (0,0)--(-1.5,1.5);
        \node[above left] at (-1.5,1.5) {$3$};
        \node[above] at (1,2) {$4$};
    \draw[very thick,->] (0,0)--(1,2);
    \fill[red] (0,3) circle (0cm);
    \fill[red] (\x,\y+\tval) circle (.1cm);
    \fill[red] (\x+\uval,\y+2*\uval) circle (.1cm);
    \fill[red] (\x+\vval,\y+2*\vval) circle (.1cm);

    \fill[red] (\xx+\sval,\yy) circle (.1cm);
    \draw[very thick,->,red] (\x,\y)--(\x-1/2,\y-3/2);
    \draw[very thick,->,red] (\xx,\yy)--(\xx-1,\yy+1);
        \draw[very thick,->,red] (\xx,\yy)--(\xx+3,\yy);

    \draw[very thick,red] (\x,\y)--(\xx,\yy);
    \draw[very thick,->,red] (\x,\y)--(\x+2.5,\y+5);
    \end{tikzpicture}
\]

The denominator is determined by the linear
forms constraining the vertices.  Consider the tropical curve on the left.  The left vertex $(a_1,b_1)$ lies in the cone defined by the inequalities $a_1+b_1>0$ and $b_1-3a_1>0$ (note the dashed line).  The right vertex $(a_2,b_2)$
must satisfy $2a_2-b_2>0.$  Note the balancing
condition says $a_2-a_1=b_2-b_1$ so the cone
is three real dimensional.  The intersections
with the rays (marked by red dots) must
be lattice points, and this constrains
$b_1-3a_1\in 4\bZ$, thus also
$a_1+b_1\in 4\bZ$.
Note the internal edge has multiplicity $2$.
Requiring the segment from the right vertex to the
intersection with Ray 4 to be an integer multiple
of $(2,2)$ constrains $2a_2-b_2\in 2\bZ.$
Thus the minimal integer-valued forms defining
the monoid at this fixed point are
$$\frac{a_1+b_1}{4},\frac{b_1-3a_1}{4},
\frac{2a_2-b_2}{2},$$
and the denominator contribution is the product of these factors. 
We take the torus acting on this cone
to be lifted from the fan, so set $a = a_1 = a_2, b = b_1 = b_2.$
Then note that the linear factors all cancel with the numerator, leaving a numerical factor of 32.  Now
Dividing by the order of the automorphism group here {(4)} gives a total contribution
from this graph of 8.

Now consider the curve on the right.
The lower-left vertex $(a_1,b_1)$ must lie in the sliver
between the dotted line and Ray 2, so the defining linear forms are $b_1-3a_1$ and $2a_1-b_1$.
The upper-right vertex $(a_2,b_2)$ must satisfy $a_2+b_2>0.$  Integrality of intersections with
Rays $2$ and $4$ give $b_1\in 2\bZ$ and $b_2\in 2\bZ,$ so the integer-values linear forms determining the monoid are $$b_1-3a_1, \frac{2a_1-b_1}{2},a_2+b_2.$$
After setting $a_1=a_2=a,$ etc., we again
get cancellation with the numerator,
yielding a combinatorial factor of $2$.  {This
curve has automorphism group of order $2$,}
giving a contribution of $1$ from this graph.
Total contribution:  $9.$

\vskip0.1in
\noindent{\bf Tropical Curve Counting.}

After choosing a set of boundary conditions for the rays at infinity, we have the following
tropical curves:

\[
\begin{tikzpicture}
\pgfmathsetmacro{\t}{.5};
    \coordinate (v1) at (0,0);
    \coordinate (v2) at (\t,\t);
    \coordinate (A) at (\t+1,\t+0);
    \coordinate (B) at (-1,-3);
    \coordinate (C) at (-1.5,1.5);
    \coordinate (D) at (\t+1,\t+2);
    \draw[red,very thick] (C)--(v1)--(v2)--(A);
    \draw[red,very thick] (B)--(v1)--(v2)--(D);
    \node[above right,xshift=2mm,yshift=0mm] at (v2) {$2$};
    \node[left,xshift=-1mm] at (v1) {$4$};
        \fill[red] (v1) circle (.1cm);
        \fill[red] (v2) circle (.1cm);

\begin{scope}[shift = {(6,0)}]
    \coordinate (v1) at (0,0);
    \coordinate (v2) at (\t,\t);
    \coordinate (A) at (\t+1,\t+0);
    \coordinate (B) at (-1,-3);
    \coordinate (C) at (-1.5,1.5);
    \coordinate (D) at (\t+1,\t+2);
    \coordinate (F) at (-\t,\t);
    \coordinate (E) at (-\t,-3*\t);
    \node[left,xshift=-1mm] at (F) {$1$};
    \node[left,xshift=-1mm] at (E) {$1$};

    \draw[red,very thick] (B)--(E)--(F)--(C);
    \draw[red,very thick] (F)--(A);
    \draw[red,very thick] (E)--(D);
        \fill[red] (E) circle (.1cm);
        \fill[red] (F) circle (.1cm);

\end{scope}    
\end{tikzpicture}
\]

\noindent The contributions are the products of the vertex multiplicities shown in the figures, i.e.~$4\cdot 2$ and $1\cdot 1$, respectively, for a total of $9$.

\begin{remark}
Note the contribution of each graph matches
that from the localization calculation,
i.e. the tropical multiplicity.  We observe
this in all examples but have no general
proof.  Further,
the localization calculation depends on the
choice of linearization (which fixed points to map marked points to) while the tropical calculation
depends on boundary conditions.  These dependencies may
conspire to all line up, but we do not yet
know precisely how.
\end{remark}

\vskip0.1in
\noindent{\bf Constructible Sheaf Counting.}

The standard ruling form a bipartite graph is shown in Figure~\ref{subfig:3b_standard}. 
The rational rulings are shown in Figure~\ref{subfig:3b_B} -- \ref{subfig:3b_D}. Each picture represents a symmetric class. There are 9 in total. 

\begin{figure}[ht]
    \centering

    \begin{subfigure}[t]{.24\linewidth}
        \centering
        \includegraphics[width = \linewidth]{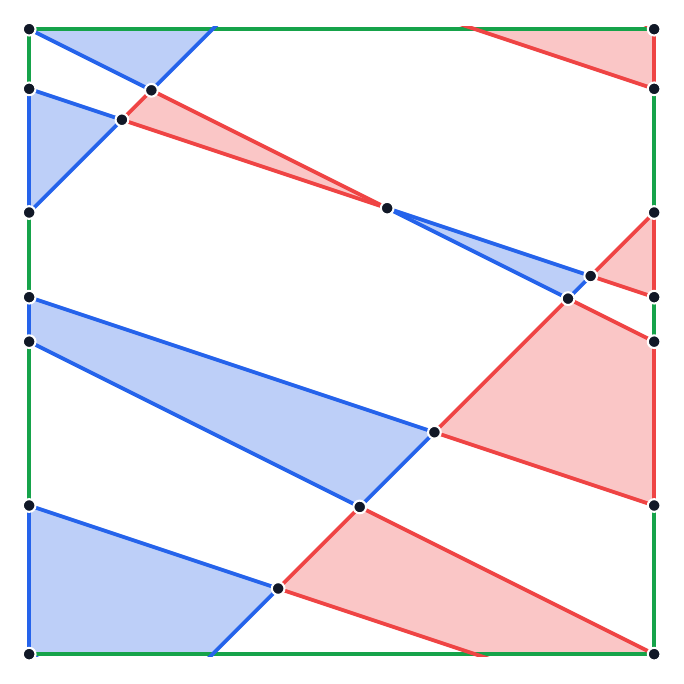}
        \caption{The standard ruling is genus two (not rational).}
        \label{subfig:3b_standard}
    \end{subfigure}
    \begin{subfigure}[t]{.24\linewidth}
        \centering
        \includegraphics[width = \linewidth]{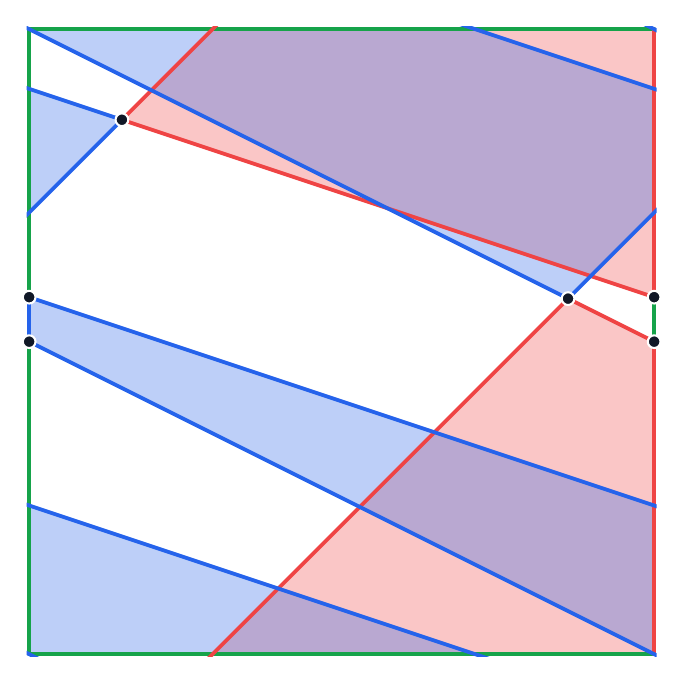}
        \caption{This rational ruling and others related by symmetry produce 6.}
        \label{subfig:3b_B}
    \end{subfigure}
    \begin{subfigure}[t]{.24\linewidth}
        \centering
        \includegraphics[width = \linewidth]{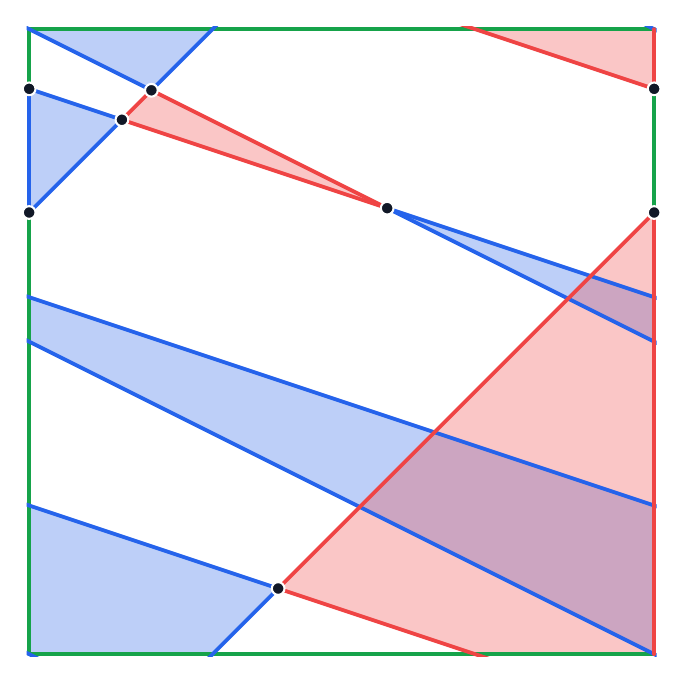}
        \caption{2 of this kind.}
        \label{subfig:3b_C}
    \end{subfigure}
    \begin{subfigure}[t]{.24\linewidth}
        \centering
        \includegraphics[width = \linewidth]{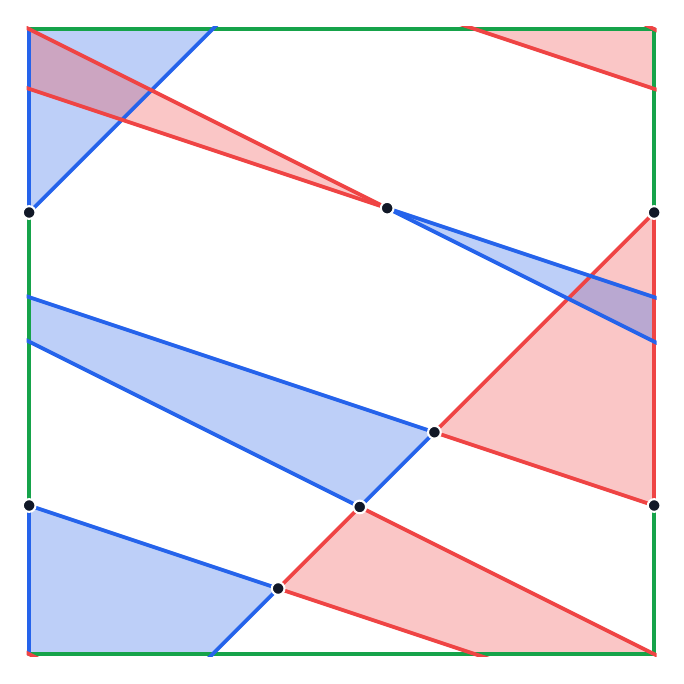}
        \caption{1 of this kind.}
        \label{subfig:3b_D}
    \end{subfigure}
    \caption{6+2+1=9 rational rulings for (3b).  Black dots are switches.}
    \label{fig:3b-rulings}
\end{figure}

\subsection{Unit-edge parallelograms}
\label{sec:unitedgeparallelograms}

Similar to Section \ref{sec:unitedgetriangles}, we can
construct unit-edge parallelograms.
Given a pair of relatively prime
numbers $k$ and $n$ with $0<k<n,$
the parallelogram with vertices $(0,0),
(1,0),(k,n),(k+1,n)$ has unit-length
edges.  Note here there we do \emph{not}
have to impose that $k-1$ and $n$ are
relatively prime.

$$
\begin{tikzpicture}
    \draw[very thick] (0,0)--(1,0)--(3,3)--(2,3)--(0,0);
    \node[above left] at (1,1.5) {$1$};
    \node[below right] at (2,1.5) {$3$};
    \node[above] at (2.5,3) {$2$};
    \node[below] at (.5,0) {$4$};
\end{tikzpicture}\qquad\qquad
\begin{tikzpicture}
\pgfmathsetmacro{\shift}{2}
    \fill (0,0) circle (.1cm);
    \draw[very thick,->] (0,0)--(0,1);
    \node[above] at (0,1) {$4$};
    \draw[very thick,->] (0,0)--(3,-2);
    \node[below right] at (3,-2) {$1$};
    \draw[very thick,->] (0,0)--(0,-1);
        \node[below] at (0,-1) {$2$};
    \draw[very thick,->] (0,0)--(-3,2);
    \node[above left] at (-3,2) {$3$};
        \node at (0.2,.2) {$n$};
        \node at (-.2,-0.2) {$n$};
        \node at (-.2,.3) {$n$};
        \node at (.2,-.3) {$n$};
    \fill (0,3) circle (0cm);
\end{tikzpicture}
$$

\vskip0.1in
\noindent{\bf Tropical Curve Counting.}
Lemma \ref{lem:comblemma} says that for the
Mikhalkin curve counting, $v = 2.$  Given
incidence conditions in $N_\bR/(\bR\rho)$ for each marked
point mapping to a curve defined by a ray $\rho$, there are three possible abstract tropical graphs,
depending on which pairs of rays meet at a vertex:
$(12)(34), (13)(24),$ and $(14)(23).$  

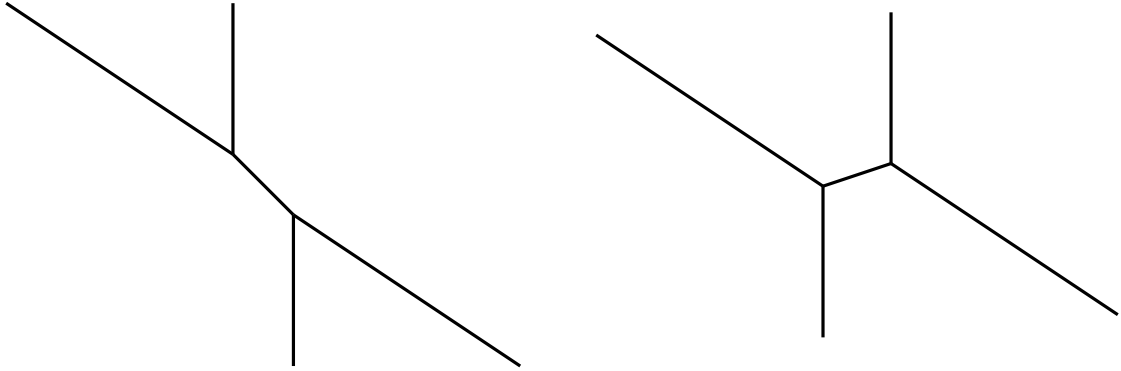
\begin{figure}[ht]
\begin{tikzpicture}
    
\pgfmathsetmacro{\xshift}{.8}
\pgfmathsetmacro{\yshift}{-\xshift}

\draw[very thick](0,0)--(\xshift,\yshift);
    \draw[very thick] (0,0)--(0,2);

    \draw[very thick] (0,0)--(-3,2);

    \draw[very thick] (0+\xshift,0+\yshift)--(+\xshift0,-2\yshift);
    \draw[very thick] (0+\xshift,0\yshift)--(\xshift+3,-2+\yshift);
    \fill (0,3) circle (0cm);
\end{tikzpicture}
\qquad
\begin{tikzpicture}
\pgfmathsetmacro{\shift}{2}
\pgfmathsetmacro{\yshift}{-.3}
\pgfmathsetmacro{\xshift}{3*\yshift}
\draw[very thick](0,0)--(\xshift,\yshift);
    \draw[very thick] (0,0)--(0,2);

    \draw[very thick] (0,0)--(3,-2);

    \draw[very thick] (0+\xshift,0+\yshift)--(+\xshift0,-2+\yshift);
    \draw[very thick] (0+\xshift,0\yshift)--(\xshift-3,2+\yshift);
    \fill (0,3) circle (0cm);

    \fill (0,2.7) circle (0cm);
    \fill (0,-2.7) circle (0cm);
    
\end{tikzpicture}
\caption{Tropical curves for different boundary conditions.}
\label{fig:parallelogram-bcs}
\end{figure}

However, $(13)(24)$ is impossible for a parallelogram ---
see Figure \ref{fig:parallelogram-bcs} ---
and both other configurations give the answer $n^2$.

\vskip0.1in
\noindent{\bf Constructible Sheaf Counting.}
A dimer model can be constructed as in Figure \ref{fig:parallelogram} below,
where the case $(k,n) = (3,5)$ is shown.  For the general case,
put white vertices at $\bZ^2 + (\frac{1}{4},\frac{1}{4})$, black vertices at $\bZ^2 + (\frac{3}{4},\frac{3}{4})$,
and edges along $\bZ^2$ translates of the lines $y=x$ and $y = \frac{1}{2} - x$,
then quotient by the lattice of translations generated by $(n,0)$ and $(-k,1).$

 \begin{figure}[ht]
  \centering
  \begin{tikzpicture}[scale=0.7]
  \pgfmathsetmacro{\A}{8}
  \pgfmathsetmacro{\B}{1}
  \pgfmathsetmacro{\a}{2*\A-1}
  \pgfmathsetmacro{\b}{2*\B-1}
\foreach \x in {-1,0,...,\a}{  
\draw (\x+1/2,-1/2) -- (\x + 1/2,\b + 1/2);}
\foreach \y in {-1,0,...,\b}{  
\draw (-1/2,\y+1/2) -- (\a + 1/2,\y + 1/2);}
\foreach \x in {0,1,...,\a}{
\foreach \y in {0,1,...,\b}{
\pgfmathparse{mod(\x,2)} \let\xm\pgfmathresult 
\pgfmathparse{mod(\y,2)} \let\ym\pgfmathresult 
\ifthenelse{\equal{\xm}{1.0} \AND \equal{\ym}{1.0}}
{
\draw (\x-1/2,\y-1/2) -- (\x+1/2,\y+1/2);\draw (\x-1/2,\y+1/2) -- (\x+1/2,\y-1/2);
\filldraw (\x,\y) circle (3pt);  
\draw[->,blue] (\x-1/2,\y-1/2) -- (\x,\y-1/2);
\draw[->,blue] (\x+1/2,\y-1/2) -- (\x+1/2,\y);
\draw[->,blue] (\x+1/2,\y+1/2) -- (\x,\y+1/2);
\draw[->,blue] (\x-1/2,\y+1/2) -- (\x-1/2,\y);
}{};
\ifthenelse{\equal{\xm}{0.0} \AND \equal{\ym}{0.0}}
{
\draw (\x-1/2,\y-1/2) -- (\x+1/2,\y+1/2);\draw (\x-1/2,\y+1/2) -- (\x+1/2,\y-1/2);
\draw[draw=black,fill=white] (\x,\y) circle (3pt); 
\draw[->,blue] (\x-1/2,\y-1/2) -- (\x-1/2,\y);
\draw[->,blue] (\x-1/2,\y+1/2) -- (\x,\y+1/2);
\draw[->,blue] (\x+1/2,\y+1/2) -- (\x+1/2,\y);
\draw[->,blue] (\x+1/2,\y-1/2) -- (\x,\y-1/2);
}{};

}}
\draw[ultra thick,dashed,red](6+-1/2,-1/2)--(6+9.5-.06,-1/2)--(-6+15.5,1.5)--(-6+5.5+.02,1.5)--(6+-1/2,-1/2);
   \end{tikzpicture}
  \caption{Dimer model for the unit-edge parallelogram defined by $(k,n) = (3,5).$
  Fundamental domain is dotted red parallelogram.}
  \label{fig:parallelogram}
\end{figure}
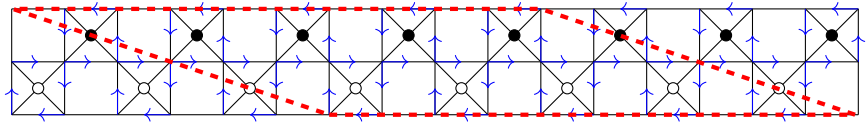 

Now we count rational rulings. Take the ``vertical'' direction to be $v = (1, 1)$.  Assume \(\rho\) is a rational ruling, and \(\Gamma(\rho)\) is its skeleton. In this case, since \(\Lambda \simeq (S^1)^{\coprod 4} \), we know
\begin{align*}
    \chi (\Gamma(\rho) ) = \chi (\Sigma(\rho)) = \chi (\Sigma_{0, 4}) = -2.
\end{align*}
On the other hand, assuming that \(\Gamma(\rho)\) has vertices \(v_1, v_2, \dots, v_m\),   we have
\begin{equation}\label{eq:ruling_eq_for_parallelogram}
    - 2 = m - \frac{1}{2} \sum_i \deg v_i = - \frac{1}{2}\sum_i (\deg v_i - 2). 
\end{equation}

First, there cannot be annular eyes. Otherwise, each of the two paths bounding such a annular eye must pass through two switches. This creates four vertices of degree \(3\) on the skeleton of this annular eye.  By~\eqref{eq:ruling_eq_for_parallelogram}, we know that those have to be all the vertices of \(\Gamma(\rho)\). But this leads to a contradiction since there has to be more than one eye.

Thus, all eyes are contractible. So we must have \( \deg v_i \ge 4\) for all \(i\). By~\eqref{eq:ruling_eq_for_parallelogram}, there can only be two rectangular eyes, one blue and one red, of the same area. Assume they meet at a switch \(p\), as drawn in Figure~\ref{fig:ruling_of_parallelogram}. Assume that the blue eye is a rectangle of size \( (a + \frac{1}{2}) \times (b + \frac{1}{2})\), where \( 0 \le a , b  \le n-1\). Since the boundaries of the two eyes have to cover \(\pi(\Lambda)\), the red eye has to be a rectangle of size \( (n - a - \frac{1}{2}, n -b - \frac{1}{2} ) \). On the other hand, the balancing condition tells us that the area of these two rectangles should be the same, which forces \( b = n - 1 -a \). So we have \(n\) choices of rational rulings fixing the crossing \(p\). Conversely, every rational ruling consists of a pair of rectangular eyes meeting along \( 4 \) points. Therefore, the total number of rational rulings should be the number of crossings times \(\frac{n}{4}\), which is \(n^2\). 

We can summarize succinctly.  Pick one black and one white node from the fundamental domain --- so, $n^2$ choices.  We will make the black node the lower-right node contained in a blue lattice rectangle and the white node the upper-right node contained in a red lattice rectangle.  Do this such that the height of the blue equals the width of the red, and \emph{vice versa}.

 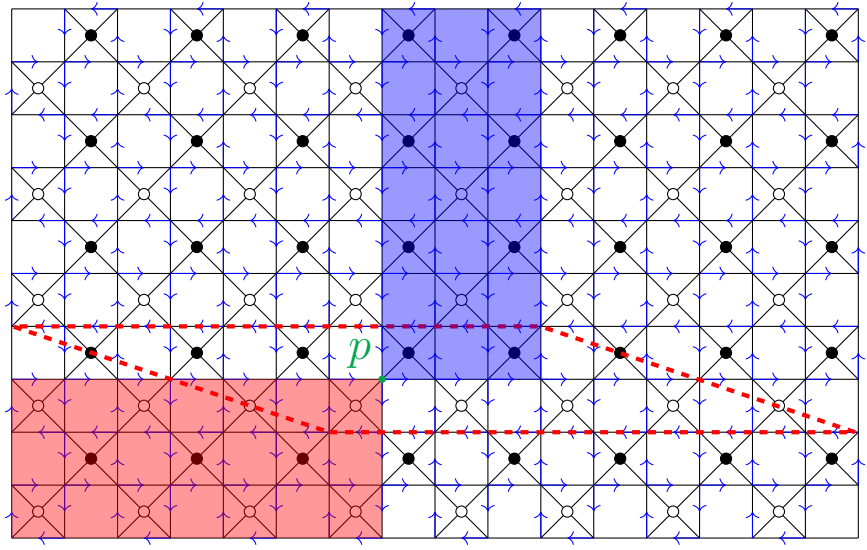
\begin{figure}[htbp]
  \centering
  \begin{tikzpicture}[scale=0.7]

\begin{scope}[shift = {(.5, .5)}]
  \pgfmathsetmacro{\A}{8}
  \pgfmathsetmacro{\B}{4}
  \pgfmathsetmacro{\a}{2*\A-1}
  \pgfmathsetmacro{\b}{2*\B-1}
\foreach \x in {-1,0,...,\a}{  
\draw (\x+1/2,-5/2) -- (\x + 1/2,\b + 1/2);}
\foreach \y in {-3,-2,...,\b}{  
\draw (-1/2,\y+1/2) -- (\a + 1/2,\y + 1/2);}
\foreach \x in {0,1,...,\a}{
\foreach \y in {-2,-1,...,\b}{
\pgfmathparse{mod(\x,2)} \let\xm\pgfmathresult 
\pgfmathparse{mod(\y,2)} \let\ym\pgfmathresult 
\ifthenelse{\isodd{\x} \AND \isodd{\y}}
{
\draw (\x-1/2,\y-1/2) -- (\x+1/2,\y+1/2);\draw (\x-1/2,\y+1/2) -- (\x+1/2,\y-1/2);
\filldraw (\x,\y) circle (3pt); 
\draw[->,blue] (\x-1/2,\y-1/2) -- (\x,\y-1/2);
\draw[->,blue] (\x+1/2,\y-1/2) -- (\x+1/2,\y);
\draw[->,blue] (\x+1/2,\y+1/2) -- (\x,\y+1/2);
\draw[->,blue] (\x-1/2,\y+1/2) -- (\x-1/2,\y);
}{};
\ifthenelse{\NOT\isodd{\x} \AND \NOT\isodd{\y}}
{
\draw (\x-1/2,\y-1/2) -- (\x+1/2,\y+1/2);\draw (\x-1/2,\y+1/2) -- (\x+1/2,\y-1/2);
\draw[draw=black,fill=white] (\x,\y) circle (3pt); 
\draw[->,blue] (\x-1/2,\y-1/2) -- (\x-1/2,\y);
\draw[->,blue] (\x-1/2,\y+1/2) -- (\x,\y+1/2);
\draw[->,blue] (\x+1/2,\y+1/2) -- (\x+1/2,\y);
\draw[->,blue] (\x+1/2,\y-1/2) -- (\x,\y-1/2);
}{};

}}
\draw[ultra thick,dashed,red](6+-1/2,-1/2)--(6+9.5-.06,-1/2)--(-6+15.5,1.5)--(-6+5.5+.02,1.5)--(6+-1/2,-1/2);
\end{scope}


\fill[blue, opacity = .4] (7, 1) -- (10, 1) -- (10, 8) -- (7, 8);
\fill[red, opacity=. 4] (7, 1) -- (0, 1) -- (0, -2) -- (7, -2) -- cycle;

\fill[Green] (7, 1) circle(2pt); 
\node[above left] at (7, 1) {\Large \color{Green} \(p\) };

   \end{tikzpicture}
  \caption{A rational ruling of the unit-edge parallelogram with $(k,n) = (3,5)$.}
  \label{fig:ruling_of_parallelogram}
\end{figure}

\vskip0.1in
\noindent{\bf Localization Calculation.}

To perform a localization calculation, we choose
equivariant lifts of the point classes.
So again, we write $P_{ij}^\vee \in H_T^2(C_i)$ for the equivariant lift
of the Poincar\'e dual of the fixed point $P_{ij}$, and note this is \emph{not} $i\leftrightarrow j$ symmetric.
We consider $\int_{\overline{\cM}^\mathrm{log}_{0,4}(\bP_,\beta)}\mathrm{ev}_2(P^\vee_{23})\cdot \mathrm{ev_3}(P_{32}^\vee)\cdot \mathrm{ev_4}(P_{41}^\vee)$.
To prep the calculation, we should compute for the numerator (no matter the graph) the torus weights at
$T_{P_{23}}C_2,$ $T_{P_{32}}C_3$, and $T_{P_{41}}C_4$.
These are computed as in previous calculations to be
$$T_{P_{23}}C_2: -a\qquad  T_{P_{32}}C_3: -ka-nb \qquad T_{P_{41}}C_4:  a$$
so the product of these three terms is the numerator.
For convenience, for the localization
calculation we assume $\gcd(k-1,n)=1$, which renders
the internal edge of the tropical curves
multiplicity free and makes
some of weight and automorphism calculations simpler.


Note that the tropical curve cannot have a single four-valent vertex, due to the point conditions on $2$ and $4$ appearing at opposite corners.  (This is not the case
for other equivariant lifts.)
In addition, since the curve class requires that there are exactly two non-contracted components, there are not enough marked points for a stable curve with three or more contracted components.  The tropical curve
therefore has two vertices.
With two domain points mapping to $P_{23},$ we know there is one contracted component mapping to $P_{23}.$  
Given the other point conditions, 
only the following three combinatorial
types of maps remain.
$$
\begin{tikzpicture}
    \draw[thick] (-1.2,0)--(.2,0);
    \draw[thick] (0,-.2)--(0,1.2);
    \draw[thick] (-.2,1)--(1.2,1);
    \draw[thick] (1,-.2)--(1,1.2);
    \node[below] at (-.5,0) {$1$};
        \node[above] at (.5,1) {$2$};
        \fill[black] (-1,0) circle (.1cm);
        \fill[black] (0,.5) circle (.1cm);
        \fill[black] (1,0) circle (.1cm);
        \fill[black] (1,.5) circle (.1cm);
\end{tikzpicture}
\hskip1in
\begin{tikzpicture}
    \draw[thick] (-1.2,0.2)--(.2,-.2);
    \draw[thick] (-.2,-.2)--(1.2,0.2);
    \draw[thick] (-1,-.2)--(-1,1.2);
    \draw[thick] (1,-.2)--(1,1.2);
    \node[below] at (-.5,0) {$1$};
        \node[below] at (.5,0) {$2$};
        \fill[black] (-1,.5) circle (.1cm);
        \fill[black] (-1,1) circle (.1cm);
        \fill[black] (1,.5) circle (.1cm);
        \fill[black] (1,1) circle (.1cm);
\end{tikzpicture}
\hskip1in
\begin{tikzpicture}
    \draw[thick] (-1.2,0.2)--(.2,-.2);
    \draw[thick] (-.2,-.2)--(1.2,0.2);
    \draw[thick] (-1,-.2)--(-1,1.2);
    \draw[thick] (1,-.2)--(1,1.2);
    \node[below] at (-.5,0) {$3$};
        \node[below] at (.5,0) {$4$};
        \fill[black] (-1,.5) circle (.1cm);
        \fill[black] (-1,1) circle (.1cm);
        \fill[black] (1,.5) circle (.1cm);
        \fill[black] (1,1) circle (.1cm);
\end{tikzpicture}
$$
Here are the corresponding combinatorial types of tropical curves.
\[
\begin{tikzpicture}[scale=.6]
\pgfmathsetmacro{\xblueshift}{-.7}
\pgfmathsetmacro{\yblueshift}{1.2}
\pgfmathsetmacro{\xredshift}{-.3+3/5}
\pgfmathsetmacro{\yredshift}{-.3-1/5}
\pgfmathsetmacro{\xredshifttwo}{-.3}
\pgfmathsetmacro{\yredshifttwo}{-.3}

    \draw[thick, dotted] (-3,1)--(3,-1);
    \draw[very thick,->] (0,0)--(0,1);
    \node[above] at (0,1) {$4$};
    \draw[very thick,->] (0,0)--(3,-2);
    \node[below right] at (3,-2) {$1$};
    \draw[very thick,->] (0,0)--(0,-1);
        \node[below] at (0,-1) {$2$};
        \node[above left] at (-3,2) {$3$};
    \draw[very thick,->] (0,0)--(-3,2);
    \fill[red] (0+\xredshift,3+\yredshift) circle (0cm);
    \fill[red] (0,0+\yredshift+1*\xredshift/3) circle (.1cm);
    \fill[red] (\xredshift,-2*\xredshift/3) circle (.1cm);
    \draw[very thick,->,red] (0+\xredshift,0+\yredshift)--(0+\xredshift,1+\yredshift);
    \draw[very thick,->,red] (0+\xredshift,0+\yredshift)--(3+\xredshift,-2+\yredshift);
    \draw[red,very thick,->] (0+\xredshifttwo,0+\yredshifttwo)--(0+\xredshifttwo,-1+\yredshifttwo);
    \draw[red,very thick,->] (0+\xredshifttwo,0+\yredshifttwo)--(-3+\xredshifttwo,2+\yredshifttwo);
    \draw[red,very thick] (\xredshift,\yredshift)--(\xredshifttwo,\yredshifttwo);
\end{tikzpicture}
\qquad
\begin{tikzpicture}[scale=.6]
\pgfmathsetmacro{\xblueshift}{-.7}
\pgfmathsetmacro{\yblueshift}{1.2}
\pgfmathsetmacro{\xredshift}{-.3+15/5}
\pgfmathsetmacro{\yredshift}{-.3-5/5}
\pgfmathsetmacro{\xredshifttwo}{-.3}
\pgfmathsetmacro{\yredshifttwo}{-.3}
    \draw[thick, dotted] (-3,1)--(3,-1);

    \draw[very thick,->] (0,0)--(0,1);
    \node[above] at (0,1) {$4$};
    \draw[very thick,->] (0,0)--(3,-2);
    \node[below right] at (3,-2) {$1$};
    \draw[very thick,->] (0,0)--(0,-1);
        \node[below] at (0,-1) {$2$};
        \node[above left] at (-3,2) {$3$};
    \draw[very thick,->] (0,0)--(-3,2);
    \fill[red] (0+\xredshift,3+\yredshift) circle (0cm);
    \fill[red] (0-\xredshift-3*\yredshift,2*\xredshift/3+2*\yredshift) circle (.1cm);
    \fill[red] (0,\yredshift + \xredshift/3) circle (.1cm);
    \draw[very thick,->,red] (0+\xredshift,0+\yredshift)--(0+\xredshift,1+\yredshift);
    \draw[very thick,->,red] (0+\xredshift,0+\yredshift)--(3+\xredshift,-2+\yredshift);
    \draw[red,very thick,->] (0+\xredshifttwo,0+\yredshifttwo)--(0+\xredshifttwo,-1+\yredshifttwo);
    \draw[red,very thick,->] (0+\xredshifttwo,0+\yredshifttwo)--(-3+\xredshifttwo,2+\yredshifttwo);
    \draw[red,very thick] (\xredshift,\yredshift)--(\xredshifttwo,\yredshifttwo);
\end{tikzpicture}
\qquad
\begin{tikzpicture}[scale=.6]
\pgfmathsetmacro{\xblueshift}{-.7}
\pgfmathsetmacro{\yblueshift}{1.2}
\pgfmathsetmacro{\xredshift}{.5}
\pgfmathsetmacro{\yredshift}{.3}
\pgfmathsetmacro{\xredshifttwo}{\xredshift-3}
\pgfmathsetmacro{\yredshifttwo}{\yredshift+1}
    \draw[thick, dotted] (-3,1)--(3,-1);

    \draw[very thick,->] (0,0)--(0,1);
    \node[above] at (0,1) {$4$};
    \draw[very thick,->] (0,0)--(3,-2);
    \node[below right] at (3,-2) {$1$};
    \draw[very thick,->] (0,0)--(0,-1);
        \node[below] at (0,-1) {$2$};
        \node[above left] at (-3,2) {$3$};
    \draw[very thick,->] (0,0)--(-3,2);
    \fill[red] (0+\xredshift,3+\yredshift) circle (0cm);
    \fill[red] (0-\xredshift-3*\yredshift,2*\xredshift/3+2*\yredshift) circle (.1cm);
    \fill[red] (0,\yredshift + \xredshift/3) circle (.1cm);
    \draw[very thick,->,red] (0+\xredshift,0+\yredshift)--(0+\xredshift,1+\yredshift);
    \draw[very thick,->,red] (0+\xredshift,0+\yredshift)--(3+\xredshift,-2+\yredshift);
    \draw[red,very thick,->] (0+\xredshifttwo,0+\yredshifttwo)--(0+\xredshifttwo,-1+\yredshifttwo);
    \draw[red,very thick,->] (0+\xredshifttwo,0+\yredshifttwo)--(-3+\xredshifttwo,2+\yredshifttwo);
    \draw[red,very thick] (\xredshift,\yredshift)--(\xredshifttwo,\yredshifttwo);
\end{tikzpicture}
\]


For each of these cases, there is a 3-dimensional cone of possibilities determined by the locations of the two vertices inside their respective cones, which are subject to the constraint that the slope of the line segment between them is $\frac{1-k}{n}.$  Below are the linear functions determining these cones, where $(a_i,b_i)$ are coordinates for the vertices, $i=1,2,$ and $1$ labels the $(23)$ cone.
\[
\begin{array}{c}{-a_1}\\-(k-1)a_1-nb_1\\ a_2\\
-ka_2-nb_2\\
(1-k)(a_2-a_1)=n(b_2-b_1)
\end{array}
\hskip.2in
\begin{array}{c}{-a_1}\\-(k-1)a_1-nb_1\\ ka_2 + nb_2\\-(k-1)a_2-nb_2\\
(1-k)(a_2-a_1)=n(b_2-b_1)
\end{array}
\hskip.2in
\begin{array}{c}{ka_1-nb_1}\\(k-1)a_1+nb_1\\ a_2\\
(k-1)a_2+nb_2\\
(1-k)(a_2-a_1)=n(b_2-b_1)
\end{array}
\]

We compute the cones corresponding to the three types of tropical curves.  All three of these cones naturally live in a three dimensional subspace $N_?$ of the $N \times N$ which contains the diagonal $N$.  The embedding is by mapping a tropical curve to the location of its two trivalent vertices, and the linear equation determining the three dimensional subspace $N_?$ is the one requiring the slope of the internal edge to be exactly $\frac{1-k}{n}$.  Within this three dimensional lattice, we will see that our cones are all simplicial, defined by three linear functions. We compile the requisite linear functions below

\[
\begin{array}{c} {\frac{a_2}{n}}\\\frac{-w_2}{n}\\ \frac{-a_1}{n}
\end{array}
\hskip.2in
\begin{array}{c} \frac{w_2}{n}\\\frac{-w_2+a_2}{n}  \\ \frac{-a_1}{n}
\end{array}
\hskip.2in
\begin{array}{c} \frac{a_2}{n}\\ \frac{w_2-a_2}{n}\\ \frac{-w_1}{n}
\end{array}
\]

 Each inequality here applies to the location of a single vertex, we used the subscripts 1 and 2 to refer to the leftmost and rightmost vertices. The weight $w_i = ka_i + nb_i$.   For example, the middle entry in the above grid reflects the fact that the right hand vertex in the second tropical curve above must lie below the dotted line.   The factors of $\frac1n$ are to get the minimal linear functional which is integral on all of our vertices.  The claim here is that a 4-tuple of real numbers $(a_1,b_1,a_2,b_2)$ corresponds to the locations of the first and second vertex of an allowed tropical curve of a given type if and only if $(1-k)(a_2-a_1)=n(b_2-b_1)$ and each of the three linear functions in the corresponding column takes a positive integer value.  

 Setting $a_1=a_2$ and $b_1=b_2$, we see the contributions from the middle and right terms cancel, leaving the left term contribution $\frac{a^2w}{n^3}.$  This is the denominator.

 Combining with the numerator $a^2w$ gives a net $n^3$,
 and dividing by automorphisms {(of order $n$)} yields $n^2$. 
\appendix

\section{Heuristic Calculations}

We present here some numerical and heuristic calculations, just to give a feel for the problem.

\begin{example}[Direct Algebraic Approach]
    Here we perform in some more detail an example
    similar to Example \ref{ex:p2ex}.  Consider $\triangle = \mathrm{conv}\{(0,0),(2,0),(0,2),(2,2)\}$ so $L_\triangle = \cO_{\bP^1\times \bP^1}(2,2)$ and choose points $(\frac{1}{2},0),(-2,0),(\pm i,\infty),(0,\pm 1),(\infty,\pm i)$ along the toric boundary, so that the restricted linear system of degree (2,2) polynomials defining curves meeting these points is given by
    $f = x^2y^2 + x^2 + y^2 + \frac{3}{2}x - 1 + axy$.  If we eliminate $x$ and $y$ from the equations $f_x = f_y = f = 0,$ we find the polynomial $P(a) = a^8 + 12a^6 + 216a^4 - 3312a^2 + 20736,$ which is multiplicity-free of degree 8.  There are thus 8 nodal curves.
\end{example}

\begin{example}[Numerical Calculations]
    Ding-Wei \cite{DW} classify convex lattice polygons $\Delta$ with up to two interior lattice points, 
    up to integral unimodular affine transformation.
    For each, we have computed numerically using Mathematica the number of nodal
    curves with fixed boundary intersection in the corresponding linear systems of
genus-two curves --- see Table \ref{table:data}.  We do this by
finding singular points on the discriminant
which correspond to two distinct singular points
of the curve.  
All of these numbers have been confirmed by counts
of tropical curnve and counts and by ruling
    counts, most by hand or by our computer code.  For some we used
    computing assistance from ChatGPT 5.6 Sol.
    In addition, an \emph{ad hoc} argument for the count of Polygon 3q is given in Example \ref{ex:jac} below.

\begin{table}[ht]
    \centering
    \setlength{\tabcolsep}{3pt}
    \begin{tabular}{c|ccccccccccccccccccccc}
        \cite{DW} Fig. & a & b & c & d & e & f & g & h & i & j & k & l & m & n & o & p & q & r & s \\
        \hline
         Fig.~{2} & {9} & {5} & {16} & {30} & 48 \\
         Fig.~3 & {9} & {9} & {9} & 48 & {30} & {18} & {30} & {31} & {16} &{17} & {13} & {39} & {8} & 23 & 48 & {24} & {48} & 12 & {16}\\
         Fig.~10 & 13 & {13} & {18} & 18 & 24 & 31 & 39 & {31} & 24 & 18 & 39  & 31 & 23 & 17 & 12 & 23 &\\
         Fig.~12 & {18} & 24 & 31 & {18} & {31} \\ 
    \end{tabular}
    \vskip.1in
    \caption{For each convex lattice polytope with two interior lattice points, we compute numerically using Mathematica the number of singular points on the discriminant corresponding to two distinct critical points on the curve.  The figure numbers and letters correspond to figures and sub-figures of \cite{DW}.}
        \label{table:data}
\end{table}

\end{example}

\begin{example}[{\it Ad hoc} Calculations]
    \label{ex:jac}
    Recall the argument from \cite{YZ} for the case of K3.  We have a complete linear system and associated map $\mathrm{K3}\to \bP^g$ whose
    divisors have genus $g$, then identify the moduli of pairs of curve plus line bundle birationally with $\mathrm{Sym}^g(\mathrm{K3})$:  $g$ points determine a hyperplane in $\bP^g$ and a divsisor $C$ of genus $g$ in $\mathrm{K3}$ together with $g$ points, so an element in $\mathrm{Pic}^g(C).$  Then an Euler characteristic calculation of $\mathrm{Sym}^g(\mathrm{K3})$ can be performed explicitly.

    In the present case, we can try to apply a version of this argument, heuristically at least.  We have a linear system $\bP_\triangle \to |L_\Delta|.$
    Fixing the toric intersection imposes linear constraints on the linear system.
    Then an open set of the moduli space
    can be parametrized as $\mathrm{Sym}^g(T),$ where $T\subset \bP$ is the open torus boundary.  The true moduli space is a partial compactification.

    In the case $g=1$, the fixed-boundary condition means that in the compactification, the point in $T$ can only approach the boundary in a direction \emph{not} tangent to it.  There are an $\bA^1$ worth of directions --- so each boundary contribues $1$ to the Euler characteristic.  Numerical calculations confirm that for lattice polygons $\triangle$ with a single interior point that $N_\triangle$ is equal to the lattice length of the perimeter.

    Let us do a case with $g=2.$  Take $\triangle = \mathrm{conv}\{(0,0),(3,0),(3,2),(0,2)\}$,
    so $L_\triangle = \cO_{\bP^1\times \bP^1}(3,2).$
    Now if one point approaches the boundary, since $\chi(T)=0$, there is no contribution from the Euler characteristic unless the other point does, as well.  Since there are 10 boundary points, there are $\binom{10}{2} = 45$ points where the two points approach different boundary points.
    
    A subtlety now arises here:  if the two points have the same $x$ coordinate --- say $x = x_0$ then the same curve $C$ arises from both this pair and the pair $C\cap \{ x-x_1 = 0\}$ for any $x_1$.  Note the corresponding curve $C$ for this pair is the same, by construction, and the line bundle over $C$ is linearly equivalent, as the difference between the two sets of points is related by the rational function $\frac{x-x_1}{x-x_0}.$  The Euler characteristic for such pairs is $0,$ however, and we only need to focus on the implications of this rational equivalence for the compactification.
    For instance, two of the 45 pairs above are equivalent:  the pair of distinct points at $x=0$ and the pair at $x=\infty$.  Our running count is thus 44.

    We now consider the case where the two interior points approach the same boundary point.  A local study of the two-jet near such configurations gives rise to a set of possibilities with Euler characteristic 1.  {\it However}, this positive contribution can be canceled out by the case where two points with the same $x$ coordinate approach the same boundary point.  As a result, the 6 boundary points at $y = 0,\infty$ have no net contribution.  The four points at $x=0,\infty$ do, contributing 4. 

   Total:  44+4 = 48.  This agrees with the count for Figure 3q in Table \ref{table:data}.

\end{example}

\begin{example}[More on the case \texorpdfstring{$\cO_{\bP^1\times\bP^1}(3,2)$}{O(3,2) on P1xP1}]
\label{ex:coameba}

Consider a curve $C\subset \bP^1\times \bP^1$ defined in inhomogeneous coordinates as the zero locus of 
$$F = -1 + z^3 + w^2 + z^3w^2 + azw + bz^2w$$
thought of as section of $\cO_{\bP^1\times\bP^1}(3,2)$.  Note at $z = 0$ we have $w^2 = 1$ and at $z=\infty$ we have $w^2 = -1,$ while at $w = 0$ we have $z^3 = 1$ and at $w = \infty$ we have $z^3 = -1.$

When $a=b=0$ the polynomial is ``maximally sparse,'' meaning its nonzero exponents are the vertices of its Newton polytope.  Coamebas of maximally sparse polynomials were studied by Nisse and have support pattern the same as the mirror constructible sheaves --- see \cite[Section 7 and Figure 11]{Nisse}.  Here is a picture of the coameba when $a=b=0$, which is maximally sparse:

\[\includegraphics[scale=.3]{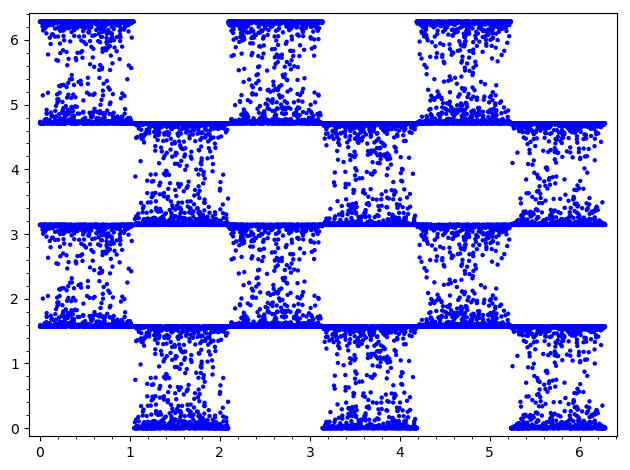}\]
For further observations on the relation between
coamebas and constructible sheaves, see Example \ref{ex:o23}.


Now let us restore $a$ and $b$.  The singular locus $S \subset \bC^2_{a,b}$ is the zero locus of the resultant
\begin{multline*}f= a^{10}b^4-a^4b^{10}-64a^{12}+164a^6b^6-64b^{12}
-6336a^8b^2+6336a^2b^8\\+141696a^4b^4+55296a^6-55296b^6-3234816a^2b^2-11943936\end{multline*}
These curves generically have single nodes (codimension one in parameter space).  We want more singular curves (two nodes), so
we investigate where
$\Delta$ itself is singular.  Let $g = f_a, h = f_b.$  We are interested in the locus of values $(a,b)$ where $f = g = h = 0.$
We can study this locus set-theoretically or scheme-theoretically.
Programming in Mathematica and solving only numerically,
one finds $96$ singular points, $48$ of which have
a unique point $(z,w)\in T$ where $C$ is singular (a cusp, then, generically),
and $48$ of which have a pair of distinct points, presumably
two nodes.  We take this as numerical evidence that
the number of nodal curves in this example is $48.$
Similar calculations produced the
results in Table \ref{table:data}.

\end{example}

\section{Further Examples}
\label{sec:furtherexs}

\subsection{Tropical curves}
\label{sec:3p}
When $\Delta = \mathrm{conv}\left((0,0),(3,0),(3,1),(0,2)\right) = 
\begin{tikzpicture}[scale=.22]
    \draw[very thick](0,0)--(3,0)--(3,1)--(0,2)--(0,0);
    \fill[blue] (1,0) circle (.25cm);
    \fill[blue] (2,0) circle (.25cm);
    \fill[blue] (3,0) circle (.25cm);
    \fill[blue] (0,1) circle (.25cm);
    \fill[blue] (1,1) circle (.25cm);
    \fill[blue] (2,1) circle (.25cm);
    \fill[blue] (3,1) circle (.25cm);
    \fill[blue] (0,2) circle (.25cm);
    \fill[blue] (0,0) circle (.25cm);
    
\end{tikzpicture}$, we computed numerically in Table \ref{table:data} (3p) that $N_\Delta = 24.$
Figure \ref{fig:tropical-curves}
illustrates the tropical count,
with curves and multiplicities shown.
Note the images of tropical trees
may have positive genus, due
to crossings:  the four-valencies are not
singularities of the domain, just the image.
\begin{figure}[htbp]
\centering

\begin{subfigure}{0.24\textwidth}
    \centering
    \includegraphics[width=\linewidth]{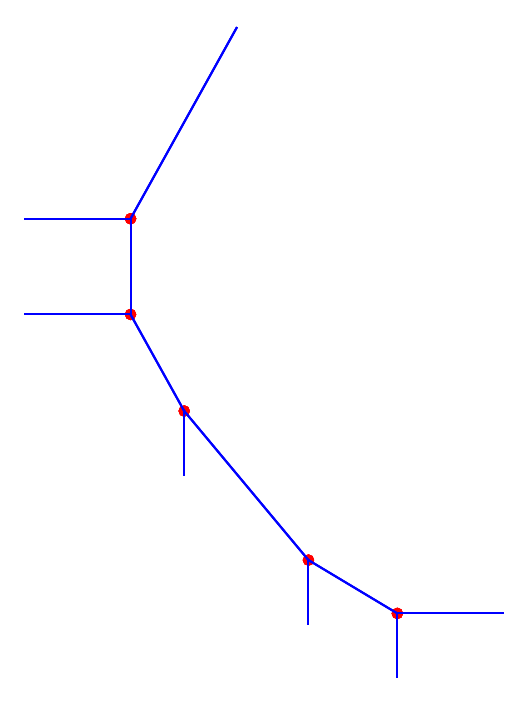}
    \caption{Multiplicity $9 = 3^2$}
\end{subfigure}
\begin{subfigure}{0.24\textwidth}
    \centering
    \includegraphics[width=\linewidth]{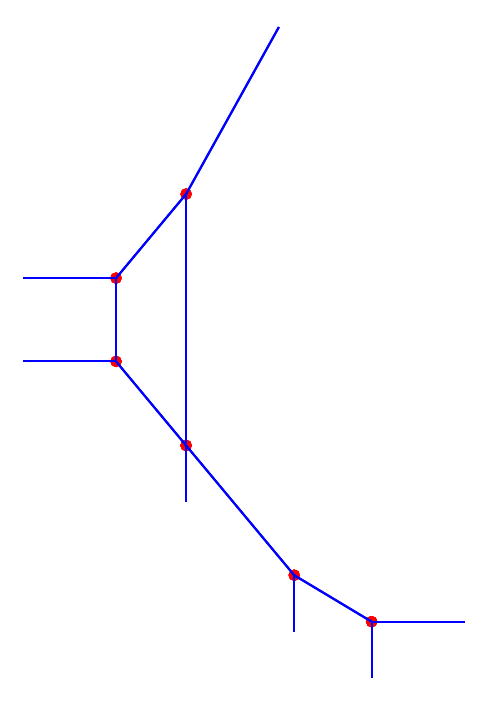}
    \caption{Multiplicity $4 = 2^2$}
\end{subfigure}
\begin{subfigure}{0.24\textwidth}
    \centering
    \includegraphics[width=\linewidth]{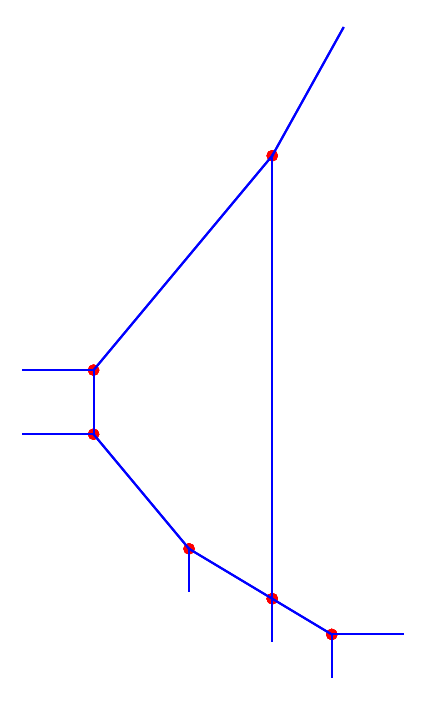}
    \caption{Multiplicity $4 = 2^2$}
\end{subfigure}
\begin{subfigure}{0.24\textwidth}
    \centering
    \includegraphics[width=\linewidth]{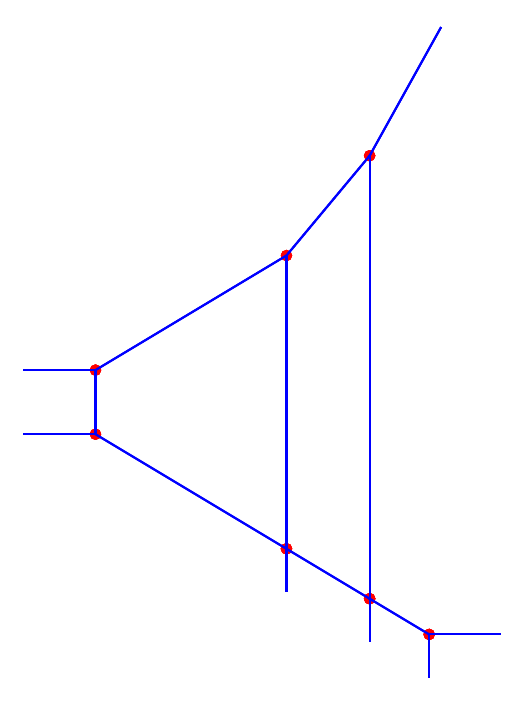}
    \caption{Multiplicity 1}
\end{subfigure}

\begin{subfigure}{0.24\textwidth}
    \centering
    \includegraphics[width=\linewidth]{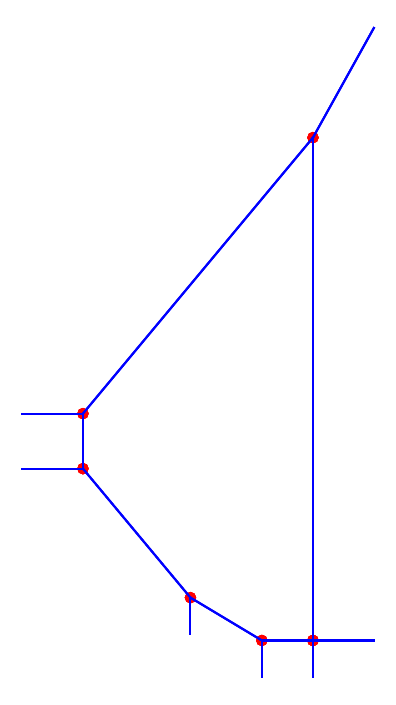}
    \caption{Multiplicity $4 = 2^2$}
\end{subfigure}
\begin{subfigure}{0.24\textwidth}
    \centering
    \includegraphics[width=\linewidth]{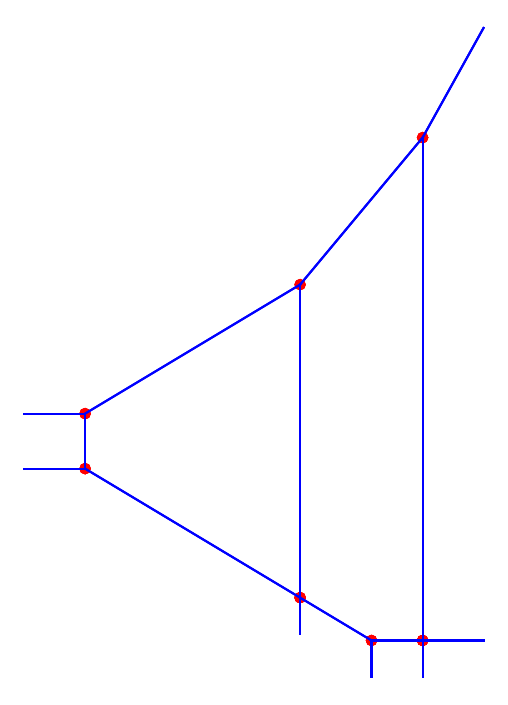}
    \caption{Multiplicity 1}
\end{subfigure}
\begin{subfigure}{0.24\textwidth}
    \centering
    \includegraphics[width=\linewidth]{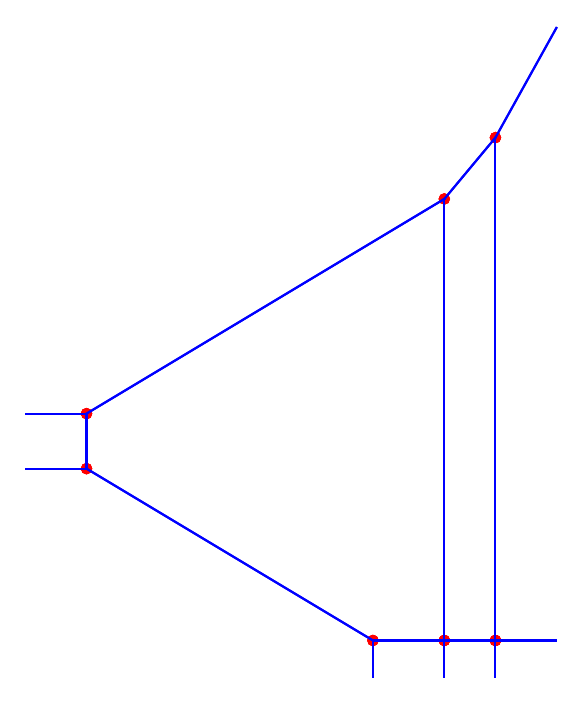}
    \caption{Multiplicity 1}
\end{subfigure}

\caption{The sum of multiplicities
is $N_\Delta = 24$.}
\label{fig:tropical-curves}
\end{figure}

The count of $q$-multiplicities is
$$3 + 3(q^\frac{1}{2}+q^{-\frac{1}{2}})^2 + (q + 1 + q^{-1})^2 = q^2 + 5q + 12 + 5q^{-1} + q^{-2} = z^{4} + 9z^2 + 24,$$
where $z = q^\frac{1}{2}-q^{-\frac{1}{2}}$.
This agrees with the counts of rulings
in different genera --- see Example \ref{ex:3p}.

\subsection{More examples of rulings}\footnote{With the programming assistance from ChatGPT 5.6 Sol, we are able to compute rulings in many examples, including all the cases in Table~\ref{table:data}. A list of rulings for various polygons, together with the code used to compute them, is available at \url{https://hu-mingyuan.github.io/ruling-lab/index.html}.  However, the examples below, as well as many others in Table~\ref{table:data}, were obtained before the AI era by hand.}

\begin{example}
\label{ex:o23}
    Consider \(\cO(2, 3) \) on \( \bP^1 \times \bP^1\), which is (3q) in Table~\ref{table:data}. The rulings are shown in Figure~\ref{fig:ruling_O23}. There are 36 rational rulings of the form \ref{subfig:filling_two_holes} (filling two non-diagonal holes), and 12 of the form \ref{subfig:diamond}, so 48 in total. 

    We remark that for some smooth and singular curves in the linear system $|\cO(2, 3)|$, the coamebas look a lot like rulings of constructible sheaves --- see Figures \ref{subfig:standardruling},\ref{subfig:filling_two_holes},\ref{subfig:diamond} and their counterparts in Figure \ref{fig:aztec}.  The relationship between rulings and coamebas remains unclear to us. 

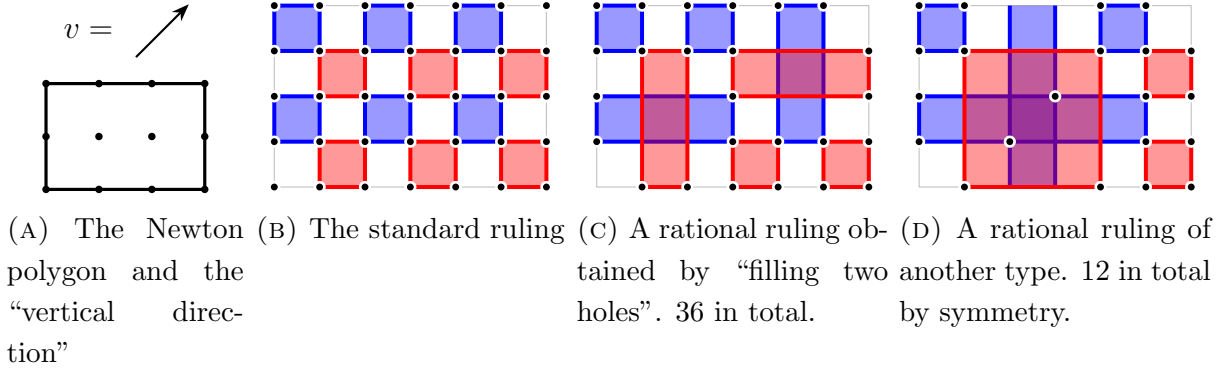
\begin{figure}[ht]
\centering

    \begin{subfigure}[t]{.19\linewidth}
    \centering
            
    \begin{tikzpicture}[scale = .7]
    \foreach\i in {0,...,3}{
    \foreach\j in {0,...,2}{
    \fill (\i,\j) circle (2pt);   }}
    \draw[very thick] (0,0)--(3, 0)--(3,2)--(0,2)--(0,0);    

    \draw[thick, -Stealth] (1.7, 2.5) -- (2.7, 3.5);
    \node at (.8, 3) {$v = $};
    
    \end{tikzpicture}
    \caption{The Newton polygon and the ``vertical direction''}
    \end{subfigure}
    \begin{subfigure}[t]{.25 \linewidth}
    
    \centering

    \begin{tikzpicture}[scale = .6,yscale=-1]
    \draw[step=1cm, gray!50,thin] (0,0) grid (6,4);
    
    \foreach\i in {0, 2}{
    \foreach\j in {0, 2, 4}{
        \filldraw[draw = blue, line width = 1.5pt, fill = blue, fill opacity = .4] (\j, \i) rectangle ++(1, 1);
    }
    }
    \foreach\i in {1, 3}{
    \foreach\j in {1, 3, 5}{
        \filldraw[draw = red, line width = 1.5pt, fill = red, fill opacity = .4] (\j, \i) rectangle ++(1, 1);
    }
    }

    \foreach\i in {0, ..., 6}{
        \foreach\j in {0, ..., 4}{
            \filldraw[black, draw=white, line width=1pt] (\i, \j) circle (3pt);
        }
    }
    
    \end{tikzpicture}
    \caption{The standard ruling}
    \label{subfig:standardruling}
    \end{subfigure}
    \begin{subfigure}[t]{.25\linewidth}

    \centering
    \begin{tikzpicture}[scale = .6,yscale=-1]
        \draw[step=1cm, gray!50,thin] (0,0) grid (6,4);
        \filldraw[draw = blue, line width = 1.5pt, fill = blue, fill opacity = .4] (0, 0) rectangle ++ (1, 1);
        \filldraw[draw = blue, line width = 1.5pt, fill = blue, fill opacity = .4] (0, 2) rectangle ++ (3, 1);
        \filldraw[draw = blue, line width = 1.5pt, fill = blue, fill opacity = .4] (2, 0) rectangle ++ (1, 1);
        \filldraw[draw = blue, line width = 1.5pt, fill = blue, fill opacity = .4] (4, 0) rectangle ++ (1, 3);
        \filldraw[draw = red, line width = 1.5pt, fill = red, fill opacity = .4] (1, 1) rectangle ++ (1, 3);
        \filldraw[draw = red, line width = 1.5pt, fill = red, fill opacity = .4] (3, 1) rectangle ++ (3, 1);
        \filldraw[draw = red, line width = 1.5pt, fill = red, fill opacity = .4] (3, 3) rectangle ++ (1, 1);
        \filldraw[draw = red, line width = 1.5pt, fill = red, fill opacity = .4] (5, 3) rectangle ++ (1, 1);

        \filldraw[black, draw=white, line width=1pt] (1, 1) circle (3pt);
        \filldraw[black, draw=white, line width=1pt] (0, 0) circle (3pt);
        \filldraw[black, draw=white, line width=1pt] (0, 1) circle (3pt);
        \filldraw[black, draw=white, line width=1pt] (1, 0) circle (3pt);
        \filldraw[black, draw=white, line width=1pt] (0, 3) circle (3pt);
        \filldraw[black, draw=white, line width=1pt] (0, 2) circle (3pt);
        \filldraw[black, draw=white, line width=1pt] (1, 4) circle (3pt);
        \filldraw[black, draw=white, line width=1pt] (2, 1) circle (3pt);
        \filldraw[black, draw=white, line width=1pt] (3, 3) circle (3pt);
        \filldraw[black, draw=white, line width=1pt] (4, 4) circle (3pt);
        \filldraw[black, draw=white, line width=1pt] (4, 3) circle (3pt);
        \filldraw[black, draw=white, line width=1pt] (3, 1) circle (3pt);
        \filldraw[black, draw=white, line width=1pt] (4, 0) circle (3pt);
        \filldraw[black, draw=white, line width=1pt] (5, 0) circle (3pt);
        \filldraw[black, draw=white, line width=1pt] (6, 1) circle (3pt);
        \filldraw[black, draw=white, line width=1pt] (6, 2) circle (3pt);
        \filldraw[black, draw=white, line width=1pt] (3, 2) circle (3pt);
        \filldraw[black, draw=white, line width=1pt] (5, 3) circle (3pt);
        \filldraw[black, draw=white, line width=1pt] (6, 3) circle (3pt);
        \filldraw[black, draw=white, line width=1pt] (6, 4) circle (3pt);
        \filldraw[black, draw=white, line width=1pt] (5, 4) circle (3pt);
        \filldraw[black, draw=white, line width=1pt] (5, 4) circle (3pt);
        \filldraw[black, draw=white, line width=1pt] (3, 0) circle (3pt);
        \filldraw[black, draw=white, line width=1pt] (2, 0) circle (3pt);
        \filldraw[black, draw=white, line width=1pt] (2, 4) circle (3pt);
        \filldraw[black, draw=white, line width=1pt] (3, 4) circle (3pt);

    \end{tikzpicture}
    
    \caption{A rational ruling obtained by ``filling two holes''. 36 in total.}
    \label{subfig:filling_two_holes}
    \end{subfigure}
    \begin{subfigure}[t]{.25\linewidth}
    \centering
    \begin{tikzpicture}[scale = .6,yscale= - 1]
        \draw[step=1cm, gray!50,thin] (0,0) grid (6,4);
        
        \filldraw[draw = blue, line width = 1.5pt, fill = blue, fill opacity = .4] (0, 0) rectangle ++ (1, 1);
        
        \filldraw[draw = blue, line width = 1.5pt, fill = blue, fill opacity = .4] (4, 0) rectangle ++ (1, 1);
        \filldraw[draw = blue, line width = 1.5pt, fill = blue, fill opacity = .4] (2, 4) -- (2, 3) -- (0, 3) -- (0, 2) -- (3, 2) -- (3, 4); 
        \filldraw[draw = blue, line width = 1.5pt, fill = blue, fill opacity = .4] (2, 0) -- (2, 3) -- (5, 3) -- (5, 2) -- (3, 2) -- (3, 0);
        \filldraw[draw = red, line width = 1.5pt, fill = red, fill opacity = .4] (1, 1) rectangle ++ (3, 3);
        \filldraw[draw = red, line width = 1.5pt, fill = red, fill opacity = .4] (5, 1) rectangle ++ (1, 1);
        
        \filldraw[draw = red, line width = 1.5pt, fill = red, fill opacity = .4] (5, 3) rectangle ++ (1, 1);

        \filldraw[black, draw=white, line width=1pt] (1, 1) circle (3pt);
        \filldraw[black, draw=white, line width=1pt] (0, 0) circle (3pt);
        \filldraw[black, draw=white, line width=1pt] (0, 1) circle (3pt);
        \filldraw[black, draw=white, line width=1pt] (1, 0) circle (3pt);
        \filldraw[black, draw=white, line width=1pt] (0, 3) circle (3pt);
        \filldraw[black, draw=white, line width=1pt] (0, 2) circle (3pt);
        \filldraw[black, draw=white, line width=1pt] (1, 4) circle (3pt);
        \filldraw[black, draw=white, line width=1pt] (2, 3) circle (3pt);
        \filldraw[black, draw=white, line width=1pt] (3, 2) circle (3pt);
        \filldraw[black, draw=white, line width=1pt] (4, 4) circle (3pt);
        \filldraw[black, draw=white, line width=1pt] (4, 1) circle (3pt);
        \filldraw[black, draw=white, line width=1pt] (5, 1) circle (3pt);
        \filldraw[black, draw=white, line width=1pt] (4, 0) circle (3pt);
        \filldraw[black, draw=white, line width=1pt] (5, 0) circle (3pt);
        \filldraw[black, draw=white, line width=1pt] (6, 1) circle (3pt);
        \filldraw[black, draw=white, line width=1pt] (6, 2) circle (3pt);
        \filldraw[black, draw=white, line width=1pt] (5, 2) circle (3pt);
        \filldraw[black, draw=white, line width=1pt] (5, 3) circle (3pt);
        \filldraw[black, draw=white, line width=1pt] (6, 3) circle (3pt);
        \filldraw[black, draw=white, line width=1pt] (6, 4) circle (3pt);
        \filldraw[black, draw=white, line width=1pt] (5, 4) circle (3pt);
    \end{tikzpicture}
    
    \caption{A rational ruling of another type. 12 in total by symmetry.}
    \label{subfig:diamond}
    \end{subfigure}
    \caption{Rulings for $\cO(2, 3)$. The ``vertical direction'' is taken to be $v = (1, 1)$, and the black dots are the switches.}
    \label{fig:ruling_O23}
\end{figure}

\begin{figure}[ht]
    \centering
\includegraphics[scale=.25]{o32coameba.png}
\includegraphics[scale=.3]{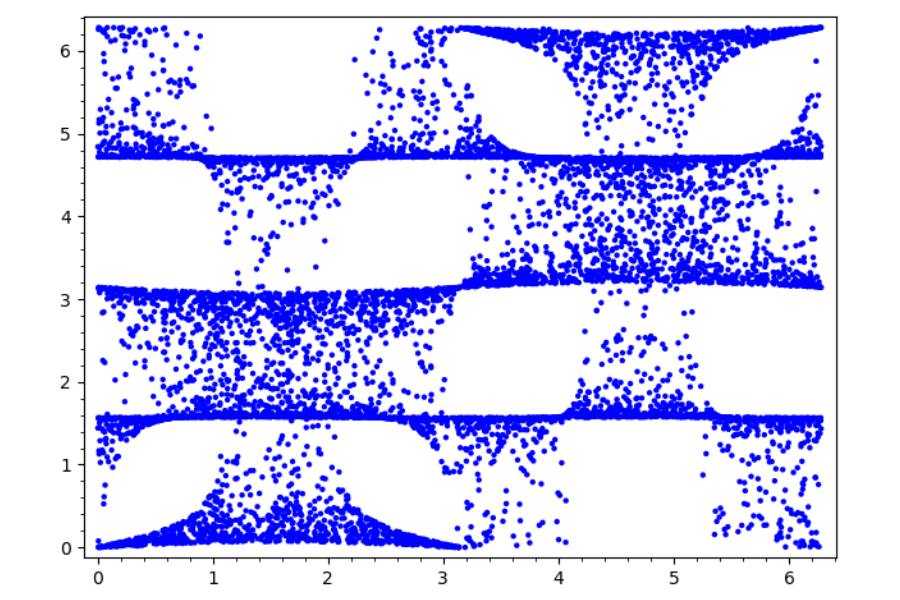}
\includegraphics[scale=.145]{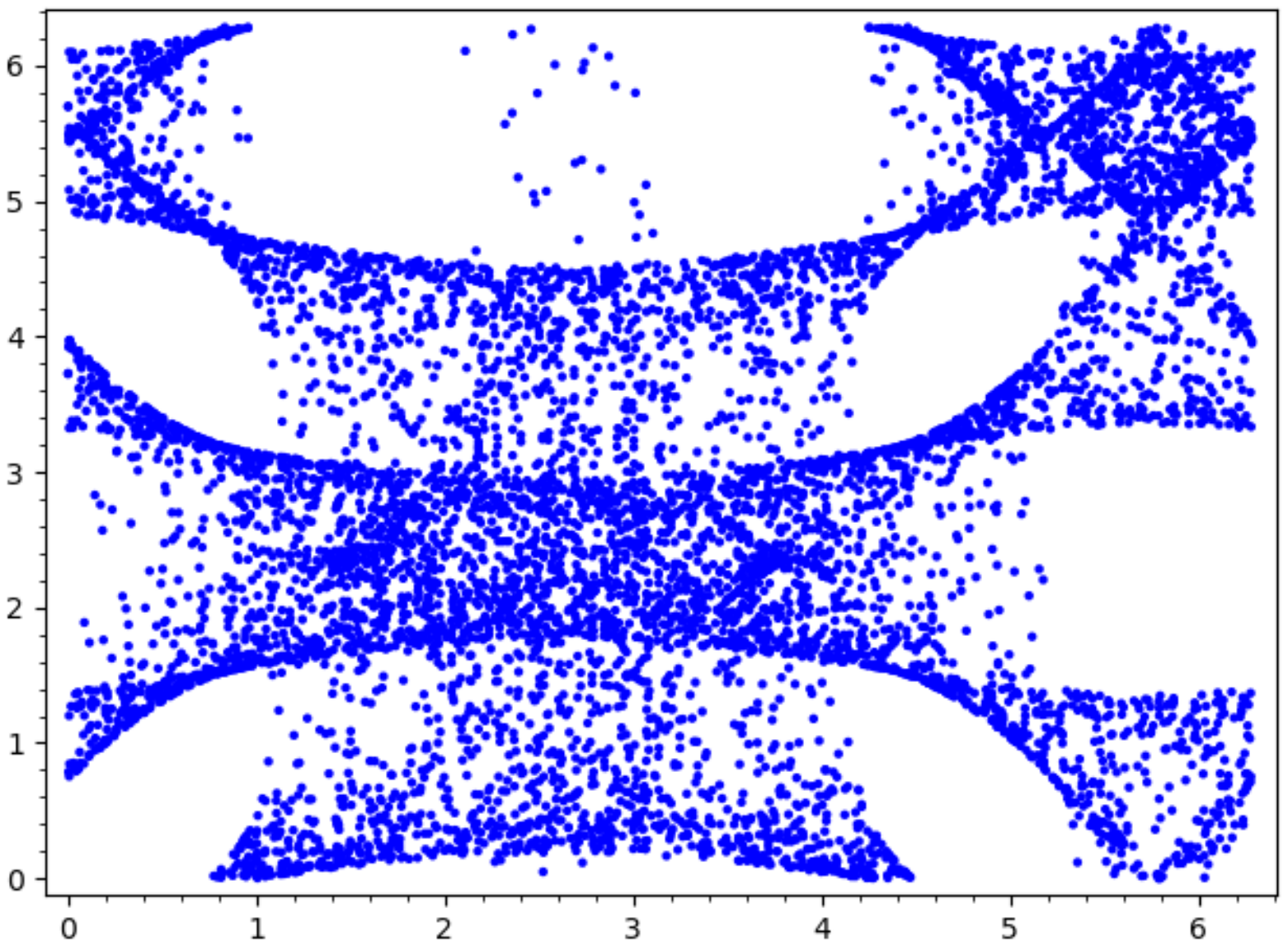}
    \caption{Coaoembas of a smooth (left) and two singular curves in $|O_{\bP^1\times \bP^1}(2, 3)|$.  The pictures look a lot like the corresponding constructible sheaves of Figure \ref{fig:ruling_O23}.  Note the Aztec
    diamond of Example \ref{ex:oablocal} at right.}
    \label{fig:aztec}
\end{figure}
\end{example}

\begin{example}
    \label{ex:3p}
    For Polygon (3p) of \cite{DW} there are 24 rational rulings, 9 genus-one rulings which ``fill one hole,'' and the unique standard ruling of genus two, giving a ruling polynomial of $z^4 + 9z^2 + 24,$ in agreement with the example
    of Section \ref{sec:3p}.
    See Figure~\ref{fig:3p_rulings}. 

    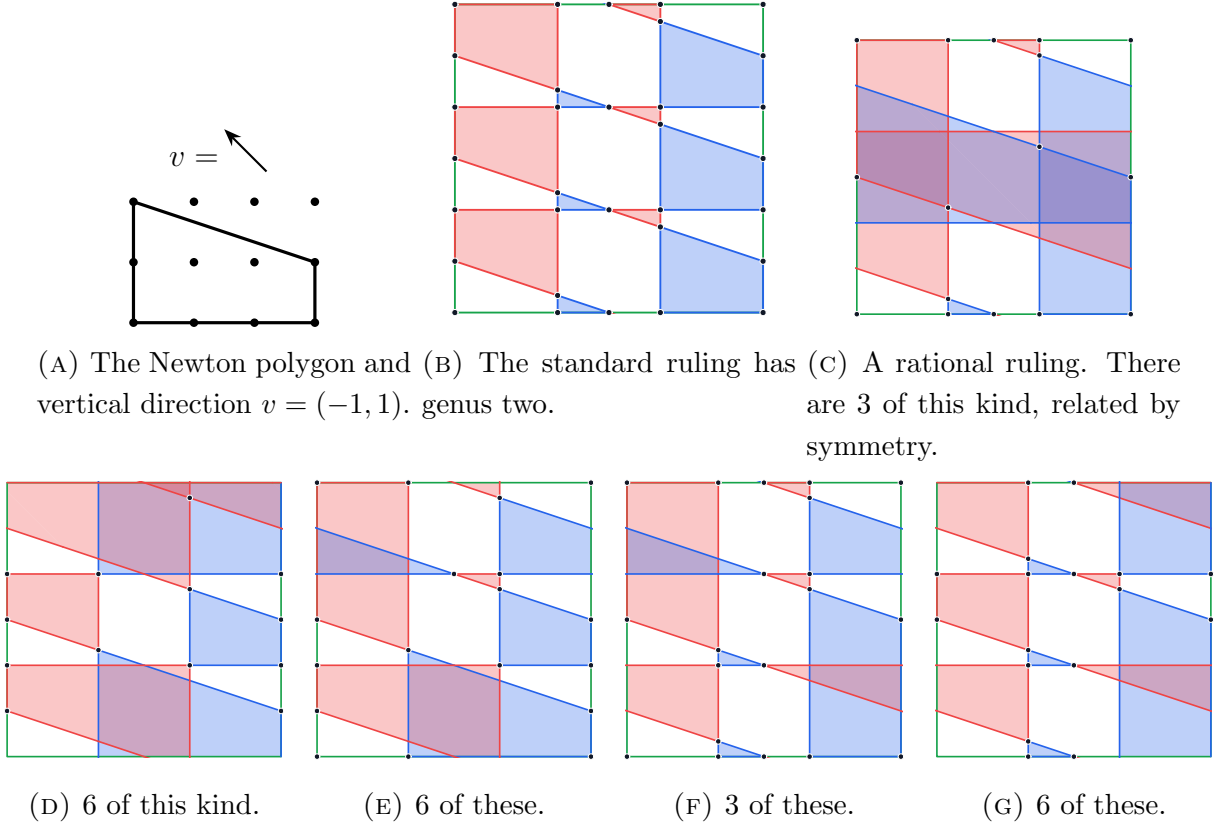
\begin{figure}[ht]
    \centering
    \begin{subfigure}[t]{.3\linewidth}
        \centering
        \begin{tikzpicture}[scale = .8]
             \foreach \i in {0,...,3}{
                \foreach \j in {0,...,2}{
                    \fill (\i,\j) circle (2pt);
                }
            }
            \draw[very thick] (0, 0) -- (3, 0) -- (3, 1) -- (0, 2) -- cycle;
            \draw[thick, -Stealth] (2.2, 2.5) -- (1.5, 3.2);
            \node at (1, 2.7) {$v=$};
        \end{tikzpicture}
        \caption{The Newton polygon and vertical direction $v = (-1, 1)$.}
    \end{subfigure}
        \begin{subfigure}[t]{.3\linewidth}
        \centering
        \includegraphics[width= .9 \linewidth]{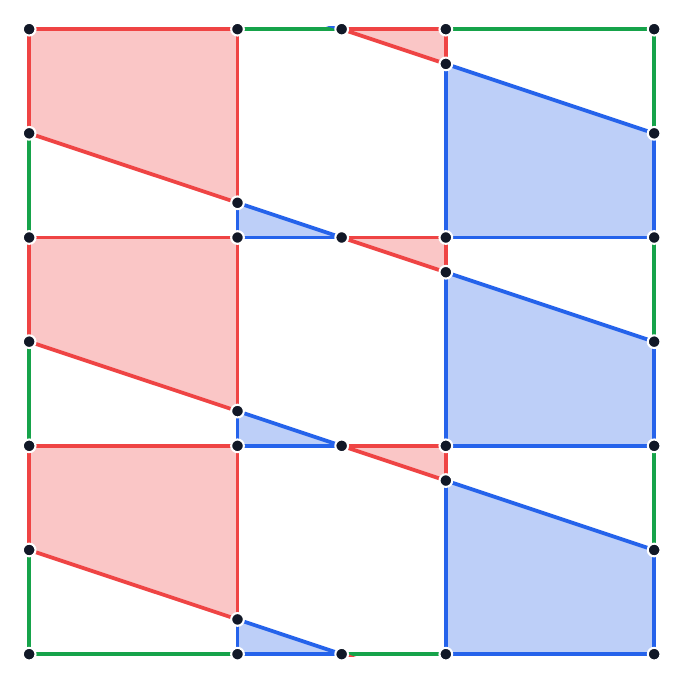}
        \caption{The standard ruling has genus two.}
        \end{subfigure}
        \begin{subfigure}[t]{.3\linewidth}
        \centering
        \includegraphics[width= .8 \linewidth]{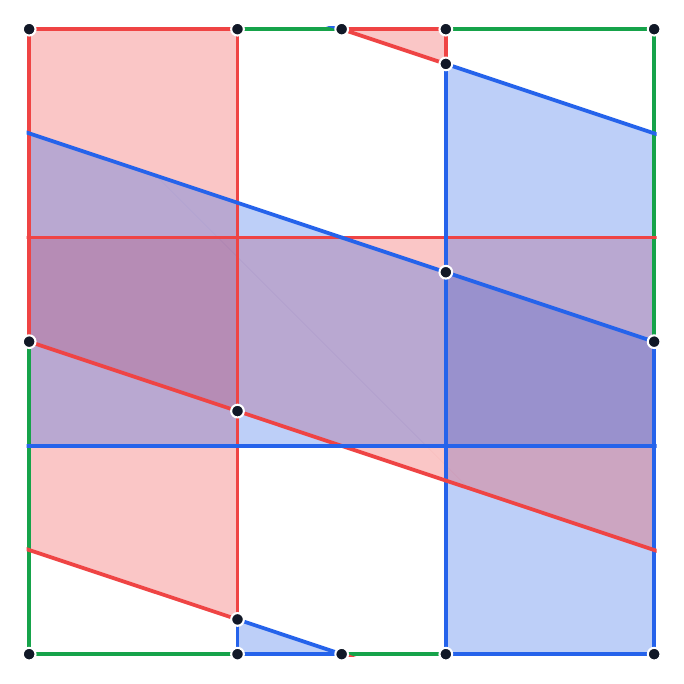}
        \caption{A rational ruling.  There are 3 of this kind, related
        by symmetry.}
        \end{subfigure}
        \begin{subfigure}[t]{.24\linewidth}
        \centering
        \includegraphics[width= \linewidth]{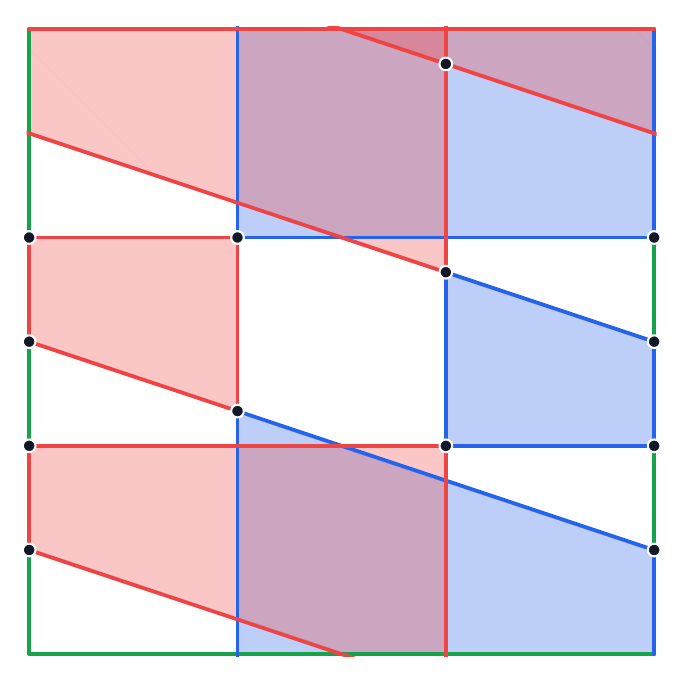}
        \caption{6 of this kind.}
        \end{subfigure}
        \begin{subfigure}[t]{.24\linewidth}
        \centering
        \includegraphics[width= \linewidth]{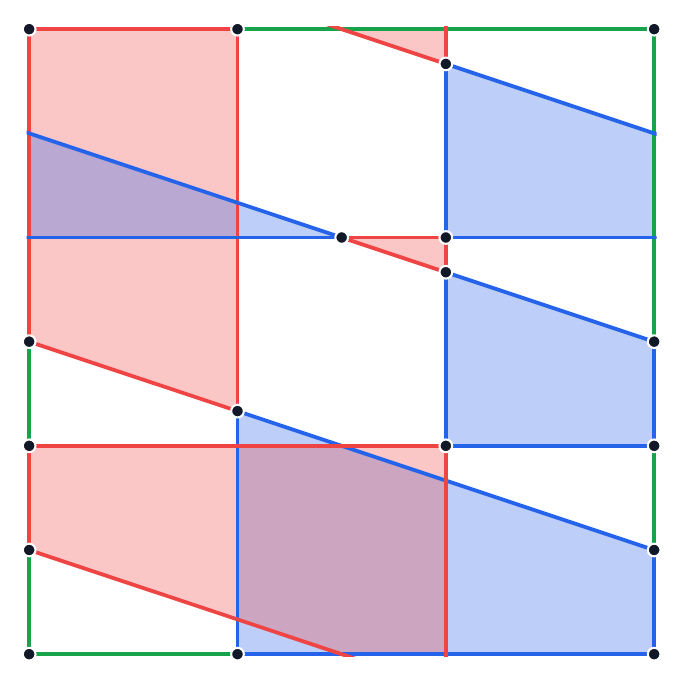}
        \caption{6 of these.}
        \end{subfigure}
        \begin{subfigure}[t]{.24\linewidth}
        \centering
        \includegraphics[width= \linewidth]{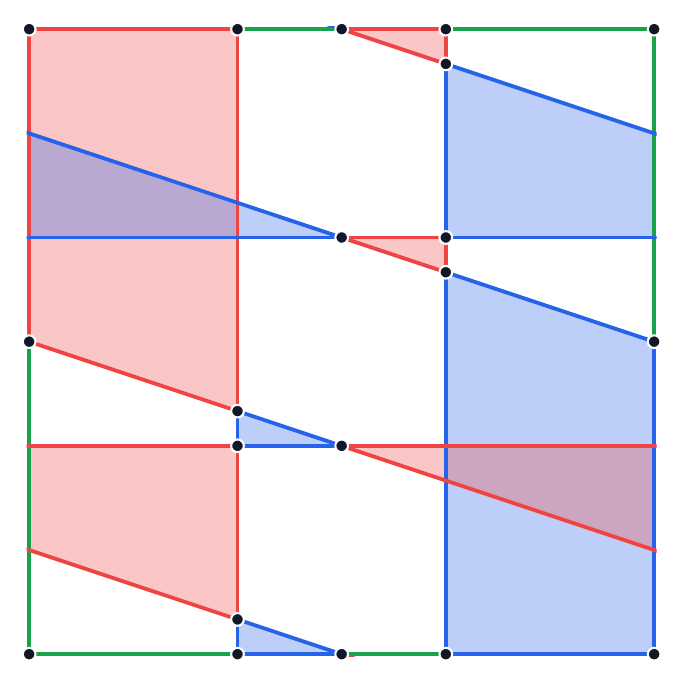}
        \caption{3 of these.}
        \end{subfigure}
        \begin{subfigure}[t]{.24\linewidth}
        \centering
        \includegraphics[width= \linewidth]{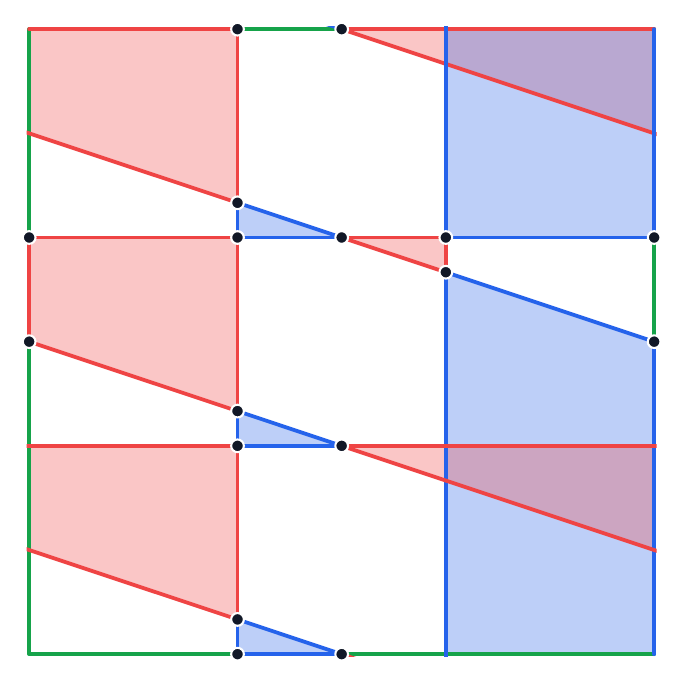}
        \caption{6 of these.}
        \end{subfigure}
    \caption{Rulings for (3p) Table~\ref{table:data}. The black dots are the switches. There are $24$ rational rulings in total.}   
    \label{fig:3p_rulings}
    \end{figure}
    
\end{example}

\begin{example}
\label{ex:2c}
Consider case (2c) in Table~\ref{table:data}. Its Newton polygon is shown in Figure~\ref{subfig:2c_NP}, and the standard ruling arising from a bipartite graph is shown in Figure~\ref{subfig:2c_standard}. There are two types of rational rulings, illustrated in Figures~\ref{subfig:2c_ration_ruling_C} and~\ref{subfig:2c_rational_ruling_D}. By symmetry, each type consists of 8 rulings, so 16 in total.

\begin{figure}[ht]
    \centering
    \begin{subfigure}[t]{.2\linewidth}
        \centering
        \begin{tikzpicture}[scale=.7]
            \foreach \i in {0,...,4}{
                \foreach \j in {0,...,2}{
                    \fill (\i,\j) circle (2pt);
                }
            }
            \draw[very thick] (0,0)--(4,0)--(1,2)--(0,0);
        \node at (1.5, 3) { $v = $};
        \draw[thick, -Stealth] (2.5, 2.5) -- (2.5, 3.5);
               
        \end{tikzpicture}
        \caption{The Newton polygon.}
        \label{subfig:2c_NP}
    \end{subfigure}
    \begin{subfigure}[t]{.25\linewidth}
        \centering
        \includegraphics[width=\linewidth]{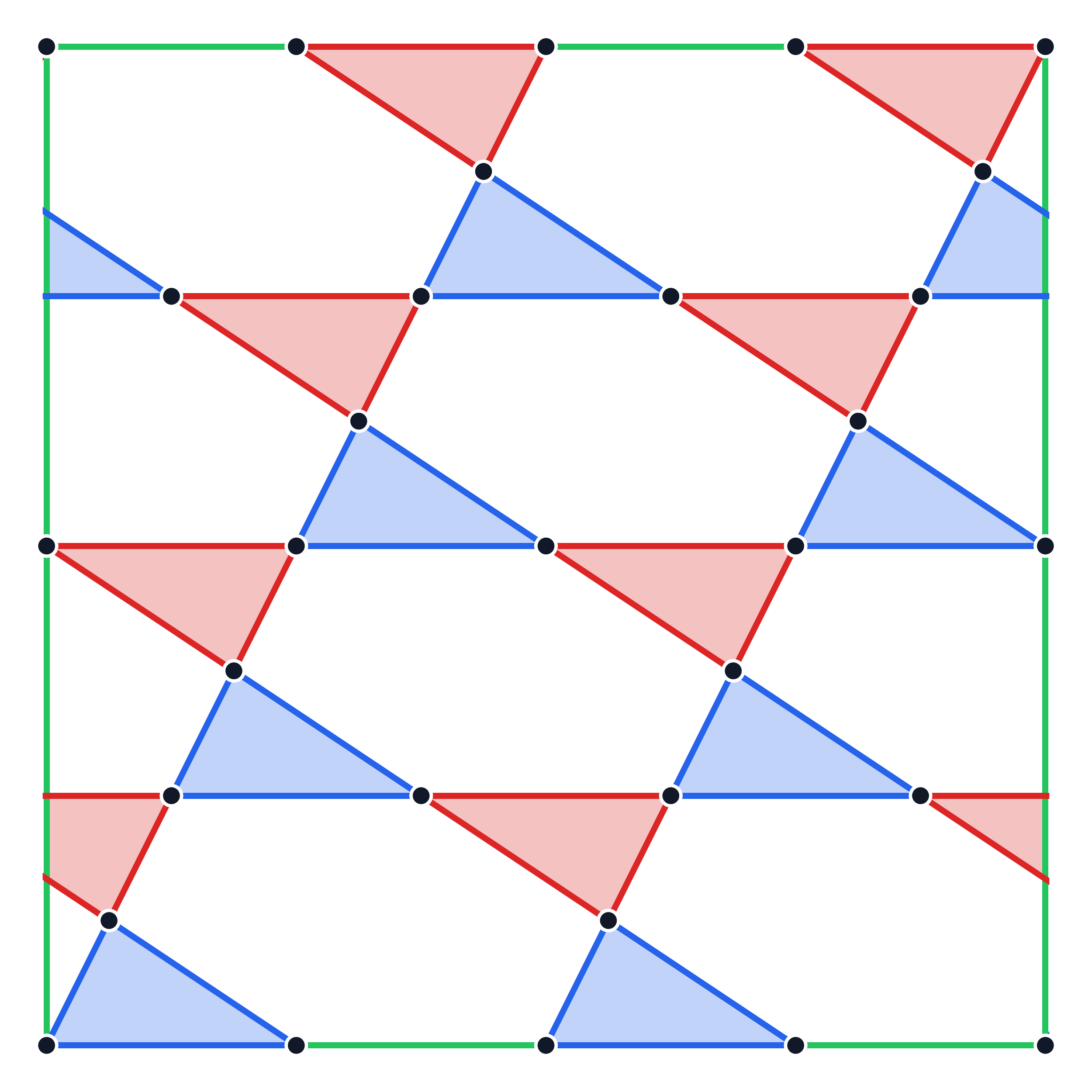}
        \caption{The standard ruling.}
        \label{subfig:2c_standard}
    \end{subfigure}
    \begin{subfigure}[t]{.25\linewidth}
        \centering
        \includegraphics[width=\linewidth]{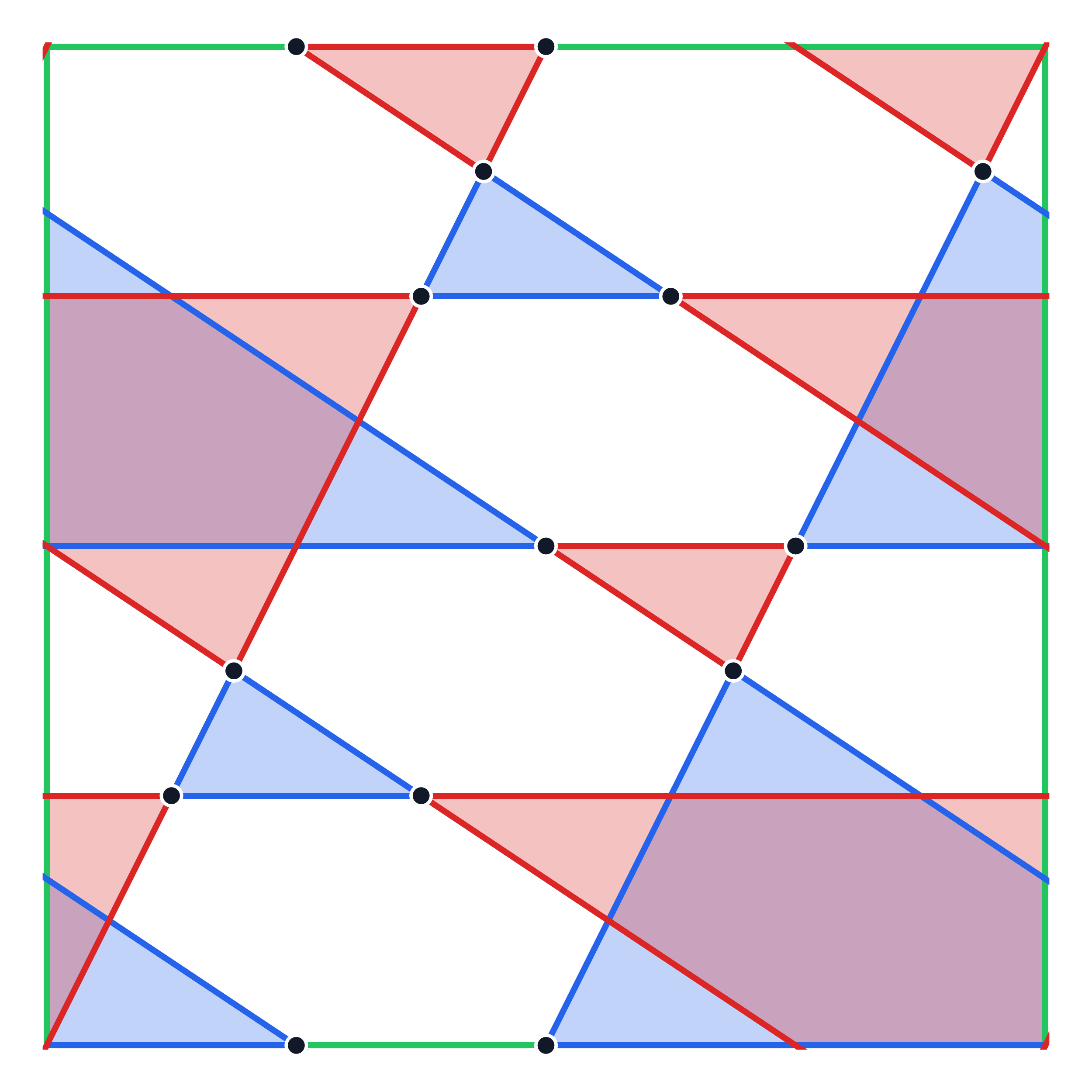}
        \caption{A type of rational ruling. 8 in total by symmetry.}
        \label{subfig:2c_ration_ruling_C}
    \end{subfigure}
    \begin{subfigure}[t]{.25\linewidth}
        \centering
        \includegraphics[width=\linewidth]{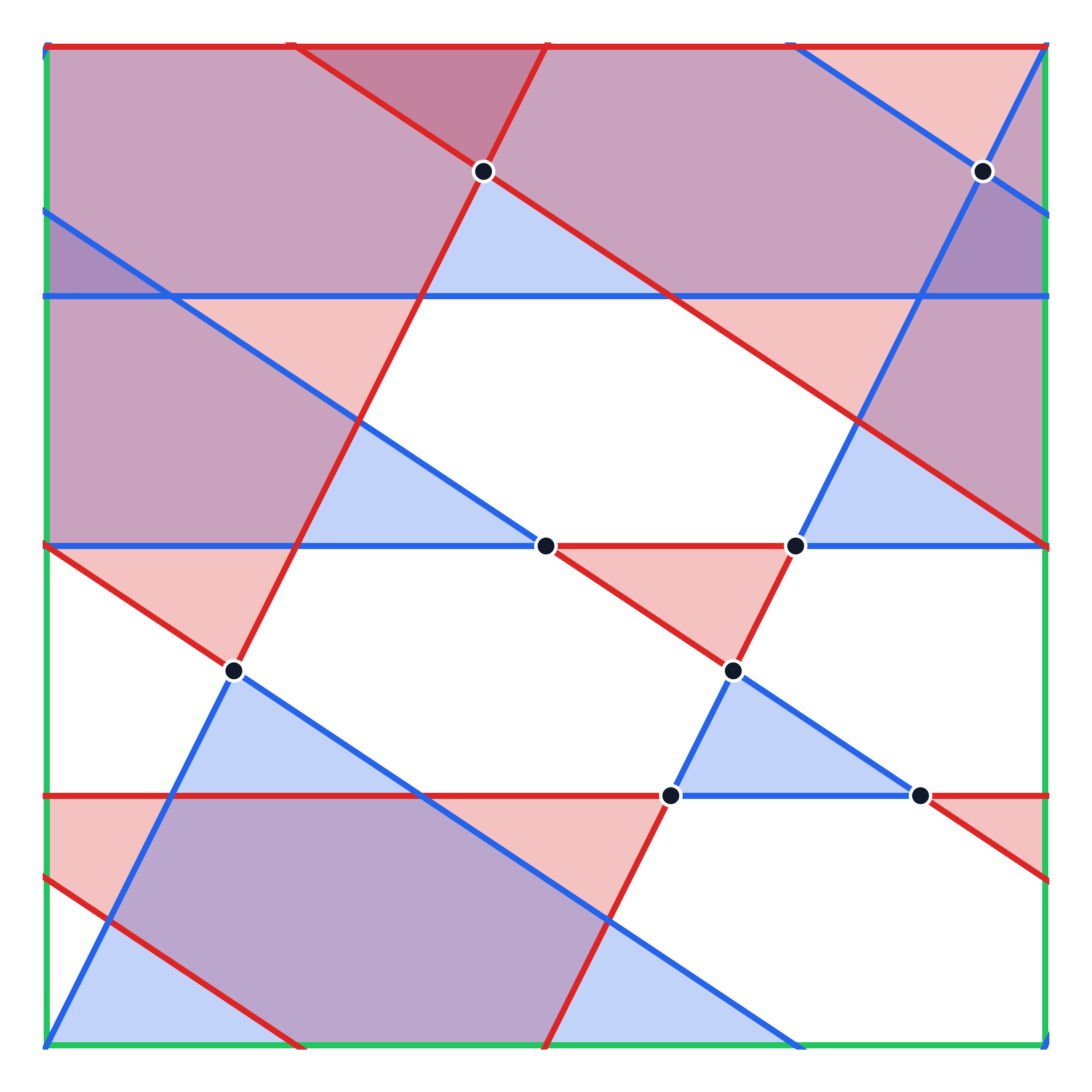}
        \caption{Another type of rational ruling. 8 in total by symmetry.}
        \label{subfig:2c_rational_ruling_D}
    \end{subfigure}
    \caption{Rulings for (2c) in Table~\ref{table:data}. The black dots are the switches. There are 16 rational rulings in total.}
\end{figure}
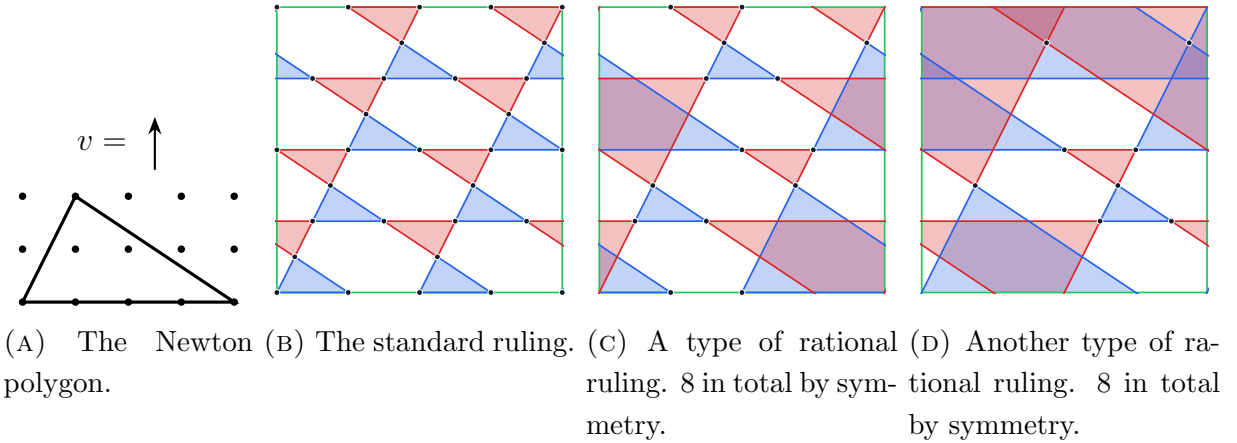

\end{example}

\begin{example}
Consider $\cO(4)$ on $\bP^2$. The number of interior integral points is $g_\triangle = 3$. The rational rulings are shown in Figure~\ref{fig:O4}. There are 304 in total. 

\begin{figure}[ht]
\centering
\begin{subfigure}[t]{.3\linewidth}
\centering
\begin{tikzpicture}[scale=.8]
\foreach \i in {-1,...,3}{
\foreach \j in {-1,...,3}{
\fill (\i,\j) circle (2pt);
}
}
\draw[very thick] (-1,-1)--(3,-1)--(-1,3)--cycle;
\node at (1.2,3.7) {$v=(-1,2)$};
\end{tikzpicture}
\caption{The Newton polygon}
\end{subfigure}
\begin{subfigure}[t]{.4\linewidth}
\centering
\includegraphics[width=.7\linewidth]{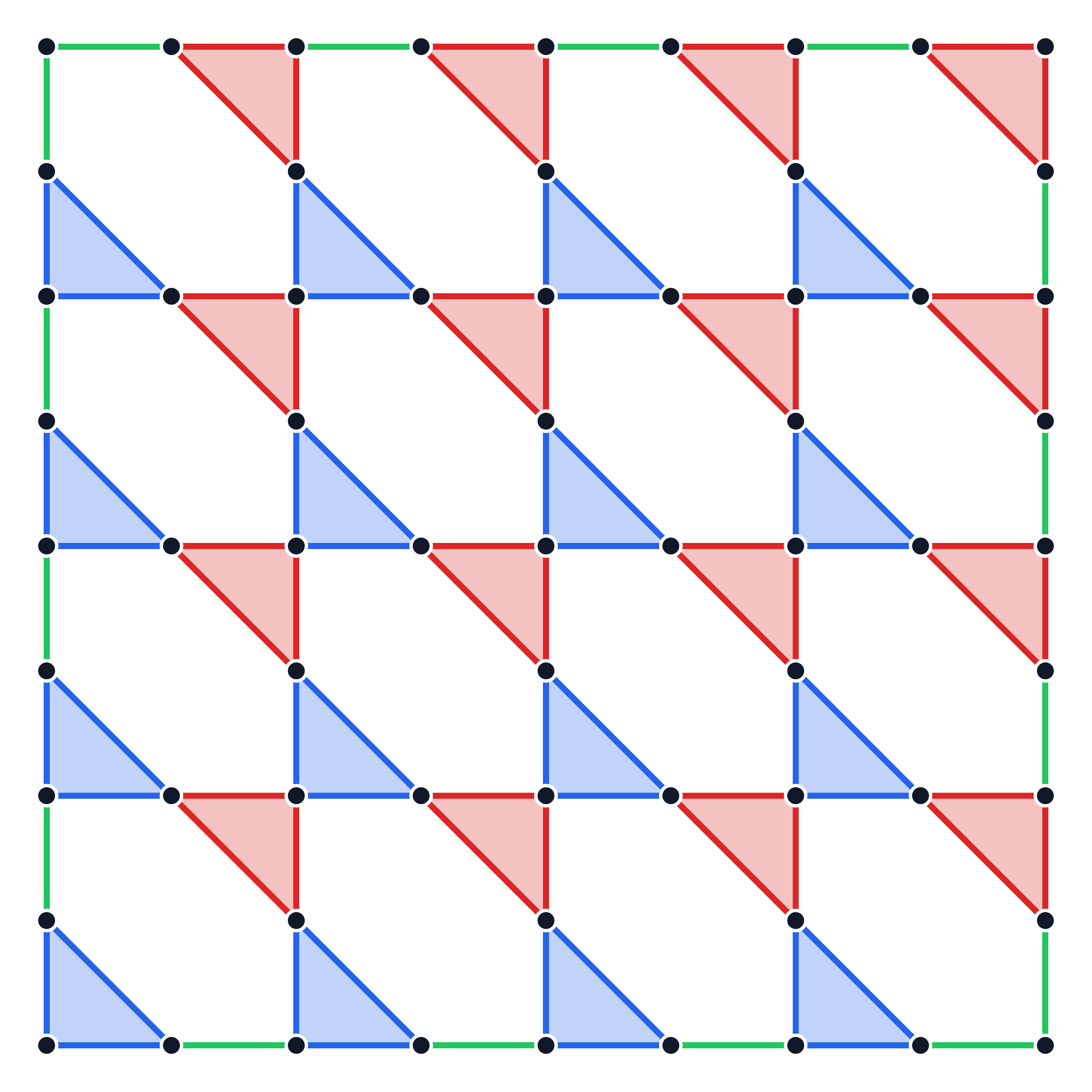}
\caption{The standard ruling}
\end{subfigure}
\begin{subfigure}[t]{.24\linewidth}
\centering
\includegraphics[width=\linewidth]{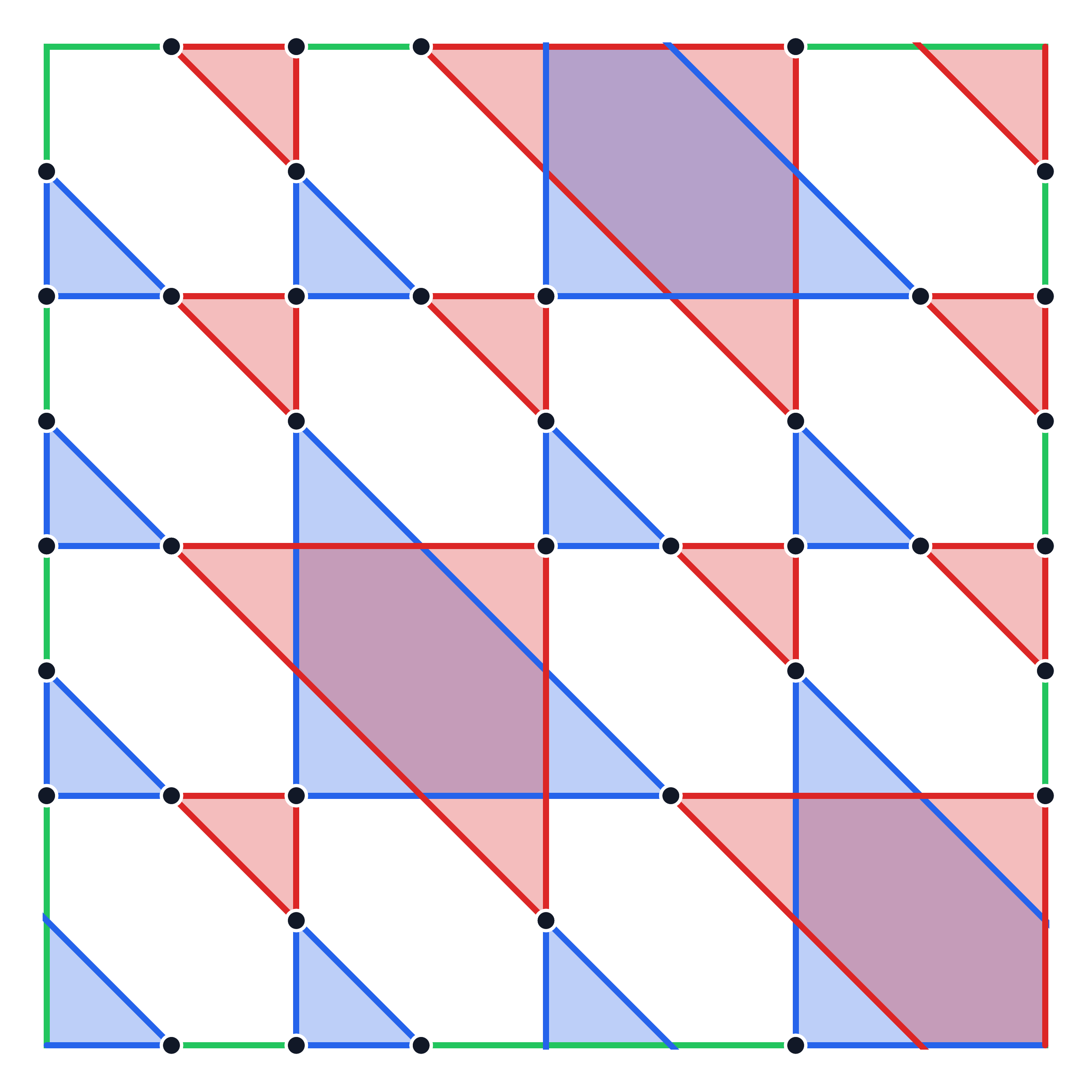}
\caption{32 in total}
\end{subfigure}
\begin{subfigure}[t]{.24\linewidth}
\centering
\includegraphics[width=\linewidth]{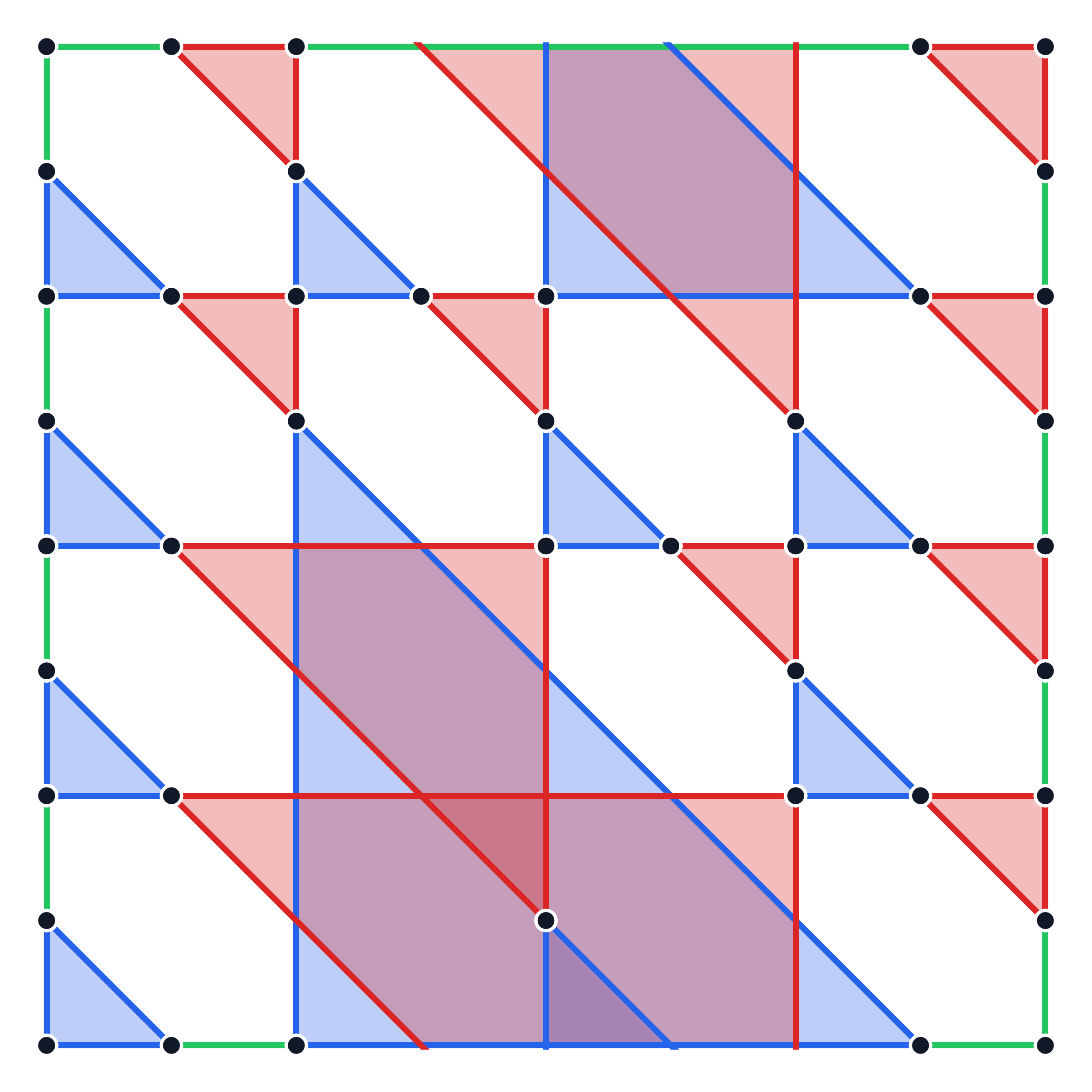}
\caption{16 in total}
\end{subfigure}
\begin{subfigure}[t]{.24\linewidth}
\centering
\includegraphics[width=\linewidth]{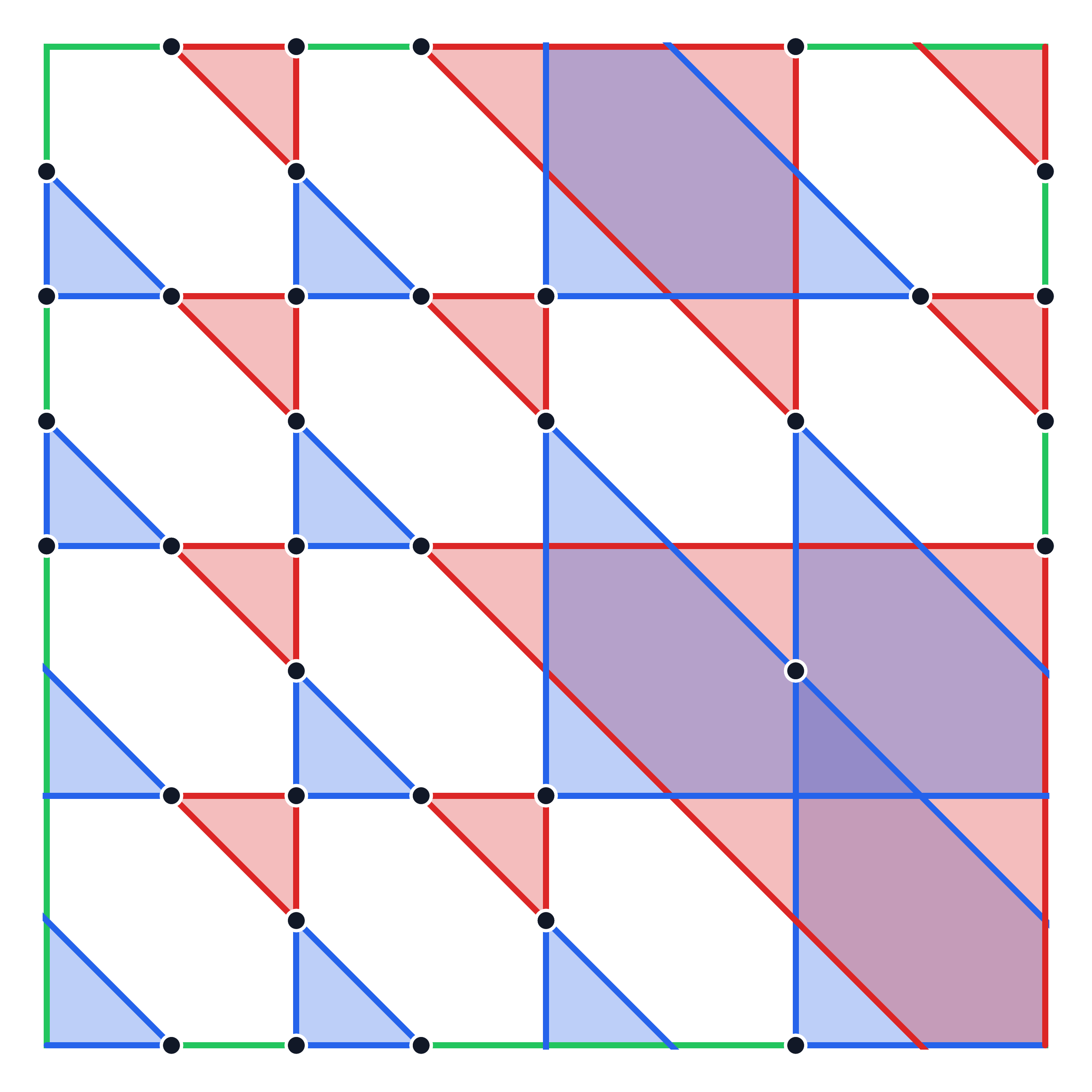}
\caption{32 in total}
\end{subfigure}
\begin{subfigure}[t]{.24\linewidth}
\centering
\includegraphics[width=\linewidth]{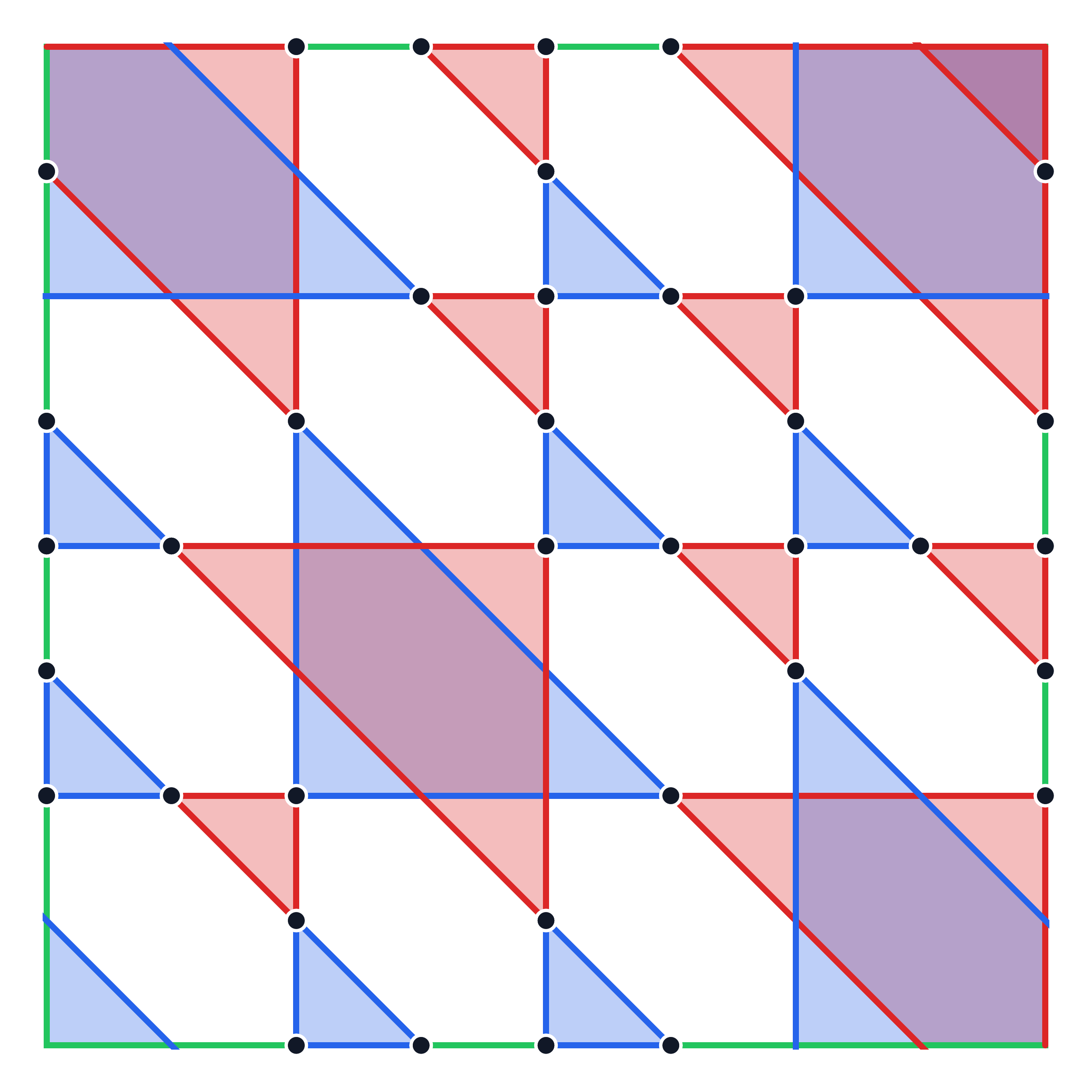}
\caption{32 in total}
\end{subfigure}
\begin{subfigure}[t]{.24\linewidth}
\centering
\includegraphics[width=\linewidth]{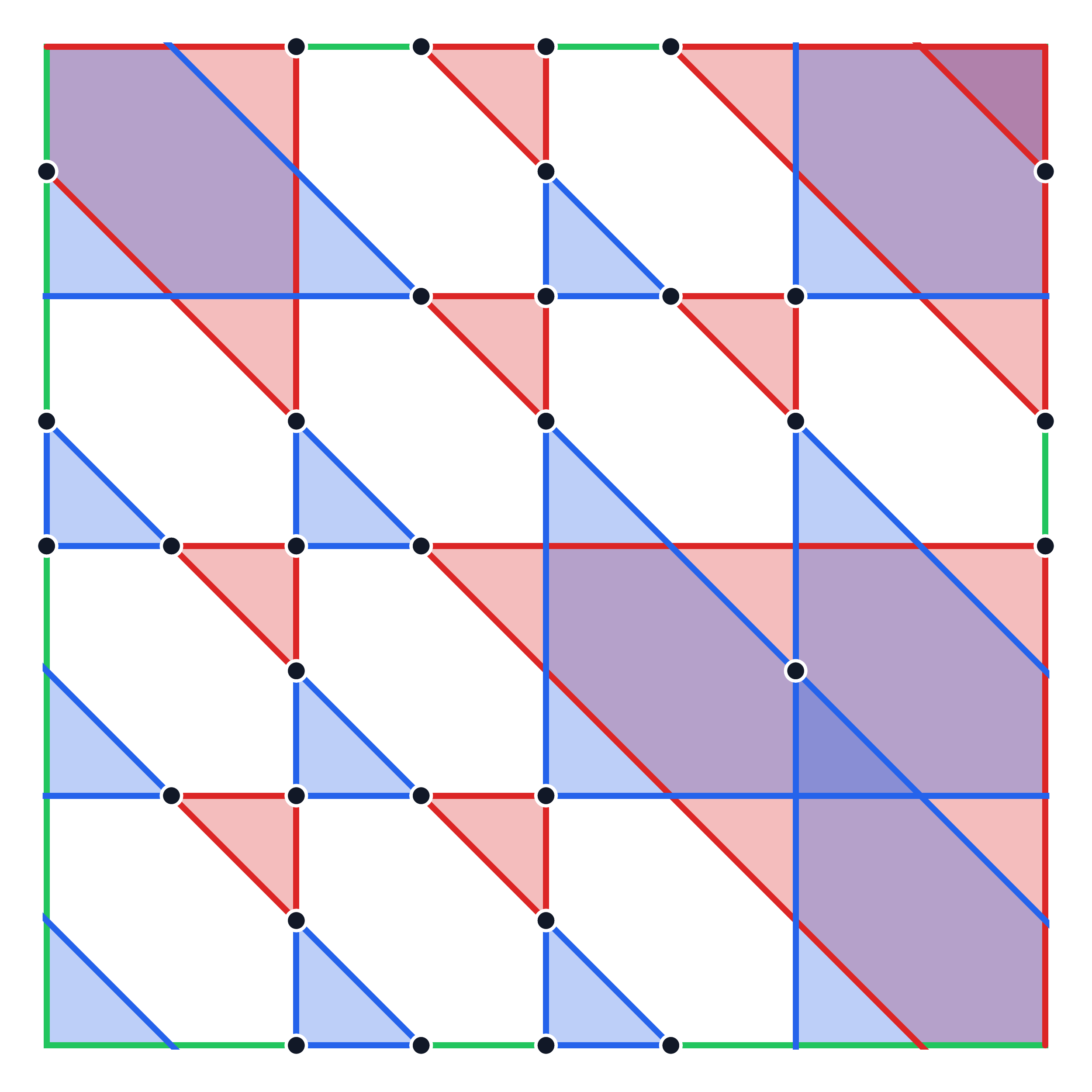}
\caption{16 in total}
\end{subfigure}
\begin{subfigure}[t]{.24\linewidth}
\centering
\includegraphics[width=\linewidth]{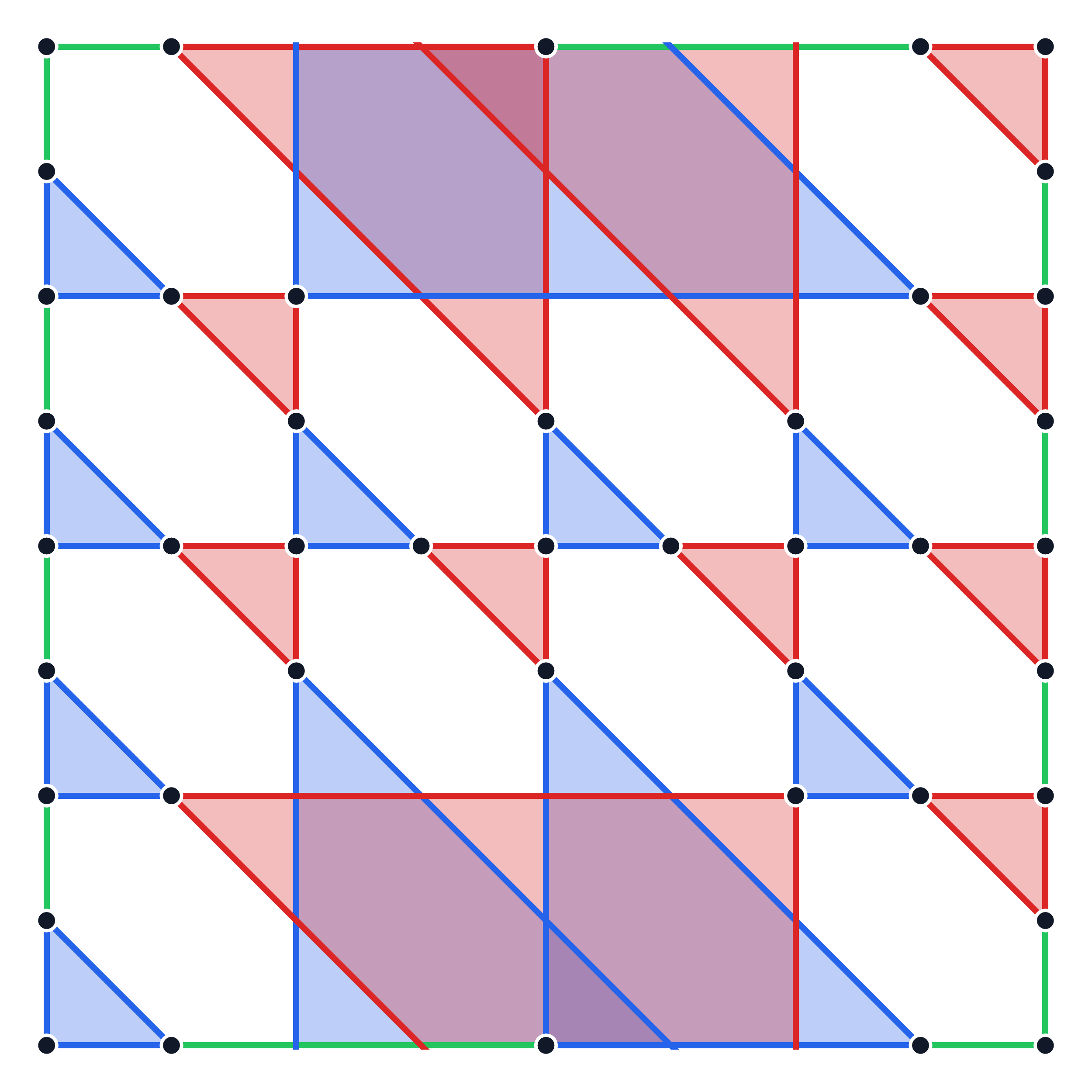}
\caption{32 in total}
\end{subfigure}
\begin{subfigure}[t]{.24\linewidth}
\centering
\includegraphics[width=\linewidth]{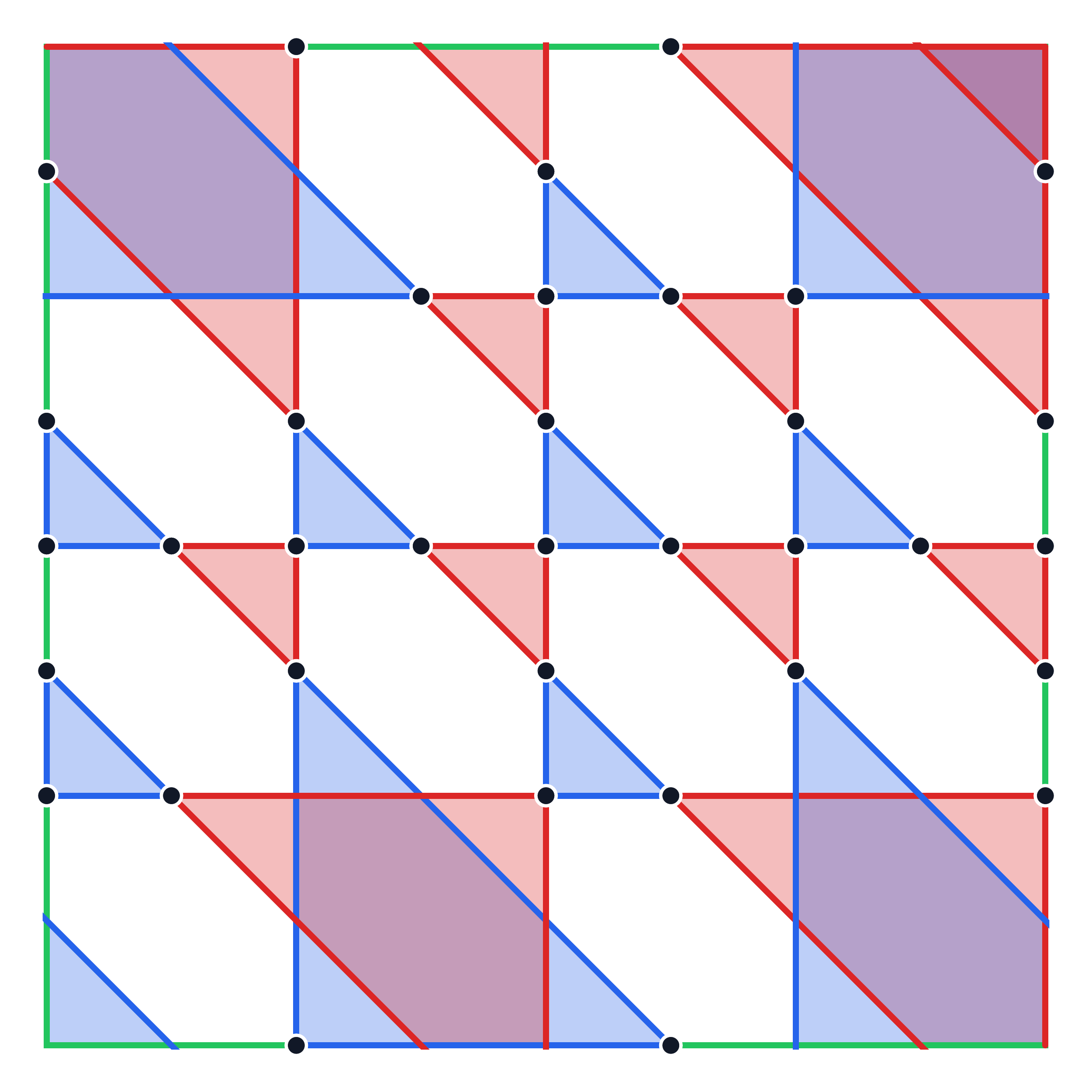}
\caption{16 in total}
\end{subfigure}
\begin{subfigure}[t]{.24\linewidth}
\centering
\includegraphics[width=\linewidth]{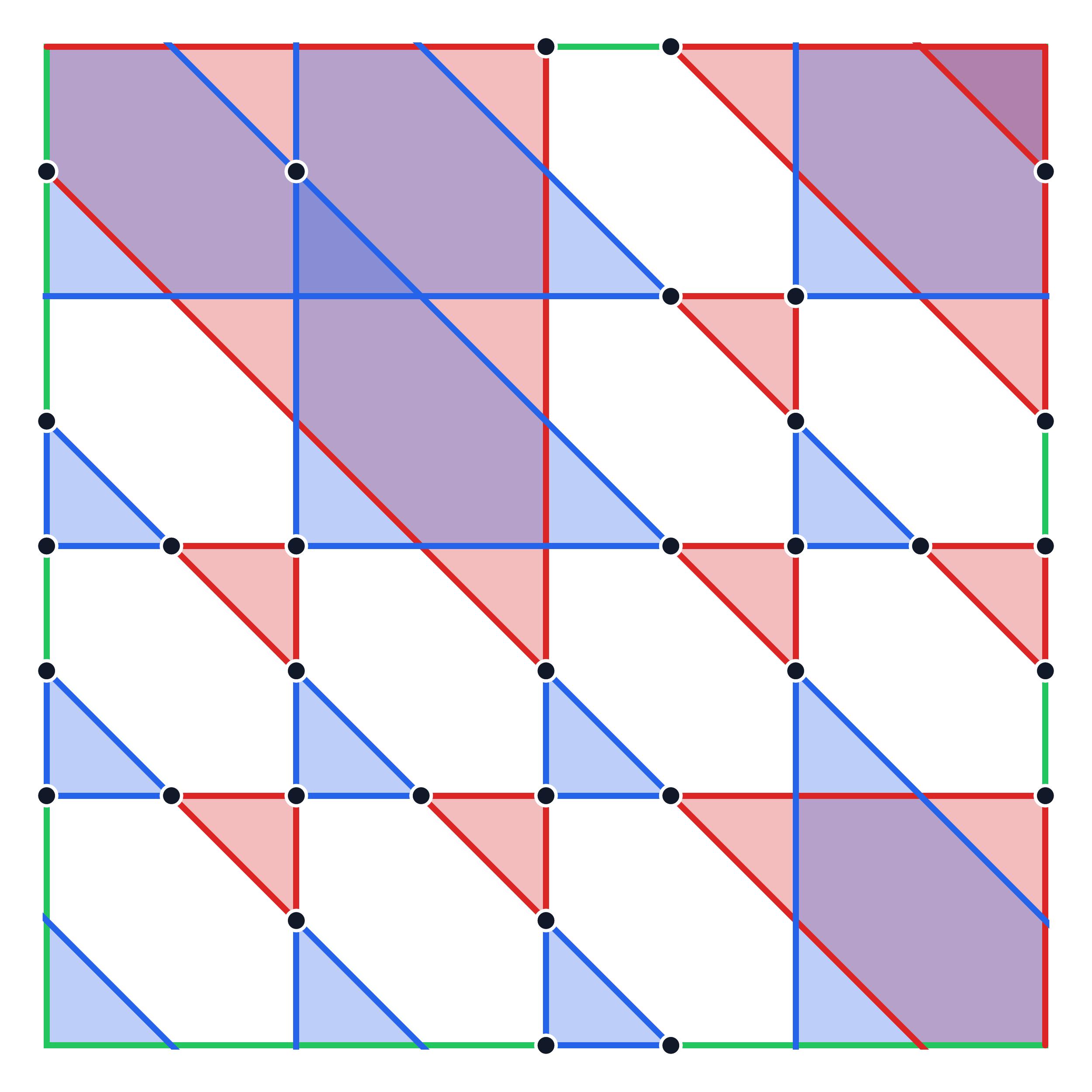}
\caption{32 in total}
\end{subfigure}
\begin{subfigure}[t]{.24\linewidth}
\centering
\includegraphics[width=\linewidth]{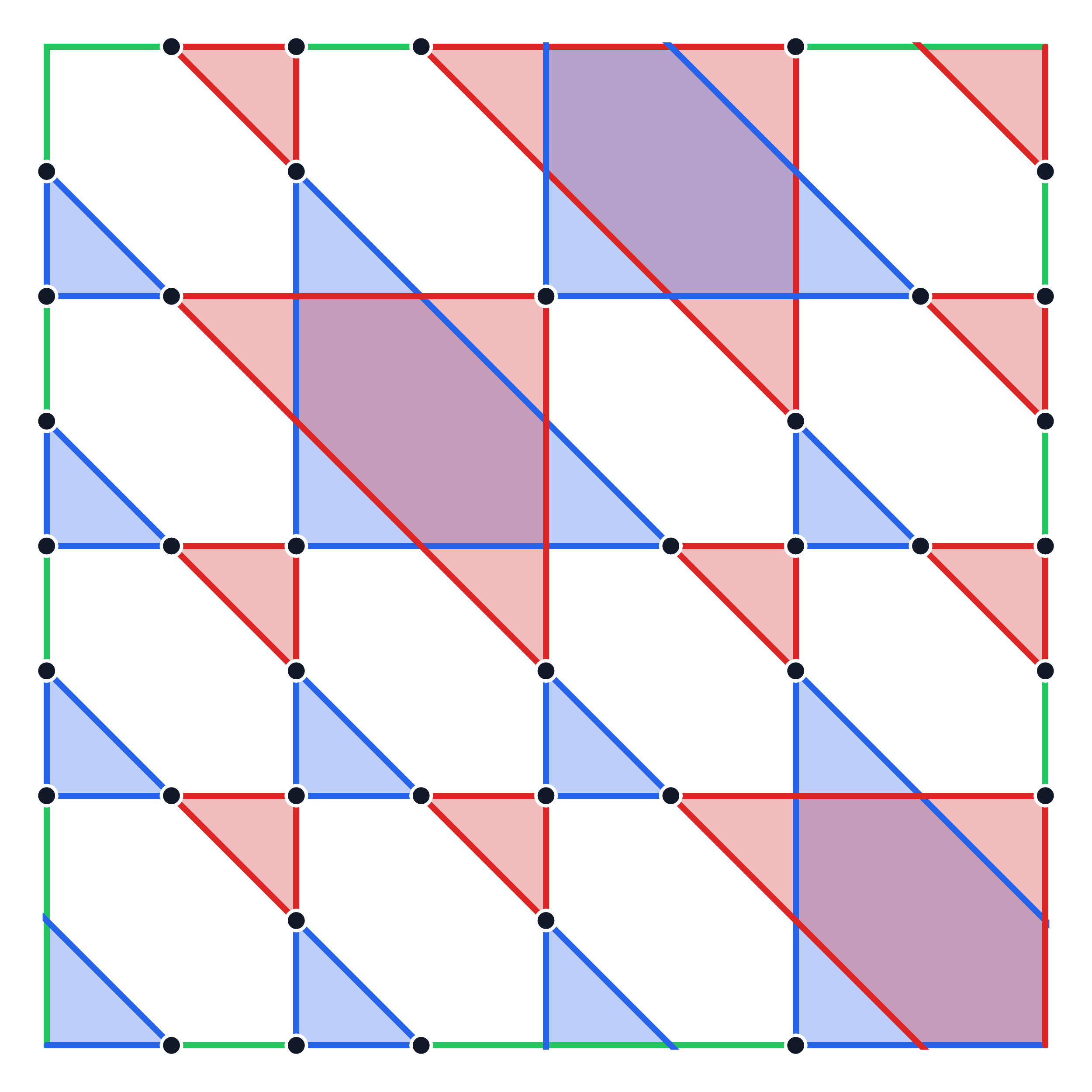}
\caption{32 in total}
\end{subfigure}
\begin{subfigure}[t]{.24\linewidth}
\centering
\includegraphics[width=\linewidth]{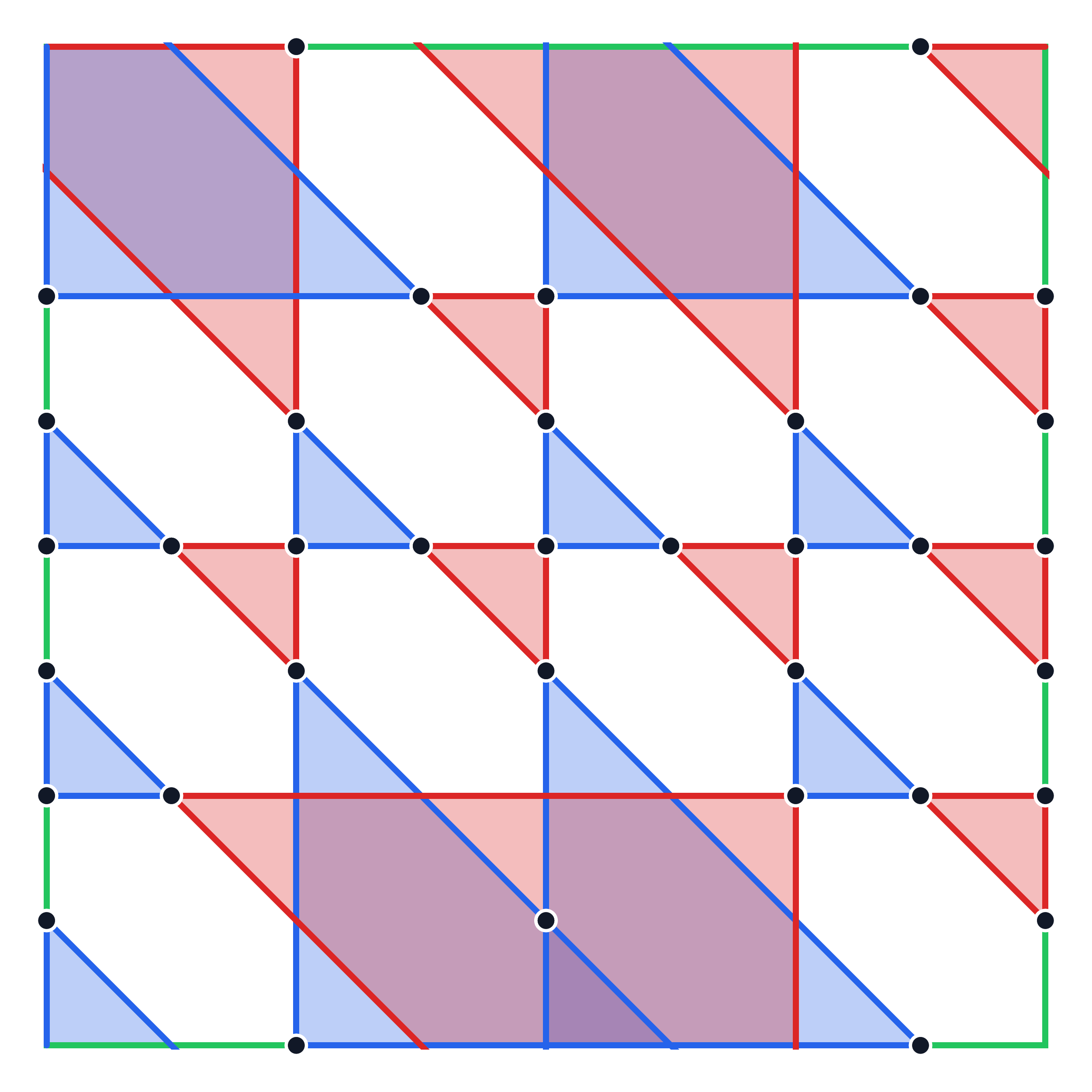}
\caption{16 in total}
\end{subfigure}
\begin{subfigure}[t]{.24\linewidth}
\centering
\includegraphics[width=\linewidth]{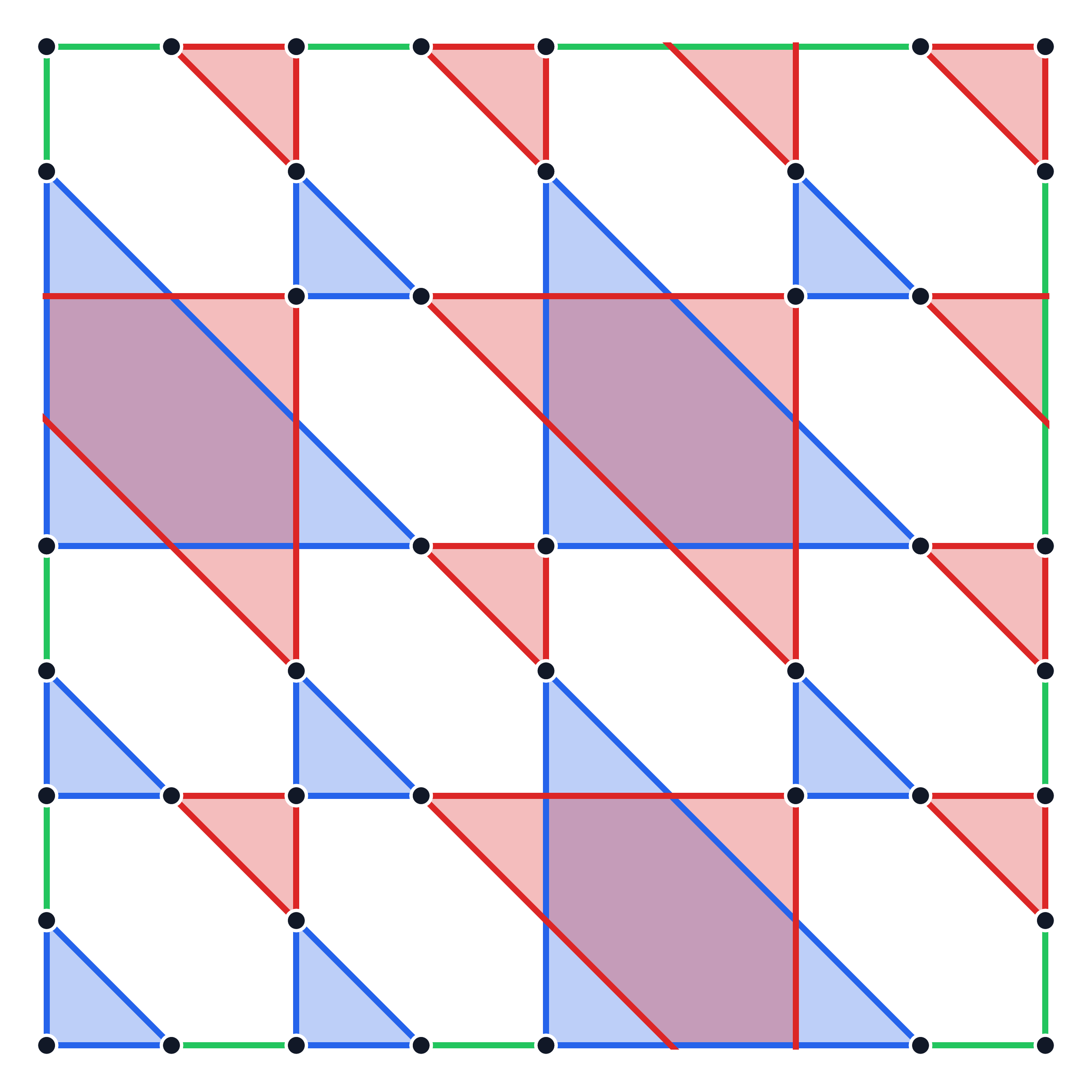}
\caption{16 in total}
\end{subfigure}
\begin{subfigure}[t]{.24\linewidth}
\centering
\includegraphics[width=\linewidth]{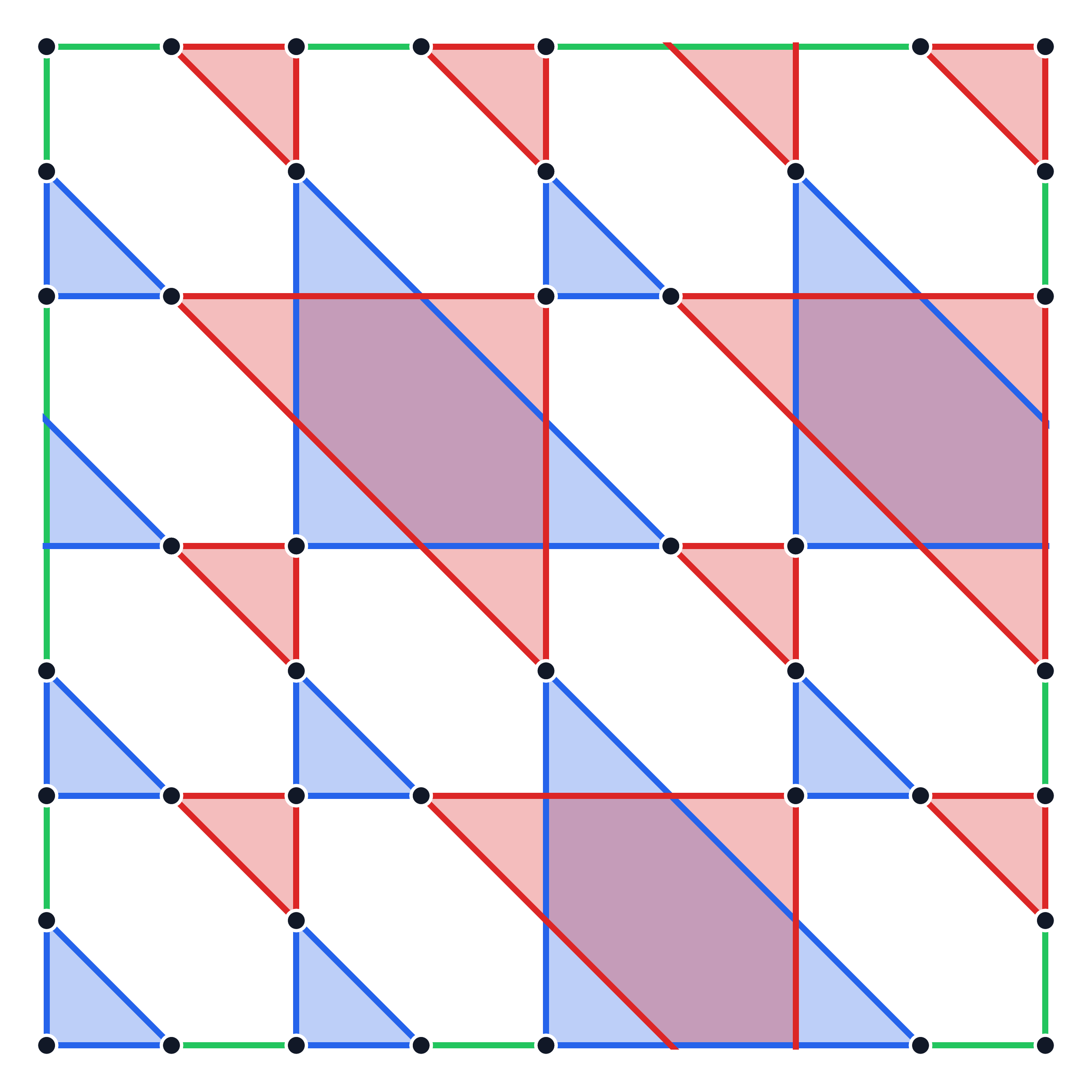}
\caption{16 in total}
\end{subfigure}
\begin{subfigure}[t]{.24\linewidth}
\centering
\includegraphics[width=\linewidth]{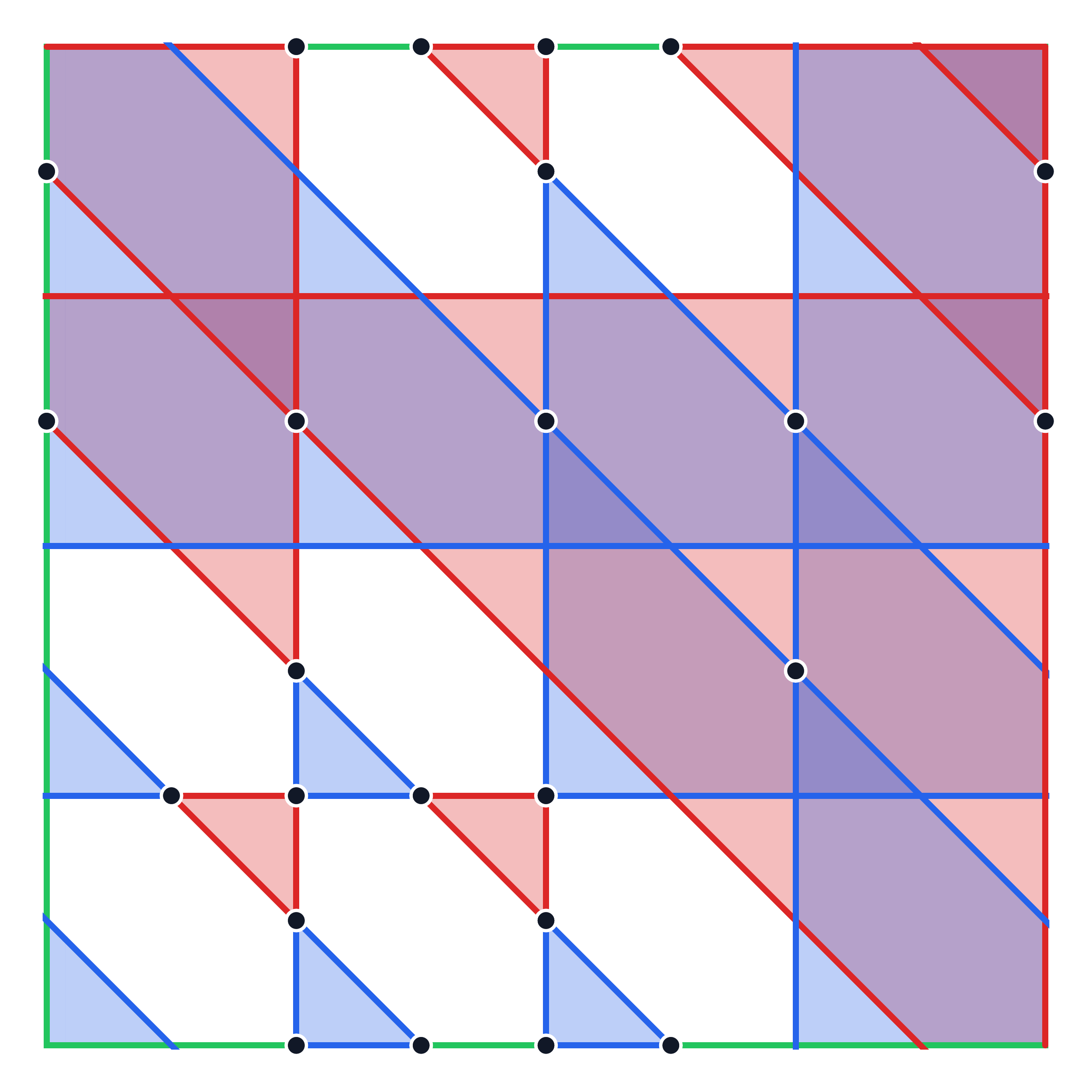}
\caption{16 in total}
\end{subfigure}

\caption{Rational rulings for $\mathcal O_{\mathbb P^2}(4)$. 304 in total.}
\label{fig:O4}

\end{figure}

\end{example}

\section{The Beauville Integrable System}
\label{sec:beauville}

In this paper, we consider a $g$-dimensional family of genus-$g$ curves $i:C\hookrightarrow \bP$ in a toric
surface $\bP$ along with their $g$-dimensional Jacobian tori. 
We have benefited from the perspective that this is an integrable system.  We prove this here, following
Beauville, Mukai and others --- see below.

Consider a linear system
of curves $i:C\hookrightarrow S$
on a complex surface, $S$.
A line bundle $L\to C$
defines a coherent sheaf $i_*L$.
The relative Jacobian thus defines a family of sheaves which fibers over the linear system of curves, with Jacobian fibers.
Mukai \cite{Mukai1984} showed that the moduli of sheaves on K3 carries an anti-symmetric two-form, and Tyurin generalized the construction to Poisson surfaces.  Bottacin  \cite{Bottacin1995} showed the
almost Poisson structure was integrable.
Beauville \cite{Beauville1991} showed that the family of
Jacobians formed an integrable
system on K3 with respect to
the Mukai-Bottacin symplectic structure.
Biswas-Gomez \cite{Biswas-Gomez} showed
that symplectic leaves for the Mukari-Bottacin Poisson structure for sheaves
of the form $i_*L$
on a Poisson surface defined by an 
anticanonical section, are obtained by fixing
the intersection of $C$ with the anticanonical divisor.
In \cite{GI} the authors
call this the Beauville integrable
system.  Beauville's proof was
just for the K3 case, so for completeness
and since it may be illustrative,
we include here a short proof that the
Jacobian fibers are isotropic for the Mukai form. 

The toric geometry is not
relevant here, so we consider
the general of $S$ a complex surface
with Poisson structure defined by
an anticanonical section $s\in \Gamma(S,K_S^{-1}).$
As above, let $i:C\hookrightarrow S$, let $L\to C$ be a line bundle, and call $\widetilde{\cM}$ the (Poisson) moduli space of sheaves $E = i_*L.$
We put $\cM\subset \widetilde{\cM}$ for the (symplectic) space 
of sheaves for which $C\cap s^{-1}(0)$
is fixed.

The pushforward map $i_*$ gives
$$H^1(C,\cO) = H^1(C,End(L)) \xrightarrow[]{\;\;i_*\;\;} H^1(S,End(E)) = T_E \widetilde{\cM}.$$
Now $T^*_E\widetilde{\cM}= Ext^1_S(E,E\otimes K_S),$
and the pairing with $T_E\widetilde{\cM}$ is given by Yoneda.
Then multiplication by $s$ gives the Poisson structure $\Pi$
on $\widetilde{\cM}$, a map
$$\Pi: T^*_E\widetilde{\cM}=Ext^1_S(E,E\otimes K) \xrightarrow[]{\;\;\cdot s\;\;} Ext^1_S(E,E) = T_E\widetilde{\cM}.$$

Recall the adjunction formula for our closed immersion, $i$:
$$Ext_S^k(i_*F,i_*G) \cong Ext^k_C(F,G)\oplus Ext^{k-1}_C(F,G\otimes N_{C/S}),$$ so when $k=1$ and $F = G = L$ we get
$T_E \widetilde{\cM} \cong H^1(C,\cO) \oplus H^0(C,N_{C/S}),$
and $i_*$ is inclusion in the first factor.
The second term represents deformations of $C$,
so for $T_E \cM$, this term is cut down to $H^0(C,K_C),$
as the space of sections vanishing on $C\cap s^{-1}(0).$  Note by Serre duality the
fiber directions are dual to the base, and
same in number.  We have a symplectic
structure on a $g$-dimensional
family of genus-$g$ curves.  It remains to show
the fibers are isotropic.

Now for $x\in T_E\cM$, the vanishing condition of $x$
along $s^{-1}(0)$ means $x/s$ has no poles,
and is therefore a $\Pi$ preimage 
in $T^*_E\cM$.  Put $\widetilde{x} = x/s
\in H^1(S,i_*\cO\otimes K^{-1}_S) =
H^1(S,i_*(K^{-1}_S\vert_C)).$
Now for $x,y\in H^1(C,\cO)\subset T_E\cM$ we want to evaluate
\[
\omega(x,y) = \omega^{-1}(\widetilde{x},\widetilde{y}) = \langle \widetilde{x},\Pi(\widetilde{y})\rangle = \int_S tr(\widetilde{x}y)
\]
where $\widetilde{x}\in H^1(S,i_*(K_S\vert_C))$
and $y\in H^1(S,i_*\cO_C)$.
But now when we take $\int_S tr(\widetilde x y)$
note the cup product $\widetilde{x}y$ lands in $H^2(S,i_*K_S\vert_C) \cong H^2(C,K_S\vert_C).$
But $C$ is a curve, so this vanishes.

We have proven the following.
\begin{proposition}
    $\cM = \{i_*(L\to C)\}$ is a complete integrable system.
\end{proposition}

\bibliographystyle{alpha}
\bibliography{refs}

\end{document}